\documentclass[11pt]{article}
\usepackage[utf8]{inputenc}
\usepackage{amssymb,amsmath,amsthm,graphicx,geometry,hyperref,xcolor,bm,mathtools,authblk,subfigure,tikz,multirow,tensor}
\usepackage[numbers]{natbib}
\usepackage{hyperref}
\usepackage{multicol}
\usepackage{multirow}
\usepackage{makecell}
\usepackage{booktabs}
\usepackage{arydshln}
\newgeometry{left=1in,right=1in,top=1in,bottom=1in}

\newcommand{\real}{\mathbb{R}}
\hypersetup{
    colorlinks=true,
    linkcolor=blue,
    filecolor = blue
, urlcolor = blue,
citecolor = blue
}

\newcounter{lst}

\newtheorem{cex}{Counterexample}
\newtheorem{conj}{Conjecture}
\newtheorem{cond}{Condition}
\newtheorem{thm}{Theorem}

\newtheorem{prop}{Proposition}
\newtheorem{cor}{Corollary}
\newtheorem{lemma}{Lemma}
\newtheorem{definition}{Definition}

\makeatletter
\newcommand*{\rom}[1]{\expandafter\@slowromancap\romannumeral #1@}
\makeatother

\newcommand{\norm}[1]{\left|\left|#1\right|\right|_2}

\newcommand{\parentheses}[1]{\left(#1\right)}
\newcommand{\tr}{\text{tr}}
\newcommand{\calN}{\mathcal{N}}

\newcommand{\argmin}{\text{argmin}}
\newcommand{\argmax}{\text{argmax}}

\newcommand{\bbP}{\mathbb{P}}
\newcommand{\calC}{\mathcal{C}}
\newcommand{\bbE}{\mathbb{E}}
\newcommand{\set}[1]{\left\{#1\right\}}
\newcommand{\calL}{\mathcal{L}}
\newcommand{\calS}{\mathcal{S}}
\newcommand{\calO}{\mathcal{O}}
\newcommand{\calH}{\mathcal{H}}
\newcommand{\calR}{\mathcal{R}}
\newcommand{\calP}{\mathcal{P}}

\newcommand{\bbN}{\mathbb{N}}
\newcommand{\bfu}{\mathbf{u}}
\newcommand{\bfv}{\mathbf{v}}
\newcommand{\bfw}{\mathbf{w}}

\newcommand{\bbS}{\mathbb{S}}
\newcommand{\bbH}{\mathbb{H}}
\begin{document}

\title{Testing Separability of High-Dimensional Covariance Matrices}
\author{Bongjung Sung} 
\author{Peter D. Hoff}
\affil{Department of Statistical Science, Duke University} 
\date{\today}

\maketitle 

\begin{abstract}
Due to their parsimony, separable covariance models have been popular in modeling matrix-variate data. However, the inference from such a model may be misleading if the population covariance matrix $\Sigma$ is actually non-separable, motivating the use of statistical tests of separability. The existing separability tests suffer mainly from two issues: 1) test statistics that are not well-defined in high-dimensional settings, 2) low power for small sample sizes and null distributions that depend on unknown parameters, preventing exact error rate control. To address these issues, we propose novel invariant tests using the core covariance matrix, a complementary notion to a separable covariance matrix. We show that testing separability of $\Sigma$ is equivalent to testing sphericity of its core component. With this insight, we construct test statistics that are well-defined in high-dimensional settings and have distributions that are invariant under the null hypothesis of separability, allowing for exact simulation of null distributions. We establish the asymptotic properties of some test statistics by proving the asymptotic spectral equivalence between the sample covariance matrix and its core in a $p/n\rightarrow\gamma\in(0,\infty)$ regime. The large power of our proposed tests relative to existing procedures is demonstrated numerically.
\end{abstract}


\section{Introduction}\label{sec1:intro}
Many modern statistical applications involve matrix-valued data, say $\set{Y_1,\ldots,Y_n }\subseteq \real^{p_1\times p_2}$, e.g., fMRI data \cite{lindquist2008}, microarray data \cite{allen2012}, phonetic data \cite{pigoli2014}, and financial time series \cite{horvath2014}. In many applications, estimating the population covariance matrix may be useful for a variety of statistical procedures, e.g., regression or classification, and can also be used to reveal underlying features of the population. While the sample covariance matrix may provide an unbiased estimate of the population covariance, it may be statistically unstable in high-dimensional settings where $n<p=p_1p_2$, without structural assumptions on the population covariance matrix. Due to their parsimony, separable covariance models have been widely used to model matrix-variate data \cite{drton2021,soloveychik2016,drton2024,linton2022,masak2023b}. The model assumes that the population covariance matrix is of the form $\Sigma=\Sigma_2\otimes \Sigma_1$. Here $\Sigma_1$ and $\Sigma_2$ are covariance matrices of dimensions $p_1\times p_1$ and $p_2\times p_2$, respectively, and $\otimes$ denotes the Kronecker product. While an unstructured covariance model has $O(p_1^2p_2^2)$ parameters, the separable covariance model has only $O(p_1^2+p_2^2)$, resulting in covariance estimates that are statistically stable in high-dimensional settings. It is also interpretable via the separate row and column covariance matrices $\Sigma_1$ and $\Sigma_2$. 

However, the appropriateness of the separability assumption is often questionable. The assumption tends to oversimplify the covariance structure as $p$ grows; ignoring significant dependencies between variables. For instance, criticism has been raised regarding the assumption in the spatio-temporal domain \cite{gneiting2002,gneiting2007}. This concern suggests that the assumption be evaluated before using it for estimation. To state the problem formally, denote the set of $p\times p$ positive definite matrices by $\calS_p^+$, and let $\calS_{p_1,p_2}^+:=\set{\Sigma_2\otimes \Sigma_1:\Sigma_1\in\calS_{p_1}^+,\,\Sigma_2\in\calS_{p_2}^+}$ be the set of separable covariance matrices for given values of $p_1$ and $p_2$ such that $p=p_1p_2$. We are interested in testing   
\begin{align}\label{sec1.eq1}
    H_0: \Sigma\in\calS_{p_1,p_2}^+\quad \text{ versus } \quad H_1:\Sigma\notin\calS_{p_1,p_2}^+
\end{align}
for a population covariance matrix $\Sigma$. While several covariance models have been proposed to generalize a separable covariance model, e.g., separable covariance expansion model \cite{tsiligkaridis2013,tsiligkaridis2013b,puchkin2024}, the separable covariance model may be still preferred in practice due to its simplicity. Nevertheless, its appropriateness should be empirically evaluated, such as with a statistical test (\ref{sec1.eq1}).  

Existing tests of (\ref{sec1.eq1}) include the classical likelihood ratio test (LRT) assuming the normality of the data \cite{mitchell2006,lu2005}. The distribution of the log-likelihood ratio statistic is invariant under the null hypothesis, and its asymptotic null distribution may be obtained from Wilks' phenomenon, which induces a natural rejection rule. However, $n$ should be sufficiently large to observe this phenomenon, which may not hold in high-dimensional settings where $n<p$. \cite{manceur2013b} proposed a modified LRT by correcting the bias in the original LRT. Nevertheless, their algorithm for bias correction may have a slow convergence rate, even when $p_1$ and $p_2$ are of moderate size. Furthermore, since the sample covariance matrix $S$ is not strictly positive definite if $n<p$, the LRT is not well-defined in high-dimensional settings when $n<p$. While \cite{simpson2014} modified the LRT to accommodate such settings, they considered repeated measurements where the dimension may vary by the subject, and additionally imposed the linear exponent autoregressive correlation structure on the separable covariance. Hence, their test is not applicable to general settings.  

On the other hand, nonparametric tests have been derived for a more general covariance operator, e.g., functional data \cite{constantinou2017,aston2017} and spatio-temporal data \cite{ghorbani2021}. Many of these tests are constructed based on the discrepancy between a sample covariance operator and the estimated separable covariance operator. For example, for functional data in a separable Hilbert space, the separability of covariance operator implies the separability of the eigenfunctions. Since the separable eigenfunctions can form a basis representation for the data, \cite{aston2017} proposed a test statistic by projecting the discrepancy between sample covariance operator and the estimated separable covariance operator onto a few eigenfunctions that encode most of the variability of the data. While the test is valid even when $n<p$, the tests may have low power as demonstrated in Section \ref{sec5:illus}. Alternatively, \cite{yu2023} derived a test based on linear spectral statistics to assess the null hypothesis regarding a specific component of a Kronecker-structured covariance matrix. For $\Sigma=\Sigma_2\otimes \Sigma_1$, they  considered testing whether $\Sigma_1=\Sigma_{1,0}$ $(\text{resp. }\Sigma_2=\Sigma_{2,0})$ for given $\Sigma_{1,0}\in\calS_{p_1}^+$ $(\text{resp. } \Sigma_{2,0}\in\calS_{p_2}^+)$. Under the null hypothesis of separability, their test statistic has a distribution independent of the unspecified component. The asymptotic null distribution and power analysis have been provided for their proposed test. Nonetheless, in the context of general separability tests, the size and power of the test with finite samples depend on the choice of $\Sigma\in\calS_{p_1,p_2}^+$. Thus, the desired invariance does not hold under the null hypothesis of separability. Also, their results are valid when $p/n\rightarrow 0$ as $n\rightarrow\infty$. 

In summary, existing separability test statistics are limited by three issues; 1) the null hypothesis of separability is a composite null, and the null distribution may depend on unknown parameters, 2) the test may not be defined in high-dimensional settings, and 3) the test may have low power. To remedy these issues, we introduce a new notion of invariance--Kronecker-invariance. We say a statistic is Kronecker-invariant if its distribution does not depend on $K=k(\Sigma)$, which we define formally in Section \ref{sec3:test}. We emphasize that this distributional freeness with respect to $K$ is \emph{exact}, i.e., the freeness holds for any fixed value of $(n,p)$. Hence, the Kronecker-invariance induces the size and power of the associated test independent of the unknown value of $K$, enabling exact error rate control.

Consequently, we propose three novel Kronecker-invariant tests that are well-defined in high-dimensional settings and have large power compared to alternatives. We develop such tests by leveraging the core covariance matrix \cite{hoff2023a}, a complementary notion to the separable covariance matrix. \cite{hoff2023a} showed that every covariance matrix can be decomposed into a core component and a separable (Kronecker) component, referred to as a Kronecker-core decomposition (KCD). A covariance matrix is separable if and only if its core component, denoted by $C$, is $I_p$. This implies that testing the separability of $\Sigma$ can be reframed as testing the sphericity of its core covariance $C$. Namely, testing (\ref{sec1.eq1}) is equivalent to testing 
\begin{align}\label{sec1.eq2}
    H_0:C=I_p \text{ versus } H_1:C\neq I_p.
\end{align}

In the light of (\ref{sec1.eq2}), we propose three novel test statistics whose distributions are invariant to the unknown value of $K$. These are based on either the eigenvalues of the sample core covariance matrix, $\hat{C}$, or its separable covariance expansion. As we show, since the empirical spectral distribution of $\hat{C}$ depends on $\Sigma$ only through $C$, the distribution of a test statistic based on the eigenvalues of $\hat{C}$ satisfies the desired Kronecker-invariance. On the other hand, we also consider the test based on the separable covariance expansion of $\hat{C}$ by employing the rearrangement operator. The separable covariance expansion embeds $\calS_{p_1,p_2}^+$ as a special case of separability rank $1$ \cite{puchkin2024,masak2023a}, and the separability rank can be examined through the rearrangement operator (\cite{puchkin2024}, Definition 2.1). As with the spectral distribution of $\hat{C}$, that of the rearrangement of $\hat{C}$ has a distribution that is independent of the true separable component $K$. Therefore, a test based on singular values of the rearrangement of $\hat{C}$ also satisfies Kronecker-invariance. All these tests are well-defined even in a high-dimensional regime where $n<p$.

In this article, we incorporate both practical and theoretical aspects of the tests, as is common in several works in the testing literature \cite{wang2013,wang2021,li2016}. In many practical applications, researchers implement a test by comparison of the test statistic to null distributions derived from a particular choice of $\calP_0$. However, if knowledge of an appropriate $\calP_0$ is not available, one may prefer an asymptotic test. Such a test provides a closed form of asymptotic null distribution that is valid for a broad range of data-generating distributions. In this case, the empirical size and power are evaluated against this asymptotic null distribution. 

We consider both approaches in this article. From a practical perspective, by the Kronecker-invariance and (\ref{sec1.eq2}), the null distribution does not depend on the unknown separable covariance in a non-asymptotic sense, and thus may be obtained via Monte Carlo simulation for a range of sampling distributions $\calP_0$. From a theoretical perspective, with $k(\Sigma)=I_p$ by virtue of the Kronecker-invariance, we provide the asymptotics for some Kronecker-invariant test statistics proposed in this article, without specifying $\calP_0$. Specifically, we prove the asymptotic distributions under the null hypothesis and some local alternative regime, consistency, and the first-order limit.  

These results are achieved by establishing the asymptotic spectral equivalence between the sample covariance matrix and its core when $p_i/\sqrt{n}\rightarrow\gamma_i\in(0,\infty)$ as $n$ grows for $i=1,2$ under some regularity conditions. Note that this spectral equivalence is implied by the convergence of the Kronecker MLE. Generalizing the results of \cite{oliveira2026}, we prove the almost sure convergence of the Kronecker MLE when $C$ exhibits a partial-isotropy rank$-r$ structure for Gaussian distributions in Section \ref{sec4.2:spec.equiv}. Motivated by this result, we derive the desired asymptotics in Section \ref{sec4.3:asymp.null}--\ref{sec4.4:consist}.

The rest of the article is organized as follows. In Section \ref{sec2:prelim}, we introduce notation and review the KCD and separable covariance expansion. We discuss how a Kronecker-invariant statistic can be constructed using the core of the sample covariance matrices. In Section \ref{sec3:test}, we formally define the notion of Kronecker-invariance and propose three novel Kronecker-invariant tests based on either the eigenvalues of the sample core covariance matrix or its separable covariance expansion. These tests are motivated by simple characterizations of a separable covariance matrix in terms of the core covariance matrix and the separable expansion, inducing the desired invariance. The aforementioned theoretical guarantees related to the proposed tests are provided in Section \ref{sec4:theory}. All the proofs of the theoretical results in this article are deferred to Appendix. We provide simulated null distributions and demonstrate the large power of the proposed tests through Monte Carlo simulations in Section \ref{sec5:illus}. The article concludes with a discussion in Section \ref{sec6:disc}. 

\section{Preliminaries}\label{sec2:prelim}
\subsection{Notation}
We denote the eigenvalues of $\Sigma\in\calS_p^+$ by $\lambda_1(\Sigma)\geq \cdots \geq \lambda_p(\Sigma)>0$. For a general matrix $M\in\real^{u\times v}$, denote the singular values of $M$ by $\sigma_1(M)\geq \cdots\geq  \sigma_{\min\set{u,v}}(M)\geq 0$. Let $\calL_{p}^+$ be the set of all $p\times p$ lower-triangular matrices with positive diagonal entries, and $\calL_{p_1,p_2}^+=\set{L_2\otimes L_1:L_1\in\calL_{p_1}^+,L_2\in\calL_{p_2}^+}$. Likewise, we write $\calO_{p}$ and $GL_p$ to denote orthogonal and general linear groups of order $p$, respectively, and define $\calO_{p_1,p_2}$ and $GL_{p_1,p_2}$ similarly. For $M\in\real^{u\times v}$, let $||M||_F=\sqrt{\tr\parentheses{M^\top M}}$ and $||M||_2=\sup_{x:||x||_2=1} ||Mx||_2$. Note that $||M||_2\leq ||M||_F\leq\sqrt{\max\set{p,q}}||M||_2$. For matrices $B_1,\ldots,B_m$ of possibly different dimensions, we denote their direct sum by $\bigoplus_{i=1}^m B_i=\text{diag}(B_1,\ldots,B_m)$. For two nonnegative sequences $\set{a_n}_{n=1}^\infty$ and $\set{b_n}_{n=1}^\infty$, if both $a_n/b_n=O(1)$ and $b_n/a_n=O(1)$ hold, write $a_n\asymp b_n$. 

\subsection{Kronecker-Core Decomposition and Kronecker-Invariant Statistics}\label{sec2.2:KCD}
We review the Kronecker-core decomposition (KCD) introduced by \cite{hoff2023a} from which we derive a statistic whose distribution is invariant to bilinear transformations of the data, and thus does not depend on the separable (Kronecker) component of the population covariance matrix. We refer to this property of a statistic as \emph{Kronecker-invariance}. The KCD parameterizes $\calS_p^+$ in terms of the separable component $K$ and the core component $C$. That is, every $\Sigma\in\calS_p^+$ can be written as $\Sigma=K^{1/2}CK^{1/2,\top}$, where $K^{1/2}$ is any square root of $K\in\calS_{p_1,p_2}^+$. Note that if $C=I_p$, then $\Sigma=K$, suggesting that $\calS_{p_1,p_2}^+$ can be embedded into $\calS_p^+$ via KCD. Indeed, any element of $\calS_{p_1,p_2}^{+}$ is positive definite by the properties of Kronecker product \cite{neudecker1969,loan2000}. To see the embedding more clearly, we formally define the separable component and the core component of $\Sigma$. For given $\Sigma$, define a function $d:\calS_{p_1,p_2}^+\mapsto \real$ by 
\begin{align}\label{sec2.2.eq1}
    d(\Omega;\Sigma)=\tr\parentheses{\Omega^{-1}\Sigma}+\log|\Omega|.
\end{align}
Then the separable component of $\Sigma$, $K=\Sigma_2\otimes \Sigma_1$, is defined to be a minimizer of $d$ over $\Omega$, i.e., $K:=k(\Sigma)=\argmin_{\Omega=\Omega_2\otimes \Omega_1\in \calS_{p_1,p_2}^+}d(\Omega;\Sigma)$. Since the function $d$ is a Kullback-Leibler (KL) divergence between $N_{p_1\times p_2}(0,\Omega)$ and $N_{p_1\times p_2}(0,\Sigma)$, $K$ is equivalent to a Kronecker maximum likelihood estimator (MLE) based on $\Sigma$, representing the most separable component of $\Sigma$. Note that $K$ uniquely exists (see Theorem $3$ of \cite{ros2016}) and so the map $k:\calS_p^+\mapsto \calS_{p_1,p_2}^+$ is well-defined. We shall call the map $k$ as a Kronecker map.  

The core component of $\Sigma$ is then defined by whitening $\Sigma$ through the identifiable square root of its separable component $K$. Namely, suppose $\calH\subseteq GL_{p_1,p_2}$ such that the function $s:\calH\mapsto \calS_{p_1,p_2}^+$, defined by $s(H)=HH^\top$, is a bijection, e.g., $\calL_{p_1,p_2}^+$ and $\calS_{p_1,p_2}^+$. If $H\in\calH$ is a square root of $k(\Sigma)$, then the core component of $\Sigma$ is defined as  $c(\Sigma):=H^{-1}\Sigma H^{-\top}$. Let $\calC_{p_1,p_2}^+=\set{H^{-1}\Sigma H^{-\top}:k(\Sigma)=HH^\top,H\in\calH}$. Since $c(\Sigma)$ is also unique as the square root of $k(\Sigma)$ is so within $\calH$, the map $c:\calS_p^+\mapsto \calC_{p_1,p_2}^+$ is well-defined and referred to as a core map. Therefore, every $\Sigma$ has a unique and identifiable KCD as $K^{1/2}CK^{1/2,\top}$ for some $(K,C)\in\calS_{p_1,p_2}^+\times \calC_{p_1,p_2}^+$ with a suitable choice of $\calH$. In the rest of the article, we shall fix $\calH$ as either $\calS_{p_1,p_2}^+$ or $\calL_{p_1,p_2}^+$ for the identifiable and unique KCD. We summarize the definitions of the separable component and the core component below.  
\begin{definition}\label{sec2.2.def1}
    Suppose $k:\calS_{p}^+\mapsto \calS_{p_1,p_2}^+$ is a Kronecker map defined by  $k(\Sigma)=\argmin_{K\in\calS_{p_1,p_2}^+}d(K;\Sigma)$. Also, let $\calH\subseteq GL_{p_1,p_2} $ be fixed as either $\calS_{p_1,p_2}^{++}$ or $\calL_{p_1,p_2}^{++}$. Then the core map $c:\calS_p^+ \mapsto \calC_{p_1,p_2}^+$ is defined as $c(\Sigma)=H^{-1}\Sigma H^{-\top}$ for unique $H\in \calH$ satisfying $k(\Sigma)=HH^\top$. For $\Sigma \in \calS_p^+$, we refer to $k(\Sigma)$ and $c(\Sigma)$ as the separable (Kronecker) and core components of $\Sigma$, respectively.    
\end{definition}

As a remark, $\Sigma$ is separable if and only if $k(\Sigma)=\Sigma$, or equivalently, $c(\Sigma)=I_p$. Thus, $\calS_{p_1,p_2}^{+}$ is embedded in $\calS_p^+$ as $\calS_{p_1,p_2}^{+}\equiv c^{-1}(\set{I_p})$. Also, for any $G\in GL_{p_1,p_2}$, $k(G\Sigma G^\top)=Gk(\Sigma) G^\top$ (see Proposition $2$ of \cite{hoff2023a}). That is, the Kronecker map $k$ is equivariant with respect to $GL_{p_1,p_2}$. This equivariance implies that the separable component of the core is always $I_p$. Specifically, if $H\in\calH\subset GL_{p_1,p_2}$ is an identifiable square root of $K=k(\Sigma)$ so that the induced core is $C=c(\Sigma)=H^{-1}\Sigma H^{-\top}$, it holds that
\begin{align}\label{sec2.2.eq2}
k(C)=k(H^{-1}\Sigma H^{-\top})=H^{-1}KH^{-\top}=H^{-1}HH^\top H^{-\top}=I_p.   
\end{align}
Thus, this implies that $\calC_{p_1,p_2}^+$ can be equivalently defined as $\set{C\in\calS_p^+:k(C)=I_p}$. Also, note that the KCD may exist for a positive semidefinite matrix, possibly not strictly definite. For instance, if $S$ is a sample covariance matrix of $n$ random samples and $n> p_1/p_2+p_2/p_1$, the minimizer of $d(\Omega;S)$ in $\Omega$ uniquely exists \cite{drton2021,drton2024,derksen2021}, and hence the KCD of $S$ is uniquely defined, where the above equivariance for the separable component and the property of the core component as in (\ref{sec2.2.eq2}) still hold. Thus, we assume $n>p_1/p_2+p_2/p_1$ throughout the paper to ensure the existence of KCD for the sample covariance matrix.

By the definitions of $\calS_{p_1,p_2}^+$ and $\calC_{p_1,p_2}^+$, $\calC_{p_1,p_2}^+\cap \calS_p^+=\set{I_p}$. Thus, testing the null hypothesis of separability is equivalent to testing the null hypothesis that $C=I_p$. Since this null hypothesis is represented in terms of only $c(\Sigma)$ for the population covariance matrix $\Sigma$, this may motivate using $\hat{C}=c(S)$ to test the separability of $\Sigma$. 

\subsection{Separable Covariance Expansion}\label{sec2.3:sep.cov.exp}

Another parameterization of $\mathcal S_p^+$ is provided by the separable 
covariance expansion 
\cite{puchkin2024,loan2000,masak2023a,tsiligkaridis2013}. In this parameterization, the covariance matrix $\Sigma$ is expressed as a finite sum of linearly independent separable matrices, not necessarily positive definite. Namely, 
\begin{align}\label{sec2.3.eq1}
    \Sigma=\sum_{i=1}^R A_i\otimes B_i 
\end{align}
for $1\leq R\leq p_1^2\wedge p_2^2$, and linearly independent sequences of matrices $\set{A_i}_{i=1}^R\subseteq \real^{p_2\times p_2}$ and $\set{B_i}_{i=1}^R\subseteq \real^{p_1\times p_1}$. Note that any symmetric $\Sigma$, not necessarily positive definite, admits the expansion in (\ref{sec2.3.eq1}) \cite{puchkin2024,loan2000,masak2023a,tsiligkaridis2013}. On the matrix domain, \cite{loan2000} showed this is possible with an orthogonal expansion of Kronecker products. In general, (\ref{sec2.3.eq1}) holds due to the following facts (Section $3.1$ of \cite{masak2023a}). Suppose $\calS_2(\calH)$ is the set of Hilbert-Schmidt operators on a separable Hilbert space $\calH$. For any two separable Hilbert spaces $\calH_1$ and $\calH_2$, it holds that $\calS_2(\calH_1\otimes \calH_2)\cong \calS_2(\calH_1)\otimes \calS_2(\calH_2)$ under some isomorphism. Then (\ref{sec2.3.eq1}) follows by identifying the SVD of the operator $\mathcal{T}\in\calS_2(\calH_1\otimes \calH_2)$ in $ \calS_2(\calH_1)\otimes \calS_2(\calH_2)$. Note that $\Sigma\in\calS_{p_1,p_2}^+$ if and only if $R=1$. We refer to the value of $R$ as a separability rank of $\Sigma$. The separability rank $R$ can be examined through the rearrangement operator, defined as follows: 
\begin{definition}\label{sec2.3.def1}
    For $M\in\real^{p_1p_2\times p_1p_2}$, partition $M$ as 
    \begin{align}\label{sec2.3.def1.eq1}
        M=\left[\begin{array}{cccc}
         M_{[1,1]} & M_{[1,2]} & \cdots & M_{[1,p_2]} \\
              M_{[2,1]}   &M_{[2,2]} & \cdots &  M_{[2,p_2]}\\
            \vdots & \vdots & \ddots & \vdots \\
          M_{[p_2,1]}   &M_{[p_2,2]} & \cdots &  M_{[p_2,p_2]}
        \end{array}\right],
    \end{align}
    where $M_{[i,j]}\in\real^{p_1\times p_1}$ for $i,j=1,\ldots,p_2$. Then the rearrangement operator $\calR:\real^{p_1p_2\times p_1p_2}\mapsto \real^{p_2^2\times p_1^2}$ is defined by $R(M)\in \real^{p_2^2\times p_1^2}$ whose $((j-1)p_2+i)-$th row is $\text{vec}(M_{[i,j]})^\top$.  
\end{definition}
By Proposition 2.2 of  \cite{puchkin2024}, the linear map $\calR$ is a bijection. Also, it holds that 
\begin{align*}
    \calR\parentheses{\sum_{j=1}^R A_j\otimes B_j}=\sum_{j=1}^R \text{vec}(A_j)\text{vec}(B_j)^\top 
\end{align*}
for $A_j\in\real^{p_2\times p_2}$ and $B_j\in\real^{p_1\times p_1}$. The above implies that if two sequences $\set{A_j}_{j=1}^R$ and $\set{B_j}_{j=1}^R$ are linearly independent, the algebraic rank of $ \calR\parentheses{\sum_{j=1}^R A_j\otimes B_j}$ is $R$. Namely, $R=\text{rank}(\calR(\Sigma))$. Therefore, the test of separability (\ref{sec1.eq1}) can be reformulated as  
\begin{align}\label{sec2.3.eq2}
    H_0:\text{rank}(\calR(\Sigma))=1 \text{ versus } H_1:\text{rank}(\calR(\Sigma))>1. 
\end{align}
Recall that the algebraic rank of a generic matrix is equivalent to its number of nonzero singular values. Thus, if we simply write $\sigma_i:=\sigma_i(\calR(\Sigma))$ for $i=1,\ldots,p_1^2\wedge p_2^2$, one can observe that $\Sigma$ has a separability rank $1$ if and only if $\sigma_1>0,\sigma_2=\cdots=\sigma_{p_1^2\wedge p_2^2}=0$. Equivalently, $\Sigma$ is separable if and only if  
\begin{align}\label{sec2.3.eq3}
\sigma_1>0,\quad  \frac{\sigma_2^2+\cdots+\sigma_{p_1^2\wedge p_2^2}^2}{\sigma_1^2}=0.
\end{align}

\section{Kronecker-Invariant Tests of Separability}\label{sec3:test}
In this section, we introduce three tests of separability based on test statistics whose distributions under the null hypothesis do not depend on any unknown parameters. To achieve this, we show in Proposition \ref{sec3.prop1}, that under mild conditions, the empirical spectral distribution (ESD) of the sample core does not depend on any unknown parameters under the null hypothesis. Here the ESD of a given $p\times p$ positive semidefinite matrix $\Omega$, denoted $F_{\Omega}$, is defined as 
\begin{align*}
    F_{\Omega}(x):=\frac{1}{p}\sum_{i=1}^p \mathbf{1}_{\lambda_i(\Omega)}(x).
\end{align*}
Therefore, any test statistic based on the eigenvalues of the sample core satisfies the desired property. Such a test statistic is desirable as it permits the size of the test to be set exactly and non-asymptotically. In contrast, a test based on a statistic lacking this property will either only have approximate error rate control or have very low power, as a result of maintaining error control over all distributions in the composite null hypothesis. To the best of our knowledge, the only existing separability test whose test statistic satisfies such a property is the LRT, which is not well-defined if $n<p$.

To construct a test statistic with the desired property, we first introduce the following model for the observed data $Y_1,\ldots,Y_n\in\real^{p_1\times p_2}$: letting $K=k(\Sigma)$ and $C=c(\Sigma)$ for population covariance matrix $\Sigma$,
\begin{align}\label{sec3.eq1}
    y_i:=\text{vec}(Y_i)\overset{d}{\equiv} K^{1/2}(UDV^\top)z_i,
\end{align}
where $z_1,\ldots,z_n\overset{i.i.d.}{\sim}\calP_0$ for some known distribution $\calP_0$ satisfying $\bbE[z_1]=0$ and $V[z_1]=I_p$. Note that we assume the zero mean following a standard approach in testing literature \cite{ding2020,wang2013,wang2013} and random matrix theories \cite{bao2015,chen2012,lee2016}. Here $K^{1/2}$ is either symmetric square root ($\calS_{p_1,p_2}^+$) or Cholesky factor ($\calL_{p_1,p_2}^+$) for unique and identifiable KCD of $\Sigma$ as in Section \ref{sec2.2:KCD}. Also, $U,V\in\calO_p$ and $D$ is a diagonal matrix with positive diagonal entries such that $C\equiv UD^2U^\top$. Hence, $UDV^\top$ represents a square root of $C$. Note that any square root $\Sigma^{1/2}$ of $\Sigma$ is written as $K^{1/2}(UDV^\top)$ for some $V\in\calO_p$. 

The model (\ref{sec3.eq1}) is assumed for the results in the remainder of this article. Note that we consider the known distribution $\calP_0$ to account for practical aspects of the test as discussed in Section \ref{sec1:intro}. Nevertheless, we will also incorporate theoretical aspects by establishing the asymptotics, including the asymptotic distribution, consistency, and the first-order limit, in Section \ref{sec4.3:asymp.null}--\ref{sec4.4:consist} for some of the Kronecker-invariant statistics proposed in this article, without specifying $\calP_0$.

Now recall that the null parameter space of (\ref{sec1.eq1}) is the set of separable, or Kronecker-structured, covariance matrices. We shall refer to statistics with constant distributions across this null model as Kronecker-invariant statistics:
\begin{definition}\label{def3}
For given $\Sigma\in\calS_p^+$, suppose random matrices $Y_1,\ldots,Y_n\in \real^{p_1\times p_2}$ are generated according to the model (\ref{sec3.eq1}). Let $T(Y)$ be a statistic based on $Y_1,\ldots,Y_n$. Parameterizing the distribution of $T(Y)$ by $K:=k(\Sigma)$ and $C:=c(\Sigma$), we say $T(Y)$ is Kronecker-invariant if for any Borel-measurable $A\subseteq \real$ and $K'\in\calS_{p_1,p_2}^+$,
    \begin{align*}
        \bbP(T(Y)\in A|K,C)=\bbP(T(Y)\in A|K',C).
    \end{align*}
 \end{definition}
 
Indeed, the following proposition implies that under mild assumptions, the spectral distribution of $\hat{C}$ does not depend on $k(\Sigma)$ regardless of the choice of $\calH$. Therefore, any test statistic based on the eigenvalues of $\hat{C}$ satisfies the Kronecker-invariance. 

\begin{prop}\label{sec3.prop1}
For a given $\Sigma\in\calS_p^+$, suppose random matrices $Y_1,\ldots,Y_n\in \real^{p_1\times p_2}$ are generated according to the model (\ref{sec3.eq1}). Let $S$ be the sample covariance matrix $S$ of $Y_1,\ldots,Y_n$, i.e., $S=1/n\sum_{i=1}^ny_iy_i^\top$, and write $\hat{C}=c(S)$, where the square root function of $\hat{K}=k(S)$ that defines $\hat{C}$ may differ from that of $K=k(\Sigma)$ defining $C=c(\Sigma)$. Assume one of the following conditions:
\begin{itemize}
    \item[](\rom{1}) The choice of $K^{1/2}$ and $UDV^\top$ in the model (\ref{sec3.eq1}) is fixed. 
    \item[](\rom{2}) The distribution of $z_1$ is orthogonally invariant, i.e., $z_1\overset{d}{\equiv} Oz_1$ for any $O\in\calO_p$.  
\end{itemize}
Then the empirical spectral distribution of $\hat{C}$, $F_{\hat{C}}$, does not depend on $K$ if either condition (\rom{1}) or (\rom{2}) is met. 
\end{prop}

As a remark, we emphasize that the above distributional freeness with respect to $K$ is \emph{exact}; that is, the freeness holds non-asymptotically. Also, from the proof of Proposition \ref{sec3.prop1}, one can deduce that the equivariance of the separable component with respect to $GL_{p_1,p_2}$ in Section \ref{sec2.2:KCD} induces the Kronecker-invariance. In the remainder of the article, we shall assume (\rom{1}) of Proposition \ref{sec3.prop1}. Using the result of Proposition \ref{sec3.prop1}, we introduce the test statistics based on the eigenvalues of the sample core covariance matrix that satisfy the Kronecker-invariance property in the next three subsections. Consequently, the level of tests based on these statistics may be set exactly and uniformly across the null hypothesis of separability.

\subsection{Test Based on The Largest Eigenvalue of The Sample Core}\label{sec3.1}
Our first test statistic is based on the following result:
\begin{lemma}\label{sec3.1.lemma1}
   For $\Sigma\in\calS_p^+$, let $C=c(\Sigma)$. Then $\lambda_1(C)\geq 1$ with equality if and only if $\Sigma$ is separable. 
\end{lemma}
We shall introduce the partial trace operators \cite{zhang2011,zhang2012} to verify the above result.
\begin{definition}\label{sec3.1.def1}
 Let $\calS_q$ be the set of $q\times q$ symmetric matrices. For $M\in \calS_{p}$, let $(M_{[i,j]})$ be a partition of $M$ as in (\ref{sec2.3.def1.eq1}). Then the partial trace operators $\tr_1$ and $\tr_2$ are defined as 
 \begin{align*}
     \tr_1&:M\in \calS_{p_1p_2} \rightarrow \sum_{i=1}^{p_2} M_{[i,i]}\in \calS_{p_1},\\
     \tr_2&:M\in \calS_{p_1p_2}\rightarrow N=(n_{ij})\in \calS_{p_2} \text{ for } n_{ij}=\tr(M_{[i,j]}).
 \end{align*}
\end{definition}
Using the partial trace operators, we introduce an alternative characterization of the core covariance matrices, which is equivalent to Proposition $3$ of \cite{hoff2023a}, from which Lemma \ref{sec3.1.lemma1} follows. 
\begin{prop}\label{sec3.1.prop1}
Let $C$ be a $p\times p$ positive semidefinite matrix. Then $C$ is a core covariance matrix if and only if 
\begin{align*}
    \tr_1(C)=p_2I_{p_1}, \quad \tr_2(C)=p_1 I_{p_2}.
\end{align*}
\end{prop}
Since the traces of all partial traces of a given matrix are the same as its trace, we consequently have that $\tr(C)=p$ for any core $C$. Because the maximum is always at least the mean, this implies that $\lambda_1(C)\geq \tr(C)/p=\sum_{i=1}^p \lambda_i(C)/p=1$ for any $C\in \calC_{p_1,p_2}^+$. If $\lambda_1(C)=1$ and $\lambda_r(C)<1$ for at least one $r=2,\ldots,p$, then $\tr(C)<p$, leading to the contradiction. Thus, whenever $\lambda_1(C)=1$, $\lambda_i(C)=1$ for all $i$ so that $C=I_p$, i.e., $\Sigma$ is separable, proving Lemma \ref{sec3.1.lemma1}. 

Using the result of Lemma \ref{sec3.1.lemma1}, we propose a test that rejects the null hypothesis of separability for large values of $T_1(Y)=\lambda_1(\hat{C})$, where $\hat{C}=c(S)$, and $S$ is the sample covariance matrix of $Y_1,\ldots,Y_n$ sampled according to the model (\ref{sec3.eq1}). Specifically, we propose a test $\phi_1(Y)=I(T_1(Y)>q_{1-\alpha})$, where $q_{1-\alpha}$ is the $1-\alpha$ quantile of $T_1(Y)$ under the null hypothesis of separability for a given level $\alpha\in(0,1)$. 

As a consequence of Proposition \ref{sec3.prop1}, the distribution of $T_1(Y)$ does not depend on the unknown parameter $K:=k(\Sigma)$, and thus $T_1(Y)$ is Kronecker-invariant. As such, the size and power of the test $\phi_1$ also do not depend on $K$. For the asymptotic size of the test, we derive the asymptotic distribution of a suitably transformed $T_1(Y)$ under some regularity conditions in Section \ref{sec4.3:asymp.null}, maintaining Kronecker-invariance non-asymptotically. We also study the consistency of (a version of) $\phi_1$ under some regime of the population cores. That is, letting  $\beta_n(\phi_1):=\bbP(\phi_1(Y)=1|I,C)$ be the power of $\phi_1$, we show that $\beta_n(\phi_1)\rightarrow 1$ as $n\rightarrow\infty$. When obtaining the value of $q_{1-\alpha}$ or studying the power, we can assume $K=I_p$ without loss of generality, due to the Kronecker-invariance.

While the asymptotic null distribution may give some reasonable approximation of the value of $q_{1-\alpha}$ for very large sample sizes, such approximation may not be accurate for moderate sample sizes, as shown in Section \ref{sec5.1:null.simul}. Thus, with moderate sample sizes, the value of $q_{1-\alpha}$ may instead be approximated to an arbitrary degree of accuracy via Monte Carlo simulation, as long as the distribution $\calP_0$ can be specified. For example, under the null hypothesis that $Y_1,\ldots,Y_n\overset{i.i.d.}{\sim}N_{p_1\times p_2}(0,\Sigma_2\otimes \Sigma_1)$, the value of $q_{1-\alpha}$ can be approximated according to the following Monte Carlo simulation scheme. For $j=1,\ldots,J$, 
\begin{align}\label{mc.simul}
    \begin{split}
        & \textbf{Step 1) } \text{Simulate } Y_1^j,\ldots,Y_n^j\overset{i.i.d.}{\sim} N_{p_1\times p_2}(0,I_p).\\
        & \textbf{Step 2) } \text{Compute the sample covariance matrix } S_j \text{ of } Y_1^j,\ldots,Y_n^j, \text{ and thus } \hat{C}_j=c(S_j).\\
        & \textbf{Step 3) } \text{Compute } \ell_j=\lambda_1(\hat{C}_j).
    \end{split}
\end{align}

Then the Monte Carlo approximation of $q_{1-\alpha}$, denoted by $\hat{q}_{1-\alpha}$, is given by the $1-\alpha$ sample quantile of $\ell_1,\ldots,\ell_J$. Note that one can assume $\Sigma_2\otimes \Sigma_1=I_p$ in the above Monte Carlo simulation, as $T_1(Y)$ is Kronecker-invariant.

\subsection{Generic Sphericity Tests}\label{sec3.2}
The population covariance matrix $\Sigma$ is separable if and only if its core component $C:=c(\Sigma)$ is the identity matrix $I_p$. This suggests using a generic sphericity test of $C$ to assess the separability of $\Sigma$. In general, these test the null hypothesis that $\Psi=\sigma^2 I_p$ for some unknown $\sigma^2>0$ using a test statistic $f(S)$ based on the sample covariance matrix $S$. Because the only $p\times p$ spherical core covariance matrix is the identity $I_p$, by replacing $\Psi$ and $S$ with $C$ and $\hat{C}$, respectively, the test statistic $U(Y):=f(\hat{C})$ can be used to test whether $C=I_p$. However, the Kronecker-invariance of the test statistic $U(Y)$ is also desired so that exact error rate control can be maintained regardless of the values of $k(\Sigma)$. By Proposition \ref{sec3.prop1}, the Kronecker-invariance is guaranteed if $f$ is a spectral statistic, i.e., $f(\hat{C})$ depends on $\hat{C}$ only through its eigenvalues. Indeed, spectral statistics are often used in the more general problem of testing the sphericity of a covariance matrix \cite{john1940,ledoit2002,fisher2010,wang2013,li2016,ding2020,wang2021}. 

A classical sphericity test based on a spectral statistic is the likelihood ratio test (LRT) \cite{wang2013,anderson2003}. However, its test statistic is not defined in high-dimensional settings. Consequently, we shall adopt a modified version suggested by \cite{wang2021} instead as a test statistic to evaluate the separability of $\Sigma$. To accommodate high-dimensional settings, \cite{wang2021} introduced a modification to the LRT, resulting in the statistic   
\begin{align*}
    W(S)=-\sum_{i=1}^p \log\parentheses{u_i+\sum_{j=1}^p u_j/p}+p\log\parentheses{\sum_{j=1}^p u_j},
\end{align*}
where $u_1\geq \cdots\geq  u_p\geq 0$ are eigenvalues of a sample covariance matrix $S$. They then proposed the extended LRT (ELRT) that rejects the null of sphericity in favor of the alternative hypothesis for large values of $W$. Replacing $S$ with $\hat{C}$, our second test statistic is then
\begin{align*}
    T_2(Y):=W(\hat{C})=-\sum_{i=1}^p \log\parentheses{\hat{u}_i+\sum_{j=1}^p \hat{u}_j/p}+p\log\parentheses{\sum_{j=1}^p \hat{u}_j}=-\sum_{i=1}^p \log\parentheses{\hat{u}_i+1}+p\log p
\end{align*}
for the eigenvalues $\hat{u}_1\geq \cdots\geq  \hat{u}_p\geq 0$ of $\hat{C}:=c(S)$. The last equality holds because the trace of a core is always $p$ as in Section \ref{sec3.1}. The corresponding separability test is $\phi_2(Y)=I(T_2(Y)>q_{1-\alpha})$, where $q_{1-\alpha}$ is the $1-\alpha$ quantile of the distribution of $T_2(Y)$ for a given level $\alpha\in(0,1)$. As with the test described in Section 3.1, by the invariance of $T_2$, its null distribution may be approximated to an arbitrary degree of accuracy via simulation, employing the Monte Carlo simulation described in (\ref{mc.simul}). 

Note that \cite{wang2021} derived the asymptotics of the ELRT, including the asymptotic null distribution and the formula of the asymptotic power for the regime $p/n\rightarrow\gamma\in(0,\infty)$ as $n$ grows. Given the construction of $T_2(Y)$, one may expect that these are directly translated to the asymptotics of $T_2(Y)$. Unfortunately, empirical evidence suggests that this may not be true, as in Figure \ref{sec5.1.figure3} of Section  \ref{sec5.1:null.simul}. From Figure \ref{sec5.1.figure3}, one can observe that the variability of the simulated $T_2(Y)$ for Gaussian random samples with a suitable transformation across $1000$ iterations differs from the limiting law by \cite{wang2021}, when $(n,p_1,p_2)=(1600,20,20)$. We assume $k(\Sigma)=I_p$ without loss of generality by virtue of Kronecker-invariance. This may be due to the trace constraint on the core covariance matrices, i.e., $\tr(C)=p$ for any core $C$. Also, note that $T_2(Y)$ is a linear spectral statistic with respect to $\hat{C}$, i.e., 
\begin{align}\label{sec3.2.eq1}
    T_2(Y)\equiv p\int (-\log(x+1)+\log p) dF_{\hat{C}}(x)
\end{align}
for the empirical spectral distribution $F_{\hat{C}}$ of $\hat{C}$. Thus, it is technically difficult to analyze the asymptotics of $T_2(Y)$ based on the reference results, as the convergence rate of $\hat{K}$ is not fast enough for such analysis, unless the rate is $O(n^{-\tau})$ for some constant $\tau>1$. Such a rate may be too strong, as to the best of our knowledge, for Gaussian populations, the fastest known rate is $O(n^{-1/2}\log n)$ when $\Sigma$ is separable \cite{oliveira2026}. Nevertheless, we can still use the test $\phi_2(Y)$ with given parametric distributions. Moreover, we numerically demonstrate that the empirical power of $\phi_2(Y)$ is larger than that of $\phi_1(Y)$ in general, which is shown to be consistent under some regularity conditions in Section \ref{sec4.3:asymp.null}. Thus, we leave the conjecture that the test $\phi_2(Y)$ is consistent under the same regularity conditions.

\subsection{Test Based on Separable Covariance Expansion}\label{sec3.3}
The last test statistic we consider is motivated by both the separable expansion of $\Sigma$ and the core covariance matrix $C=c(\Sigma)$, based on the following observation, whose proof is trivial and thus omitted.

\begin{lemma}\label{sec3.3.lemma1}
 For $\Sigma\in\calS_p^+$, let $C=c(\Sigma)$. Then $\Sigma$ is separable if and only if $C$ has a separability rank of $1$. 
\end{lemma}
From the definition of the separability rank, it is immediate to see that the only $C\in\calC_{p_1,p_2}^+$ with separability rank $1$ is $I_p$. Also, the result of Lemma \ref{sec3.3.lemma1} follows from the definition of the separability rank because $C$ is obtained by whitening $\Sigma$ through some $A\in GL_{p_1,p_2}$. Using the rearrangement operator, it is straightforward to verify that the algebraic ranks of $\calR(C)$ and $\calR(\Sigma)$ are the same. Hence, Lemma \ref{sec3.3.lemma1} implies that testing the separability of $\Sigma$ is equivalent to testing
\begin{align}\label{sec3.3.eq1}
    H_0: \text{rank}(\calR(C))=1 \text{ versus } H_1:\text{rank}(\calR(C))>1.
\end{align}
To test the hypothesis (\ref{sec3.3.eq1}), we shall use the empirical version of (\ref{sec2.3.eq3}), based on $\hat{C}=c(S)$ for sample covariance matrix $S$, as a test statistic. Specifically, we propose the test statistic 
\begin{align}\label{sec3.3.eq2}
    T_3(Y)=\frac{\sum_{j=2}^{p_1^2\wedge p_2^2}\hat{\sigma}_j^2 }{\hat{\sigma}_1^2}=\frac{\sum_{j=1}^{p_1^2\wedge p_2^2}\hat{\sigma}_j^2}{\hat{\sigma}_1^2}-1=\frac{||\calR(\hat{C})||_F^2}{\hat{\sigma}_1^2}-1=\frac{||\hat{C}||_F^2}{\hat{\sigma}_1^2}-1
\end{align}
for random samples $Y=(Y_1,\ldots,Y_n)$ generated according to the model (\ref{sec3.eq1}), where $\hat{\sigma}_j:=\sigma_j(\calR(\hat{C}))$ and the last equality holds as $\calR$ is an isometry between $(\real^{p_1p_2\times p_1p_2},||\cdot||_F)$ and $(\real^{p_2^2\times p_1^2},||\cdot||_F)$ \cite{puchkin2024}. The associated test $\phi_3(Y)$ is to reject the null hypothesis of (\ref{sec3.3.eq1}) for values of $T_3$ that are larger than the $1-\alpha$ quantile $q_{1-\alpha}$ of this statistic under the null hypothesis, i.e., $\phi_3(Y)=I(T_3(Y)>q_{1-\alpha})$. In fact, for any core $C$, possibly positive semidefinite, it turns out that $\sigma_1(\calR(C))$ is always $\sqrt{p}$ as a consequence of \cite{kwok2021}.

\begin{prop}\label{sec3.3.prop1}
Let $C$ be a core covariance matrix, possibly positive semidefinite. For the rearrangement operator $\calR$ defined in Definition \ref{sec2.3.def1}, $\sigma_1(\calR(C))=\sqrt{p}$.     
\end{prop}
Therefore, the test statistic $T_3(Y)$ can be simplified as 
\begin{align*}
    T_3(Y)=\frac{||\hat{C}||_F^2}{p}-1.
\end{align*}
This implies that the test $\phi_3(Y)$ is equivalent to applying John's sphericity test \cite{john1940} to the sample core $\hat{C}$ to test the sphericity of $C=c(\Sigma)$, further motivating the use of $T_3(Y)$. Indeed, for the sample covariance matrix $S$, the test statistic of John's test is given by 
\begin{align*}
    V(S)=\frac{||S||_F^2/p}{\parentheses{\tr(S)/p}^2}-1.
\end{align*}
Recalling that $\tr(C)=p$ for any core $C$, it immediately follows that $T_3(Y)\equiv V(\hat{C})$. Also, by Proposition \ref{sec3.prop1}, we see that $T_3(Y)$ maintains the Kronecker-invariance as $||\hat{C}||_F^2$ depends only on the eigenvalues of $\hat{C}$, thereby obtaining the level of the test $\phi_3(Y)$ via Monte Carlo simulation. Nevertheless, we emphasize that the Kronecker-invariance of $T_3(Y)$ can arise purely from the separable expansion perspective by analogy of Proposition \ref{sec3.prop1}.   
\begin{prop}\label{sec3.3.prop2}
Assume the same as Proposition \ref{sec3.prop1}. Then the empirical distribution of the non-zero singular values of $\calR(\hat{C})$ for the sample core $\hat{C}$ does not depend on the true separable component $K$ if either condition (\rom{1}) or (\rom{2}) of Proposition \ref{sec3.prop1} is met.     
\end{prop}
Hence, any test statistic based on the singular values of $\calR(\hat{C})$ can achieve the Kronecker-invariance. For example, one can alternatively consider the test statistic as 
\begin{align*}
    \tilde{T}_3(Y)=\frac{\sum_{j=1}^{p_1^2\wedge p_2^2}\hat{\sigma}_j}{\hat{\sigma}_1}-1=\frac{\sum_{j=1}^{p_1^2\wedge p_2^2}\hat{\sigma}_j}{\sqrt{p}}-1.
\end{align*}

As an analogy to Section \ref{sec3.2}, since the asymptotics of John's sphericity test are well studied in the literature \cite{ledoit2002,wang2013,li2016,qiu2023}, particularly when $p/n\rightarrow\gamma\in(0,\infty)$ as $n$ grows, one may hope that these asymptotics are well transferred to those of $T_3(Y)$. However, it turns out that this may not be true. To illustrate this, we compare the empirical distribution of simulated $nT_3(Y)-p-1$ across $1000$ iterations for Gaussian random samples and $(n,p_1,p_2)=(1600,20,20)$ when $\Sigma=I_p$ so that $c(\Sigma)=I_p$ with the limiting law of $nV(S)-p-1$, $N(0,4)$, by \cite{ledoit2002} in Figure \ref{sec5.1.figure3} of Section \ref{sec5.1:null.simul}. Again, note that we assume $k(\Sigma)=I_p$ without loss of generality due to Kronecker-invariance.

From Figure \ref{sec5.1.figure3}, one can see that the asymptotic distribution of $nV(S)-p-1$ with $\Sigma=I_p$ does not imply that of $nT_3(Y)-p-1$. As discussed in Section \ref{sec3.2}, the trace constraint on the core may shift the location where $nT_3(Y)-p-1$ is concentrated, as there may be a restriction in the value of $||\hat{C}||_F^2$ due to the trace constraint. Also, the theoretical analysis is difficult, as $nT_3(Y)-p-1$ can be interpreted similarly to (\ref{sec3.2.eq1}). Nevertheless, as an analogy to Section \ref{sec3.2}, the test $\phi_3(Y)$ can still be used with given parametric distributions. Moreover, we derive the first-order limit of $T_3(Y)$ in Section \ref{sec4.4:consist} under the null hypothesis of separability to infer where it is concentrated. Furthermore, in Section \ref{sec5.2:emp.power}, we empirically illustrate that the power of $\phi_3$ is larger than $\phi_1$ in general. This, combined with the consistency of $\phi_1$ established in Section \ref{sec4.3:asymp.null}, leads to the conjecture that $\phi_3$ is also consistent under the same conditions.

\section{High-Dimensional Asymptotic Power under a Partial Isotropic Core}\label{sec4:theory}

\subsection{Core Covariance Matrix with a Low-Rank Partial Isotropic Structure}\label{sec4.1}

Recall that testing the separability of $\Sigma$ is equivalent to testing the sphericity of its core $C=c(\Sigma)$, where the only spherical $C$ is $I_p$. Thus, the test statistics proposed in Section \ref{sec3:test} are constructed with this insight, two of which are obtained by applying existing sphericity tests to the sample core. Note that the power in the high-dimensional sphericity testing literature \cite{wang2013,wang2021,li2016,ding2020} is commonly analyzed under alternative regimes with partial-isotropy rank fixed in $n,p$. Here we say $\Omega\in\calS_p^+$ has a partial-isotropy rank$-r$ if $\lambda_1(\Omega)\geq \cdots \geq \lambda_r(\Omega)>\lambda_{r+1}(\Omega)=\cdots=\lambda_p(\Omega)>0$. 

Inspired by such power analysis, we study the asymptotics of the proposed test statistics under the null hypothesis of separability and the alternative regime of partial-isotropy cores in the next three subsections, except for $T_2(Y)$. These include asymptotic distributions, consistency of the test, and the first-order limit. We also evaluate the empirical power on such regimes in Section \ref{sec5.2:emp.power}. Note that these asymptotics follow from the asymptotic spectral equivalence between the sample covariance matrix and its core, which is implied by the convergence of the Kronecker MLE $\hat{K}$ (see Section \ref{sec4.2:spec.equiv}--\ref{sec4.3:asymp.null}). Hence, for most scenarios $(p_1,p_2,r)$ admitting the partial-isotropy rank$-r$ core with fixed $r$ in $(p_1,p_2,n)$, we will prove the convergence of $\hat{K}$ under such regimes in Section \ref{sec4.2:spec.equiv}. Due to the equivariance of $\hat{K}$ as discussed in Section \ref{sec2.2:KCD}, this can be done assuming $k(\Sigma)=I_p$. By Proposition 2.2 of \cite{sung2025}, the core with a partial-isotropy rank$-r$ is represented as below.
\begin{lemma}[\textbf{Proposition 2.2 of \cite{sung2025}}]\label{sec4.1.lemma1}
Suppose $C\in\calC_{p_1,p_2}^+$ has a partial-isotropy rank$-r$ structure for some $r<p$. Then $C=(1-\lambda)AA^\top+\lambda I_p$ for some $A\in \real^{p\times r}$ of full-column rank such that $AA^\top$ is a core, and $\lambda\in(0,1)$.  
\end{lemma}
In view of the above lemma, the core $C$ can have a partial-isotropy rank$-r$ if there exists a rank$-r$ core. As discussed in Section \ref{sec2.2:KCD}, the core, possibly positive semidefinite, always has the separable component or the Kronecker MLE as $I_p$. That is, a valid choice of $r$ is possible only if there exists a rank$-r$ positive semidefinite matrix whose Kronecker MLE is $I_p$. By Theorem $1.2$ of \cite{derksen2021}, it turns out that the possible values of $r$ depend on the value of $(p_1,p_2)$ as follows; 
\begin{itemize}
    \item $p_1^2+p_2^2-rp_1p_2<0\Leftrightarrow p_1/p_2+p_2/p_1<r$; 
    \item $p_1^2+p_2^2-rp_1p_2=0$;
    \item $p_1^2+p_2^2-rp_1p_2=d^2$ for $d=\text{gcd}(p_1,p_2)$.
\end{itemize}
For all other scenarios, the Kronecker MLE does not exist in a strict algebraic sense \cite{derksen2021}, and so neither does the rank$-r$ partial-isotropy core. For the first scenario, there is a sort of freeness in the choice of $r$ if $\max\set{p_1/p_2,p_2/p_1}$ is fixed in $n$, as a generic $(p_1,p_2,r)$ would satisfy the first scenario. While $r$ can be taken to be fixed for the other scenarios also, there is a restriction on the possible values of $r$ depending on $(p_1,p_2)$.    
\begin{lemma}\label{sec4.1.lemma2}
Assume $p_1\geq p_2$ and $r\in \bbN$, and let $d=\text{gcd}(p_1,p_2)$. Then the following are true:
\begin{itemize}
    \item $p_1^2+p_2^2-rp_1p_2=0$ if and only if $(p_1,p_2,r)=(p_1,p_1,2)$.
    \item $p_1^2+p_2^2-rp_1p_2=d^2$ if and only if $r\geq 1$ and $(p_1,p_2)$ is a multiple of one of the elements with positive entries of a sequence that  $\set{(x_{n+1},x_{n})}_{n=1}^{\infty}$ such that $x_0=0$ and $x_n$ for $n\geq 1$ is the largest solution to the quadratic equation $x^2-rx_{n-1}x+x_{n-1}^2-1=0$. In particular, $(p_1,p_2,r)=((k+1)m,km,2)$ for $k,m\in\bbN$ and $(p_1,p_2,r)=(p_2r,p_2,r)$ satisfy this condition.   
\end{itemize}
\end{lemma}
As a remark, the above implies that it is still possible to take a fixed value of $r$ in the regime where $p_1,p_2$ may grow with $n$, even when $p_1^2+p_2^2-rp_1p_2=0,d^2$. Moreover, the value of $r$ may depend on $(p_1,p_2)$. For instance, if $p_1\geq p_2$, $r=1$ is possible only when $p_1=p_2$, and $r=2$ only when $p_1=p_2$ or $(p_1,p_2)=((k+1)m,km)$. The latter case follows because if $r=2$, the quadratic equation $x^2-2x_{n-1}x+x_{n-1}^2-1=0$ implies that $x_n=x_{n-1}+1$. Observe that $r=2$ cannot satisfy $p_1^2+p_2^2-rp_1p_2<0$ as $2\leq p_1/p_2+p_2/p_1<r$.

Note that the aforementioned asymptotics of the test statistics may require the constant or converging ratio of the spiked eigenvalue to the non-spiked eigenvalue in the core, referred to as the signal-to-noise ratio (SNR). To motivate such assumptions, we study the eigenstructure of the cores corresponding to some cases in Lemma \ref{sec4.1.lemma2}, using the result of Proposition \ref{sec3.1.prop1}.
    \begin{thm}\label{sec4.1.thm1}
Suppose $A\in\real^{p\times r}$ is of full-column rank such that $AA^\top$ is a rank$-r$ core. Write $A=[\text{vec}(A_1),\ldots,\text{vec}(A_r)]$ for $A_i\in \real^{p_1\times p_2}$ and assume $p_1\geq p_2$ without loss of generality. Then the following are true: 
\begin{itemize}
    \item If $(p_1,p_2,r)=(p_1,p_1,2)$, there exist $1\geq d_1>\cdots>d_k\geq 0$, $n_1+\cdots+n_k=p_1$, $O_j\in\calO_{n_j}$, and $U,V\in \calO_{p_1}$, such that for $D_1=\bigoplus_{i=1}^k d_i I_{n_i}$ and $D_2=\bigoplus_{i=1}^k \sqrt{1-d_i^2}O_i$,  
    \begin{align*}
        A_1=\sqrt{p_1}UD_1V^\top,\quad A_2=\sqrt{p_1}UD_2V^\top.
    \end{align*}
     If $k=1$, $O_1\neq \pm I_{p_1}$.  
    \item If $(p_1,p_2,r)=((k+1)m,km,2)$, there exist $O_1,\ldots,O_k\in \calO_m$, $U=[U_1,\ldots,U_{k+1}]\in \calO_{p_1}$, and $V\in \calO_{p_2}$, where each block of $U$ has $m$ columns, such that for $D_1=\bigoplus_{i=1}^k\sqrt{k+1-i}I_m$ and $D_2=\bigoplus_{i=1}^k \sqrt{i}O_i$
       \begin{align*}
        A_1=\sqrt{m}[U_1,\ldots,U_k]D_1V^\top,\quad A_2=\sqrt{m}[U_2,\ldots,U_{k+1}]D_2V^\top.
    \end{align*}
    \item If $(p_1,p_2,r)=(p_2r,p_2,r)$, there exists $O\in \calO_{p_1}$ such that 
      \begin{align*}
        [A_1,\ldots,A_r]=\sqrt{p_2}O.
    \end{align*}
\end{itemize}
\end{thm}
The eigenvalues of the partial-isotropy rank-$r$ core covariance matrices for $(p_1,p_2,r)$ studied in Theorem \ref{sec4.1.thm1} are given below. The result implies that a constant or converging SNR is possible for the core based on a suitable choice of the non-spiked eigenvalue $\lambda\in(0,1)$, as shown in Section \ref{sec4.3:asymp.null}.  

\begin{cor}\label{sec4.1.cor1}
Suppose $C=(1-\lambda)AA^\top+\lambda I_p$ for $\lambda\in(0,1)$ and $A\in\real^{p\times r}$ of full-column rank such that $AA^\top$ is a rank$-r$ core. For given $(p_1,p_2,r)$, $\lambda_{r+1}(C)=\cdots=\lambda_p(C)=\lambda$ and the $r$ spiked eigenvalues of $C$ are given as follows: 
\begin{itemize}
    \item If $(p_1,p_2,r)=(p_1,p_1,2)$, $\lambda_i(C)=p_1^2(1-\lambda)\alpha_i+\lambda$
  \begin{align*}
        \alpha_1,\alpha_2=\frac{1\pm\sqrt{1-4\beta/p_1^2}}{2}
    \end{align*}
    for 
    \begin{align*}
        \beta=\sum_{j=1}^kn_jd_j^2\cdot \sum_{j=1}^k n_j(1-d_j^2)-\parentheses{\sum_{j=1}^k d_j\sqrt{1-d_j^2}\tr(O_j)}^2.
    \end{align*}
    Here $n_j,d_j$ and $O_j$ are those in Theorem \ref{sec4.1.thm1}.
    \item If $(p_1,p_2,r)=((k+1)m,km,2)$ or $(p_2r,p_2,r)$,  $\lambda_1(C)=\cdots=\lambda_r(C)=(1-\lambda)p/r+\lambda$.
\end{itemize}
\end{cor}

\subsection{Almost Sure Convergence of the Kronecker MLE}\label{sec4.2:spec.equiv}
Suppose $Y_1,\ldots,Y_n\overset{i.i.d.}{\sim}N_{p_1\times p_2}(0,\Sigma)$. Let $K=k(\Sigma)$, and $\hat{K}(Y)$ be the Kronecker MLE of $\Sigma$ based on $Y=(Y_1,\ldots,Y_n)$, i.e., the separable component of the sample covariance matrix based on $Y$. We introduce the metric $d_{op}(\cdot,\cdot)$ defined by 
\begin{align*}
    d_{op}(A,B)=||B^{-1/2}AB^{-1/2}-I_p||_2
\end{align*}
for $A,B\in\calS_{p}^+$, where $B^{1/2}$ is the symmetric square root of $B$. Built upon the proof techniques of \cite{oliveira2026}, we prove that $d_{op}(\hat{K},K)=O(n^{-\delta}\log n)$ almost surely for some fixed constant $\delta\in(0,1/2)$, generalizing the result of \cite{oliveira2026}. The choice of $\delta$ depends on the behavior of the non-spiked eigenvalue $\lambda$ of a partial-isotropy $C=c(\Sigma)$. Note that \cite{oliveira2026} proved $d_{op}(\hat{K},K)=O(n^{-1/2}\log n)$ almost surely when $\Sigma$ is separable, i.e., $C=I_p$. We consider most scenarios of $(p_1,p_2,r)$ for which a partial-isotropy rank$-r$ core exists as studied in Section \ref{sec4.1} for fixed $r$ in $(n,p_1,p_2)$. In the next subsection, under some regularity conditions, the asymptotic spectral equivalence between the sample covariance matrix $S$ and its core $\hat{C}$ will be deduced from the convergence of $\hat{K}$ when $K=I_p$ by virtue of the Kronecker-invariance. This enables studying the asymptotics of $T_1(Y)$ and $T_3(Y)$ and the test $\phi_1(Y)$ based on existing random matrix theories in Section \ref{sec4.3:asymp.null}--\ref{sec4.4:consist}.

By the equivariance of the Kronecker MLE (the separable component of the sample covariance matrix) as in Section \ref{sec2.2:KCD}, we assume $K=I_p$ without loss of generality, and thus $\Sigma=C$. Also, by Lemma \ref{sec4.1.lemma1}, if $C$ exhibits a partial-isotropy rank$-r$ structure, then $C=(1-\lambda)AA^\top+\lambda I_p$ for $A\in \real^{p\times r}$ of full column-rank such that $AA^\top$ is a rank$-r$ core, and the non-spiked eigenvalue $\lambda\in(0,1)$. Write $A=[\text{vec}(A_1),\ldots,\text{vec}(A_r)]$ for $A_i\in\real^{p_1\times p_2}$. Then we have that 
\begin{align}\label{sec4.2.eq1}
    Y_i\overset{d}{\equiv}\sqrt{1-\lambda}\sum_{j=1}^rA_jz_{ij}+\sqrt{\lambda}E_i
\end{align}
for $E_i\overset{i.i.d.}{\sim}N_{p_1\times p_2}(0,I_p)$ and independently $z_{ij}\overset{i.i.d.}{\sim}N(0,1)$. Now, we provide the non-asymptotic concentration inequality for $\hat{K}(Y)$ from which the almost sure convergence of $\hat{K}(Y)$ under the metric $d_{op}$ follows.
\begin{thm}\label{sec4.2.thm1}
Let $Y_1,\ldots,Y_n$ be $p_1\times p_2$ random matrices generated according to the model in (\ref{sec4.2.eq1}), and $p_{\max}=\max\set{p_1,p_2},p_{\min}=\min\set{p_1,p_2}$. Assume that there exists a constant $\tau>1$ such that $\tau^{-1}<p_i/\sqrt{n}<\tau$ for $i=1,2$. Suppose $r$ is fixed in $(n,p_1,p_2)$ such that a rank$-r$ core exists given $(p_1,p_2)$, $\lambda\in(0,1)$ and $1-\lambda\asymp n^{-\alpha}$ for some constant $\alpha\geq 0$. Take a fixed constant $\delta\in (0,\min\set{(\alpha+1/2)/2,1/2})$. Then there exist universal constants $C,c_1,c_2,c_3,c_4,c_5,c_6>0$ satisfying the following properties. If $n\geq C(p_{\max}/p_{\min})$,
\begin{align}\label{sec4.2.thm1.eq1}
    d_{op}(\hat{K}(Y),I_p)=O(n^{-\delta}\log n)
\end{align}
with probability at least $1-\calP_{n,p_1,p_2}$, where 
\begin{align}\label{sec4.2.thm1.eq2}
\begin{split}
   \calP_{n,p_1,p_2}&=1-6\exp\parentheses{-c_1n^{1-2\delta}p_{\min}}-6\exp\parentheses{-c_2n^{1+\alpha-2\delta}}-6\exp\parentheses{-c_3n}-\parentheses{\frac{\sqrt{np}}{{p_1+p_2}}}^{-c_4(p_1+p_2)}\\
   &-\parentheses{\frac{1}{\sqrt{1-\lambda}}\frac{\sqrt{np_{\min}}}{p_1+p_2}}^{-c_5(p_1+p_2)}-\parentheses{\frac{1}{1-\lambda}\frac{\sqrt{np_{\min}}}{\sqrt{p_1}+\sqrt{p_2}}}^{-c_6(\sqrt{p_1}+\sqrt{p_2})}.
\end{split}
\end{align}
Consequently, $\hat{K}(Y)$ converges to $I_p$ almost surely. 
\end{thm}

As a remark, when $\lambda=1$, $\delta$ can be taken as $1/2$ in the above theorem with the improved upper bound on the failure event as shown in Theorem $1.11$ of \cite{oliveira2026}. This improvement is achieved by the fact that $Y_C^\top(Y_CY_C^\top)^{-1} Y_C$ is a uniformly random projection for $Y_C=[Y_1^\top,\ldots,Y_n^\top]$ when $Y_i\overset{i.i.d.}{\sim}N_{p_1\times p_2}(0,I_p)$. Moreover, if $\alpha=0$, the spiked eigenvalues of a partial-isotropy core $C$ may diverge as $p_1,p_2$ grow with $n$ because the trace of $AA^\top$ is always $p$ regardless of its rank $r$ by Proposition \ref{sec3.1.prop1}. Thus, Theorem \ref{sec4.2.thm1} implies that the Kronecker MLE $\hat{K}(Y)$ almost surely converges even when the spiked eigenvalues of $C$ diverge at the cost of a slower convergence rate. 

Note that the result of Theorem \ref{sec4.2.thm1} does not alter if the metric $d_{op}$ is replaced with the metric
\begin{align*}
    \tilde{d}_{op}(A,B)=||B_L^{-1}AB_L^{-1,\top}-I_p||_2
\end{align*}
for the Cholesky factor $B_L$ of $B$ by the Kronecker-invariance. Specifically, recall that the identifiable square root of $K$ for defining the core $C$ can be either the symmetric square root or the Cholesky factor, as in Section \ref{sec2.2:KCD}. Write $S=K^{1/2}(1/n\sum_{i=1}^n y_iy_i^\top)K^{1/2}=:K^{1/2}S_YK^{1/2}$, where $y_i=\text{vec}(Y_i)$ for $Y_i$'s generated according to the model in (\ref{sec4.2.eq1}). Let $K_L$ be the Cholesky factor of $K$. Since there exists $O\in \calO_{p_1,p_2}$ such that $K_L=K^{1/2}O$, the equivariance of the Kronecker MLE implies that
\begin{align*}
    \tilde{d}_{op}(k(S),K)&=||K_L^{-1}K^{1/2}k(S_Y)K^{1/2}K_L^{-1,\top}-I_p||_2=||O^\top K^{-1/2}K^{1/2}k(S_Y)K^{1/2}K^{-1/2}O-I_p||_2\\
    &=||k(S_Y)-I_p||_2=\tilde{d}_{op}(k(S_Y),I_p)=d_{op}(k(S_Y),I_p).
\end{align*}
The above shows that the result would be the same if $K^{1/2}$ in $S$ is replaced with $K_L$.

\subsection{Asymptotic Distribution of $T_1(Y)$ and Consistency of $\phi_1(Y)$}\label{sec4.3:asymp.null}
We derive the asymptotic distribution of a suitably transformed $T_1(Y)$ and the consistency of a version of $\phi_1(Y)$. Specifically, we will show that the asymptotic distribution of $T_1(Y)$ after some transformation follows the Tracy-Widom law of order $1$ ($TW_1$) under the null hypothesis of separability and some local regime. As in Section \ref{sec4.1}--\ref{sec4.2:spec.equiv}, suppose that the population core $C$ has a partial-istoropy rank$-r$ structure, i.e.,
\begin{align}\label{sec4.3.eq1}
    C=(1-\lambda)AA^\top+\lambda I_p,
\end{align}
where $\lambda\in(0,1]$ and $A\in\real^{p\times r}$ of full-column rank with fixed $r\in\bbN$ in $(n,p_1,p_2)$ for admitting a rank$-r$ core given $(p_1,p_2)$. In particular, if $\lambda=1$, $C$ is $I_p$ so that $\Sigma=K$ is separable, representing the null hypothesis of separability, as discussed in Section \ref{sec2.2:KCD}.

Assume that $Y_1,\ldots,Y_n$ are i.i.d. random matrices with 
\begin{align}\label{sec4.3.eq2}
    y_i\overset{d}{\equiv} K^{1/2}(UDV^\top)z_i
\end{align}
for $y_i=\text{vec}(Y_i)$ and $z_i=\text{vec}(Z_i)$, where $\bbE[Z_1]=0$ and $V[Z_1]=I_p$. Here $K^{1/2}$ is either the symmetric square root ($\calS_{p_1,p_2}^+$) or the Cholesky factor $(\calL_{p_1,p_2}^+$) of $K\in \calS_{p_1,p_2}^+$ and $UDV^\top$ is a SVD of square root $C^{1/2}$ of $C$ above as in (\ref{sec3.eq1}). Note that $\Sigma=V[Y_1]=K^{1/2}CK^{1/2,\top}$ so that $k(\Sigma)=K$ and $c(\Sigma)=C$. Let $S=1/n\sum_{i=1}^n y_iy_i^\top$, $\hat{K}=k(S)$, and $\hat{C}=c(S)$ throughout this section. Now we introduce the assumptions necessary to establish the results of this section.

\begin{itemize}
    \item[]\hypertarget{A1}{\textbf{(A1)}} $p_1/\sqrt{n}\rightarrow\gamma_1\in(0,\infty)$ and $p_2/\sqrt{n}\rightarrow\gamma_2\in(0,\infty)$ as $n\rightarrow \infty$.
    \item[]\hypertarget{A2}{\textbf{(A2)}} If $Z_i=(z_{i,uv})$ in (\ref{sec4.3.eq2}), $z_{i,uv}$'s are i.i.d. entries with $\bbE[z_{1,11}]=0$ and $\bbE[z_{1,11}^2]=1$. Also, there exist constants $C>0,\epsilon\geq 1$ such that for all $i,u,v$,
    \begin{align*}
        \bbP\parentheses{|z_{i,uv}|>t}\leq C\exp\parentheses{-t^\epsilon}.
    \end{align*}
   \item[]\hypertarget{A3}{\textbf{(A3)}} For fixed $r\in\bbN$ in $(n,p_1,p_2)$ admitting a rank$-r$ core and $\lambda$ in (\ref{sec4.3.eq1}), $||K^{-1/2}\hat{K}K^{-1/2,\top}-I_p||_2=O(a_n)$ with probability at least $1-o(1)$, where $a_n=\log n/\sqrt{n}$ when $\lambda=1$ and $a_n=\log n/n^{\delta}$ when $\lambda\in(0,1)$, $1-\lambda\asymp n^{-\alpha}$ for some constant $\alpha \geq 0$, and a constant $\delta\in \parentheses{0,\min\set{(\alpha+1/2)/2,1/2}}$.
\item[]\hypertarget{A4}{\textbf{(A4)}} If $r\in \bbN$ is fixed in $(n,p_1,p_2)$ for which rank$-r$ core exists given $(p_1,p_2)$, $\sigma_i^2(A)/p\rightarrow d_i\in (0,1)$ as $n\rightarrow \infty$ for $i=1,\ldots,r$.
\end{itemize}

To briefly discuss the implications of the above, \hyperlink{A1}{\textbf{(A1)}} considers the regime where an asymptotic proportional growth rate holds. In the random matrix literature, it is standard to assume that $\hat{\gamma}\rightarrow\gamma \in (0,\infty)$ as $n\rightarrow \infty$. Under \hyperlink{A1}{\textbf{(A1)}}, $p/n\rightarrow\gamma_1\gamma_2$ as $n$ grows. We consider a version of the standard assumption as \hyperlink{A1}{\textbf{(A1)}} to account for the fact that both the Kronecker and the core covariance matrices depend on $(p_1,p_2)$. On the other hand, \hyperlink{A2}{\textbf{(A2)}} assumes a sub-exponential tail bound on the $z_{i,uv}$'s. The sub-exponential tail assumption is commonly assumed for the edge universality for the largest eigenvalue of sample covariance matrices \cite{bao2015,lee2016}. We will need this universality to prove that the asymptotic distribution of the transformed $T_1(Y)$ follows $TW_1$ under the null hypothesis and some local alternative regimes. 

Moreover, \hyperlink{A3}{\textbf{(A3)}} assumes the convergence rate of the Kronecker MLE $\hat{K}=k(S)$. The convergence rate $a_n$ when $\lambda=1$ can indeed be satisfied by  Theorem $1.11$ of \cite{oliveira2026} by additionally assuming \hyperlink{A1}{\textbf{(A1)}}. For the case when $\lambda\in(0,1)$, the rate $a_n$ can indeed be satisfied by Theorem \ref{sec4.2.thm1}, generalizing the result of \cite{oliveira2026}. Lastly, \hyperlink{A4}{\textbf{(A4)}} assumes the converging $\sigma_i^2(A)$ on the scale of $p$ as $n$ grows. It is natural to scale $\sigma_i^2(A)$ by $p$ as $\sum_{i=1}^r \sigma_i^2(A)=p$ by Proposition \ref{sec3.1.prop1}. Hence, it also holds that $\sum_{i=1}^r d_i=1.$ Note that this assumption can indeed be satisfied with a specific choice of $(p_1,p_2,r)$, e.g., the second item of Corollary \ref{sec4.1.cor1}. 

We first state the result on the asymptotic spectral equivalence between the sample covariance matrix and its core under the null hypothesis of separability and some local alternative regimes. While the result is straightforward from the assumption \hyperlink{A3}{\textbf{(A3)}}, it is still noteworthy as it enables deducing the asymptotic results on the test based on $T_1(Y)$ using the existing results in random matrix theory.

\begin{lemma}\label{sec4.3.lemma1}
Suppose $Y_i$'s, $z_i$'s and $C^{1/2}=UDV^\top$ are those in (\ref{sec4.3.eq2}). Let $S=1/n\sum_{i=1}^n y_iy_i^\top$ be the sample covariance matrix of $Y_i$'s, $\hat{C}=c(S)$, and $\tilde{C}=C^{1/2}(1/n\sum_{i=1}^n z_iz_i^\top)C^{1/2,\top}$. Assume \hyperlink{A3}{\textbf{(A3)}} and either (\rom{1}) or (\rom{2}) of Proposition \ref{sec3.prop1} is met. If $||\tilde{C}||_2$ is bounded in probability as $n$ grows, then $\max_{i\in [p]}|\lambda_i(\hat{C})-\lambda_i(\tilde{C})|=O_p(a_n)$.
\end{lemma}

Note that the condition that the bounded $||\tilde{C}||_2$ in probability as $n$ grows can indeed be met with a suitable choice of $\lambda$. For example, assuming \hyperlink{A1}{\textbf{(A1)}} and $\bbE[z_{i,uv}]=0,\bbE[z_{i,uv}^2]=1,\bbE[z_{i,uv}^4]<\infty$, it is well-known that $||\tilde{C}||_2\overset{a.s.}{\rightarrow} (1+\sqrt{\gamma_1\gamma_2})^2$ \cite{baik2006,baik2005,jiang2021,bai2012} when $\lambda=1$, i.e., the null hypothesis of separability. Under some additional regularity condition that can be satisfied by \hyperlink{A2}{\textbf{(A2)}}, the condition can also hold even under the alternative regime. For instance, when $(p_1,p_2,r)=(p_2r,p_2,r)$ for some fixed $r\in\bbN$, if $\lambda=p/(p+rc)$ for some constant $c>0$, Corollary \ref{sec4.1.cor1} and Proposition $2.1$ of \cite{jiang2021} imply that 
\begin{align*}
    ||\tilde{C}||_2\overset{a.s.}{\rightarrow}\begin{cases}
        (1+\sqrt{\gamma_1\gamma_2})^2,&\quad c\leq \sqrt{\gamma_1\gamma_2},\\
        1+c+\gamma_1\gamma_2(1+c)/c,&\quad c>\sqrt{\gamma_1\gamma_2}
    \end{cases}.
\end{align*}
A more generic choice of $\lambda$ will be provided in Theorem \ref{sec4.3.thm1} and Proposition \ref{sec4.3.prop1} below, whose proofs will imply the above. As a direct consequence of Lemma \ref{sec4.3.lemma1} and Theorem $1.1$ of \cite{gotze2011}, we have that the limiting spectral distribution of the sample core $\hat{C}$ is the Marchenko-Pastur (MP) law \cite{marchenko1967} under the null hypothesis of separability.

\begin{cor}\label{sec4.3.cor1}
Assume \hyperlink{A1}{\textbf{(A1)}}, \hyperlink{A2}{\textbf{(A2)}} with $C\leq 1/\epsilon$, and \hyperlink{A3}{\textbf{(A3)}} with $K=I_p$ and $\lambda=1$. Let $\hat{C}$ be that in Lemma \ref{sec4.3.lemma1} with $UDV^\top=I_p$ in (\ref{sec4.3.eq2}). Also, for two given distribution functions $F$ and $G$, let $d_{KS}(F,G):=\sup_{x\in\real}|F(x)-G(x)|$ be their Kolmogorov-Smirnov distance. If $F_{\hat{C}}$ and $F_{MP}^{\gamma}$ are the empirical spectral distribution of $\hat{C}$ and the distribution function of the MP law with parameter $\gamma$, respectively, then we have that 
\begin{align*}
   d_{KS}(F_{\hat{C}},F_{MP}^{\gamma})=O_p(\log n/\sqrt{n}).
\end{align*}
\end{cor}

Now we derive the asymptotic distributions of $T_1(Y)=\lambda_1(\hat{C})$ for the sample core $\hat{C}$ under the null hypothesis of separability and some local alternative regimes. In view of the asymptotic spectral equivalence established in Lemma \ref{sec4.3.lemma1}, we deduce the asymptotic distributions of $T_1(Y)$ from those of its natural counterpart, $\lambda_1(\tilde{C})$ for $\tilde{C}$ as in Lemma \ref{sec4.3.lemma1}. Note that $\hat{C}$ and $c(\tilde{C})$ share the same eigenvalues. To this end, we introduce an additional condition for the edge universality, namely, the moment-matching condition (see \cite{bao2015,han2016} for example).

\begin{definition}[\textbf{Moment-Matching}]\label{sec4.3.def1}
For two random matrices $U=(u_{ij}),V=(v_{ij})\in\real^{p\times n}$, we say $U$ matches $V$ up to order $k$ if there exists a constant $\tau>$1 such that 
\begin{align*}
    \bbE[U_{ij}^\ell]=\bbE[V_{ij}^\ell]+O(\exp(-(\log p)^\tau)),
\end{align*}
for $\ell=0,\ldots,k$. If $V_{ij}\overset{i.i.d.}{\sim} F$ for distribution $F$, then we say $U$ matches $F$ up to order $k$ if the above is satisfied.
\end{definition}
We introduce the quantities necessary to transform $T_1(Y)$, inducing the closed form of asymptotic distributions. Let $F_{\Omega}$ be the empirical spectral distribution of $\Omega\in\calS_p^+$. Suppose $\xi_{+,\Omega}\in[0,\lambda_1(\Omega)^{-1})$ is a unique solution in $x$ to the equation that for $\hat{\gamma}=p/n=p_1p_2/n$, 
\begin{align*}
    \int\parentheses{\frac{tx}{1-tx}}^2 dF_{\Omega}(t)=1/\hat{\gamma}.
\end{align*}
Indeed, such a solution uniquely exists \cite{lee2016,bao2015}. Define the quantities $\gamma_{0,\Omega}$ and $E_{+,\Omega}$ as
\begin{align*}
    \frac{1}{\gamma_{0,\Omega}^3}=\hat{\gamma}\int\parentheses{\frac{t}{1-t\xi_{+,\Omega}}}^3 dF_{\Omega}(t)+\frac{1}{\xi_{+,\Omega}^3},\quad E_{+,\Omega}=\frac{1}{\xi_{+,\Omega}}\parentheses{1+\hat{\gamma}\int\frac{t\xi_{+,\Omega}}{1-t\xi_{+,\Omega}}dF_{\Omega}(t)}.
\end{align*}
In particular, if $\Omega=I_p$, we have that 
\begin{align}\label{sec4.3.eq3}
    \gamma_0:= \gamma_{0,I_p}=\parentheses{\frac{\sqrt{\hat{\gamma}}}{(1+\sqrt{\hat{\gamma}})^4}}^{1/3},\quad E_+:= E_{+,I_p}=(1+\sqrt{\hat{\gamma}})^2.
\end{align}
Also, let $\hat{E}_+:=E_{+,C^{1/2,\top}\tilde{K}^{-1}C^{1/2}}$ for $\tilde{K}=k(\tilde{C})$ and $C^{1/2}=UDV^\top$ for $UDV^\top$ in (\ref{sec4.3.eq2}). In particular, if $\lambda=1$, $\hat{E}_+\equiv E_{+,\tilde{K}^{-1}}$. Using the results of \cite{lee2016,bao2015,han2016}, the asymptotic distributions of $T_1(Y)=\lambda_1(\hat{C})$ are given as follows.

\begin{thm}\label{sec4.3.thm1}
    Assume \hyperlink{A1}{\textbf{(A1)}}--\hyperlink{A4}{\textbf{(A4)}}. Suppose $Z=[z_1,\ldots,z_n]$ for $z_i$'s in \hyperlink{A2}{\textbf{(A2)}} matches $N(0,1)$ up to order $4$ and $\lambda,\alpha$ in \hyperlink{A3}{\textbf{(A3)}} satisfy one of the following:
\begin{itemize}
    \item[](\rom{1}) $\lambda=1$,
    \item[](\rom{2}) $\alpha\in(1,\infty)$,
    \item[](\rom{3}) $\alpha=1$, and $\lambda=\sigma_1^2(A)/(\sigma_1^2(A)+c)$ for some positive constant $c<\sqrt{\gamma_1\gamma_2}$.
\end{itemize}
Then it holds that
\begin{align*}
    \gamma_0n^{2/3}(T_1(Y)-\hat{E}_+)\Rightarrow TW_1, 
\end{align*}
where $\gamma_0$ and $\hat{E}_+$ are those defined above, and $\hat{E}_+\overset{p}{\rightarrow} E_+$.
\end{thm}

As a remark, with $\lambda$ as in the condition (\rom{3}) of Theorem \ref{sec4.3.thm1} and \hyperlink{A1}{\textbf{(A1)}}, $\alpha$ is indeed $1$ as 
\begin{align*}
    1-\lambda=\frac{c}{\sigma_1^2(A)+c}=\frac{1}{p}\cdot \frac{c}{\sigma_1^2(A)/p+c/p}\asymp 1/n.
\end{align*}

Moreover, under the condition (\rom{2}) of Theorem \ref{sec4.3.thm1}, the fact that the trace of the core is $p$ implies that $\lambda_1(C)$ for $C$ in (\ref{sec4.3.eq1}) converges to $1$ as $n$ grows, which will be shown in the proof of Theorem \ref{sec4.3.thm1}. Under (\rom{1}), $C=I_p$, representing the null hypothesis of separability, so that $\lambda_1(C)=1$. While this is not the case with $\alpha=1$, note that $\lambda_1(C)/\lambda=1+c$ with $\lambda$ as in (\rom{3}) of Theorem \ref{sec4.3.thm1}. Thus, the above result implies that the asymptotic behavior of $\gamma_0n^{2/3}(T_1(Y)-\hat{E}_+)$ depends on the signal-to-noise ratio (SNR) of $C$, measured by $\lambda_1(C)/\lambda$. Indeed, as shown below, if the SNR is large enough, the test based on a version of this statistic becomes consistent with slightly modified $\hat{E}_+$.

Note that while the result of Theorem \ref{sec4.3.thm1} holds for any unknown separable component $K$ due to Kronecker-invariance, the quantity $\hat{E}_+$ depends on $\tilde{K}=k(\tilde{C})$. Since $C$ is also typically unknown in practice, $\tilde{C}$ is not typically available. Therefore, one may hope to replace $\hat{E}_+$ with some deterministic quantity. Indeed, under the assumptions of Theorem \ref{sec4.3.thm1}, we have $\hat{E}_+$ converges to $E_+$ in probability. Hence, it is natural to question if $\hat{E}_+$ can be replaced with $E_+$ to obtain the asymptotic null distribution. Some empirical evidence supporting this leads to the following conjecture (see Section \ref{sec5.1:null.simul} and Appendix \ref{secC.1}). 

\begin{conj}\label{sec4.3.conj1}
Under the same conditions as in Theorem \ref{sec4.3.thm1}, it may hold that for the quantities $\gamma_0$ and $E_+$ in (\ref{sec4.3.eq3}), 
\begin{align*}
    \gamma_0n^{2/3}(T_1(Y)-E_+)\Rightarrow TW_1.
\end{align*}
\end{conj}

Consequently, we consider two versions of $T_1(Y)=\lambda_1(\hat{C})$ as follows. For the quantities $\gamma_0$, $E_+$ in (\ref{sec4.3.eq3}) and $\hat{E}_+$ but now with $\hat{E}_+:=E_{+,\tilde{K}^{-1}}$,
\begin{align}\label{sec4.3.eq4}
   T_1^1(Y)=\gamma_0n^{2/3}(T_1(Y)-E_+), \quad  T_1^2(Y)=\gamma_0n^{2/3}(T_1(Y)-\hat{E}_+)
\end{align}
We adopt $\hat{E}_+$ different from Theorem \ref{sec4.3.lemma1} for the following reason: as shown in the proof of the proposition below, under some conditions on $\alpha$ and $\lambda$, $\hat{E}_+$ converges to $E_+$ in probability even under the alternative regime. Thus, $\hat{E}_+$ is still a small random perturbation of $E_+$ under the alternative regime, unlike $E_{+,C^{1/2,\top}\tilde{K}^{-1}C^{1/2}}$. Also, the simulation studies in Section \ref{sec5.1:null.simul} and Appendix \ref{secC.1}--\ref{secC.2} imply that the asymptotic behaviors of $T_1^1(Y)$ and $T_1^2(Y)$ resemble each other. 

The corresponding tests are given by $\phi_1^1(Y)=I(T_1^1(Y)>C_{1,n})$ and $\phi_1^2(Y)=I(T_1^2(Y)>C_{2,n})$ for some critical values $C_{1,n}$ and $C_{2,n}$. For given a Kronecker-invariant test $\varphi$, let $\beta_n(\varphi):
=\bbP(\varphi(Y)=1|I,C)$, namely the power of $\varphi$ when $\Sigma=C\in\calC_{p_1,p_2}^+$. Now we prove the consistency of the tests $\phi_1^1(Y)$ and $\phi_1^2(Y)$, i.e., $\beta_n(\phi_1^i)\rightarrow 1$ as $n\rightarrow \infty$ for $i=1,2$, using the results of \cite{jiang2021,cai2020,jiang2021b}.
\begin{prop}\label{sec4.3.prop1}
Assume \hyperlink{A1}{\textbf{(A1)}}--\hyperlink{A4}{\textbf{(A4)}}. Suppose $\lambda,\alpha$ in \hyperlink{A3}{\textbf{(A3)}} satisfy one of the following:
\begin{itemize}
    \item[](\rom{1}) $\alpha\in[0,1)$,
    \item[](\rom{2}) $\alpha=1$, and $\lambda=\sigma_1^2(A)/(\sigma_1^2(A)+c)$ for some positive constant $c>\sqrt{\gamma_1\gamma_2}$.
\end{itemize}
If $C_{1,n},C_{2,n}=o(n^{5/3-\alpha})$, $\beta_n(\phi_1^1),\beta_n(\phi_1^2)\rightarrow 1$ as $n\rightarrow\infty$.
\end{prop}

Note that under the condition of (\rom{1}) and the assumptions \hyperlink{A1}{\textbf{(A1)}}, \hyperlink{A4}{\textbf{(A4)}}, the spiked eigenvalues of $C$ in (\ref{sec4.3.eq1}) diverge as $n$ grows because the trace of the core is always $p$. Therefore, it is natural for the tests $\phi_1^1(Y)$ and $\phi_1^2(Y)$ to be consistent under such a regime. Again, $\alpha=1$ is a critical regime, where the consistency depends on the behavior of $c$. 

Unlike the test based on $T_1(Y)$, the asymptotic results for the tests based on $T_2(Y)$ and $T_3(Y)$ are technically difficult. However, as shown in Section \ref{sec5.2:emp.power}, the empirical power of these tests is generally higher than that of the test based on $T_1(Y)$. Thus, this gives another conjecture that the consistency of the tests based on $T_2(Y)$ and $T_3(Y)$ may also hold under the same regime as in Proposition \ref{sec4.3.prop1}.   

\subsection{First-order Limit of $T_3(Y)$}\label{sec4.4:consist}
As discussed in Section \ref{sec3.3}, we derive the first-order limit of $T_3(Y)$. While deriving the asymptotic null distribution of $T_3(Y)$ is technically hard, the first-order limit of $T_3(Y)$ under the null hypothesis of separability tells us the center of the asymptotic distribution, provided that it exists. The result follows from the asymptotic spectral equivalence as in Lemma \ref{sec4.3.lemma1} along with Theorem $2.2$ of \cite{heiny2023}. 
\begin{prop}\label{sec4.4.prop1}
Let $S=K^{1/2}(1/n\sum_{i=1}^n z_iz_i^\top)K^{1/2,\top}$ for random vectors $z_i=\text{vec}(Z_i)$ whose entries are i.i.d. with mean zero, unit variance, and finite fourth moment, and $\hat{C}=c(S)$. Here $K^{1/2}$ is either the symmetric square root ($\calS_{p_1,p_2}^+)$ or the Cholesky factor ($\calL_{p_1,p_2}^+$) of $K\in\calS_{p_1,p_2}^+$. Assume \hyperlink{A1}{\textbf{(A1)}} and \hyperlink{A3}{\textbf{(A3)}} with $\lambda=1$. Then we have that
\begin{align*}
    T_3(Y)=\frac{||\hat{C}||_F^2}{p}-1\overset{p}{\rightarrow}\gamma_1\gamma_2.
\end{align*}
\end{prop}
The above result follows from the result of \cite{heiny2023} and the Kronecker-invariance of $T_3(Y)$. Based on the discussion of Section \ref{sec3.3}, one may consider the test statistic $\tilde{T}_3(Y):=||S||_F^2/\sigma_1(\calR(S))^2-1$ as a main motivation of $T_3(Y)$ is to test the separability rank of $\Sigma$, which is the same as that of $C=c(\Sigma)$. However, as numerically shown in Appendix \ref{sec.B:misc}, the limit of $\tilde{T}_3(Y)$ depends on the value of $K$. This further motivates the use of $T_3(Y)$ rather than $\tilde{T}_3(Y)$.  

\section{Illustration}\label{sec5:illus}

\subsection{Simulated Null Distributions}\label{sec5.1:null.simul}

We compare the Monte Carlo approximations of the null distributions of the proposed test statistics under a given data-generating distribution $\calP_0$ to their relevant asymptotic null distributions. If reasonable choices of $\calP_0$ are available, we recommend users to implement the tests using Monte Carlo-approximated (simulated) null distributions. Otherwise, one may use the asymptotic tests by comparing their test statistics to their asymptotic null distributions. However, an asymptotic test may have low power if the asymptotic null distributions are not well approximated with finite $(n,p_1,p_2)$. In particular, this is the case when the limit approaches from below as shown in Figures \ref{sec5.1.figure1}--\ref{sec5.1.figure2}. Therefore, it is important to examine the behavior of the simulated null distributions and the empirical sizes of the tests across various values of $(n,p_1,p_2)$ and choices of $\calP_0$. While we focus on Gaussian populations in this section, results with non-Gaussian populations are provided in Appendix \ref{secC.1}. 

To be specific, for $T_1(Y)$, we consider its two versions, $T_1^1(Y)$ and $T_1^2(Y)$, defined in (\ref{sec4.3.eq4}), and compare their simulated null distributions against $TW_1$ as studied in Theorem \ref{sec4.3.thm1}. For $T_2(Y)$ and $T_3(Y)$, we compare their simulated null distributions to the asymptotic null distributions by \cite{wang2021} and \cite{ledoit2002}, respectively. For a given $(n,p_1,p_2)$, we generate data as $Y_1,\ldots,Y_{n}\overset{i.i.d.}{\sim}N_{p_1\times p_2}(0,I_p)$, and then compute each test statistic across $1000$ Monte Carlo simulations. Note that by the Kronecker-invariance, we assume $k(\Sigma)=I_p$ without loss of generality. The Monte Carlo-approximated distributions of $T_1^1(Y)$ and $T_1^2(Y)$ under some (local) alternative regime for Gaussian populations are provided in Appendix \ref{secC.2}. Moreover, the Monte Carlo-approximated critical values of suitably transformed test statistics across additional values of $(n,p_1,p_2)$ are provided in Appendix \ref{sec.D:cutoff}.

We first examine the simulated null distributions of $T_1^1(Y)$ and $T_1^2(Y)$, defined in (\ref{sec4.3.eq4}). Note that $T_1^1(Y)$ is obtained by transforming $T_1(Y)$ using the quantities depending only on $(n,p_1,p_2)$, and $T_1^2(Y)$ by those additionally depending on $\Sigma$ only through its core component. Figures \ref{sec5.1.figure1}--\ref{sec5.1.figure2} present the empirical distributions of simulated $T_1^1(Y)$ and $T_1^2(Y)$ based on $1000$ Monte Carlo simulations with $(n,p_1,p_2)=(1600,80,40)$ and $(1600,80,80)$, respectively, along with the comparison to the asymptotic null distribution, $TW_1$. We observe that both empirical distributions of $T_1^1(Y)$ and $T_1^2(Y)$ approximate $TW_1$ well for large values of $(n,p_1,p_2)$. However, the approximation for $T_1^1(Y)$ is better than that of $T_1^2(Y)$, supporting Conjecture \ref{sec4.3.conj1}. Also, the limit seems to approach from below. The slower convergence for $T_1^2(Y)$ may be due to the estimation error of $\hat{K}$. We believe this estimation error impedes the convergence of the distribution of $T_1^1(Y)$ and $T_1^2(Y)$ so that the values of $(n,p_1,p_2)$ should be large to get a reasonable approximation. Indeed, in Tables \ref{sec5.1.table1}--\ref{sec5.1.table2}, we provide the empirical sizes of the tests based on $T_1^1(Y)$ and $T_1^2(Y)$ against $TW_1$ based on $1000$ Monte Carlo simulations for different values of $(n,p_1,p_2)$ at levels $\alpha=0.01$ and $0.05$, respectively. We observe that these sizes are approaching the level $\alpha$ as $(n,p_1,p_2)$ increases. However, these values should be large enough to provide a reasonable approximation.  

Consequently, the asymptotic test, where the empirical power of the test is evaluated based on the asymptotic null distribution, may be conservative due to the slow convergence. In fact, this makes the asymptotic test useful, as one would adopt a parsimonious separable covariance model when there is no sufficient evidence to adopt a more complex model. Hence, one can avoid making a Type-\rom{1} error. However, while we believe that the conservativeness of the asymptotic test based on $T_1^2(Y)$ is likely to hold for a range of $\calP_0$, we believe that the conservativeness of a test based on $T_1^1(Y)$ may depend on the heavy-tailedness of $\calP_0$. Namely, the less heavy-tailed $\calP_0$ is, the more conservative the asymptotic test based on $T_1^1(Y)$ is, as shown in Appendix \ref{secC.1}. 

For the purpose of obtaining more higher empirical power of the test, one may hope to employ Monte Carlo simulations using a reasonable distribution $\calP_0$. Since the distribution of $T_1^1(Y)$ provides a better approximation of $TW_1$, we shall evaluate the empirical power of the test based on $T_1^1(Y)$ rather than $T_1^2(Y)$ in Section \ref{sec5.2:emp.power}. Because $T_1^1(Y)$ is a strictly increasing transformation of $T_1(Y)$ involving quantities depending only on $(n,p_1,p_2)$, implementing the test $\phi_1^1(Y)$ based on $T_1^1(Y)$ is equivalent to implementing the test $\phi_1(Y)$ introduced in Section \ref{sec3.1}. Moreover, both $T_1^1(Y)$ and $T_1^2(Y)$ secure the Kronecker-invariance non-asymptotically.  

\begin{figure}[!ht]
    \centering
    \includegraphics{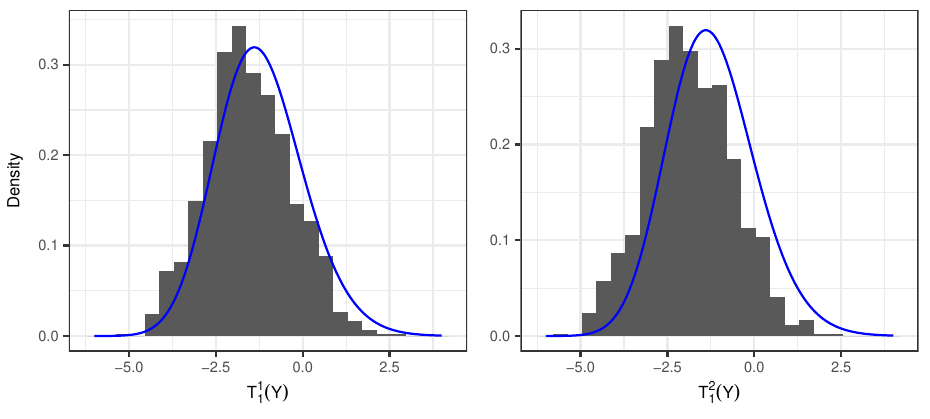}
    \caption{Monte Carlo-approximated null distributions of $T_1^1(Y)$ (left panel) and $T_1^2(Y)$ (right panel) based on $1000$ simulations under $N_{p_1\times p_2}(0,I_p)$ and the asymptotic null distribution $TW_1$ (blue line) with $(n,p_1,p_2)=(1600,80,40)$.}
    \label{sec5.1.figure1}
\end{figure}

\begin{figure}[!ht]
    \centering
    \includegraphics{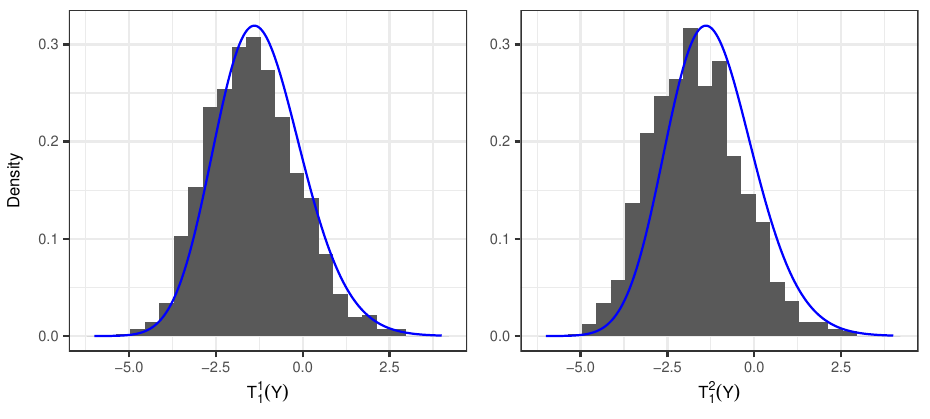}
    \caption{Monte Carlo-approximated null distributions of $T_1^1(Y)$ (left panel) and $T_1^2(Y)$ (right panel) based on $1000$ simulations under $N_{p_1\times p_2}(0,I_p)$ and the asymptotic null distribution $TW_1$ (blue line) with $(n,p_1,p_2)=(1600,80,80)$.}
    \label{sec5.1.figure2}
\end{figure}

\begin{table}[!ht]
    \centering
    \small 
    \begin{tabular}{ccccccc}
    \hline 
         & \multicolumn{6}{c}{$(n,p_1,p_2)$} \\ \cmidrule{2-7}
         & $(1600,20,20)$ & $(1600,40,20)$ & $(1600,80,40)$ & $(1600,80,80)$ & $(14400,60,60)$ & $(25600,80,80)$ \\ \hline\\[-0.3cm]
      $T_1^1(Y)$   & 0.002 & 0.001 & 0.002 & 0.006 & 0.007 & 0.007 \\ [0.1cm]
      $T_1^2(Y)$   & 0.000 & 0.001 & 0.002 & 0.005 & 0.004 & 0.006\\   \hline
    \end{tabular}
    \caption{Empirical sizes of the tests associated with $T_1^1(Y)$ and $T_1^2(Y)$ against $TW_1$ based on $1000$ Monte Carlo simulations under $N_{p_1\times p_2}(0,I_p)$ for different values of $(n,p_1,p_2)$ with $\alpha=0.01$.}
    \label{sec5.1.table1}
\end{table}

\begin{table}[!ht]
    \centering
    \small 
    \begin{tabular}{ccccccc}
    \hline 
         & \multicolumn{6}{c}{$(n,p_1,p_2)$} \\ \cmidrule{2-7}
         & $(1600,20,20)$ & $(1600,40,20)$ & $(1600,80,40)$ & $(1600,80,80)$ & $(14400,60,60)$ & $(25600,80,80)$ \\ \hline\\[-0.3cm]
      $T_1^1(Y)$   & 0.022 & 0.026 & 0.019 & 0.036 & 0.041 & 0.046 \\ [0.1cm]
      $T_1^2(Y)$   & 0.006 & 0.011 & 0.012 & 0.026 & 0.033 & 0.031\\   \hline
    \end{tabular}
    \caption{Empirical sizes of the tests associated with $T_1^1(Y)$ and $T_1^2(Y)$ against $TW_1$ based on $1000$ Monte Carlo simulations under $N_{p_1\times p_2}(0,I_p)$ for different values of $(n,p_1,p_2)$ with $\alpha=0.05$.}
    \label{sec5.1.table2}
\end{table}

Lastly, we provide the empirical distributions of a suitably transformed $T_2(Y)$ (see Theorem $1$--$2$ of \cite{wang2021}) and $nT_3(Y)-p-1$ based on $1000$ Monte Carlo simulations with $(n,p_1,p_2)=(1600,20,20)$ in Figure \ref{sec5.1.figure3}, along with the comparison of reference asymptotic null distributions by \cite{wang2021} and \cite{john1940,ledoit2002}, respectively. Recall that both these test statistics are obtained by applying sphericity tests by \cite{wang2021} and \cite{john1940,ledoit2002} to the sample core, respectively. They derived the asymptotic null distribution for their test statistic under the regime where $p/n\rightarrow\gamma\in(0,\infty)$ as $n$ grows, both of which are Gaussian distributions. We refer to Theorem $1$--$2$ of \cite{wang2021} for the formula of the mean and variance of the asymptotic null distribution of the test statistic by \cite{wang2021}. For John's sphericity test by \cite{john1940}, it was shown to be $N(0,4)$ \cite{ledoit2002}.

From Figure \ref{sec5.1.figure3}, one can clearly observe that the simulated null distributions of $T_2(Y)$ and $T_3(Y)$ are shifted to the left compared to those of the reference asymptotic distributions. Consequently, their associated tests will be also conservative as an analogy to those based on $T_1^1(Y)$ and $T_1^2(Y)$, which may hold for a range of $\calP_0$ as shown in Appendix \ref{secC.1}. As discussed in Section \ref{sec3.2}-\ref{sec3.3}, we believe the trace constraint on the core may shift the center of transformed $T_2(Y)$ and $nT_3(Y)-p-1$. Also, viewing these statistics as linear spectral statistics with respect to the sample core multiplied by $p$ as in (\ref{sec3.2.eq1}), deducing their asymptotic null distributions from the reference results is technically difficult, unless the Kronecker MLE $\hat{K}$ converges very fast. 

\begin{figure}[!ht]
    \centering
    \includegraphics{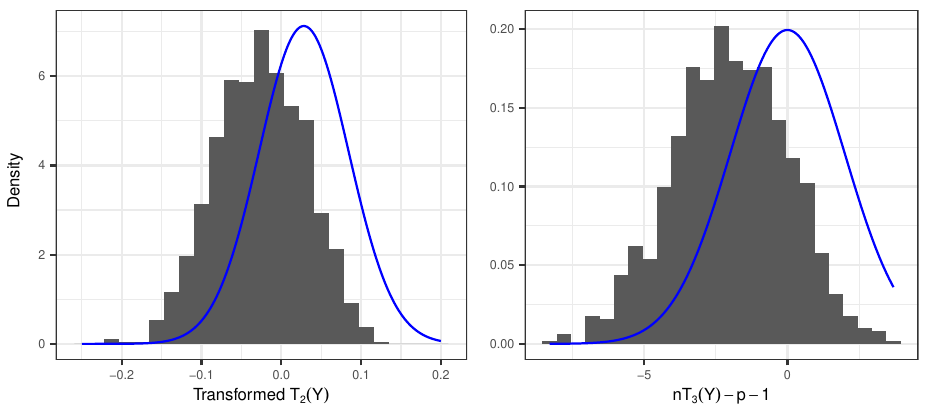}
    \caption{Monte Carlo-approximated null distributions of $T_2(Y)$ (left panel) and $T_3(Y)$ (right panel) based on $1000$ simulations under $N_{p_1\times p_2}(0,I_p)$ and the asymptotic null distributions in blue line by \cite{wang2021} (left panel) and \cite{ledoit2002} (right panel) with $(n,p_1,p_2)=(1600,20,20)$.}
    \label{sec5.1.figure3}
\end{figure}

\subsection{Empirical Power}\label{sec5.2:emp.power}

We illustrate the empirical power of the tests proposed in this article via several numerical studies. Here the power is evaluated by comparing a test statistic against the Monte Carlo-approximated null distribution obtained from a specified data-generating distribution $\calP_0$. We demonstrate that the empirical power of the proposed tests in this article is larger than that of existing separability tests, including the LRT, in general. While the LRT shows power comparable to some of the proposed tests, it is not well-defined in the high-dimensional regime where $n<p$, whereas the proposed test statistics are well-defined. Since the proposed test statistics are Kronecker-invariant by construction, it suffices to compute empirical power under the scenario where the population covariance matrix itself is a core $C\in \calC_{p_1,p_2}^+ $. Consequently, we generate data as $Y_1,\ldots,Y_n\overset{i.i.d.}{\sim} N_{p_1\times p_2}(0,C)$. We evaluate the power assuming the known mean, which is set to be $0$ without loss of generality, as is common in testing literature \cite{wang2013,ding2020,wang2021}. Nevertheless, in Appendix \ref{secC.4}, we also provide simulation studies with the unknown mean case, where the tendency of the empirical power is shown to be the same as that in this section. Additional simulation studies for non-Gaussian populations are provided in Appendix \ref{secC.3}.   

As competitors of the tests proposed in this article, we consider the LRT \cite{mitchell2006,lu2005} (\texttt{LRT}) and the nonparametric test by \cite{aston2017} (\texttt{PTCLT}), which is based on the projection of the difference between the sample covariance matrix and the partial trace estimator onto a low-dimensional subspace. For \texttt{PTCLT}, we use the subspace spanned by the leading eigenvector of the partial trace estimator. The test statistic of \texttt{LRT} satisfies the Kronecker-invariance, allowing us to approximate critical values via the Monte Carlo simulation scheme in (\ref{mc.simul}) by replacing the test statistic in \textbf{Step 3}. Using these approximated critical values, the empirical power is then computed using $K$ Monte Carlo simulations. In contrast, the test statistic of \texttt{PTCLT} is not Kronecker-invariant. We therefore compute its power using its asymptotic null distribution based on the central limit theorem (CLT) (Corollary 2.4 of \cite{aston2017}). For each $r=1,\ldots,R$, we simulate $Y_{1}^r,\ldots,Y_{n}^r\sim N_{p_1\times p_2}(0,C)$ for $C\in \calC_{p_1,p_2}^+$ described below, to compute the test statistic and the corresponding $p-$value under the CLT, denoted $\tilde{p}_{r}$. We implement \texttt{PTCLT} using the \texttt{R} function \texttt{\small clt$\_$test} from the $\texttt{R}$ package \texttt{covsep} \cite{tavakoli2018}. Given a level $\alpha\in(0,1)$, the empirical power is given by $1/R\sum_{r=1}^R I(\tilde{p}_r<\alpha)$, where we take $\alpha=0.05$ in our simulation studies. Although \cite{aston2017} recommended a bootstrap approach due to the slow convergence rate of the CLT, there were no significant differences in power, with the bootstrap only increasing the computational costs. In this section, we set $J=K=R=1000$, where $J$ is that in (\ref{mc.simul}). 

We discuss the choice of the population core $C$. Because the only separable $C$ is $I_p$, a large deviation of $C$ from $I_p$, such as a large value of $||C-I_p||_2$, is likely to increase the power of the proposed tests. Hence, we initially consider two values of $C$ with different spiked eigenstructures, $C_1$ and $C_2$, generated as follows: let $\Sigma_1$ and $\Sigma_2$ be $p\times p$ covariance matrices whose spectral decompositions are $\Gamma_1\Lambda_1\Gamma_1^\top$ and $\Gamma_2\Lambda_2\Gamma_2^\top$, respectively, where $\Gamma_1,\Gamma_2\in \calO_p$ are randomly generated, and 
\begin{align*}
    \Lambda_1&=\text{diag}(10,\underbrace{4,\ldots,4}_{p/2-1},\underbrace{1,\ldots,1}_{p/2}),\quad \Lambda_2=\text{diag}(4,\underbrace{3,\ldots,3}_{p/2-1},\underbrace{2,\ldots,2}_{p/2}).
\end{align*}
We then take $C_1=c(\Sigma_1)$ and $C_2=c(\Sigma_2)$. Note that $C_2$ exhibits a more moderate spiked eigenstructure than $C_1$. For instance, with $(p_1,p_2)=(16,32)$, the first $10$ largest eigenvalues of randomly generated matrices $C_1$ and $C_2$ are as follows:
\begin{align*}
    &C_1:3.936,\, 1.694,\, 1.686,\, 1.682,\, 1.680,\, 1.679,\, 1.678,\, 1.677,\, 1.675,\, 1.673,\\
    &C_2:1.595,\, 1.224,\, 1.221,\, 1.221,\, 1.220,\, 1.220,\, 1.220,\, 1.219,\, 1.219,\, 1.218.
\end{align*}
Therefore, the tests are expected to have higher power under $N_{p_1\times p_2}(0,C_1)$ than under $N_{p_1\times p_2}(0,C_2)$.

To more broadly examine how spikeness in the eigenstructure of $C$ affects power, we shrink $C_1$ and $C_2$ toward $I_p$; that is, we consider $C_{1,w}=wC_1+(1-w) I_p$ and $C_{2,w}=wC_2+(1-w)I_p$ for $w\in (0,1)$. Here, $1-w$ represents the shrinkage amount of $C_1$ and $C_2$ toward $I_p$. Thus, larger values of $w$ induce more spikeness in the eigenstructures of $C_{1,w}$ and $C_{2,w}$. Note that both $C_{1,w}$ and $C_{2,w}$ are also core covariance matrices as $\calC_{p_1,p_2}^+$ is convex by Proposition \ref{sec3.1.prop1}. In our simulation studies, we consider $w\in\set{1/10,\ldots,1}$. In addition, to investigate whether the empirical power converges or grows as $n$ grows under the regime in \hyperlink{A1}{\textbf{(A1)}}, we consider different values of $(\hat{\gamma}_1,\hat{\gamma}_2):=(p_1/\sqrt{n},p_2/\sqrt{n})$ and $n$. Specifically, we consider $(\hat{\gamma}_1,\hat{\gamma}_2)\in \set{(1/2,1/2),(1/2,1),(1,2),(2,2)}$ and $n\in\set{144,256,400,576,784,1024}$. The values of $(p_1,p_2)$ are determined according to those of $(\hat{\gamma}_1,\hat{\gamma}_2,n)$. 

 Figures \ref{sec5.2.figure1}--\ref{sec5.2.figure2} report the empirical power of the proposed tests and their competitors as functions of $n$ for $C_{1,w}$ and $C_{2,w}$, respectively, across selected values of $w$ and $\hat{\gamma}:=\hat{\gamma}_1\hat{\gamma}_2=p/n$. From these figures, one can observe that $\phi_2$, $\phi_3$, and \texttt{LRT} achieve substantially higher power than the other tests, with $\phi_3$ slightly dominating both $\phi_2$ and \texttt{LRT} in power. However, we emphasize that \texttt{LRT} is not well-defined if $n<p$. The power of \texttt{PTCLT} is generally very low. Note that the power of $\phi_2$ and $\phi_3$ approach $1$ with moderate sample sizes for all values of $(\hat{\gamma}_1,\hat{\gamma}_2)$ when $w=0.4,0.6,0.8$ for $C_{1,w}$ and $w=0.6,0.8$ for $C_{2,w}$. This implies that they can effectively detect non-separability in high-dimensional settings even with weak signals in the core component, e.g., a small value of $||C-I_p||_2$. Even when $w$ is small, the power increases with $n$, and eventually approaches $1$ as seen from Figure \ref{sec5.2.figure1} with $w=0.2$. This result holds even when the population covariance matrix involves a non-trivial separable component, i.e., $I_p$, as the statistics $T_2$ and $T_3$ are Kronecker-invariant by construction. Also, we note that the power of $\phi_2$ and $\phi_3$ is generally higher than $\phi_1$. Therefore, this suggests that the consistency of $\phi_2$ and $\phi_3$ may also hold under the regime where the consistency of $\phi_1$ is analyzed. In terms of numerical efficiency, $\phi_3$ is more numerically efficient to implement than $\phi_2$. Under the regime of \hyperlink{A1}{\textbf{(A1)}}, $T_2$ requires computing all the eigenvalues of $\hat{C}$ with complexity $O(p^3)=O(n^3)$. In contrast, the computational complexity of computing $T_3$ arises from computing $||\hat{C}||_F^2$, which is $O(p^2)=O(n^2)$.

The relatively low power of $\phi_1$ and \texttt{PTCLT} is expected, because their test statistics leverage the projection of certain random matrices onto low-dimensional subspaces, whereas the other tests utilize all the eigenvalues of $\hat{C}$ or $\calR(\hat{C})\calR(\hat{C})^\top$. For example, the test statistic of $\phi_1$ is based on $\lambda_1(\hat{C})=<\hat{C}\hat{v},\hat{v}>$, where $\hat{v}$ is the leading eigenvector of $\hat{C}$. A similar interpretation can be made for \texttt{PTCLT} (see (2.4) of \cite{aston2017}). Furthermore, the slow rate of convergence to the asymptotic null distribution for the test statistic of \texttt{PTCLT} may account for its low power. However, as shown in Figure \ref{sec5.2.figure1}, if $w$ is not too small, $\phi_1$ can achieve power comparable to $\phi_2$, $\phi_3$, and \texttt{LRT} with moderately large $n$. This suggests that $\phi_1$ performs well if $C$ has sufficient spikeness, e.g., a large value of $\lambda_1(C)$. 

\begin{figure}[!ht]
    \centering
    \includegraphics{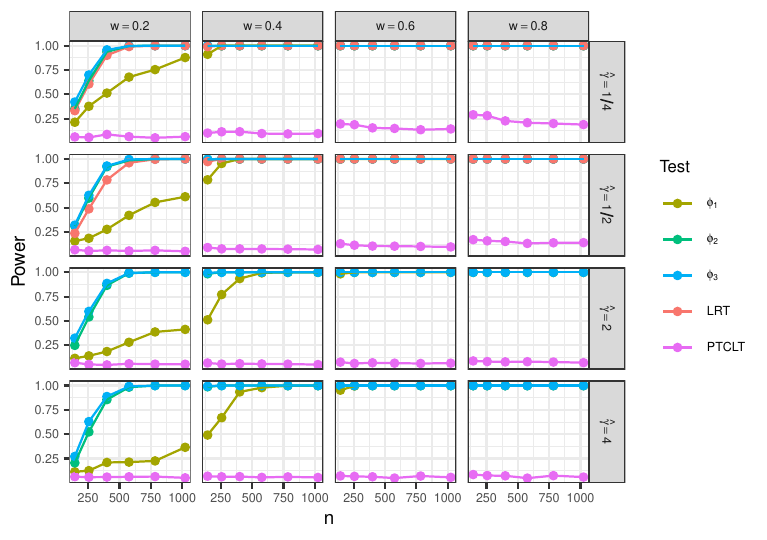}
    \caption{Empirical power of the separability tests under $N_{p_1\times p_2}(0,C_{1,w})$ as functions of $n$ across each $w=0.2,0.4,0.6,0.8$, and $\hat{\gamma}:=\hat{\gamma}_1\hat{\gamma}_2=p_1p_2/n$. The tests $\phi_1,\phi_2,$ and $\phi_3$ are those proposed in this article. \texttt{LRT} and $\texttt{PTCLT}$ denote the tests by \cite{mitchell2006,lu2005} and \cite{aston2017}, respectively. \texttt{LRT} is not present when $\hat{\gamma}>1$, since its test statistic is not well-defined in this case.}
    \label{sec5.2.figure1}
\end{figure}

\begin{figure}[!ht]
    \centering
    \includegraphics{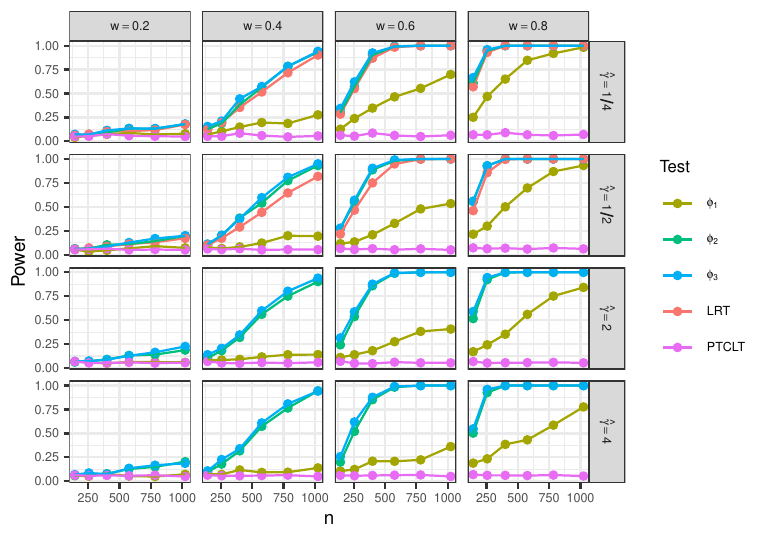}
    \caption{Empirical power of the separability tests under $N_{p_1\times p_2}(0,C_{2,w})$ as functions of $n$ across each $w=0.2,0.4,0.6,0.8$, and $\hat{\gamma}:=\hat{\gamma}_1\hat{\gamma}_2=p_1p_2/n$. For the notations of $\phi_1,\phi_2,\phi_3$, \texttt{LRT}, and \texttt{PTCLT}, see Figure \ref{sec5.2.figure1}. \texttt{LRT} is not present when $\hat{\gamma}>1$ as in Figure \ref{sec5.2.figure1}.}
    \label{sec5.2.figure2}
\end{figure}

To better illustrate the trend of power by shrinkage $w$, Figures \ref{sec5.2.figure3}--\ref{sec5.2.figure4} also present the empirical power as functions of $w$ with $n=256$ for each $C_{1,w}$ and $C_{2,w}$, respectively, across different values of $(\hat{\gamma}_1,\hat{\gamma}_2)$. While similar trends are observed in Figures \ref{sec5.2.figure3}--\ref{sec5.2.figure4} as in Figures \ref{sec5.2.figure1}--\ref{sec5.2.figure2}, it is clearer that the power of the tests other than \texttt{PTCLT} quickly reach $1$ even with moderately small $w$ for $C_{1,w}$ in Figure \ref{sec5.2.figure3}. In contrast, $w$ should be large enough to see this for $C_{2,w}$ in Figure \ref{sec5.2.figure4}. This implies that the power of the tests may be sensitive to the strength of the signal in the core component, as quantified by $||C-I_p||_2$. 

\begin{figure}[!ht]
    \centering
    \includegraphics{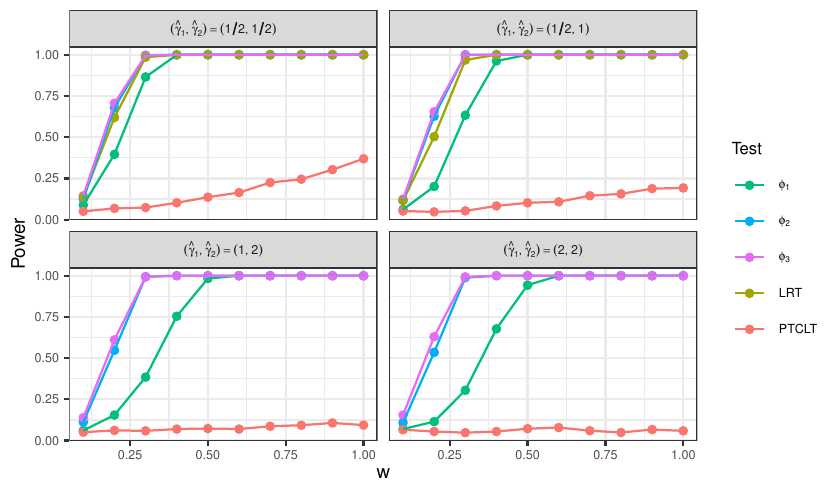}
    \caption{Empirical power of the separability tests under $N_{p_1\times p_2}(0,C_{1,w})$ as functions of $w$ across $(\hat{\gamma}_1,\hat{\gamma}_2)$.  Here $n=256$ and $(p_1,p_2)$ is determined according to the value of $(\hat{\gamma}_1,\hat{\gamma}_2)$. For the notations of $\phi_1,\phi_2,\phi_3$, \texttt{LRT}, and \texttt{PTCLT}, see Figure \ref{sec5.2.figure1}. \texttt{LRT} is not present when $\hat{\gamma}>1$ as in Figure \ref{sec5.2.figure1}.}
    \label{sec5.2.figure3}
\end{figure}

\begin{figure}[!ht]
    \centering
    \includegraphics{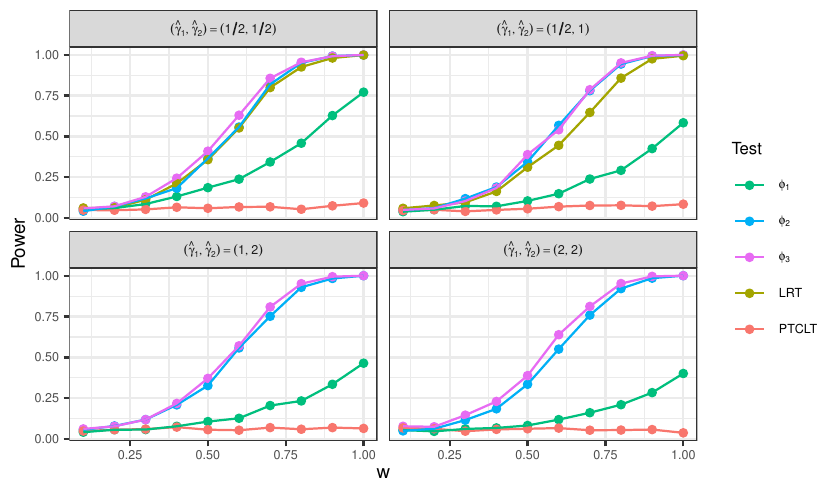}
    \caption{Empirical power of the separability tests under $N_{p_1\times p_2}(0,C_{2,w})$ as functions of $w$ across $(\hat{\gamma}_1,\hat{\gamma}_2)$. Here $n=256$ and $(p_1,p_2)$ is determined according to the value of $(\hat{\gamma}_1,\hat{\gamma}_2)$. For the notations of $\phi_1,\phi_2,\phi_3$, \texttt{LRT}, and \texttt{PTCLT}, see Figure \ref{sec5.2.figure1}. \texttt{LRT} is not present when $\hat{\gamma}>1$ as in Figure \ref{sec5.2.figure1}.}
    \label{sec5.2.figure4}
\end{figure}

From Figures \ref{sec5.2.figure1}--\ref{sec5.2.figure4}, we clearly observe that the powers of $\phi_2$ and $\phi_3$ are comparable. Hence, we examine whether the empirical power of $\phi_2$ resembles the asymptotic power of ELRT (\texttt{Asymp}) in Corollary 2.4 of \cite{wang2021} for large $n$. Specifically, we assess the empirical power of these two tests under $N_{p_1\times p_2}(0,C)$, where $C$ exhibits a partial-isotropy rank$-r$ structure with $(p_1,p_2,r)=((k+1)m,km,2)$ for some $k,m\in\bbN$ or $(p_1,p_2,r)=(p_2r,p_2,r)$. To this end, we randomly generate $C$ using the result of Lemma \ref{sec4.1.lemma1} and Theorem \ref{sec4.1.thm1}, and run $1000$ Monte Carlo simulations as before. To compute \texttt{Asymp}, recall that \cite{wang2021} assumed that $\Sigma$ has a rank$-r$ partial isotropic structure, with both $r$ and the values of spike eigenvalues fixed in $n$, and the non-spike eigenvalues all equal to $1$. However, their results remain invariant as long as the ratio of the value of each spike eigenvalue to the non-spike eigenvalue is fixed for all $r$ spikes. To align with this setting, we shall take the non-spiked eigenvalue $\lambda$ in Corollary \ref{sec4.1.cor1} as $1/(1+rc/p_1p_2)$ for some constant $c>0$ and $(p_1,p_2,r)$ above. By Corollary \ref{sec4.1.cor1}, we have that $\lambda_i(C)/\lambda=1+c$ for $i=1,\ldots,r$. Then \texttt{Asymp} can be computed by setting $a_1,\ldots,a_r$ as $1+c$ with $k=r$ in Corollary $1$ of \cite{wang2021}.

Tables \ref{sec5.2.table1}--\ref{sec5.2.table2} report the empirical power of $\phi_2$ and $\phi_3$, and \texttt{Asymp} for $C$ of rank$-2$ and rank$-6$ partial isotropic structures, respectively, across various values of $(p_1,p_2,n)$, $\hat{\gamma}=p_1p_2/n$, and $c$. The power of $\phi_2$ and $\phi_3$ are comparable when the number of spikes or $c$ is small, but differences between them become more evident as either increases. However, they are both generally much smaller than \texttt{Asymp}. This may suggest that the convergence rates of the asymptotic distributions of $T_2$ and $T_3$ may be slow. Moreover, since $T_2$ and $T_3$ are based on $\hat{C}=c(S)$ not $S$ as in Corollary $2.4$ of \cite{wang2021}, this may induce an additional error in the power analysis of \cite{wang2021}, which may align with empirical studies in Section \ref{sec5.1:null.simul}.
\begin{table}[!ht]
    \centering \footnotesize
    \begin{tabular}{c ccccc ccccc}
    \hline 
      \multirow{5}{*}{$c$}   & \multicolumn{5}{c}{$\hat{\gamma}=1/2$}  & \multicolumn{5}{c}{$\hat{\gamma}=2$} \\ \cmidrule(r){2-6} \cmidrule{7-11}
       & \multirow{3}{*}{\texttt{Asymp}}  & \multicolumn{4}{c}{$(p_1,p_2,n)$}   & \multirow{3}{*}{\texttt{Asymp}}  & \multicolumn{4}{c}{$(p_1,p_2,n)$} \\ \cmidrule(r){3-6} \cmidrule{8-11}
       &  & \multicolumn{2}{c}{(20,16,640)} & \multicolumn{2}{c}{(24,20,960)}    & & \multicolumn{2}{c}{(36,30,540)} & \multicolumn{2}{c}{(40,30,600)} \\ \cmidrule(r){3-4} \cmidrule{5-6} \cmidrule(r){8-9} \cmidrule{10-11}
          &  & $\phi_2$ & $\phi_3$ & $\phi_2$ & $\phi_3$ & & $\phi_2$ & $\phi_3$ & $\phi_2$ & $\phi_3$ \\
       \hline 
   $0.2$ & $0.105$ & $0.068$ & $0.054$ & $0.068$ & $0.063$  & $0.074$ & $0.052$ & $0.048$ & $0.043$ & $0.043$ \\
   $0.3$ & $0.143$ & $0.069$ & $0.058$ & $0.082$ & $0.081$   & $0.088$ & $0.049$ & $0.054$ & $0.041$ & $0.040$ \\
   $0.4$ & $0.189$ & $0.089$ & $0.090$ & $0.079$ & $0.090$  & $0.104$ & $0.061$ & $0.050$ & $0.055$ & $0.055$ \\
   $0.5$ & $0.242$ & $0.121$ &  $0.118$& $0.105$ & $0.123$   & $0.122$ & $0.049$ & $0.049$ & $0.052$& $0.055$  \\ \hline
    \end{tabular}
    \caption{The comparison of the empirical power of $\phi_2$ and $\phi_3$, and \texttt{Asymp} when the population core covariance matrix exhibits a rank$-2$ partial isotropic structure with $\lambda=1/(1+2c/p_1p_2)$ and $(p_1,p_2)$ as in the second item of Theorem \ref{sec4.1.thm1}. The values of $(p_1,p_2,n)$ are chosen to satisfy $\hat{\gamma}=p_1p_2/n$ for a given $\hat{\gamma}$. Here $\phi_2$ and $\phi_3$ are the tests proposed in this article.}
    \label{sec5.2.table1}
\end{table}

\begin{table}[!ht]
    \centering \footnotesize
    \begin{tabular}{c ccccc ccccc}
    \hline 
      \multirow{5}{*}{$c$}   & \multicolumn{5}{c}{$\hat{\gamma}=1/2$}  & \multicolumn{5}{c}{$\hat{\gamma}=2$} \\ \cmidrule(r){2-6} \cmidrule{7-11}
       & \multirow{3}{*}{\texttt{Asymp}}  & \multicolumn{4}{c}{$(p_1,p_2,n)$}   & \multirow{3}{*}{\texttt{Asymp}}  & \multicolumn{4}{c}{$(p_1,p_2,n)$} \\ \cmidrule(r){3-6} \cmidrule{8-11}
       &  & \multicolumn{2}{c}{(36,6,432)} & \multicolumn{2}{c}{(60,10,1200)}    & & \multicolumn{2}{c}{(72,12,432)} & \multicolumn{2}{c}{(84,14,588)} \\ \cmidrule(r){3-4} \cmidrule{5-6} \cmidrule(r){8-9} \cmidrule{10-11}
          &  & $\phi_2$ & $\phi_3$ & $\phi_2$ & $\phi_3$ & & $\phi_2$ & $\phi_3$ & $\phi_2$ & $\phi_3$ \\
       \hline 
   $0.2$ & $1.000$ & $0.093$ &  $0.096$ & $0.072$ & $0.068$ & $0.984$ & $0.057$ & $0.063$ & $0.047$ & $0.056$ \\
   $0.3$ & $1.000$  & $0.121$ & $0.127$  & $0.118$ & $0.122$ & $1.000$ & $0.062$ & $0.065$ & $0.055$ & $0.068$  \\
   $0.4$ &$1.000$   &  $0.222$ &  $0.246$ & $0.193$ & $0.222$  & $1.000$ & $0.076$ & $0.079$ & $0.064$ &  $0.081$ \\
   $0.5$ & $1.000$ & $0.330$  & $0.401$ & $0.309$ & $0.380$ & $1.000$ & $0.093$ & $0.100$ & $0.083$ & $0.110$  \\  \hline
    \end{tabular}
    \caption{The comparison of the empirical power of $\phi_2$ and $\phi_3$, and \texttt{Asymp} when the population core covariance matrix exhibits a rank$-6$ partial isotropic structure with $\lambda=1/(1+6c/p_1p_2)$ and $(p_1,p_2)$ as in the third item of Theorem \ref{sec4.1.thm1} for $r=6$. The values of $(p_1,p_2,n)$ are chosen to satisfy $\hat{\gamma}=p_1p_2/n$ for a given $\hat{\gamma}$. For the notations of $\phi_2$ and $\phi_3$, see Table \ref{sec5.2.table1}. }
    \label{sec5.2.table2}
\end{table}

\section{Discussion}\label{sec6:disc}
We introduced three Kronecker-invariant test statistics to test the separability of a population covariance matrix $\Sigma$. These statistics are well-defined in high-dimensional settings where $p>n$, and their distributions depend on $\Sigma$ only through its core component $C=c(\Sigma)$. Under the null hypothesis of separability, the core component reduces to $I_p$, rendering their null distributions free of any unknown parameters, referred to as Kronecker-invariance. Therefore, the level of the associated tests can be controlled exactly for any given $\alpha\in(0,1)$. As illustrated through numerical studies, these tests exhibit power higher than the competitors, such as the test by \cite{aston2017} and the classical LRT, even in high-dimensional regimes. In particular, the tests based on the ELRT and separable covariance expansion approach the power $1$ even with moderately large $(n,p_1,p_2)$. 

The test based on the largest eigenvalue of the core $\hat{C}$ of the sample covariance matrix $S$ is theoretically supported by its asymptotic distributions under the null hypothesis of separability and local alternative regimes. Also, we established its consistency under the regime where $C$ exhibits a partial isotropic structure. These results are achieved by the asymptotic spectral equivalence between $S$ and $\hat{C}$ when $k(\Sigma)=I_p$ due to Kronecker-invariance. This equivalence comes from the consistency of the Kronecker MLE $\hat{K}=k(S)$ for any $k(\Sigma)$. On the other hand, the test statistic based on the separable expansion of $\hat{C}$ is shown to converge to some limit in probability under the same regimes. While the empirical evidence suggests that consistency holds for the tests based on this statistic and the ELRT, the asymptotic results, such as the consistency and asymptotic distributions, for these tests seem to be challenging. We leave their theoretical proofs as a future direction.  

The separability tests proposed in this article can be naturally extended to test the
separability of the covariance array for tensor data. For example, using the negative log-likelihood of an array normal distribution as an objective function, one can define the separable component of the covariance array, which in turn induces the core component through the Tucker product \cite{hoff2011,manceur2013a,derksen2022}. As an analogy to the Kronecker-invariance notion in this article, the induced core component of the sample covariance array has a distribution that is independent of the separable component of the population covariance array. Thus, one can construct the test statistic based on this core component, thereby inducing a Kronecker-invariant test. To examine separability along specific modes of the data array, one may collapse the remaining indices and then apply the test discussed above (see \cite{gerard2016} also). 

The \texttt{R} package for implementing the tests proposed in this article is available as \texttt{kro.inv.test} \cite{sung2026}.\footnote{The package is also available at \url{https://github.com/Seungbongjung/kro.inv.test}.}
 
\bibliographystyle{plain}
\bibliography{sep_test} 

@article{hoff2023a,
    author={Hoff, Peter and Mccormack, Andrew and Zhang, Anru R.},
    title ={Core shrinkage covariance estimation for matrix-variate data},
fjournal={Journal of the Royal Statistical Society Series B: Statistical Methodology},
    journal ={J. R. Stat. Soc. Ser. B Methodol.},
volume={85},
number={5},
pages={1659--1679},
    year ={2023} 
}

@article{lindquist2008,
title={ The statistical analysis of fMRI data},
author={Lindquist, Martin A.},
fjournal={Statistical Science},
journal={Stat. Sci.}, 
volume={23},
number={4},
year={2008},
pages={439--464}
}

@article{loan2000,
title={The ubiquitous Kronecker product},
author={Loan, Charles F. Van},
fjournal={Journal of Computational and Applied Mathematics},
journal={J. Comput. Appl. Math.},
volume={123},
number={1},
year={2000},
pages={85--100}
}

@article{neudecker1969,
title={A Note on Kronecker Matrix Products and Matrix Equation Systems},
author={Neudecker, H.},
fjournal={SIAM Journal on Applied Mathematics},
journal={SIAM J. Appl. Math.},
volume={17},
number={3},
pages={603--606},
year={1969}
}

@article{soloveychik2016,
title = {Gaussian and robust Kronecker product covariance estimation: Existence and uniqueness},
fjournal={Journal of Multivariate Analysis},
journal = {J. Multivar. Anal.},
volume = {149},
pages = {92--113},
year = {2016},
author = {Soloveychik, I. and Trushin, D.},
}

@article{drton2021,
title={Existence and uniqueness of the Kronecker covariance
MLE},
fjournal={Annals of Statistics},
journal={Ann. Stat.},
volume={49},
number={5},
year={2021},
pages={2721--2754},
author={Drton, Mathias and Kuriki, Satoshi and Hoff, Peter}
}

@article{drton2024,
title={Rational Maximum Likelihood Estimators of Kronecker Covariance Matrices},
fjournal={Algebraic Statistics},
journal={Algebr. Stat.},
volume={15},
number={1},
year={2024},
pages={144-163},
author={Drton, Mathias and Grosdos, Alexandros and McCormack, Andrew}
}

@article{pigoli2014,
title={Distances and inference for
covariance operators},
author={Pigoli, Davide and Aston, John A. D. and Dryden, Ian L. and Secchi, Piercesare},
journal={Biometrika},
volume={101},
number={2},
year={2014},
pages={409--422}
}

@article{allen2012,
title={ Inference with transposable data: modelling the effects of
row and column correlations},
author={Allen, Genevera I. and Tibshirani, Robert},
fjournal={ournal of the Royal Statistical Society Series B: Statistical Methodology},
journal={J. R. Stat. Soc., B: Stat. Methodol.},
volume={74},
number={4},
year={2012},
pages={721--743}
}

@article{masak2023b,
title={Random Surface Covariance Estimation by Shifted Partial Tracing},
author={Masak, Tomas and Panaretos, Victor M.},
fjournal={Journal of the American Statistical Association},
journal={J. Am. Stat. Assoc.},
volume={118},
number={544},
year={2023},
pages={2562--2574}
}

@article{horvath2014,
title={Testing stationarity of functional time series},
author={Horv\'ath, Lajos and Kokoszka, Piotr and Rice, Gregory},
fjournal={Journal of Econometrics},
journal={J. Econom.},
volume={179},
number={1},
year={2014},
pages={66--82}
}

@article{kwok2021,
title={Spectral Analysis of Matrix Scaling and Operator Scaling},
author={Kwok, Tsz Chiu and Lau, Lap Chi and Ramachandran, Akshay},
fjournal={SIAM Journal on Computing},
journal={SIAM J. Comput.},
volume={50},
number={3},
year={2021},
pages={1034--1102}
}

@article{john1940,
title={Significance test for sphericity of a normal n-variate distribution},
author={Mauchly, John W.},
fjournal={The Annals of Mathematical Statistics},
journal={Ann. Math. Stat.},
volume={11},
number={2},
year={1940},
pages={204--209}
}

@article{zhang2012,
author={Zhang, Fuzhen},
title={Positivity of matrices with generalized matrix functions},
fjournal={ Acta Mathematica Sinica, English Series },
journal={Acta Math. Sin. (Engl. Ser.)},
volume={28},
pages={1779--1786},
year={2012}
}

@book{zhang2011,
author={Zhang, Fuzhen},
title={Matrix Theory: Basic Results and Technique},
publisher={Springer, New York},
edition={2nd},
year={2011}
}

@article{gotze2011,
title={On the Rate of Convergence to the Marchenko--Pastur Distribution},
author={G\"{o}tze, Friedrich and Tikhomirov, Alexander},
journal={arXiv preprint arXiv:1110.1284},
year={2011}
}

@article{diaconis1986,
title={An elementary proof of Stirling’s formula},
author={Diaconis, P. and Freedman, D.},
fjournal={The American Mathematical Monthly},
journal={Amer. Math. Monthly},
volume={93},
numer={2},
year={1986},
pages={123--125}
}

@book{artin1964,
title={The Gamma Function},
author={Artin, E.},
publisher={New York: Holt, Rinehart and Winston},
year={1964}
}

@article{gneiting2002,
title={Nonseparable, Stationary Covariance Functions for Space–Time Data},
author={Gneiting, Tilmann},
fjournal={Journal of the American Statistical Association},
journal={J. Am. Stat. Assoc.},
volume={97},
number={458},
year={2002},
pages={590--600}
}

@article{gneiting2007,
author={Gneiting, Tilmann and Genton, Marc G. and Guttorp, Peter},
title={Geostatistical space–time models, stationarity,
separability, and full symmetry.},
fjournal={Monographs On Statistics And Applied Probability},
journal={Monogr. Stat. Appl. Probab.},
volume={107},
year={2007},
pages={151--175}
}

@article{ghorbani2021,
title = {Testing the first-order separability hypothesis for spatio-temporal point patterns},
fjournal = {Computational Statistics & Data Analysis},
journal={ Comput. Stat. Data. Anal.},
volume = {161},
pages = {107245},
year = {2021},
author = {Ghorbani, Mohammad and Vafae, Nafiseh and Dvořák, Jiří and Myllymäki, Mari}
}

@article{simpson2014,
author = {Simpson, Sean L. and Edwards, Lloyd J. and Styner, Martin A. and Muller, Keith E. },
title = {Separability tests for high-dimensional, low-sample size multivariate repeated measures data},
fjournal = {Journal of Applied Statistics},
journal={J. Appl. Stat.},
volume = {41},
number = {11},
pages = {2450--2461},
year = {2014}
}

@article{constantinou2017,
author={Constantinou, Panayiotis and Kokoszka, Piotr and Reimherr, Matthew},
title={Testing separability of space–time
functional processes},
journal={Biometrika},
volume={104},
number={2},
pages={425--437},
year={2017}
}

@article{manceur2013b,
author={Manceur, A.M. and Dutilleul, Pierre},
title={Unbiased modified likelihood ratio tests for simple and double
separability of a variance–covariance structure},
fjournal={Statistics and Probability Letters
},
journal={ Stat. Probab. Lett.},
volume={83},
number={2},
pages={631--636},
year={2013}
}

@book{anderson2003,
author={Anderson, Theodore Wilbur},
title={An Introduction to Multivariate Statistical Analysis},
publisher={Wiley},
year={2003}
}

@article{masak2023a,
author={Masak, Tomas and Sarkar, Soham and  Panaretos, Victor M.},
title={Separable expansions for covariance estimation via the partial inner product},
journal={Biometrika},
volume={110},
number={1},
pages={225--247},
year={2023}
}

@article{puchkin2024,
author={Puchkin, Nikita and Rakhuba, Maxim},
title={Dimension-free Structured Covariance Estimation},
fjournal={Proceedings of Machine Learning Research},
journal={ Proc. Mach. Learn. Res.},
volume={247},
year={2024},
pages={1--31}
}

@article{hoff2011,
author={Hoff, Peter D.},
title={Separable covariance arrays via the Tucker
product, with applications to multivariate
relational data},
fjournal={Bayesian Analysis},
journal={Bayesian Anal.},
volume={6},
number={2},
year={2011},
pages={179--196}
}

@article{manceur2013a,
author={Manceur, Ameur M. and Dutilleul, Pierre},
title={Maximum likelihood estimation for the tensor normal distribution: Algorithm, minimum sample size, and empirical bias and dispersion},
fjournal={Journal of Computational and Applied Mathematics},
journal={J. Comput. Appl. Math.},
volume={239},
number={1},
year={2013},
pages={37--49}
}

@article{derksen2022,
author={Derksen, Harm and Makam, Visu and Walter, Michael},
title={Maximum likelihood estimation for tensor normal models via castling transforms},
fjournal={Forum of Mathematics, Sigma},
journal={Forum math. Sigma.},
volume={10},
number={50},
year={2022}
}

@book{lehmann2022,
title={Testing statistical hypotheses},
author={Lehmann, E.L. and Romano, Joseph P.},
publisher={Springer},
year={2022}
}

@article{silverstein1995,
author={Silverstein, Jack W. and Bai Z. D.},
title={On the empirical distribution of eigenvalues of a class of large dimensional random matrices,
},
fjournal={Journal of Multivariate Analysis},
journal={J. Multivar. Anal.},
volume={54},
number={2},
year={1995},
pages={175--192}
}

@article{wang2021,
title={High-dimensional sphericity test by extended likelihood
ratio},
author={Wang, Zhendong and Xu, Xingzhong},
journal={Metrika},
volume={84},
year={2021},
pages={1169–-1212}
}

@article{wang2013,
title={On the sphericity test with
large-dimensional observations},
author={Wang, Qinwen and Yao, Jianfeng},
fjournal={Electronic Journal of Statistics},
journal={Electron. J. Stat.},
volume={7},
year={2013},
pages={2164--2192}
}

@article{li2016,
title={Testing the sphericity of a covariance matrix when the dimension is much larger than the sample size.},
author={Li, Zeng and Yao, Jianfeng},
fjournal={Electronic Journal of Statistics},
journal={Electron. J. Stat. },
volume={10},
number={2},
year={2016},
pages={2973--3010}
}

@article{ledoit2002,
title={Some hypothesis tests for the covariance matrix when the dimension is large compared to the sample size},
author={Ledoit, Olivier and Wolf, Michael},
fjournal={Annals of Statistics},
journal={Ann. Stat.},
volume={30},
number={4},
year={2002},
pages={1081--1102}
}

@article{fisher2010,
title={A new test for sphericity of the covariance matrix for high
dimensional data},
author={Fisher, Thomas J. and Sun, Xiaoqian and Gallagher, Colin M.},
fjournal={Journal of Multivariate Anaylsis},
journal={J. Multivar. Anal.},
volume={101},
number={10},
year={2010},
pages={2554--2570}
}

@article{sung2025,
title={Covariance Estimation for Matrix-variate Data via Fixed-rank Core Covariance Geometry},
author={Bongjung Sung},
journal={arXiv preprint arXiv:2512.01070},
year={2025}
}

@article{qiu2023,
title={Asymptotic normality for eigenvalue statistics of a general sample covariance matrix when $p/n\rightarrow\infty $ and applications},
author={Qiu, Jiaxin and Li, Zeng and Yao, Jianfeng},
fjournal={Annals of Statistics},
journal={Ann. Stat.},
volume={51},
number={3},
year={2023},
pages={1427--1451}
}

@article{yu2023,
author={Yu, Long and Xie, Jiahui and Zhou, Wang},
title={Testing Kronecker product covariance matrices for high-dimensional matrix-variate data},
journal={Biometrika},
volume={110},
number={3},
year={2023},
pages={799-–814}
}

@article{ding2020,
author={Ding, Xue},
title={Some sphericity tests for high dimensional data based on ratio
of the traces of sample covariance matrices},
fjournal={Statistics and Probability Letters},
journal={Stat. Probab. Lett.},
volume={156},
year={2020},
pages={108613}
}

@article{tsiligkaridis2013,
author={Tsiligkaridis, Theodoros and Hero \rom{3}, Alfred O.},
title={Covariance Estimation in High Dimensions Via Kronecker Product Expansions},
fjournal={IEEE Transactions on Signal Processing},
journal={IEEE Trans. Signal Process.},
volume={61},
number={21},
pages={5347--5360},
year={2013}
}

@article{bao2015,
author={Bao, Zhigang and Pan, Guangmin and Zhou, Wang},
title={Universality for the largest eigenvalue of sample covariance matrices with general population},
fjournal={Annals of Statistics},
journal={Ann. Stat.},
volume={43},
number={1},
pages={382--421},
year={2015}
}

@article{lee2016,
author={Lee, Ji Oon and Schnelli, Kevin},
title={Tracy-Widom Distribution for the Largest Eigenvalue of Real Sample Covariance Matrices with General Population},
fjournal={Annals of Applied Probability},
journal={Ann. Appl. Probab.},
volume={26},
number={6},
pages={3786–-3839},
year={2016}
}

@article{tsiligkaridis2013b,
  author={Tsiligkaridis, Theodoros and Hero III, Alfred O. and Zhou, Shuheng},
 fjournal={IEEE Transactions on Signal Processing},
  journal={IEEE Trans. Signal Process.},  
  title={On Convergence of Kronecker Graphical Lasso Algorithms}, 
  year={2013},
  volume={61},
  number={7},
  pages={1743--1755}}

@Manual{tavakoli2018,
    title = {covsep: Tests for Determining if the Covariance Structure of
2-Dimensional Data is Separable},
    author = {Shahin Tavakoli},
    year = {2018},
    note = {R package version 1.1.0},
    url = {https://CRAN.R-project.org/package=covsep}
  }

@Manual{sung2026,
    title = {kro.inv.test: Kronecker-Invariant Tests for High-Dimensional Separability
Testing},
    author = {Bongjung Sung},
    year = {2026},
    note = {R package version 0.1.1},
    url = {https://CRAN.R-project.org/package=kro.inv.test}
  }

@article{gerard2016,
author={Gerard, David and Hoff, Peter},
title={A higher-order LQ decomposition for separable covariance models},
fjournal={Linear Algebra and its Applications},
journal={Linear Algebra Appl.},
volume={505},
number={15},
pages={57--84},
year={2016}
}

@article{aston2017,
author={Aston, John A. D. and Pigoli, Davide and Tavakoli, Shahin},
title={Tests for separability in nonparametric covariance
operators of random surfaces},
fjournal={Annals of Statistics},
journal={Ann. Stat.},
volume={45},
number={4},
pages={1431--1461},
year={2017}
}

@article{derksen2021,
author = {Derksen, Harm and Makam, Visu},
title = {Maximum Likelihood Estimation for Matrix Normal Models via Quiver Representations},
fjournal = {SIAM Journal on Applied Algebra and Geometry},
journal={SIAM J. Appl. Algebra Geom.},
volume = {5},
number = {2},
pages = {338-365},
year = {2021}
}

@article{linton2022,
author={Linton, Oliver B. and Tang, Haihan},
title={Estimation of the Kronecker covariance model by quadratic form.
},
fjournal={Econometric Theory},
journal={	Econom. Theory},
volume={38},
number={5},
pages={1014--1067},
year={2022}
}

@article{marchenko1967,
author={Marchenko, Volodymyr and Pastur, Leonid},
title={The eigenvalue distribution in some ensembles of random
matrices},
fjournal={Mathematics of the USSR Sbornik},
journal={Math. USSR Sbornik},
volume={1},
year={1967},
pages={457--483}
}

@article{baik2006,
author={Baik, Jinho and Silverstein, Jack W.},
title={Eigenvalues of large sample covariance matrices of spiked population models},
fjournal={Journal of Multivariate Analysis},
journal={J. Multivar. Anal.},
volume={97},
number={6},
pages={1382--1408},
year={2006}
}

@article{cai2020,
author={Cai, T. Tony and Han, Xiao and Pan, Guangming},
title={Limiting Laws for Divergent Spiked Eigenvalues and Largest Non-spiked Eigenvalue of Sample Covariance Matrices},
fjournal={Annals of Statistics},
journal={Ann. Stat.},
volume={48},
number={3},
pages={1255--1280},
year={2020}
}

@article{baik2005,
author={Baik, Jinho and Arous, G\'erard Ben and P\'ech'e, Sandrine},
title={Phase transition of the largest eigenvalue for nonnull complex sample covariance matrices},
fjournal={Annals of Probability},
journal={Ann. Probab.},
volume={33},
number={5},
pages={1643--1697},
year={2005}
}

@article{mitchell2006,
author={Mitchell, Matthew W. and Genton, Marc G. and Gumpertz, Marcia L.},
title={A likelihood ratio test for separability of covariances},
fjournal={Journal of Multivariate Analysis},
journal={J. Multivar. Anal.},
volume={97},
number={5},
pages={1025--1043},
year={2006}
}

@article{jiang2021,
author={Jiang, Dandan and Bai, Zhidong},
title={Generalized four moment theorem and an application to CLT for spiked eigenvalues of high-dimensional covariance matrices},
fjournal={Bernoulli},
journal={Bernoulli},
volume={27},
number={1},
pages={274--294},
year={2021}
}

@article{rudelson2013,
author={Rudelson, Mark and Vershynin, Roman},
title={Hanson-Wright inequality and sub-gaussian concentration},
fjournal={Electronic Communications in Probability},
journal={ Electron. Commun. Probab.},
volume={18},
number={82},
pages={1--9},
year={2013}
}

@article{ros2016,
author={Ro\'s, Beata and Bijma, Fetsje and de Mucnk, Jan C. and de Gunst, Mathisca C.M.},
title={Existence and uniqueness of the maximum likelihood estimator for models with a Kronecker product covariance structure},
fjournal={Journal of Multivariate Analysis},
journal={J. Multivar. Anal.},
volume={143},
year={2016},
pages={345--361}
}

@article{oliveira2026.supp,
author={Franks, Cole and Oliveira, Rafael and Ramachandran, Akshay and Walter, Michael},
title={Supplement to ``Near Optimal Sample Complexity for Matrix and Tensor Normal Models via Geodesic Convexity"},
fjournal={Annals of Statistics},
journal={Ann. Statist},
year={2026}
}

@article{oliveira2026,
author={Franks, Cole and Oliveira, Rafael and Ramachandran, Akshay and Walter, Michael},
title={Near Optimal Sample Complexity for Matrix and Tensor Normal Models via Geodesic Convexity},
fjournal={Annals of Statistics},
journal={Ann. Statist},
year={2026},
volume={54},
number={1},
pages={93--119}
}

@article{chen2012,
author={Chen, B.B. and Pan, G.M.},
title={Convergence of the largest eigenvalue of normalized sample covariance matrices when $p$ and $n$ both tend to infinity with their ratio converging to zero},
journal={Bernoulli},
volume={18},
number={4},
year={2012},
pages={1405--1420}
}

@article{jiang2021b,
author={Jiang, Dandan and Bai, Zhidong},
title={Partial generalized four moment theorem revisited},
journal={Bernoulli},
volume={27},
number={4},
year={2021},
pages={2337--2352}
}

@article{feral2009,
author={F\'{e}ral, Delphine and P\'{e}ch\'{e}, Sandrine},
title={The largest eigenvalues of sample covariance matrices for a spiked population: Diagonal case},
fjournal={Journal of Mathematical Physics},
journal={J. Math. Phys.},
volume={50},
number={7},
year={2009},
pages={073302}
}

@article{han2016,
author={Han, Xiao and Pan, Guangming and Zhang, Bo},
title={The Tracy-Widom law for the Largest Eigenvalue of F Type Matrix},
fjournal={Annals of Statistics},
journal={Ann. Statist.},
volume={44},
number={4},
year={2016},
pages={1564--1592}
}

@book{bai2012b,
author={Bai, Zhidong and Silverstein, Jack W.},
title={Spectral Analysis of Large Dimensional Random Matrices},
publisher={Springer Series in Statistics},
edition={2},
year={2012}
}

@article{heiny2023,
author={Heiny, Johannes and Kleemann, Carolin},
title={Asymptotic independence of point process and frobenius norm of a large sample covariance matrix},
journal={arXiv preprint arXiv:2302.13914},
year={2023}
}

@article{skovgaard1984,
author={Skovgaard, Lene Theil},
title={A Riemannian Geometry of the Multivariate Normal Model},
fjournal={Scandinavian Journal of Statistics},
journal={Scand. J. Stat.},
volume={11},
number={4},
year={1984},
pages={211--223}
}

@article{moakher2005,
author={Moakher, Maher},
title={ A differential geometric approach to the geometric mean of symmetric positive-definite matrices},
fjournal={SIAM Journal on Matrix Analysis and Applications},
journal={SIAM J. Matrix Anal. Appl. },
volume={26},
number={3},
year={2005},
pages={735--747}
}

@article{pennec2006,
author={Pennec, Xavier and Fillard, Pierre and Ayache, Nicholas},
title={A Riemannian Framework for Tensor Computing},
fjournal={International Journal of Computer Vision},
journal={Int. J. Comput. Vis.},
volume={66},
number={1},
year={2006},
pages={41--66}
}

@article{lu2005,
author={Lu, Nelson and Zimmerman, Dale L.},
title={The likelihood ratio test for a separable covariance matrix},
fjournal={Statistics & Probability Letters},
journal={Stat. Probab. Lett.},
volume={73},
number={4},
pages={449--457},
year={2005}
}

@article{bai2012,
author={Bai, Zhidong and Yao, Jianfeng},
title={On sample eigenvalues in a generalized spiked population model},
fjournal={Journal of Multivariate Analysis },
journal={J. Multivar. Anal.},
volume={106},
pages={167--177},
year={2012}
}

\section*{Appendix}

\appendix
\renewcommand\thefigure{\thesection.\arabic{figure}} 
\renewcommand\thetable{\thesection.\arabic{table}} 

\counterwithin{cor}{section} 
\counterwithin{cond}{section}
\counterwithin{lemma}{section} 
\counterwithin{cex}{section} 
\counterwithin{prop}{section} 
\makeatletter
\renewenvironment{proof}[1][\proofname.]{%
  \par\pushQED{\qed}%
  \normalfont \topsep6\p@\@plus6\p@\relax
  \trivlist
  \item[\hskip\labelsep
        \textbf{\upshape #1}]\ignorespaces
}{%
  \popQED\endtrivlist\@endpefalse
}
\makeatother

In Appendix, we provide the proofs of the theoretical results and additional theoretical results in the main text (Section \ref{sec.A:proof}) and miscellaneous figures and tables supporting the results and discussions in the main text (Section \ref{sec.B:misc}). In addition to the proofs of the theoretical results in the main text, Section \ref{sec.A:proof} provides additional results on the singular values of $\calR(\hat{C})$ and $\calR(C)$ for some $C$ studied in Theorem \ref{sec4.1.thm1}. It also presents a result on the invariance of the distribution of $\hat{C}$ when $\calH=\calL_{p_1,p_2}^+$, along with a counterexample showing that $\hat{C}$ may fail to be maximal-invariant under the action of $\calH=\calL_{p_1,p_2}^+$ when $n<p$. Additional simulation studies for non-Gaussian populations are provided in Section \ref{sec.C:add}. The approximated $q_{1-\alpha}$ values for the tests $\phi_1,\phi_2$ and $\phi_3$, based on Monte Carlo simulations under Gaussian populations for various values of $(p_1,p_2,n)$, are provided for $\alpha=0.05$ in Section \ref{sec.D:cutoff}. 

\section{Proofs of the Theoretical Results}\label{sec.A:proof}

\subsection{Proofs of the Results from Section \ref{sec3:test}}\label{secA.1}

We prove Proposition \ref{sec3.prop1}--\ref{sec3.3.prop2}.

\begin{proof}[\textbf{Proof of Proposition \ref{sec3.prop1}.}]
We first prove the result under the condition (\rom{1}). Note that the square root $K^{1/2}$ of $K$ is fixed as either symmetric square root $(\calS_{p_1,p_2}^+$) or Cholesky factor ($\calL_{p_1,p_2}^+$) for the unique and identifiable KCD. For simplicity, write $C^{1/2}\equiv UDV^\top$ for fixed $UDV^\top$. It suffices to show that the eigenvalues of $\hat{C}$ do depend on the value of $K$. Let $\tilde{C}=C^{1/2}1/n\sum_{i=1}^n z_iz_i^\top C^{1/2,\top}$. Since $S=K^{1/2}\tilde{C}K^{1/2,\top}$ and $K^{1/2}\in GL_{p_1,p_2}$, the equivariance of the Kronecker map $k$ implies that  
\begin{align*}
        \hat{K}=k(S)&=K^{1/2}k\parentheses{C^{1/2}1/n\sum_{i=1}^n z_iz_i^\top C^{1/2,\top}}K^{1/2,\top}\\
    &:=K^{1/2}k(\tilde{C})K^{1/2,\top}.
\end{align*}
    Let $\tilde{K}=k(\tilde{C})$ and suppose $\tilde{K}^{1/2}$ is the square root of $\tilde{K}$, allowing it to be either symmetric square root $(\calS_{p_1,p_2}^+$) or Cholesky factor ($\calL_{p_1,p_2}^+$). Note that the choice of $\tilde{K}^{1/2}$ may differ from that of $K^{1/2}$. Then any square root of $\hat{K}$ can be represented as $\hat{K}^{1/2}=K^{1/2}\tilde{K}^{1/2}O'$ for some $O'\in \calO_{p_1,p_2}$. Hence,
\begin{align}\label{prop1.eq2}
\begin{split}
        \hat{C}=\hat{K}^{-1/2}S \hat{K}^{-1/2,\top}&=O^{',\top} \tilde{K}^{-1/2} K^{-1/2} K^{1/2} \tilde{C}K^{1/2}K^{-1/2,\top}\tilde{K}^{-1/2,\top} O^{'}\\
    &=O^{',\top} \tilde{K}^{-1/2} \tilde{C}\tilde{K}^{-1/2,\top} O^{'}.
\end{split}
\end{align}
Note that $\tilde{K}$ depends only on $\tilde{C}$, which does not depend on $K$, and that the eigenvalues are orthogonally invariant. Also, for any fixed square root $\bar{K}^{1/2}$ of $\tilde{K}$, we can write $\tilde{K}^{1/2}=\bar{K}^{1/2}\tilde{O}$ for some $\tilde{O}\in \calO_{p_1,p_2}$, leading to
\begin{align*}
    \hat{C}=O^{',\top}\tilde{O}^\top \bar{K}^{-1/2}\tilde{C} \bar{K}^{-1/2,\top}\tilde{O}\tilde{O}'.
\end{align*}
Since $\calO_{p_1,p_2}$ is a group, $\tilde{O}O'\in\calO_{p_1,p_2}$. Therefore, the eigenvalues of $\hat{C}$ do not depend on $K$, and so is the ESD of $\hat{C}$,  $F_{\hat{C}}$.\\

To prove the result under (\rom{2}), note that for any fixed (but possibly differently defined) square roots $\check{K}^{1/2}$ and $\check{C}^{1/2}$ of $K$ and $C$, respectively, there exists $O\in\calO_p$ such that 
\begin{align*}
    K^{1/2}(UDV^\top)=\check{K}^{1/2}\check{C}^{1/2}O
\end{align*}
as both $K^{1/2}(UDV^\top)$ and $\check{K}^{1/2}\check{C}^{1/2}$ are square roots of $\Sigma$ for $K^{1/2}$ and $UDV^\top$ in the model (\ref{sec3.eq1}). Thus,
\begin{align*}
    y_i\overset{d}{\equiv}  K^{1/2}(UDV^\top)z_i= \check{K}^{1/2}\check{C}^{1/2}Oz_i\overset{d}{\equiv} \check{K}^{1/2}\check{C}^{1/2}z_i,
\end{align*}
due to the orthogonal invariance of $z_i$. Hence, applying the same argument for the proof under (\rom{1}), the invariance of the ESD of $\hat{C}$ with respect to $K$ holds without a specification of $K^{1/2}$ and $UDV^\top$ as in (\rom{1}). 
\end{proof}

\begin{proof}[\textbf{Proof of Proposition \ref{sec3.1.prop1}.}]
From the proof of Proposition $3$ of \cite{hoff2023a}, one can deduce that $C$ is a core covariance matrix, possibly positive semidefinite matrix, if and only if for a random matrix $Y\in\real^{p_1\times p_2}$ with $\bbE[Y]=0$ and $V[Y]=C$,
\begin{align*}
    \bbE[YY^\top]=p_2 I_{p_1}, \quad \bbE[Y^\top Y]=p_1 I_{p_2}.
\end{align*}
Suppose $(C_{[i,j]})$ is a partition of $C$ as in Definition \ref{sec2.3.def1}. Also, if $Y=[y_1,\ldots,y_{p_2}]$ for $y_i\in\real^{p_1}$, the assumption that $\bbE[Y]=0$ and $V[Y]=C$ implies $\bbE[y_iy_j^\top]=C_{[i,j]}$. Thus, 
\begin{align*}
    \bbE[YY^\top]=\sum_{i=1}^{p_2}\bbE[y_iy_i^\top]=\sum_{i=1}^{p_2}C_{[i,i]}\equiv \tr_1(C)=p_2I_{p_1}.
\end{align*}
Also, for $a,b\in[p_2]$,
\begin{align*}
    \parentheses{\bbE[YY^\top]}_{ab}=\bbE[y_a^\top y_b]=\tr\parentheses{\bbE[y_ay_b^\top]}=\tr(C_{[a,b]}).
\end{align*}
Hence, 
\begin{align*}
    \bbE[YY^\top]\equiv \tr_2(C)=p_1 I_{p_2}.
\end{align*}

\end{proof}

To prove Proposition \ref{sec3.3.prop1}, let $\textsf{mat}_{p_1\times p_2}(u)$ be the reshape of $u\in \real^{p_1\times p_2}$ into a $p_1\times p_2$ matrix. We introduce the following ancillary results.
\begin{lemma}\label{secA.1.lemma1}
Let $M\in\real^{p_1p_2\times p_1p_2}$ and consider its block partition as
\begin{align*}
    M=\left[\begin{array}{cccc}
    M_{[1,1]} & M_{[1,2]} & \cdots & M_{[1,p_2]}\\
      M_{[2,1]} & M_{[2,2]} & \cdots & M_{[2,p_2]}\\
      \vdots & \vdots & \ddots & \vdots \\
    M_{[p_2,1]} & M_{[p_2,2]} & \cdots & M_{[p_2,p_2]}
    \end{array}\right],
\end{align*}
where each block is of dimension $p_1\times p_1$. Recall the rearrangement $\calR:\real^{p_1p_2\times p_1p_2}\rightarrow\real^{p_2^2\times p_1^2}$, whose $((j-1)p_1+1)+i$ row of $\calR(M)$ is the vectorization of $M_{[i,j]}$. Let $a\in\real^{p_1p_2}$ and $A=\textsf{mat}_{p_1\times p_2}(a)$. 
\begin{align*}
    \calR(aa^\top)=A^\top \otimes A^\top.
\end{align*}
\end{lemma}
\begin{proof}
Partition $a$ as $[a_1^\top,\ldots,a_{p_2}^\top]^\top$ where each partition is a $p_1-$dimensional vector. Then $[i,j]$ block of $aa^\top$ is clearly $a_ia_j^\top$, whose vectorization is $a_j\otimes a_i$. On the other hand, the $((j-1)p_1+i)$-th row of $A^\top \otimes A^\top $ is a Kronecker product of $j$th row and $i$th row of $A^\top$, which are $a_j$ and $a_i$, proving the claim.
\end{proof}

\begin{lemma}\label{secA.1.lemma2}
Let $r\in\bbN$ for which a rank$-r$ core exists given $(p_1,p_2)$. Suppose $A=[\text{vec}(A_1),\ldots,\text{vec}(A_r)]\in\real^{p\times r}$ of full-column rank and $A_i\in\real^{p_1\times p_2}$. Then $AA^\top$ is a rank$-r$ core if and only if 
\begin{align*}
  \sum_{i=1}^r A_iA_i^\top=p_2I_{p_1},\quad \sum_{i=1}^r A_i^\top A_i=p_1I_{p_2}.
\end{align*}
\end{lemma}
\begin{proof}
    This is a direct consequence of Proposition \ref{sec3.1.prop1}.
\end{proof}

\begin{proof}[\textbf{Proof of Proposition \ref{sec3.3.prop1}.}]
 Given $(p_1,p_2)$, let $r$ and $A$ be those in Lemma \ref{secA.1.lemma2}. From the result of Lemma \ref{secA.1.lemma1}, the linearity of $\calR$ implies that 
 \begin{align*}
     \calR(AA^\top)=\sum_{i=1}^r A_i^\top\otimes A_i^\top=\parentheses{\sum_{i=1}^r A_i\otimes A_i}^\top.
 \end{align*}
By Lemma $3.6$ of \cite{kwok2021} with $\epsilon=0$, Lemma \ref{secA.1.lemma2} implies that $\sigma_1(\calR(AA^\top))=\sqrt{p_1p_2}$ (see p.$1037$ of \cite{kwok2021} also).
\end{proof}

\begin{proof}[\textbf{Proof of Proposition \ref{sec3.3.prop2}.}]
As an analogy to the proof of Proposition \ref{sec3.prop1}, it suffices to show that the singular values of $\calR(\hat{C})$ do not depend on the value of $K$. As in the proof of Proposition \ref{sec3.prop1}, we prove the result under the condition (\rom{1}), since the result under the assumption (\rom{2}) follows from this. Using the notations in the proof of Proposition \ref{sec3.prop1}, recall that 
    \begin{align*}
        \hat{C}=(\tilde{O}_2\otimes \tilde{O}_1)^\top \tilde{K}^{-1/2}\tilde{C}\tilde{K}^{-1/2,\top}(\tilde{O}_2\otimes \tilde{O}_1),
    \end{align*}
    where $O'=\tilde{O}_2\otimes \tilde{O}_1\in\calO_{p_1,p_2}$. Suppose $\tilde{K}^{-1/2}\tilde{C}\tilde{K}^{-1/2,\top}$ has a separable expansion as  $\sum_{i=1}^{\tilde{R}} \tilde{A}_i\otimes \tilde{B}_i$ for some $\tilde{R}\leq p_1^2\wedge p_2^2$, $\set{\tilde{A}_i}_{i=1}^{\tilde{R}}\subseteq \real^{p_2\times p_2}$, and $\set{\tilde{B}_i}_{i=1}^{\tilde{R}}\subseteq \real^{p_1\times p_1}$. Then by vec-Kronecker identity,
    \begin{align*}
           \calR(\hat{C})&=\calR\parentheses{\sum_{i=1}^{\tilde{R}} (\tilde{O}_2\tilde{A}_i\tilde{O}_2^\top)\otimes (\tilde{O}_1\tilde{B}_i\tilde{O}_1^\top)}\\
        &=\sum_{i=1}^{\tilde{R}} \text{vec}(\tilde{O}_1\tilde{B}_i\tilde{O}_1^\top)\text{vec}(\tilde{O}_2\tilde{A}_i\tilde{O}_2^\top)^\top\\
        &=\sum_{i=1}^{\tilde{R}} (\tilde{O}_1\otimes \tilde{O}_1)\text{vec}(\tilde{B}_i)\text{vec}(\tilde{A}_i)^\top (\tilde{O}_2\otimes \tilde{O}_2)^\top\\
        &=(\tilde{O}_1\otimes \tilde{O}_1)\calR(\tilde{K}^{-1/2}\tilde{C}\tilde{K}^{-1/2,\top})(\tilde{O}_2\otimes \tilde{O}_2)^\top.
    \end{align*}
    Since $\tilde{O}_i\otimes \tilde{O}_i\in\calO_{p_i^2}$ for $i=1,2$ and the singular values do not alter by rotating left and right singular vectors, the above shows that the singular values of $\calR(\hat{C})$ and those of $\calR(\tilde{K}^{-1/2}\tilde{C}\tilde{K}^{-1/2,\top})$ coincide. Because the distribution of $\tilde{K}^{-1/2}\tilde{C}\tilde{K}^{-1/2,\top}$ does not depend on $K$, this proves the desired claim.
\end{proof}

\subsection{Proofs of the Results from Section \ref{sec4.1}}\label{secA.2}
We prove Lemma \ref{sec4.1.lemma2}, Theorem \ref{sec4.1.thm1}, and Corollary \ref{sec4.1.cor1}.
\begin{proof}[\textbf{Proof of Lemma \ref{sec4.1.lemma2}.}]
Note that the proof of if part is straightforward. Hence, we verify the only if part. We assume that $d=\text{gcd}(p_1,p_2)=1$ without loss of generality as $p_1^2+p_2^2-rp_1p_2=0,d^2$ if and only if $(p_1/d)^2+(p_2/d)^2-r(p_1/d)(p_2/d)=0,1$, and $p_1/d,p_2/d\in\bbN$. Also, recall that we assume $p_1\geq p_2$ and write $a|b$ if $a$ divides $b$ for integers $a\geq 1$ and $b\geq 0$.\\ 

To prove the first item, note that if $p_1^2+p_2^2-rp_1p_2=0$, then $p_1p_2$ should divide $p_1^2+p_2^2$. Since $p_1|p_1^2$ and $p_2|p_2^2$, this implies that $p_1|p_2^2$ and $p_2|p_1^2$. If $p_1>1$, then $p_1$ has at least one prime factor, which will also divide $p_2^2$ and so $p_2$, which contradicts the assumption that $d=1$. Hence, $p_1=1$ and similarly, $p_2=1$, where $r=2$. Multiplying $p_1=1$ and $p_2=1$ with any integer $m$, we have $(p_1,p_2,r)=(m,m,2)$ satisfies $p_1^2+p_2^2-rp_1p_2=0$. \\

It remains to prove the second item. Note that for the sequence $\set{x_n}_{n=1}^\infty$ defined in the second item, $(x_2,x_1)=(r,1)$ and $(x_3,x_2)=(r^2-1,r)$. If there exists an integer $r$ such that $p_1^2+p_2^2-rp_1p_2=1$, then $\text{gcd}(p_1,p_2)=1$ is always true by B\'{e}zout's identity  as $p_1\cdot p_1+(p_2-rp_1)\cdot p_2=1$. Also, we have that $p_1p_2|p_1^2+p_2^2-1$. If $p_2=1$, then this is always satisfied for any $p_1$ with $r=p_1$, i.e., $(p_1,p_2)=(r,1)$, and thus any multiple of $(r,1)$ as a valid choice for $(p_1,p_2)$.\\

Otherwise, $p_2\geq 2$. Note that $p_1p_2|p_1^2+p_2^2-1$ implies that $p_1|p_2^2-1$ and $p_2|p_1^2-1$, as $p_1|p_1^2$ and $p_2|p_2^2$. Suppose $p_1=p_2$. Then $p_2|p_2^2-1$ only when $p_2=1$, which is a contradiction. Hence, $p_1\geq p_2+1$. Also, since $p_1|p_2^2-1$, then $p_1\leq p_2^2-1$. Thus, $p_2+1\leq p_1\leq p_2^2-1$. If $p_1=p_2+1$, $r=2$. If $r=2$, note that the quadratic equation $x^2-2x_{n-1}x+x_{n-1}^2-1=0$ yields $x_n=x_{n-1}+1$. Thus, with $r=2$, $(p_1,p_2)=(m+1,m)$ for some $m\in\bbN$. Again, any multiple of $(p_1,p_2)$ is a valid choice.\\ 

The remaining case is $p_2+1<p_1\leq p_2^2-1$ with $p_2\geq 3$ as $p_2+1<p_2^2-1$. We shall adopt Vieta jumping argument to prove the desired claim. Fix $p_2\geq 3$ and $r$ such that $p_1^2+p_2^2-rp_1p_2=1$. Consider the quadratic equation $x^2-rp_2x+p_2^2-1=0$. We have one solution as $p_1$. If $y_1$ is another solution, then we have $y_1+p_1=kp_2$, implying that $xy_1$ is also an integer as $p_1,k,p_2$ are integers. Moreover, since $y_1p_1=p_2^2-1$ and $p_2+1<p_1\leq p_2^2-1$, we have that $1\leq y_1<p_2-1$ so that $y_1+1<p_2\leq y_1^2-1$. Hence, we have found other pair of positive integers $(p_2,y_1)$ such that $r=(p_2^2+y_1^2-1)/(y_1p_2)$. Now consider the quadratic equation that $x^2-ry_1x+y_1^2-1=0$. Again, we have one solution as $p_2$. By the same argument, we find the other integer solution $0\leq y_2=(y_1^2-1)/p_2<y_1-1$ as $1\leq y_1<p_2-1$. Thus, we found the other solution $(y_1,y_2)$. Continuing this procedure, we can generate a sequence of pair of positive integers $(y_n,y_{n+1})$ such that $y_{n+1}<y_n-1$ with $r=(y_n^2+y_{n+1}^2-1)/y_ny_{n+1}$. Nevertheless, since we should only have positive integers, this sequence should terminate within finite procedures. This can happen only when $y_{n+1}=0$ given $y_n>0$. Indeed, if $y_n>1$, then we find another $0\leq y_{n+1}<y_n-1$. If $y_n=1$, then we have the quadratic equation that $x^2-rx=0$, leading to $y_{n+1}=0$. Reversing the procedure above, since $r$ was fixed, if $x_0=0$, the quadratic equation $x^2-rx_{n-1} x+x_{n-1}^2-1=0$ would imply that $x_1=1$ as $x_1$ is the largest soultion of this equation. And $x_2=r$, $x_3=r^2-1$, $x_4=r^3-2r$, and so. $(p_1,p_2)$ can be any multiple of $(x_{n+1},x_n)$ as long as $x_{n+1}\geq x_n\geq 1$. For example, if $r=1$, $x_3=0$. Thus, it should be that $r\geq 2$ if $(p_1,p_2)$ is a multiple of $(r^2-1,r)$.

\end{proof}

To prove Theorem \ref{sec4.1.thm1} and Corollary \ref{sec4.1.cor1}, we will use the result of Lemma \ref{secA.1.lemma2}, which follows from Proposition \ref{sec3.1.prop1}. We will also need the following ancillary results. 

\begin{lemma}\label{secA.2.lemma1}
  Suppose $M\in\real^{p\times p}$ is symmetric with an eigendecomposition $\Gamma\Lambda\Gamma^\top$, where $\Gamma\in \calO_p$ and $\Lambda=\text{diag}(\lambda_1,\ldots,\lambda_p)$ with $\lambda_i$'s in descending order. Partition $I_p$ as $I_p=[V_1,\ldots,V_k]$, where $V_i\in\real^{p\times n_i}$.  If $D=\text{diag}\parentheses{\underbrace{d_1,\ldots,d_1}_{n_1},\ldots,\underbrace{d_k,\ldots,d_k}_{n_k}}$, where $d_i$'s are distinct and arranged in ascending order, $M=D$ if and only if $(\lambda_1,\ldots,\lambda_p)=\parentheses{\underbrace{d_k,\ldots,d_k}_{n_k},\ldots,\underbrace{d_1,\ldots,d_1}_{n_1}},$
  \begin{align*}
      \Gamma&=[V_k,\ldots,V_k]\bigoplus_{i=1}^k U_{k+1-i},\\
  \end{align*}
  where $U_i\in\calO_{n_i}$ and $\bigoplus$ denotes the direct sum.
\end{lemma}
\begin{proof}
    The result follows from the fact that the eigendecomposition is unique up to the rotation of the orthonormal basis of the eigenspace corresponding to each distinct eigenvalue.
\end{proof}

\begin{lemma}\label{secA.2.lemma2}
    For $p_1,p_2\in\mathbb{N}$, suppose $p_2\leq p_1\leq 2p_2$. Define the equivalence relation $\sim$ on $\real_{+}^{p_1}$ by $a\sim b$ if and only if there exists a permutation $\pi:[p_1]\mapsto [p_1]$ such that $\pi(a)=(a_{\pi(1)},\ldots,a_{\pi(p_1)})^\top=b$. Define the vectors $u,v\in\real_{+}^{p_1}$ by 
    \begin{align*}
        u&=(1-w_1,\ldots,1-w_{p_2},\underbrace{0,\ldots,0}_{p_1-p_2})^\top,\\
        v&=(p_2/p_1-w_1,\ldots,p_2/p_1-w_{p_2},\underbrace{p_2/p_1,\ldots,p_2/p_1}_{p_1-p_2})^\top,
    \end{align*}
    for $w\in\real_+^{p_2}$. If $p_1>p_2$ and $p_1-p_2|p_1$, take any $w$ such that 
    \begin{align*}
        w\sim (\underbrace{1,\ldots,1}_{m},\ldots,\underbrace{k,\ldots,k}_{m})^\top/(k+1)
    \end{align*}
    for which $p_1=(k+1)m$ and $p_2=km$ for some $k,m\in\mathbb{N}$. Alternatively, if $p_1=p_2$, take any $w$ such that $\max_i w_i\leq 1$. Then $u\sim v$ if and only if $p_1=p_2$ or $p_1>p_2$ and $p_1-p_2|p_2$, and $w$ is chosen accordingly as described.
\end{lemma}
\begin{proof}
    We first prove if part. If $p_1=p_2$, then $u=v$ and $u\in \real_+^{p_1}$ as long as $\max_iw_i\leq 1$. Otherwise, $p_1>p_2$ and $p_1-p_2|p_2$. Then there exists $k\in\mathbb{N}$ such that $p_2=(p_1-p_2)k$ or $p_2=k/(k+1)p_1$. Since $k+1$ cannot divide $k$, $k+1$ should divide $p_1$. Thus, there exists $m\in\mathbb{N}$ such that $p_1=(k+1)m$ and so $p_2=km$. Then 
    \begin{align*}
        u&=(1-w_1,\ldots,1-w_{km},\underbrace{0,\ldots,0}_m)^\top=(\bfu_1^\top,\ldots,\bfu_{k+1}^\top)^\top,\\
        v&=(k/(k+1)-w_1,\ldots,k/(k+1)-w_{km},k/(k+1),\ldots,k/(k+1))^\top=(\bfv_1^\top,\ldots,\bfv_{k+1}^\top)^\top,
    \end{align*}
    for $\bfu_i,\bfv_i\in\real_+^{m}$. Similarly partition $w$ as $w=(\bfw_1^\top,\ldots,\bfw_{k}^\top)^\top$. To find $w$ such that $u\sim v$, without loss of generality, equating $\bfu_1$ to $\bfv_{k+1}$ yields that $\bfw_1=1/(k+1)\mathbf{1}_m$. Now noting that $\bfv=(k-1)/(k+1)\mathbf{1}_m$, equate $\bfu_2$ to $\bfv_1$, and obtain $\bfw_2=2/(k+1)\mathbf{1}_m$. Repeating the procedure $\bfu_{i}=\bfv_{i-1}$ to obtain $\bfw_i$ for $i=3,\ldots,k$, we have $\bfw_{i}=i/(k+1)$. Since $\bfu_{k+1}=\bfv_{k}=\mathbf{0}_m$, $u\sim v$. \\

    It remains to prove the only if part. Since $p_2\leq p_1$, $p_2=p_1$ or $p_2<p_1$. Because the proof for the former case follows from that for if part, we assume $p_2<p_1$. Since $p_1-p_2\leq p_2$, there uniquely exist $q,r\in\mathbb{N}\cup\set{0}$ such that $p_2=(p_1-p_2)q+r$ for $q\geq 1$ and $r\in\set{0,\ldots,p_1-p_2-1}$. Rearranging the terms, $p_2=(p_1q+r)/(1+q)$. Assume $0<r$, and thus $p_1-p_2\nmid p_2$. Slightly modifying the notations in the proof of if part, define the vectors 
    \begin{align*}
        \bfu_j&=(1-w_{(j-1)(p_1-p_2)+1},\ldots,1-w_{j(p_1-p_2)})^\top,\\
        \bfv_j&=(p_2/p_1-w_{(j-1)(p_1-p_2)+1},\ldots,p_2/p_1-w_{j(p_1-p_2)+1})^\top,\\
    \end{align*}
    for $j=1,\ldots,q$. Also, let 
    \begin{align*}
        \bfu_{q+1}&=(1-w_{q(p_1-p_2)+1},\ldots,1-w_{p_2})^\top,\quad \bfv_{q+1}=(p_2/p_1-w_{q(p_1-p_2)+1},\ldots,p_2/p_1-w_{p_2})^\top,\\
        \bfu_{q+2}&=(\underbrace{0,\ldots,0}_{p_1-p_2})^\top,\quad \bfv_{q+2}=(\underbrace{p_2/p_1,\ldots,p_2/p_1}_{p_1-p_2})^\top.
    \end{align*}
Write $w=(\bfw_1^\top,\ldots,\bfw_q^\top,\bfw_{q+1}^\top)^\top$, where $\bfw_1,\ldots,\bfw_q\in\real^{p_1-p_2}$ and $\bfw_{q+1}\in\real^{r}$. Following the iterative procedure described in the proof of if part, equating $\bfu_1$ to $\bfv_{q+2}$ yields that $\bfw_1=(1-p_2/p_1)\mathbf{1}_{p_1-p_2}$. This gives $\bfv_1=(2p_2/p_1-1)\mathbf{1}_{p_1-p_2}$. If $2p_2/p_1-1=0$, $p_1-p_2=p_2$, which contradicts the assumption that $r>0$. Also, if $2p_2/p_1-1=p_2/p_1$, $p_1=p_2$, which is again a contradiction. Thus, $\bfu_1=\bfv_{q+2}$, $\bfv_1$, and $\bfu_{q+2}$ are mutually different, and so we can proceed the iterative procedure by equating $\bfu_2$ to $\bfv_1$. This gives $\bfw_2=(2p_2/p_1-2)\mathbf{1}_{p_1-p_2}$ and so $\bfv_2=(3p_2/p_1-2)\mathbf{1}_{p_1-p_2}$. Iterating this procedure for $q$ times, one obtain $\bfv_i=((i+1)p_2/p_1-i)\mathbf{1}_{p_1-p_2}$ for $i=1,\ldots,q$. Clearly, $\bfv_1,\ldots,\bfv_q$ are distinct. Also, $\bfv_i\neq 0$ since otherwise $p_1-p_2|p_2$, which again contradicts the assumption. Hence, at $q$th step of this iterative procedure, we have $2p_2/p_1-1,\ldots,(q+1)p_2/p_1-q$ as unique elements of $v$ for which values of $w$ have been assigned. On the other hand, $p_2/p_1,\ldots,qp_2/p_1-(q-1)$, are unique elements of $u$ for which values of $w$ have been assigned. So assume $1-w_{q(p_1-p_2)+1}=(q+1)p_2/p_1-q$ without loss of generality. Then $w_{q(p_1-p_2+1}=(q+1)(1-p_2/p_1)>0$, which gives $p_2/p_1-w_{q(p_1-p_2+1}=(q+2)p_2/p_1-(q+1)$. We claim that $p_2/p_1<(q+1)/(q+2)$, which yields the contradiction as $\bfv\notin \real_{+}^{p_1}$. Since $q$ is a quotient, we have $p_2-(p_1-p_2)(q+1)<0$ or $0\leq p_2/(p_1-p_2)-1=(2p_2-p_1)/(p_1-p_2)<q$. The last inequality holds because $p_1-p_2\nmid p_2$. This implies that $(q+1)/(q+2)>p_2/p_1$. Hence, we have $r=0$ and thus $p_1-p_2|p_2$. Then the desired $w$ can be found as in the proof of if part.
\end{proof}

\begin{proof}[\textbf{Proof of Theorem \ref{sec4.1.thm1}.}]
By Lemma \ref{secA.1.lemma2}, we seek to find $A_1,\ldots,A_r$ satisfying 
\begin{align}\label{sec4.1.thm1.proof.eq1}
    \sum_{i=1}^r A_iA_i^\top=p_2 I_{p_1},\quad \sum_{i=1}^r A_i^\top A_i=p_1 I_{p_2}
\end{align}
for given $(p_1,p_2,r)$. Recall that $A=[\text{vec}(A_1),\ldots,\text{vec}(A_r)]$.

\begin{itemize}
    \item $(p_1,p_2,r)=(p_1,p_1,2),((k+1)m,km,2)$: In this case, we have that $p_2\leq p_1\leq 2p_2$. For the latter case, $(p_1,p_2)=((k+1)m,km)$ for some $k,m\in\bbN$ if and only if $p_1>p_2$ and $p_1-p_2|p_1$. We shall rewrite $A_i=\sqrt{p_1}W_i$ so that (\ref{sec4.1.thm1.proof.eq1}) implies that 
\begin{align}\label{sec4.1.thm1.proof.eq2}
\begin{split}
            W_1W_1^\top+W_2W_2^\top &=p_2/p_1I_{p_1},\\
    W_1^\top W_1+W_2^\top W_2 &=I_{p_2}.
\end{split}
\end{align}

Suppose $W_i$ has a SVD $U_i\Sigma_iV_i^\top$, where $U_i\in\calO_{p_1}$, $V_i\in\calO_{p_2}$, and
\begin{align*}
    \Sigma_i=\left[\begin{array}{cccc}
    \sigma_{1,i} & & & \\
    & \ddots  & & \\
    & & \sigma_{p_2,i} & \\
    & & & 
    \end{array}\right]\in\real^{p_1\times p_2}
\end{align*}
for $\sigma_{1,i}\geq \cdots \geq \sigma_{p_2,i}\geq 0$. Let $U=U_1^\top U_2\in\calO_{p_1}$ and $V=V_1^\top V_2\in\calO_{p_2}$. Then the first equation of (\ref{sec4.1.thm1.proof.eq2}) yields that 
\begin{align}\label{sec4.1.thm1.proof.eq3}
\begin{split}
       &\text{diag}(\sigma_{1,1}^2,\ldots,\sigma_{p_2,1}^2,\underbrace{0,\ldots,0}_{p_1-p_2})+U\text{diag}(\sigma_{1,2}^2,\ldots,\sigma_{p_2,2}^2,\underbrace{0,\ldots,0}_{p_1-p_2})U^\top=p_2/p_1I_{p_1}\\
  \Rightarrow & U\text{diag}(\sigma_{1,2}^2,\ldots,\sigma_{p_2,2}^2,\underbrace{0,\ldots,0}_{p_1-p_2})U^\top=\text{diag}(p_2/p_1-\sigma_{1,1}^2,\ldots,p_2/p_1-\sigma_{p_2,1}^2,\underbrace{p_2/p_1,\ldots,p_2/p_1}_{p_1-p_2}).
\end{split}
\end{align}
Similarly, one can obtain that 
\begin{align}\label{sec4.1.thm1.proof.eq4}
    V\text{diag}(\sigma_{1,2}^2,\ldots,\sigma_{p_2,2}^2)V^\top=\text{diag}(1-\sigma_{1,1}^2,\ldots,1-\sigma_{p_2,1}^2).
\end{align}
Comparing the eigenvalues, one can deduce from (\ref{sec4.1.thm1.proof.eq3})--(\ref{sec4.1.thm1.proof.eq4}) that 
\begin{align*}
    (\sigma_{1,2}^2,\ldots,\sigma_{p_2,2}^2,\underbrace{0,\ldots,0}_{p_1-p_2})&\sim (p_2/p_1-\sigma_{1,1}^2,\ldots,p_2/p_1-\sigma_{p_2,1}^2,\underbrace{p_2/p_1,\ldots,p_2/p_1}_{p_1-p_2}),\\
    (\sigma_{1,2}^2,\ldots,\sigma_{p_2,2}^2)&\sim (1-\sigma_{1,1}^2,\ldots,1-\sigma_{p_2,1}^2),
\end{align*}
which implies
\begin{align}\label{sec4.1.thm1.proof.eq5}
        (1-\sigma_{1,1}^2,\ldots,1-\sigma_{p_2,1}^2,\underbrace{0,\ldots,0}_{p_1-p_2})&\sim (p_2/p_1-\sigma_{1,1}^2,\ldots,p_2/p_1-\sigma_{p_2,1}^2,\underbrace{p_2/p_1,\ldots,p_2/p_1}_{p_1-p_2}).
\end{align}
By Lemma \ref{secA.2.lemma2}, we should have that either $p_1=p_2$ or $p_1>p_2$ and $p_1-p_2|p_2$, which is equivalent to the existence of $k,m\in\bbN$ such that $(p_1,p_2)=((k+1)m,km)$. \\

Suppose $p_1=p_2$. Then by Lemma \ref{secA.2.lemma2}, we can take any $\sigma_{1,1}\geq \cdots\geq \sigma_{p_2,1}$ on $[0,1]$. Let $1\geq d_1>\cdots>d_k\geq 0$ be distinct values of $\sigma_{i,1}$'s with each multiplicity of $n_k$. By Lemma \ref{secA.2.lemma1}, we have that 
\begin{align*}
    (\sigma_{1,2},\ldots,\sigma_{p_2,2})=(\underbrace{\sqrt{1-d_k^2},\ldots,\sqrt{1-d_k^2}}_{n_k},\ldots,\underbrace{\sqrt{1-d_1^2},\ldots,\sqrt{1-d_1^2}}_{n_1}).
\end{align*}
Also, partitioning $I_{p_1}$ as $[M_1,\cdots,M_k]$, where each block has $n_i$ columns, there exist $Y_i,Z_i\in\calO_{n_i}$ such that 
\begin{align*}
    U&=U_1^\top U_2=[M_k,\ldots,M_1]\bigoplus_{j=1}^k Y_{k+1-j},\\
    V&=V_1^\top V_2=[M_k,\ldots,M_1]\bigoplus_{j=1}^k Z_{k+1-j}.
\end{align*}
Fixing $U_1$ and $V_1$, 
\begin{align*}
    U_2&=U_1[M_k,\ldots,M_1]\bigoplus_{j=1}^k Y_{k+1-j}=[U_k^1,\dots,U_1^1]\bigoplus_{j=1}^k Y_{k+1-j},\\
    V_2&=V_1[M_k,\ldots,M_1]\bigoplus_{j=1}^k Z_{k+1-j}=[V_k^1,\dots,V_1^1]\bigoplus_{j=1}^k Z_{k+1-j}.
\end{align*}
Combining the results above,
\begin{align*}
    W_1&=U_1\Sigma_1 V_1^\top=\sum_{j=1}^k d_j U_j^1(V_j^1)^\top,\\
    W_2&=U_2\Sigma_2 V_2^\top\\
    &=[U_k^1,\dots,U_1^1]\bigoplus_{j=1}^k Y_{k+1-j}\bigoplus_{j=1}^k \sqrt{1-d_{k+j-1}^2}I_{n_{k+j-1}}\bigoplus_{j=1}^k Z_{k+1-j}^\top [V_k^1,\dots,V_1^1]^\top\\
    &=[U_k^1,\dots,U_1^1]\parentheses{\bigoplus_{j=1}^k  \sqrt{1-d_{k+j-1}^2} O_{k+1-j}}[V_k^1,\dots,V_1^1]^\top\\
    &=\sum_{j=1}^k \sqrt{1-d_j^2}U_j^1 O_j(V_j^1)^\top,
\end{align*}
where $O_j=Y_jZ_j^\top\in\calO_{n_j}$.\\

Now we seek the condition on $d_j$ and $O_j$ for which $A=\sqrt{p_1}[\text{vec}(W_1),\text{vec}(W_2)]$ is of full-column rank. Since $\text{rank}(A)=\text{rank}(A^\top A)$, it suffices to find the condition such that $\text{det}(A^\top A)\neq 0$. We claim that such condition is either $k\geq 2$ or $k=1,0<d_1<1$, and $O_1\neq \pm I_{p_1}$. Note that 
\begin{align}\label{sec4.1.thm1.proof.eq6}
\begin{split}
    \text{det}(A^\top A)&=\left[p_1(1-\lambda)\right]^2\left[\text{vec}(W_1)^\top \text{vec}(W_1)\cdot \text{vec}(W_2)^\top \text{vec}(W_2)-(\text{vec}(W_1)^\top \text{vec}(W_2))^2\right]\\
    &=\left[p_1(1-\lambda)\right]^2\left[\sum_{j=1}^kn_jd_j^2\cdot \sum_{j=1}^k n_j(1-d_j^2)-\parentheses{\sum_{j=1}^k d_j\sqrt{1-d_j^2}\tr(O_j)}^2\right]\\
    &\geq \left[p_1(1-\lambda)\right]^2\left[\sum_{j=1}^kn_jd_j^2\cdot \sum_{j=1}^k n_j(1-d_j^2)-\parentheses{\sum_{j=1}^k n_j d_j\sqrt{1-d_j^2}}^2\right]\\ 
    &\geq 0,
\end{split}
\end{align}
where the last two inequalities hold due to Cauchy-Schwartz inequality. Let $a=(a_i),b=(b_i)\in\real^k$, where $a_i=\sqrt{n_i}d_i$ and $b_i=\sqrt{n_i}\sqrt{1-d_i^2}$. Suppose $k=2$. The equality condition for Cauchy-Schwartz inequality implies that the equality in the last inequality of (\ref{sec4.1.thm1.proof.eq6}) holds if $a_i=tb_i$ for some $t\in\real$. However, since $d_1>\cdots >d_k$ and $k\geq 2$, such $t$ does not exist. Indeed, consider $d_1$ and $d_k$. If $d_1=1$ or $d_k=0$, it is easy to see that $a_1\neq tb_1$ and $a_k\neq tb_k$ for any $t$. On the other hand, if $1>d_1>d_k>0$, if $a_1=tb_1$ and $a_k=tb_k$ for some $t$, $d_1/\sqrt{1-d_1^2}=d_k/\sqrt{1-d_k^2}=t$. This is impossible because the map $x\mapsto x/\sqrt{1-x^2}$ is strictly increasing on $(0,1)$. Hence, whenever $k\geq 2$, $\text{det}(A^\top A)>0$. Now suppose $k=1$. Then
\begin{align*}
      \text{det}(A^\top A)&=[p_1(1-\lambda)]^2d_1^2(1-d_1^2)(p_1^2-\text{tr}(O_1)^2).
\end{align*}
Obviously, we need that $d_1\in(0,1)$. Also, by Cauchy-Schwartz inequality, $p_2^2\geq \text{tr}(O_1)^2$ and the equality holds only when $O_1=\pm I_{p_1}$. This concludes the proof for the case when $p_1=p_2$. \\

The remaining case is when $p_1>p_2$ and $p_1-p_2|p_2$. Note that $k$ and $m$ in the proof of Lemma \ref{secA.2.lemma2}, are in fact $\text{gcd}(p_1,p_2)$ and $p_2/k$. Also, applying Lemma \ref{secA.2.lemma2} to (\ref{sec4.1.thm1.proof.eq3}) and (\ref{sec4.1.thm1.proof.eq4}) yields that 
\begin{align*}
    (\sigma_{1,1}^2,\ldots,\sigma_{p_2,1}^2)&=(\sigma_{1,2}^2,\ldots,\sigma_{p_2,2}^2)=(k/(k+1)\mathbf{1}_m^\top,\ldots,1/(k+1)\mathbf{1}_m^\top).
\end{align*}
As an analogy to the proof when $p_1=p_2$, if we partition $I_{p_1}$ and $I_{p_2}$ as $[M_1,\ldots,M_{k+1}]$ and $[N_1,\ldots,N_k]$, respectively, where each block has $m$ columns, one can obtain that 
\begin{align*}
    U_2&=U_1[M_{k+1},\ldots,M_1]\bigoplus_{j=1}^{k+1}Y_{k+2-j}=[U_{k+1}^1,\ldots,U_{1}^2]\bigoplus_{j=1}^{k+1}Y_{k+2-j},\\
    V_2&=V_1[N_1,\ldots,N_k]\bigoplus_{j=1}^k X_{k+1-j}=[V_k^1,\ldots,V_1^1]\bigoplus_{j=1}^k X_{k+1-j},
\end{align*}
where $Y_i,X_j\in\calO_m$. Hence, 
\begin{align*}
    W_1&=\sum_{j=1}^k\sqrt{\frac{k+1-j}{k+1}}U_j^1(V_j^1)^\top,\quad W_2=\sum_{j=2}^{k+1} \sqrt{\frac{j-1}{k+1}}U_{j}^1O_{j}(V_{j-1}^1)^\top,
\end{align*}
where $O_j\in\calO_m$. Note that $A$ is of full-column rank because 
\begin{align*}
    \text{vec}(W_1)^\top \text{vec}(W_2)&=\tr\parentheses{W_1^\top W_2}\\
    &=\tr\parentheses{\sum_{j=2}^k \sqrt{\frac{(k+1-j)(j-1)}{(k+1)^2}}V_j^1 O_j(V_{j-1}^1)^\top}\\
    &=\tr\parentheses{\sum_{j=2}^k \sqrt{\frac{(k+1-j)(j-1)}{(k+1)^2}} O_j(V_{j-1}^1)^\top V_j^1}=0.
\end{align*}
    \item $(p_1,p_2,r)=(p_2r,p_2,r)$: Letting $A_R=[A_1,\ldots,A_R]$, $A_R$ is a $p_1\times p_1$ matrix and(\ref{sec4.1.thm1.proof.eq1}) implies that 
    \begin{align*}
        A_RA_R^\top=p_2 I_{p_1}
    \end{align*}
    so that $A_R=\sqrt{p_2}O$ for some $O\in\calO_{p_1}$. If $O$ is partitioned as $[O_1,\ldots,O_r]$, where each block has $p_2$ columns, $A_i=\sqrt{p_2}O_i$. Then 
    \begin{align*}
        \sum_{i=1}^r A_i^\top A_i=\sum_{i=1}^r p_2 I_{p_2}=rp_2I_{p_2}=p_1I_{p_2}.
    \end{align*}
\end{itemize}
\end{proof}

\begin{proof}[\textbf{Proof of Corollary \ref{sec4.1.cor1}.}]
Since $AA^\top$ is of rank$-r$, it is straightforward to see that $\lambda_{r+1}(C)=\cdots=\lambda_p(C)=\lambda$. Also, the spiked eigenvalues of $C$ are given as $\lambda_i(C)=(1-\lambda)\lambda_i(AA^\top)+\lambda I_p$ for $i=1,\ldots,r$. Because $\lambda_i(AA^\top)=\lambda_i(A^\top A)$ for $i=1,\ldots,r$, we investigate the first $r$ eigenvalues of $A^\top A$ for each scenario of $(p_1,p_2,r)$.
\begin{itemize}
    \item $(p_1,p_2,r)=(p_1,p_1,2),((k+1)m,km,2)$: Recall $A=\sqrt{p_1}[\text{vec}(W_1),\text{vec}(W_2)]$ in the proof of Theorem \ref{sec4.1.thm1}. Also, $AA^\top$ and $A^\top A$ share the same $2$ non-zero eigenvalues. For $2\times 2$ symmetric matrix $B$, the eigenvalues of $B$ are $(\text{tr}(B)\pm\sqrt{\text{tr}(B)^2-4\text{det}(B)})/2$. If $(p_1,p_2)=((k+1)m,km)$, $w_1=\text{vec}(W_1)$ and $w_2=\text{vec}(W_2)$ are orthogonal, and $||w_1||_2^2=||w_2||_2^2=p_2/2$, and so the conclusion follows since $\text{tr}(A^\top A)=2p_1||w_1||_2^2$ and $\det(A^\top A)=p_1^2||w_1||_2^4$. Lastly, if $p_1=p_2$, using $\beta=\det(A^\top A)$ in (\ref{sec4.1.thm1.proof.eq6}) and noting that $\text{tr}(A^\top A)=p_1^2$, the result follows.

\item $(p_1,p_2,r)=(p_2r,p_2,r)$: Recall from the proof of Theorem \ref{sec4.1.thm1}, if $O$ is partitioned as $[O_1,\ldots,O_r]$, where each block has $p_2$ columns, $A_i=\sqrt{p_2}O_i$. Consequently, 
\begin{align*}
    (A^\top A)_{ij}=p_2 \text{vec}(O_i)^\top\text{vec}(O_j)=p_2^2\delta_{ij}
\end{align*}
and thus $A^\top A=p_2^2 I_{r}$. Note that $p_2^2=p/r$ if $(p_1,p_2,r)=(p_2r,p_2,r)$.
\end{itemize}

\end{proof}

\subsection{Proofs of the Results from Section \ref{sec4.2:spec.equiv}}\label{secA.3}
We prove Theorem \ref{sec4.2.thm1}, generalizing the result of \cite{oliveira2026}. We first review the proof strategy from \cite{oliveira2026} in Section \ref{secA.3.1} and then prove Theorem \ref{sec4.2.thm1} in Section \ref{secA.3.2}. The proofs of the ancillary results to prove Theorem \ref{sec4.2.thm1} are deferred to Section \ref{secA.3.3}.\\

Suppose $\Omega=\Omega_2\otimes \Omega_1\in\calS_{p_1,p_2}^+$. Note that each factor in $\Omega$ is identifiable only up to a constant factor. Namely, $\Omega_2\otimes \Omega_1=(c\Omega_2)\otimes (\Omega_1/c)$ for any constant $c>0$. Thus, we will identify $(\Omega_1,\Omega_2)$ through the constraint that $|\Omega_1|^{1/p_1}=|\Omega_2|^{1/p_2}$. Also, we write $p_{\max}=\max\set{p_1,p_2}$ and $p_{\min}=\min\set{p_1,p_2}$.

\subsubsection{Review of the Proof Strategy from \cite{oliveira2026}}\label{secA.3.1}
We review the proof strategy from \cite{oliveira2026}. Define the objective function $f$ on $\calS_{p_1,p_2}^+$ by 
\begin{align}\label{secA.3.1.eq1}
    f(\Theta_2\otimes \Theta_1)=\frac{1}{np}\sum_{i=1}^n\tr\parentheses{\Theta_1Y_i\Theta_2 Y_i^\top}-\frac{1}{p_2}\log|\Theta_2|-\frac{1}{p_1}\log |\Theta_1|.
\end{align}
for given random matrices $Y_1,\ldots,Y_n\in\real^{p_1\times p_2}$ generated according to the model in (\ref{sec4.2.eq1}). Note that $f(\Theta_2\otimes \Theta_1)\equiv d(\Theta_2^{-1}\otimes \Theta_1^{-1}|\rho)$, where $\rho=1/n\sum_{i=1}^n y_iy_i^\top$ for $y_i=\text{vec}(Y_i)$ and $d(\cdot|\cdot)$ is that defined in (\ref{sec2.2.eq1}). Suppose 
\begin{align*}
    \bbP=\set{\Theta=\nu(\Theta_2\otimes \Theta_1):\nu>0,|\Theta_1|^{1/p_1}=|\Theta_2|^{1/p_2}=1}\equiv \calS_{p_1,p_2}^{++}.
\end{align*}
Under affine-invariant geometry, the function $f$ in (\ref{secA.3.1.eq1}) is geodesically convex over $\bbP$. By the equivariance of the Kronecker MLE as in Section \ref{sec2.2:KCD}, it suffices to study the convergence when $k(\Sigma)=I_p$. Given $Y_1,\ldots,Y_n$, let $\nabla f$ denote the Riemannian gradient of $f$ at $\Theta=\Theta_2\otimes \Theta_1=I_p$ under affine-invariant metric \cite{skovgaard1984,pennec2006,moakher2005}.\\

Note that the Kronecker MLE (separable component) $\hat{K}(Y)=\hat{K}_1(Y)\otimes \hat{K}_2(Y)$ based on $Y_1,\ldots,Y_n$ is $\hat{\Theta}^{-1}(Y)=\hat{\Theta}_2^{-1}(Y)\otimes \hat{\Theta}_1^{-1}(Y)$, where 
\begin{align*}
   \hat{\Theta}(Y)=\hat{\Theta}_2(Y)\otimes \hat{\Theta}_1(Y):=\argmin_{\Theta_2\otimes \Theta_1\in\bbP}f(\Theta_2\otimes \Theta_1).
\end{align*}
We shall derive the non-asymptotic concentration inequality for $d_{op}(\hat{\Theta}_i(Y),I_{p_i})$ for the following reasons. First, note that
\begin{align}\label{secA.3.1.eq2}
\begin{split}
        d_{op}(\hat{\Theta}(Y),I_p)&=||\hat{\Theta}_2(Y)\otimes\hat{\Theta}_1(Y)-I_p||_2\\
    &=||(\hat{\Theta}_2(Y)-I_{p_2})\otimes(\hat{\Theta}_1(Y)-I_{p_1})+(\hat{\Theta}_2(Y)-I_{p_2})\otimes I_{p_1}+I_{p_2}\otimes (\hat{\Theta}_1(Y)-I_{p_1})||_2\\
    &\leq ||\hat{\Theta}_2(Y)-I_{p_2}||_2||\hat{\Theta}_1(Y)-I_{p_1}||_2+ ||\hat{\Theta}_2(Y)-I_{p_2}||_2+||\hat{\Theta}_1(Y)-I_{p_1}||_2.
\end{split}
\end{align}
Therefore, if $||\hat{\Theta}_2(Y)-I_{p_2}||_2,||\hat{\Theta}_1(Y)-I_{p_1}||_2=O(a_n)$ for some positive sequence $a_n$ converging to $0$ with high probability, then $d_{op}(\hat{\Theta}(Y),I_p)=O(a_n)$ with the same probability guarantee. Moreover, with the same probability guarantee, we have that 
\begin{align}\label{secA.3.1.eq3}
\begin{split}
       d_{op}(\hat{K}(Y),I_p)&=||\hat{K}(Y)-I_p||_2=||\hat{\Theta}^{-1}(Y)-I_p||_2\\
   &\leq ||\hat{\Theta}(Y)-I_p||_2/||\hat{\Theta}(Y)||_2\leq ||\hat{\Theta}(Y)-I_p||_2/(1-||\hat{\Theta}(Y)-I_p||_2)=O(a_n)
\end{split}
\end{align}
for sufficiently large $n$.\\

We derive the non-asymptotic concentration inequality for $d_{op}(\hat{\Theta}_i(Y),I_{p_i})$ built upon the proof techniques from \cite{oliveira2026}. To this end, we introduce the relevant notions from \cite{oliveira2026}. Define the following sets:
\begin{align*}
\calS_p:=\set{X\in\real^{p\times p}:X=X^\top},\quad \calS_p^0:=\set{X\in\calS_p:\tr(X)=0}.    
\end{align*}
Now we introduce $(\delta,\eta)-$quantum expander (see \cite{oliveira2026} also).
\begin{definition}[\textbf{$(\delta,\eta)$-Quantum Expander}]\label{secA.3.1.def1}
Given $Y=(Y_1,\ldots,Y_n)\in(\real^{p_1\times p_2})^n$, define the quantum expanders
\begin{align*}
    &\Phi_Y^{(1,2)}:Z\in \calS_{p_2}\rightarrow \frac{1}{np}\sum_{i=1}^n Y_iZY_i^\top\in  \calS_{p_1},\\
    &\Phi_Y^{(2,1)}:Z\in \calS_{p_1}\rightarrow \frac{1}{np}\sum_{i=1}^n Y_i^\top Z Y_i\in \calS_{p_2}.
\end{align*}
Note that $\Phi_Y^{(2,1)}$ is an adjoint operator of $\Phi_Y^{(1,2)}$. Let $\calS_q^0$ be the set of $q\times q$ symmetric matrices of trace $0$. Define the norm $||\cdot||_0$ of $\Phi_Y^{(1,2)}$ and $\Phi_Y^{(2,1)}$ by 
\begin{align*}
||\Phi_Y^{(1,2)}||_0&=\sup_{H_i\in\calS_{p_i}^0,||H_i||_F=1}<H_1,\Phi_Y^{(1,2)}(H_2)>,\\
||\Phi_Y^{(2,1)}||_0&=\sup_{H_i\in\calS_{p_i}^0,||H_i||_F=1}<H_2,\Phi_Y^{(2,1)}(H_1)>,
\end{align*}
As an adjoint operator, we see that $||\Phi_Y^{(1,2)}||_0=||\Phi_Y^{(2,1)}||_0$. Here $<A,B>=\tr(A^\top B)$. We say the operator $\Phi_Y^{(1,2)}$ and $\Phi_Y^{(2,1)}$ are $\delta-$doubly balanced if
\begin{align}\label{secA.3.1.def1.eq1}
    \norm{\frac{\Phi_Y^{(1,2)}(I_{p_2})}{\tr(\Phi_Y^{(1,2)}(I_{p_2}))}-\frac{I_{p_1}}{p_1}}\leq \frac{\delta}{p_1},\quad \norm{\frac{\Phi_Y^{(2,1)}(I_{p_1})}{\tr(\Phi_Y^{(2,1)}(I_{p_1}))}-\frac{I_{p_2}}{p_2}}\leq \frac{\delta}{p_2}
\end{align}
and $(\delta,\eta)-$quantum expander if the above and 
\begin{align}\label{secA.3.1.def1.eq2}
  ||\Phi_Y^{(1,2)}||_0, ||\Phi_Y^{(2,1)}||_0 \leq \eta\frac{\tr(\Phi_Y^{(1,2)}(I_{p_2}))}{\sqrt{p_1p_2}}=\eta\frac{\tr(\Phi_Y^{(2,1)}(I_{p_1}))}{\sqrt{p_1p_2}}=\eta\frac{\tr(\rho)}{\sqrt{p_1p_2}}
\end{align}
are satisfied. We say $\Phi_Y^{(1,2)}$ has a spectral gap $\gamma\in(0,1)$ if 
\begin{align*}
    \sigma_2(\Phi_Y^{(1,2)})\leq (1-\gamma)\frac{\tr(\Phi^{(1,2)}(I_{p_2}))}{\sqrt{p_1p_2}}.
\end{align*}
The analogue of the above can be defined for $\Phi_Y^{(2,1)}$.\\
\end{definition}

Note that by vectorization,
\begin{align}\label{secA.3.1.eq4}
||\Phi_Y^{(1,2)}||_0,||\Phi_Y^{(2,1)}||_0&\leq \max\set{\frac{1}{np}||(\sum_{i=1}^n Y_i\otimes Y_i)\circ \Pi_2||_2, \frac{1}{np}||(\sum_{i=1}^n Y_i^\top\otimes Y_i^\top)\circ \Pi_1||_2}
\end{align}
for the orthogonal projection operators $\Pi_1,\Pi_2$ defined as
\begin{align*}
    \Pi_i=I_{p_i^2}-\frac{1}{p_i}\text{vec}(I_{p_i})\text{vec}(I_{p_i})^\top.
\end{align*}

In view of the proof of Theorem $1.11$ from \cite{oliveira2026}, it suffices to show the followings:
\begin{itemize}
    \item[]\hypertarget{cond1}{(\rom{1})} $\nabla f$ is small with high probability. 
    \item[]\hypertarget{cond2}{(\rom{2})} $\Phi_Y^{(1,2)}$ and $\Phi_Y^{(2,1)}$ are $(\delta,\eta)-$quantum expanders for $\delta,\eta\in(0,1)$ with high probability.
\end{itemize}
The following lemmas illustrate why it suffices to show (\rom{1})--(\rom{2}).

\begin{lemma}[Lemma E.2 of \cite{oliveira2026.supp}]\label{secA.3.1.lemma1}
    There exists a universal constant $c>0$ with the following property. If $\Phi$ is an $(\epsilon,\eta)$-quantum expander and $\epsilon\leq c(1-\eta)$, then $\Phi$ has spectral gap $1-\eta-O(\epsilon)$.
\end{lemma}

\begin{lemma}[Lemma E.9 of \cite{oliveira2026.supp}]\label{secA.3.1.lemma2}
    There is a constant $c>0$ with the following property : let $X=(X_1,\ldots,X_n), W =(W_1\ldots,W_n)\in(\real^{p_1\times p_2})^n$ such that $W_i=X_iR$ for some $R\in GL_{p_2}$. Let $0<\epsilon,\eta<1$. If $\Phi_X$ is an $(\epsilon,\eta)$-quantum expander and $||R^\top R-I_{p_2}||_2\leq \delta$ for some $\delta\leq c$, then $\Phi_W$ is an $(\epsilon+O(\delta),\eta+O(\delta))-$quantum expander.
\end{lemma}

\begin{lemma}[Corollary E.4 of \cite{oliveira2026.supp}]\label{secA.3.1.lemma3}
    There is a universal constant $C>0$ such that the following holds. Let $\epsilon,\gamma\in(0,1)$, $1<p_1,p_2$, and suppose the map $\Phi_W^{(1,2)}$ ($\Phi_W^{(2,1)}$ ) is $\epsilon$-doubly balanced and has spectral gap $\gamma$, where $\gamma^2\geq C\epsilon\log p_{\min}$. Further assume that $\tr\parentheses{\sum_{i=1}^n W_iW_i^\top}=np$. Then the MLE $\hat{\Theta}:=\hat{\Theta}(W)=\hat{\Theta}_1(W)\otimes \hat{\Theta}_2(W)$ exists, is unique, and satisfies 
    \begin{align*}
        \max\set{||\hat{\Theta}_1(W)-I_{p_1}||_2,||\hat{\Theta}_2(W)-I_{p_2}||_2}\leq O\parentheses{\frac{\epsilon \log p_{\min}}{\gamma}}
    \end{align*}
\end{lemma}
To outline the proof strategy, note that (\rom{1}) will be used to show that (\rom{2}) is a $(\delta,\eta)-$quantum expander with high probability. For random matrices $Y_1,\ldots,Y_n$ generated according to (\ref{sec4.2.eq1}), let $R=\parentheses{\frac{1}{np_1}\sum_{i=1}^n Y_i^\top Y_i}^{-1/2}$, which is indeed non-singular if $p_{\max}/p_{\min}<n$ with probability $1$. Let $Z_i=Y_iR$ and $Z=(Z_1,\ldots,Z_n)$. Suppose $\Phi_Y^{(1,2)}$ is an $(\epsilon,\eta)-$quantum expander and $||R^2-I_p||_2=O(\delta)$ is small enough. Then Lemma \ref{secA.3.1.lemma2} implies that $\Phi_Z$ is an $(\epsilon+O(\delta),\eta+O(\delta))-$quantum expander. If $\epsilon$ is so small enough, then we see $\Phi_Z$ has a spectral gap $1-\eta-O(\epsilon+O(\delta))$. Since $\tr\parentheses{\sum_{i=1}^n Z_iZ_i^\top}=np$, Lemma \ref{secA.3.1.lemma3} implies that $\hat{\Theta}(Z)$ exists. By the equivariance of the Kronecker MLE, one can observe that $\hat{\Theta}_1(Y)=\hat{\Theta}_1(Z)$ and $\hat{\Theta}_2(Y)=R\hat{\Theta}_2(Z)R$. If $||R^2-I_p||_2$ and $\epsilon \log p_{\min}/\gamma$ in Lemma \ref{secA.3.1.lemma2}--\ref{secA.3.1.lemma3} are small enough, we consequently have that 
\begin{align*}
    \max\set{||\hat{\Theta}_1(Y)-I_{p_1}||_2,||\hat{\Theta}_2(Y)-I_{p_2}||_2}\leq O\parentheses{\frac{\epsilon \log p_{\min}}{\gamma}} 
\end{align*}
as a consequence of Lemma $A.5$ of \cite{oliveira2026.supp} (see the proof of Theorem $3.1$ of \cite{oliveira2026} also). Therefore, we aim to prove (\rom{1})--(\rom{2}) with $Y_i$'s generated according to (\ref{sec4.2.eq1}).\\ 

Finally, we introduce some additional preliminaries and notations. We first introduce the formula of $\nabla f$, which is available from Lemma $2.9$ of \cite{oliveira2026}. Write 
\begin{align*}
    \nabla f=\left[\begin{array}{c}
    \nabla_0 f \\
    \nabla_1 f\\
    \nabla_2 f 
    \end{array}\right].
\end{align*}
Also, write
\begin{align*}
 \bbH:=\set{(H_0;H_1,H_2):H_0\in\real, H_i\in \calS_{p_i}^0}.
\end{align*}
Here $\bbH$ is a tangent space of $\bbP$ at $\Theta=I_p$. Lastly, recall that $\rho=\frac{1}{np}\sum_{i=1}^n y_iy_i^\top$ for $y_i=\text{vec}(Y_i)$, and define
\begin{align*}
    \rho^{(1)}&=\frac{1}{np}\sum_{i=1}^n Y_iY_i^\top, \quad   \rho^{(2)}=\frac{1}{np}\sum_{i=1}^n Y_i^\top Y_i.
\end{align*}
Note that $\rho^{(1)}$ and $\rho^{(2)}$ are partial traces of $\rho$. Also, the formula of $\nabla f$ is given as follows:
\begin{align*}
    \nabla_0 f&=\tr(\rho)-1, \nabla_i f=\sqrt{p_i}\parentheses{\rho^{(i)}-\frac{\tr(\rho)}{p_i}I_{p_i}} \text{ for }i=1,2.
\end{align*}

\subsubsection{Proof of Theorem \ref{sec4.2.thm1}}\label{secA.3.2}
We first introduce the ancillary results to prove Theorem \ref{sec4.2.thm1}, namely Lemma \ref{secA.3.2.lemma1}--\ref{secA.3.2.lemma2}.\\

The following proves \hyperlink{cond1}{(\rom{1})} of Section \ref{secA.3.1}, namely the small gradient with high probability.
\begin{lemma}\label{secA.3.2.lemma1}
   Let  $Y_1,\ldots,Y_n$ be the random matrices generated according to the model (\ref{sec4.2.eq1}). Assume $\lambda\in(0,1)$ and the existence of a constant $\tau>1$ such that $\tau^{-1}<p_i/\sqrt{n}<\tau$ for $i=1,2$. Suppose $p_{\max}=\max\set{p_1,p_2}$ and $p_{\min}=\min\set{p_1,p_2}$. Then there exist universal constants $C$, $c_1,c_2,c_3>0$, such that for every constant $\epsilon\in(0,1)$, if $n>C\frac{p_{\max}}{p_{\min}}/(1-\epsilon)^2$,
   \begin{align*}
    ||\nabla_i f||_2\leq \frac{9\epsilon}{\sqrt{p_i}},\quad |\nabla_0 f|_2\leq \epsilon
\end{align*}
for $i=1,2$, so that 
\begin{align*}
    ||\nabla f||_F^2\leq 163\epsilon^2
\end{align*}
with probability at least 
\begin{align*}
1-6\exp\parentheses{-c_1np_{\min}\epsilon^2}-6\exp\parentheses{-c_2\frac{n\epsilon^2}{1-\lambda}}-6\exp\parentheses{-c_3n}.
\end{align*}
The above also holds when $\epsilon=O(n^{-\delta})$ for a fixed constant $\delta\in(0,\min\set{\frac{\alpha+1/2}{2},\frac{1}{2}})$ with $1-\lambda\asymp n^{-\alpha}$ for constant $\alpha\geq 0$.
\end{lemma}

The following proves \hyperlink{cond2}{(\rom{2})} of Section \ref{secA.3.1}, namely $(\delta,\eta)-$quantum expander with high probability.
\begin{lemma}\label{secA.3.2.lemma2}
Let $Y_1,\ldots,Y_n$ be random matrices generated according to (\ref{sec4.2.eq1}). Assume there exists a constant $\tau>1$ such that $\tau^{-1}<\frac{p_1}{\sqrt{n}},\frac{p_2}{\sqrt{n}}<\tau$. Suppose $\lambda\in(0,1)$ and $1-\lambda\asymp n^{-\alpha}$ for some constant $\alpha\geq 0$. Fix a constant $\delta\in(0,\min\set{\frac{\alpha+1/2}{2},\frac{1}{2}})$. If $\epsilon,\eta\in(0,1)$ are constants, take $(\eta,\lambda)$ such that $1-7/16\eta<\lambda$ when $\alpha=0$, whereas $\eta\in(0,1)$ can be any value if $\alpha>0$. Also, if $\epsilon=O(n^{-\delta})$, take any $\eta\in(0,1)$. Then there exist universal constants $C,c_1,c_2,c_3,c_4.c_5,c_6>0$ such that if $n\geq C(p_{\max}/p_{\min})$ for $p_{\max}=\max\set{p_1,p_2}$ and $p_{\min}=\min\set{p_1,p_2}$, 
$\Phi_Y^{(1,2)}$ is an $(\epsilon,\eta)-$quantum expander with probability at least 
\begin{align*}
  &1-6\exp\parentheses{-c_1np_{\min}\epsilon^2}-6\exp\parentheses{-c_2\frac{n\epsilon^2}{1-\lambda}}-6\exp\parentheses{-c_3n}-\parentheses{\frac{\sqrt{np}}{{p_1+p_2}}}^{-c_4(p_1+p_2)}\\
   &-\parentheses{\frac{1}{\sqrt{1-\lambda}}\frac{\sqrt{np_{\min}}}{p_1+p_2}}^{-c_5(p_1+p_2)}-\parentheses{\frac{1}{1-\lambda}\frac{\sqrt{np_{\min}}}{\sqrt{p_1}+\sqrt{p_2}}}^{-c_6(\sqrt{p_1}+\sqrt{p_2})},
\end{align*}
With the same probability guarantee, $\Phi_Y^{(2,1)}$ is also an $(\epsilon,\eta)-$quantum expander with possibly different universal constants.
\end{lemma}
As a consequence of the above, we have the following result.
\begin{lemma}\label{secA.3.2.lemma3}
Let $Y_1,\ldots,Y_n$ be random matrices generated according to (\ref{sec4.2.eq1}) and assume the same as Lemma \ref{secA.3.2.lemma2} with $\epsilon=O(n^{-\delta})$ for a fixed constant $\delta\in(0,\min\set{(\alpha+1/2)/2,1/2})'$. Suppose $Z_i=Y_i(\frac{1}{np_1}\sum_{j=1}^n Y_j^\top Y_j)^{-1/2}$ for some sufficiently large constant $C>0$ in  Lemma \ref{secA.3.2.lemma2}. Put $Z_i$'s into $Z=(Z_1,\ldots,Z_n)$. Then $\Phi_Z^{(1,2)}$ and $\Phi_Z^{(2,1)}$ are $O(n^{-\delta})-$doubly balanced for any constant $\eta\in(0,1)$ with the same probability guarantee as in Lemma \ref{secA.3.2.lemma2}. 
\end{lemma}
We finally prove Theorem \ref{sec4.2.thm1}.
\begin{proof}[Proof of Theorem \ref{sec4.2.thm1}.]
In view of (\ref{secA.3.1.eq2})--(\ref{secA.3.1.eq3}) and Lemma \ref{secA.3.1.lemma3}, it suffices to prove that with probability $1-\calP_{n,p_1,p_2}$ for $\calP_{n,p_1,p_2}$ defined in (\ref{sec4.2.thm1.eq2}), 
\begin{align*}
        \max\set{||\hat{\Theta}_1(Y)-I_{p_1}||_2,||\hat{\Theta}_2(Y)-I_{p_2}||_2}\leq O\parentheses{n^{-\delta}\log n}.
\end{align*}
Note that $1-\calP_{n,p_1,p_2}$ is the same as the probability guarantee in Lemma \ref{secA.3.2.lemma2} with possibly different universal constants.\\

Let $\hat{\Theta}(W)=\hat{\Theta}_2(W)\otimes \hat{\Theta}_1(W)$ be the inverse of the Kronecker MLE based on $W$, where $W$ is either $Y=(Y_1,\ldots,Y_n)$ as in Theorem \ref{sec4.2.thm1} or $Z=(Z_1,\ldots,Z_n)$ as in Lemma \ref{secA.3.2.lemma3}. Identify each factor via the constraint $|\hat{\Theta}_1(W)|^{1/p_1}=|\hat{\Theta}_2(W)|^{1/p_2}$. Note that $\tr(\sum_{i=1}^n Z_iZ_i^\top)=np$. Also, there exists a universal constant $c>0$ such that $n^{-\delta}\leq c(1-\eta)$ given constants $\delta,\eta\in(0,1)$ and $1-\eta-O(n^{-\delta})\geq CO(n^{-\delta})\log p_{\min}$ as $p_{\min}\asymp \sqrt{n}$. Therefore, $\hat{\Theta}(Z)$ indeed uniquely exists and satisfies
\begin{align*}
             \max\set{||\hat{\Theta}_1(Y)-I_{p_1}||_2,||\hat{\Theta}_2(Y)-I_{p_2}||_2}\leq O\parentheses{n^{-\delta}\log n}.
\end{align*}
by Lemma \ref{secA.3.1.lemma1} and \ref{secA.3.1.lemma3}, whenever $\Phi_Z^{(1,2)}$ is a $(O(n^{-\delta}),\eta+O(n^{-\delta}))-$quantum expander. By Lemma \ref{secA.3.2.lemma3}, this event holds with probability guarantee $1-\calP_{n,p_1,p_2}$ as in Theorem \ref{sec4.2.thm1}. From the equivariance of the Kronecker MLE, one can deduce that 
\begin{align*}
    \hat{\Theta}_1(Y)=\hat{\Theta}_1(Z),\quad \hat{\Theta}_2(Y)=R\hat{\Theta}_2(Z)R
\end{align*}
for $R=(\frac{1}{np_1}\sum_{j=1}^n Y_j^\top Y_j)^{-1/2}$. Following the proof of Theorem $1.11$ from \cite{oliveira2026} and noting that $\epsilon=O(n^{-\delta})$ and $1-\lambda\asymp n^{-\alpha}$, we conclude the desired claim.
\end{proof}

\subsubsection{Proofs of the Ancillary Results}\label{secA.3.3}.
We prove the ancillary results introduced in Section \ref{secA.3.2}, Lemma \ref{secA.3.2.lemma1}--\ref{secA.3.2.lemma3}. We first prove Lemma \ref{secA.3.2.lemma1}. Throughout the proof, we will frequently use the following notation. Given $A=(A_1,\ldots,A_r)\in (\real^{p_1\times p_2})^r$, 
\begin{align*}
    A_R:=[A_1,\ldots,A_r],\quad A_C:=[A_1^\top,\ldots,A_r^\top].
\end{align*}

Now we prove Lemma \ref{secA.3.2.lemma1}. We first introduce ancillary results.

\begin{lemma}\label{secA.3.3.lemma2}
Suppose $\tau^{-1}<p_i/\sqrt{n}<\tau$ for some constant $\tau>1$ with $i=1,2$. Suppose random matrices $Y_1,\ldots,Y_n\in\real^{p_1\times p_2}$ are generated according to the model (\ref{sec4.2.eq1}) and $\rho=1/n\sum_{i=1}^n y_iy_i^\top$ for $y_i=\text{vec}(Y_i)$. Assume $1-\lambda\asymp n^{-\alpha}$ for some constant $\alpha\geq 0$ and $\lambda\in(0,1)$. For every constant $\epsilon\in(0,1)$, 
\begin{align*}
    |\tr(\rho)-1|\leq \epsilon,
\end{align*}
with a failure probability at most
\begin{align*}
\begin{cases}
    2\exp\parentheses{-cn^{1+\alpha}\epsilon^2},&\quad 0\leq \alpha<1,\\
    2\exp\parentheses{-cnp\epsilon^2},&\quad 1\leq \alpha<\infty,  \\
\end{cases}
\end{align*}
for some universal constant $c>0$. If $\epsilon=O(n^{-\delta})$ for some constant $0<\delta<\min\set{\frac{\alpha+1}{2},1}$, $|\tr(\rho)-1|=O(n^{-\delta})$ almost surely. 
\end{lemma}
\begin{proof}
Let $y=(I_n\otimes C^{1/2})[z_1^\top,\ldots,z_n^\top]^\top$ for $z_i\overset{i.i.d.}{\sim} N_p(0,I_p)$ for $C=(1-\lambda)AA^\top+\lambda I_p$, where $A=[\text{vec}(A_1),\ldots,\text{vec}(A_r)]$ for $A_i$'s in (\ref{sec4.2.eq1}). Then note that $y^\top y\overset{d}{\equiv}np\cdot \tr(\rho)$. By Hanson-Wright inequality (\cite{rudelson2013}, Theorem $1.1$), there exists a universal constant $c>0$ such that for any $t>0$,
\begin{align*}
    \bbP(|y^\top y-\bbE y^\top y|>t)\leq 2\exp\parentheses{-c\min\set{\frac{t^2}{||I_n\otimes C||_F^2},\frac{t}{||I_n\otimes C||_2}}}.
\end{align*}
Note that 
\begin{align*}
    ||I_n\otimes C||_2&=||I_n||_2 ||C||_2=((1-\lambda)\sigma_1^2(A)+\lambda)\leq (1-\lambda)p+\lambda\leq 2\max\set{(1-\lambda)p,\lambda},\\
    ||I_n\otimes C||_F^2&=||I_n||_F^2 ||C||_F^2=n||C||_F^2=n(1-\lambda)^2||AA^\top||_F^2+2n\lambda(1-\lambda)\tr(AA^\top)+\lambda^2 np.\\
\end{align*}
Hence,
    \begin{align*}
       ||I_n\otimes C||_F^2=||I_n||_F^2 ||C||_F^2=n||C||_F^2&\leq n(1-\lambda)^2 \tr(AA^\top)||AA^\top||_2+2\lambda(1-\lambda)np+\lambda^2 np\\
       &=n(1-\lambda)^2 p\sigma_1^2(A)+2\lambda(1-\lambda)np+\lambda^2 np\\
       &=n(1-\lambda)^2p(\sigma_1^2(A)-1)+np\\
       &=n(1-\lambda)^2p^2\frac{\sigma_1^2(A)-1}{p}+np\\
       &\leq n(1-\lambda)^2 p^2+np\\
       &\leq 2np\max\set{1,(1-\lambda)^2p},
    \end{align*} 
    noting that $\sigma_1^2(A)\leq \sum _{i=1}^r \sigma_i^2(A)=\tr(AA^\top)=p$ by Proposition \ref{sec3.1.prop1} as $AA^\top$ is a rank$-r$ core. Furthermore,    
\begin{align*}
   \bbE y^\top y=n\bbE[y_1^\top y_1]=n\tr(V[y_1])=np.
\end{align*}    
    Thus, with $t=np\epsilon$ for some $\epsilon\in(0,1)$, we have that 
\begin{align*}
    \bbP(|\tr(\rho)-1|>\epsilon)\leq  2\exp\parentheses{-\frac{c}{2}n\min\set{\frac{p\epsilon^2}{\max\set{1,(1-\lambda)^2p}},\frac{p\epsilon}{\max\set{(1-\lambda)p,\lambda}}}}.
\end{align*}
Suppose $1-\lambda\asymp n^{-\alpha}$ for some constant $\alpha\geq 0$. If $\alpha<1/2$, the upper bound on the tail probability above becomes
\begin{align*}
    2\exp\parentheses{-\frac{c}{2}n\min\set{\frac{p\epsilon^2}{\max\set{1,(1-\lambda)^2p}},\frac{p\epsilon}{\max\set{(1-\lambda)p,\lambda}}}}&=2\exp\parentheses{-\frac{c}{2}n\min\set{\frac{\epsilon^2}{(1-\lambda)^2},\frac{\epsilon}{1-\lambda}}}\\
    &\leq 2\exp\parentheses{-\frac{c}{2}n\frac{\epsilon^2}{(1-\lambda)}}.
\end{align*}
The last inequality holds because $\epsilon\in(0,1)$ so that $\epsilon^2\leq \epsilon$ and $1/(1-\lambda)\leq 1/(1-\lambda)^2$ whenever $\lambda\in(0,1)$. Otherwise, note that $\tau^{-2}<p/n<\tau^2$ for some constant $\tau>1$ by the assumption. Thus, if $1/2\leq\alpha<1$,  
\begin{align*}
  2\exp\parentheses{-\frac{c}{2}n\min\set{\frac{p\epsilon^2}{\max\set{1,(1-\lambda)^2p}},\frac{p\epsilon}{\max\set{(1-\lambda)p,\lambda}}}}&\leq 2\exp\parentheses{-c'n\min\set{p\epsilon^2,\epsilon/(1-\lambda)}}\\
  &\leq 2\exp\parentheses{-c^{''}n\frac{\epsilon^2}{1-\lambda}}.
\end{align*}
for some constants $c',c^{''}>0$. Indeed, for given $\epsilon>0$, $\epsilon^2\leq \epsilon$, and $1/(1-\lambda)\asymp n^{\alpha}=o(p)$. Thus, 
\begin{align*}
c^{'''}\epsilon^2/(1-\lambda) \leq \min\set{p\epsilon^2,\epsilon^2/(1-\lambda)}\leq \min\set{p\epsilon^2,\epsilon/(1-\lambda)}
\end{align*}
for some universal constant $c^{'''}>0$. Now with $1-\lambda\asymp n^{-\alpha}$, we conclude the claim when $\alpha\in(0,1)$. If $\alpha\geq 1$,
\begin{align*}
  2\exp\parentheses{-\frac{c}{2}n\min\set{\frac{p\epsilon^2}{\max\set{1,(1-\lambda)^2p}},\frac{p\epsilon}{\max\set{(1-\lambda)p,\lambda}}}}&\leq 2\exp\parentheses{-c'n\min\set{p\epsilon^2,p\epsilon}}\\
  &=2\exp\parentheses{-c'np\epsilon^2}
\end{align*}
as $\epsilon<1$. If $\epsilon=n^{-\delta}$, then the above results hold when $0<\delta<\min\set{\frac{\alpha+1}{2},1}$.
\end{proof}

\begin{lemma}\label{secA.3.3.lemma3}
Assume $\lambda\in(0,1)$ and the existence of a constant $\tau>1$ such that $\tau^{-1}<p_i/\sqrt{n}<\tau$ for $i=1,2$. Suppose $p_{\max}=\max\set{p_1,p_2}$ and $p_{\min}=\min\set{p_1,p_2}$. Let $Y_R=[Y_1,\ldots,Y_n]$ and $Y_C=[Y_1^\top,\ldots,Y_n^\top]$ for random matrices $Y_1,\ldots,Y_n$ generated according to the model (\ref{sec4.2.eq1}). For every constant $\epsilon\in(0,1)$ and $n\geq C\frac{p_{\max}}{p_{\min}(1-\epsilon)^2}$ for sufficiently large constant $C>0$, there exist universal constants $c_1,c_2,c_3>0$ such that the event  
\begin{align*}
 (1-\epsilon)\sqrt{np_2}-\sqrt{p_1}    \leq \sigma_{p_1}(Y_R)  \leq \sigma_1(Y_R)\leq(1+\epsilon)\sqrt{np_2}+\sqrt{p_1}  \\
\end{align*}
holds with probability at least 
\begin{align*}
    1-2\exp\parentheses{-c_1np_{\min}\epsilon^2}-2\exp\parentheses{-c_2\frac{n\epsilon^2}{1-\lambda}}-2\exp\parentheses{-c_3n},
\end{align*}
Similarly, with the same probability guarantee, 
\begin{align*}
     (1-\epsilon)\sqrt{np_1}-\sqrt{p_2}\leq \sigma_{p_2}(Y_C)  \leq \sigma_1(Y_C)\leq(1+\epsilon)\sqrt{np_1}+\sqrt{p_2}
\end{align*}
The above also holds when $\epsilon=O(n^{-\delta})$ for $\delta\in(0,\min\set{\frac{\alpha+1/2}{2},\frac{1}{2}})$ with $1-\lambda\asymp n^{-\alpha}$ for constant $\alpha\geq 0$.
\end{lemma}
\begin{proof}
    We prove the claim for $Y_R$ as that for $Y_C$ follows similarly. The proof is again based on Hanson-Wright inequality along with a standard argument based on $\epsilon-$net. For a given $v\in \bbS^{p_1-1}$, define the $\epsilon-$net of $v$ by 
    \begin{align}
        \calN_{\epsilon}(v)=\set{u\in \bbS^{p_1-1}:||u-v||_2\leq \epsilon}.
    \end{align}
    Then it holds that $|\calN_{\epsilon}(v)|\leq (3/\epsilon)^{p_1}$. Given $\epsilon>0$, let $E$ be the event that 
    \begin{align*}
        \max_{x\in\calN_{1/2}(v_R)} |\frac{1}{np_2} ||Y_R^\top x||_2^2-1|\leq \delta/2,
    \end{align*} 
    where $v_R:=\argmax_{v\in\bbS^{p_1-1}} ||Y_R^\top v||_2$. Following the standard argument based on $\epsilon-$net, one can see that on the event $E$, 
    \begin{align*}
     ||\frac{1}{np_2}Y_RY_R^\top-I_{p_2}||_2=\max\set{|\sigma_1(Y_R)^2/np_2-1|,|\sigma_{p_1}(Y_R)^2/np_2-1|}\leq \delta
    \end{align*}
    as $n>p_1/p_2+p_2/p_1$. Take
    \begin{align*}
        \delta=-(\epsilon-1)^2+1+2(1-\epsilon)\sqrt{\frac{p_1}{np_2}}-\frac{p_1}{np_2}.
    \end{align*}
    By the assumption on $n$, $\delta\in(0,1)$ with sufficiently large constant $\delta$. Indeed, 
    \begin{align*}
        \delta=1-\parentheses{(\epsilon-1)^2+2(\epsilon-1)\sqrt{\frac{p_1}{np_2}}+\frac{p_1}{np_2}}\in (0,1)
    \end{align*}
    Also, because $2(1-\epsilon)\sqrt{\frac{p_1}{np_2}}-\frac{p_1}{np_2}>0$ and $\epsilon\in(0,1)$, $\delta\geq c\epsilon$ for sufficiently small universal constant $c>0$. On the event $E$,
    \begin{align*}
        &\sigma_1(Y_R)^2/np_2\leq 1+\epsilon^2+2\epsilon+2(1+\epsilon)\sqrt{\frac{p_1}{np_2}}+\frac{p_1}{np_2}\\
        \Rightarrow &\sigma_1(Y_R)^2\leq (1+\epsilon)^2np_2+2(1+\epsilon)\sqrt{np_1p_2}+p_1\\
            \Rightarrow &\sigma_1(Y_R)\leq (1+\epsilon)\sqrt{np_2}+\sqrt{p_1}.
    \end{align*}
    Likewise, we have that $\sigma_{p_2}(Y_R)\geq (1-\epsilon)\sqrt{np_2}-\sqrt{p_1}>0$. Hence, we obtain the upper bound on the probability of the event $E^c$.\\
    
    Fix $x\in \calN_{1/2}(v_R)$. Define the matrices
    \begin{align*}
        &E_R:=[E_1,\ldots,E_R],\quad A_R:=[A_1,\ldots,A_r], Z:=(z_{ij})\in \real^{r\times n},
    \end{align*}
    for $E_i,A_i,z_{ij}$ in (\ref{sec4.2.eq1}). Then one can observe that 
    \begin{align*}
        Y_R\overset{d}{\equiv}\sqrt{\lambda} E_R+\sqrt{1-\lambda}A_R(Z\otimes I_{p_2}).
    \end{align*}
    Observe that 
    \begin{align*}
      ||Y_R^\top x||_2^2&=\lambda(x^\top E_RE_R^\top x)+(1-\lambda) (x^\top A_R(ZZ^\top\otimes I_{p_2})A_R^\top x)\\
      &+2\sqrt{\lambda(1-\lambda)}(x^\top E_R(Z^\top \otimes I_{p_2}) A_R^\top x).
    \end{align*}
    Also, noting that $A_RA_R^\top=p_2I_{p_1}$ and $Z$ and $E_R$ are independent standard Gaussian matrices, 
    \begin{align*}
        &\bbE[E_RE_R^\top]=np_2 I_{p_1},\quad \bbE[E_R(Z^\top\otimes I_{p_2})]=0,\\
             &\bbE[A_R(ZZ^\top \otimes I_{p_2})A_R^\top]=A_R(\bbE[ZZ^\top]\otimes I_{p_2})A_R^\top=n A_RA_R^\top=np_2 I_{p_1}.
    \end{align*}
    Hence,
    \begin{align*}
        \frac{1}{np_2} \bbE[x^\top E_RE_R^\top x]=1,\quad \frac{1}{np_2}\bbE[x^\top (A_R(\bbE[ZZ^\top]\otimes I_{p_2})A_R^\top)x]=1.
    \end{align*} 
    Given a constant $\delta>0$, define the events $E_1$, $E_2$, $E_3$ that 
    \begin{align*}
        E_1&:=\set{|x^\top E_RE_R^\top x/np_2-1|>\delta/(6\lambda)},\\
        E_2&:=\set{|x^\top (A_R(ZZ^\top\otimes I_{p_2})A_R^\top)x/np_2-1|>\delta/(6(1-\lambda))},\\
        E_3&:=\set{|x^\top E_R(Z^\top \otimes I_{p_2}) A_R^\top x|/np_2>\delta/(12\sqrt{\lambda(1-\lambda)})}.
    \end{align*}
    Since 
    \begin{align*}
       \frac{1}{np_2}||Y_R^\top x||^2-1
        &=\lambda(\frac{x^\top E_RE_R^\top x}{np_2}-1)+(1-\lambda)(\frac{x^\top(A_R(ZZ^\top\otimes I_{p_2})A_R^\top)x}{np_2}-1)\\
        &+\frac{x^\top E_R(Z^\top \otimes I_{p_2})A_R^\top x}{np_2},
    \end{align*}
    it holds that 
    \begin{align*}
        \bbP\parentheses{|\frac{1}{np_2}||Y_R^\top x||_2^2-1|>\delta/2}\leq \bbP(E_1)+\bbP(E_2)+\bbP(E_3).
    \end{align*}
Now we bound $\bbP(E_i)$ for each $i=1,2,3$. To bound $\bbP(E_1)$, because $E_R^\top x\sim N_{np_2}(0,I_{np_2})$, using the tail probability for the sum of independent chi-squared random variables or simply Hanson-Wright inequality yields that 
\begin{align*}
    \bbP(E_1)\leq 2\exp\parentheses{-cnp_2\min\set{\frac{\delta}{6},\frac{\delta^2}{36}}}
\end{align*}
for some universal constant $c>0.$ By the choice of $\delta$ and the assumption on $n$, 
\begin{align*}
    \bbP(E_1)\leq 2\exp\parentheses{-cnp_2\epsilon^2}
\end{align*}
for some universal constant $c>0$. \\

Next, to bound $\bbP(E_2)$, let $u=A_R^\top x/\sqrt{p_2}$ and $z=\text{vec}(Z)\sim N_{nr}(0,I_{nr})$. Since $||x||_2=1$ and $A_RA_R^\top=p_2 I_{p_1}$, $||u||_2=1$. By vec-Kronecker identity,  
\begin{align*}
    (Z^\top\otimes I_{p_2})A_R^\top x=\sqrt{p_2}(Z^\top \otimes I_{p_2})u=\sqrt{p_2}\text{vec}(UZ)=\sqrt{p_2}(I_n\otimes U)z,
\end{align*}
where $U=\textsf{mat}_{p_2\times r}(u)$. Hence,
\begin{align*}
    x^\top A_R(ZZ^\top \otimes I_{p_2})A_R^\top x=p_2\cdot z^\top (I_n\otimes U^\top U)z.
\end{align*}
Note that the mean of the above random variable is $np_2$. Applying the Hanson-Wright inequality gives
\begin{align*}
    \bbP(| x^\top A_R(ZZ^\top \otimes I_{p_2})A_R^\top x-np_2|>t)\leq 2\exp\parentheses{-\frac{1}{8}\min\set{\frac{t}{p_2||I_n\otimes U^\top U||_2},\frac{t^2}{p_2^2||I_n\otimes U^\top U||_F^2}}}.
\end{align*}
Note that 
\begin{align*}
    ||I_n\otimes U^\top U||_2&=||U^\top U||_2\leq ||U||_2^2\leq ||U||_F^2=1,\\
    ||I_n\otimes U^\top U||_F^2&=||I_n||_F^2 ||U^\top U||_F^2\leq n||U||_2^2||U||_F^2\leq n.
\end{align*}
Taking $t=np_2\delta/(6(1-\lambda))$ leads to 
\begin{align*}
    \bbP(E_2)\leq2\exp\parentheses{-\frac{1}{8}\min\set{\frac{n\delta}{6(1-\lambda)},\frac{n\delta^2}{36(1-\lambda)^2}}}.
\end{align*}
Again by the choice of $\delta$ and the assumption on $n$, 
\begin{align*}
    \bbP(E_2)\leq2\exp\parentheses{-c\frac{n\epsilon^2}{1-\lambda}}.
\end{align*}
for some universal constant $c>0$.\\

It remains to bound $\bbP(E_3)$. Conditional on $Z$, note that $x^\top E_R(Z^\top\otimes I_{p_2})A_R^\top x$ is a mean zero Gaussian random variable as it is a linear combination of i.i.d. standard gaussian random variables. To compute its conditional variance given $Z$, letting $y=(A_R^\top x)$, note that 
\begin{align*}
    x^\top E_R(Z^\top \otimes I_{p_2})y=([y^\top(Z\otimes I_{p_2})]\otimes x^\top)e_r
\end{align*}
for $e_R=\text{vec}(E_R)$. Hence,
\begin{align*}
    V[x^\top E_R(Z^\top\otimes I_{p_2})y|Z]&=([y^\top(Z\otimes I_{p_2})]\otimes x^\top)\bbE[e_re_r^\top|Z]([y^\top(Z\otimes I_{p_2})]\otimes x^\top)^\top\\
    &=([y^\top(Z\otimes I_{p_2})]\otimes x^\top)([y^\top(Z\otimes I_{p_2})]\otimes x^\top)^\top\\
    &=y^\top (ZZ^\top\otimes I_{p_2})y.
\end{align*}
Hence, using the Chernoff bound for the mean-zero gaussian random variable,
\begin{align*}
    \bbP(E_3|Z)&\leq 2\exp\parentheses{-\frac{n^2p_2^2\delta^2}{288[\lambda(1-\lambda)]y^\top (ZZ^\top \otimes I_{p_2})y}}\\
    &=2\exp\parentheses{-\frac{n^2p_2\delta^2}{288[\lambda(1-\lambda)]u^\top (ZZ^\top \otimes I_{p_2})u}}
\end{align*}
for $u=y/\sqrt{p_2}$. Recall that $u=y/\sqrt{p_2}$ has a unit norm. Note that 
\begin{align*}
\bbP(E_3)=\bbE[\bbP(E_3|Z)]&=\bbE[\bbP(E_3|Z)\mathbf{1}_{E_4}]+\bbE[\bbP(E_3|Z)\mathbf{1}_{E_4^c}]\\
&\leq \bbP(E_4)+\bbE[\bbP(E_3|Z)\mathbf{1}_{E_4^c}],
\end{align*}
where the event $E_4$ is defined as
\begin{align*}
    E_4:=\set{|u^\top(ZZ^\top \otimes I_{p_2})u/p_2-1|>1}.
\end{align*}
Using the argument to bound $\bbP(E_2)$, we see that $\bbP(E_4)\leq 2\exp(-n/8)$. On the event $E_4^c$, $u^\top(ZZ^\top \otimes I_{p_2})u/p_2$ is upper bounded by $2$. Thus, since $\lambda<1$,
\begin{align*}
    \bbP(E_3)\leq 2\exp(-n/8)+2\exp\parentheses{-\frac{n^2\delta^2}{576\lambda(1-\lambda)}}\leq 2\exp(-n/8)+2\exp\parentheses{-c\frac{n^2\epsilon^2}{1-\lambda}}.
\end{align*}
Here $\exp\parentheses{-cn^2\epsilon^2/(1-\lambda)}$ can be absorbed into $\exp\parentheses{-c'n\epsilon^2/(1-\lambda)}$ with appropriately chosen $c'>0$.\\

To conclude the proof, combining the bounds on $\bbP(E_i)$ for $i=1,2,3$ lead to
\begin{align*}
    \bbP\parentheses{|\frac{1}{np_2}||Y_R^\top x||_2^2-1|>\delta/2}&\leq \bbP(E_1)+\bbP(E_2)+\bbP(E_3)\\
    &\leq   2\exp\parentheses{-c_1np_{\min}\epsilon^2}+2\exp\parentheses{-c_2\frac{n\epsilon^2}{1-\lambda}}+2\exp\parentheses{-c_3n}\\
    &=:\calP_{n,p_1,p_2}
\end{align*}
for some universal constants $c_1,c_2,c_3$. Since $x$ was chosen among $\calN_{1/2}(v_R)$ and $|\calN_{1/2}(v_R)|\leq 6^{p_1}$, 
\begin{align*}
  \bbP(E)\leq |\calN_{1/2}(v_R)|\calP_{n,p_1,p_2}=\exp(p_1\log 6)\calP_{n,p_1,p_2}.
\end{align*}
Under the assumption that $\tau^{-1}<p_1/\sqrt{n}<\tau$, if $\epsilon,\lambda\in(0,1)$ were constants, we conclude the claim by taking universal constants smaller than $c_1,c_2,c_3$ in $\calP_{n,p_1,p_2}$. However, if $\epsilon=O(n^{-\delta})$ for some constant $\delta>0$, absorbing $\exp\parentheses{p_1\log 6}$ into $\calP_{n,p_1,p_2}$ is possible with smaller constants $c_1,c_2,c_3>0$ when $\delta\in(0,\min\set{\frac{\alpha+1/2}{2},\frac{1}{2}})$ with $1-\lambda\asymp n^{-\alpha}$ for constant $\alpha\geq 0$.
\end{proof}
Using Lemma \ref{secA.3.3.lemma2}--\ref{secA.3.3.lemma3}, we prove Lemma \ref{secA.3.2.lemma1}.
\begin{proof}[\textbf{Proof of Lemma \ref{secA.3.2.lemma1}.}]
    We shall invoke the proof of Proposition $2.11$ from \cite{oliveira2026}. Recall that $\rho^{(1)},\rho^{(2)}$ are partial traces of $\rho$. Following the proof of Proposition $2.11$ of \cite{oliveira2026}, the eigenvalues of $p_i\rho^{(i)}$ are contained in $[1-4\frac{\epsilon}{\sqrt{p_i}},1+8\frac{\epsilon}{\sqrt{p_i}}]$ for any fixed $\epsilon\in(0,1)$ and $i=1,2$, with probability at least 
\begin{align*}
      1-4\exp\parentheses{-c_1np_{\min}\epsilon^2}-4\exp\parentheses{-c_2\frac{n\epsilon^2}{1-\lambda}}-4\exp\parentheses{-c_3n},
\end{align*}
for some universal constants $c_1,c_2$ by Lemma \ref{secA.3.3.lemma3}. Also, by Lemma \ref{secA.3.3.lemma2}, $|\tr(\rho)-1|\leq \epsilon$ for $\epsilon\in (0,1)$ with probability at least $1-2\exp(-c_4np\epsilon^2)-2\exp\parentheses{-c_5n\epsilon^2/(1-\lambda)}$ for some universal constants $c_4,c_5>0$. Consequently, 
\begin{align*}
   ||\nabla_i f||_2\leq \frac{1}{\sqrt{p_i}}|| p_i\rho^{(i)}-I_{p_i}||_2+\frac{|\tr(\rho)-1|}{\sqrt{p_i}}\leq\frac{9\epsilon}{\sqrt{p_i}}
\end{align*}
for each $i=1,2$. Taking the bound on the probability of the union of the events that the eigenvalues of $p_i\rho^{(i)}$ are contained in $[1-4\frac{\epsilon}{\sqrt{p_i}},1+8\frac{\epsilon}{\sqrt{p_i}}]$ and that $|\tr(\rho)-1|\leq \epsilon$ gives the desired probability. Also, on these events,
\begin{align*}
    ||\nabla f||_F^2\leq \sum_{i=1}^2 p_i||\nabla_i||_2^2+|\nabla_0 f|^2\leq 163\epsilon^2.
\end{align*}
The choice of $\delta$ when $\epsilon=O(n^{-\delta})$ follows from Lemma \ref{secA.3.3.lemma2}--\ref{secA.3.3.lemma3}.
\end{proof}

\begin{lemma}\label{secA.3.3.lemma5}
Let $z_1,\ldots,z_n\overset{i.i.d.}{\sim}N_r(0,I_r)$. For every constant $t\geq 2$, there exists a universal constant $c>0$ such that 
\begin{align*}
    ||\sum_{i=1}^n(z_i\otimes I_{m}\otimes z_i-\bbE[z_i\otimes I_{m}\otimes z_i])||_2\leq O(tm^{\frac{1}{2q}}\sqrt{n}q)
\end{align*}
with probability at least $1-t^{-cq}$ whenever $q\geq {r/2}$.
\end{lemma}
\begin{proof}
Let $w_1,\ldots,w_n$ be i.i.d. copies of $z_1$. Note that 
\begin{align*}
   z_i\otimes I_m\otimes z_i- \bbE[z_i\otimes I_{m}\otimes z_i]=\bbE[z_i\otimes I_m\otimes z_i-w_i\otimes I_m\otimes w_i|z_i].
\end{align*}
By Jensen's inequality, for any $p\geq 1$,
\begin{align*}
    \bbE||\sum_{i=1}^n(z_i\otimes I_m\otimes z_i- \bbE[z_i\otimes I_{m}\otimes z_i])||_2^p\leq \bbE||\sum_{i=1}^n(z_i\otimes I_m\otimes z_i-w_i\otimes I_m\otimes w_i)||_2^p.
\end{align*}
To bound the term on the right-hand side with high probability, note that $(z_i,w_i)\overset{d}{\equiv}\parentheses{\frac{z_i+w_i}{\sqrt{2}},\frac{z_i-w_i}{\sqrt{2}}}$. Also,
\begin{align*}
    &\sum_{i=1}^n\parentheses{\frac{z_i+w_i}{\sqrt{2}}\otimes I_m\otimes \frac{z_i+w_i}{\sqrt{2}}-\frac{z_i-w_i}{\sqrt{2}}\otimes I_m\otimes \frac{z_i-w_i}{\sqrt{2}}}=\sum_{i=1}^n\parentheses{z_i\otimes I_m \otimes w_i+w_i\otimes I_m\otimes z_i}.
\end{align*}
Hence, 
\begin{align*}
   \bbE||\sum_{i=1}^n(z_i\otimes I_m\otimes z_i-w_i\otimes I_m\otimes w_i)||_2^p\leq 2^p\bbE||\sum_{i=1}^n z_i\otimes I_m \otimes w_i||_2^p. 
\end{align*}
Note that $||\cdot ||_2\leq |||\cdot |||_{p}$ for Schatten $p-$norm $|||\cdot|||_p$ with $p\geq 1$. Here, the Schatten $p-$norm is defined as $|||A|||_p=\parentheses{\tr([A^\top A]^{p/2})}^{1/p}$. Take $p=2q$ for $q\geq 1$. Then we have that  
\begin{align*}
  \bbE||\sum_{i=1}^n z_i\otimes I_m \otimes w_i||_2^{2q}&\leq \bbE ||| \sum_{i=1}^n z_i\otimes I_m \otimes w_i|||_{2q}^{2q}\\
  &=\bbE\tr\left[\parentheses{\sum_{i,j\in [n]}(z_i^\top\otimes I_m\otimes w_i^\top)(z_i\otimes I_m\otimes w_i)}^{q}\right]\\
  &=\bbE\tr\left[\parentheses{\sum_{i,j\in [n]}(z_i^\top z_j)(w_i^\top w_j)}^q I_m\right]\\
  &=m\sum_{i_\ell,j_\ell \in [n],\ell\in [q]}\bbE[ (z_{i_1}^\top z_{j_1})\cdots(z_{i_q}^\top z_{j_q})] \bbE[(w_{i_1}^\top w_{j_1})\cdots(w_{i_q}^\top w_{j_q})]\\
  &\overset{(\rom{1})}{\leq} m\sum_{i_\ell,j_\ell \in [n],\ell\in [q]} \bbE[ (z_{i_1}^\top z_{j_1})\cdots(z_{i_q}^\top z_{j_q})]\bbE[||w_{i_1}||_2||w_{j_1}||_2\cdots ||w_{i_q}||_2||w_{j_q}||_2]\\
  &\leq m\sum_{i_\ell,j_\ell \in [n],\ell\in [q]} \bbE[ (z_{i_1}^\top z_{j_1})\cdots(z_{i_q}^\top z_{j_q})] \bbE[||w_1||_2^{2q}]\\
  &= m\bbE[U^q]\sum_{i_\ell,j_\ell \in [n],\ell\in [q]} \bbE[ (z_{i_1}^\top z_{j_1})\cdots(z_{i_q}^\top z_{j_q})] \\
  &= m\bbE[U^q]\bbE[||\sum_{i=1}^n z_{i} ||_2^{2q}]\\
  &\overset{(\rom{2})}{\leq} mn^q\bbE[U^q]\bbE[||z_1||_2^{2q}]=mn^q (\bbE[U^q])^2
\end{align*}
where $U\overset{d}{\equiv} ||w_1||_2^2\sim \chi_r^2$. In (\rom{1}), we used the fact that the expectation of the monomial in gaussian random variables is nonnegative and the Cauchy-Schwartz inequality. The inequality (\rom{2}) holds as $\sum_{i=1}^r z_i\sim \sqrt{n}N_r(0,I_r)$. To bound $\bbE[U^q]$, note that if $\Gamma$ is the Gamma function, for any $x>0$,
\begin{align*}
    \Gamma(x)\leq \sqrt{2\pi}x^{x-1/2}e^{-x}e^{\frac{1}{12x}}
\end{align*}
which arises from Stirling's approximation (see \cite{artin1964,diaconis1986}). Because $q\geq r/2$ by the assumption, 
\begin{align*}
    \bbE[U^q]=2^q\frac{\Gamma(q+r/2)}{\Gamma(r/2)}&\leq \frac{\sqrt{2\pi}}{\Gamma(3/2)}e^{-q-r/2}e^{\frac{1}{12(q+r/2)}}(q+r/2)^{-1/2}(q+r/2)^q(q+r/2)^{r/2}\\
    &\leq C' (2q)^q(2q)^q\leq (Cq)^q
\end{align*}
for large enough $q$ and some universal constants $C,C'>0$ as $r$ is fixed. Combining the previous bounds, the Markov's inequality implies that for any fixed $s>0$,
\begin{align*}
    \bbP\parentheses{    ||\sum_{i=1}^n(z_i\otimes I_{m}\otimes z_i-\bbE[z_i\otimes I_{m}\otimes z_i])||_2>s}&\leq \frac{\bbE[||\sum_{i=1}^n(z_i\otimes I_{m}\otimes z_i-\bbE[z_i\otimes I_{m}\otimes z_i])||_2^{2q}]}{s^{2q}}\\
    &\leq \frac{mn^q (2Cq)^{2q}}{s^{2q}}=\parentheses{\frac{2Cm^{1/(2q)}\sqrt{n}q}{s}}^{2q}
\end{align*} 
Letting $s=2Ctm^{1/(2q)}\sqrt{n}q$, 
\begin{align*}
     \bbP\parentheses{    ||\sum_{i=1}^n(z_i\otimes I_{m}\otimes z_i-\bbE[z_i\otimes I_{m}\otimes z_i])||_2>2Ctm^{1/(2q)}\sqrt{n}q}\leq t^{-2q}
\end{align*}
for some universal constant $c>0$ and any fixed constant $t\geq 2$.
\end{proof}

\begin{lemma}\label{secA.3.3.lemma6}
Let $E_1,\ldots,E_n\overset{i.i.d.}{\sim} N_{p_1\times p_2}(0,I_{p})$ and independently, $z_1,\ldots,z_n\overset{i.i.d.}{\sim}N_r(0,I_r)$. For every fixed constant $t\geq 2$, there exists a universal constant $c>0$ such that 
\begin{align*}
    ||\sum_{i=1}^n E_i\otimes z_i||_2\leq O(t\sqrt{n}(p_1+p_2))
\end{align*}
with probability at least $1-t^{-c(p_1+p_2)}$.
\end{lemma}
\begin{proof}
Note that the mean of $\sum_{i=1}^n E_i\otimes z_i$ is zero due to the independence between $E_i$ and $z_i$. Following the argument of Lemma, we shall derive the upper bound on $\bbE|||\sum_{i=1}^n E_i\otimes z_i |||_{2q}^{2q}$ for the Schatten norm $|||\cdot|||_p$ and $q\geq 1$. By direct computations,  
\begin{align*}
 \bbE||\sum_{i=1}^n E_i\otimes z_i ||_2^{2q}\leq    \bbE|||\sum_{i=1}^n E_i\otimes z_i |||_{2q}^{2q}&\leq \bbE\tr\left[\parentheses{\sum_{i,j\in [n]}(z_i^\top z_j) E_i^\top E_j}^q\right]\\
    &=\sum_{i_\ell,j_\ell\in [n],\ell\in [q]} \bbE \tr\left[E_{i_1}^\top E_{j_1}\cdots E_{i_q}^\top E_{j_q}\right]\bbE[ (z_{i_1}^\top z_{j_1})\cdots(z_{i_q}^\top z_{j_q})]\\
    &\leq \bbE[U^q]\parentheses{\bbE|||\sum_{i=1}^n E_i|||_{2q}^{2q}}\\
    &=n^{q}\bbE[U^q]\cdot \bbE|||E_1|||_{2q}^{2q},
\end{align*}
where $U\sim \chi_r^2$. From the proof of Theorem C.1 of \cite{oliveira2026.supp}, one can deduce that 
\begin{align*}
    \bbE|||E_1|||_{2q}^{2q}\leq p_{\min}^2\parentheses{(2\bbE|| E_1||_2)^{2q}+(C\sqrt{q})^{2q}}
\end{align*}
for some universal constant $C>0$ and $p_{\min}=\min\set{p_1,p_2}$. Take $q=2(p_1+p_2)$. Because $\bbE||E_1||_2\leq \sqrt{p_1}+\sqrt{p_2}$ (see (C.1) of \cite{oliveira2026.supp}) and $\sqrt{p_1+p_2}\leq \sqrt{p_1}+\sqrt{p_2}$, 
\begin{align*}
    (2\bbE|| E_1||_2)^{2q}+(C\sqrt{q})^{2q}\leq (C'(\sqrt{p_1}+\sqrt{p_2}))^{2q}
\end{align*}
Note that $p_{\min}^{1/q}$ with $q=2(p_1+p_2)$ is upper bounded. Recalling that $\bbE[U^q]\leq (C^{''}q)^q$ for some universal constant $C''>0$, we consequently have that 
\begin{align*}
    \bbE||\sum_{i=1}^n E_i\otimes z_i ||_2^{2q}\leq p_{\min}^2 n^{q}(C^{''}q)^{q} (C'(\sqrt{p_1}+\sqrt{p_2}))^{2q}&\leq p_{\min}^2 \parentheses{\sqrt{n}\sqrt{C^{''}}\sqrt{q}C'(\sqrt{p_1}+\sqrt{p_2})}^{2q}\\
    &\leq (C^{'''}\sqrt{n}(p_1+p_2))^{2q}\\
\end{align*}
for some universal constant $C^{'''}$ noting that $q=2(p_1+p_2)$. Again, as in the proof of Lemma \ref{secA.3.3.lemma5}, using Markov's inequality leads to the desired claim. 
\end{proof}

The following lemma is trivial but useful to bound $||(\sum_{i=1}^n Y_i\otimes Y_i)\circ \Pi_2||_2, ||(\sum_{i=1}^n Y_i^\top \otimes Y_i^\top)\circ \Pi_1||_2$ with high probability. 
\begin{lemma}\label{secA.3.3.lemma1}
For $A=(A_1,\ldots,A_r)\in (\real^{p_1\times p_2})^r$, if 
\begin{align*}
    A_RA_R^\top=p_2I_{p_1}, \quad A_CA_C^\top=p_1I_{p_2},
\end{align*}
it holds that $\sigma_1(A_R)=\sqrt{p_2}$ and $\sigma_1(A_C)=\sqrt{p_1}$.
\end{lemma}

\begin{lemma}\label{secA.3.3.lemma7}
    Suppose random matrices $Y_1,\ldots,Y_n$ are generated according to (\ref{sec4.2.eq1}). Let $\Pi_i$ be an orthogonal projection defined by
    \begin{align*}
        \Pi_i&=I_{p_i^2}-\frac{1}{p_i}\text{vec}(I_{p_i})\text{vec}(I_{p_i})^\top\\
    \end{align*}
    for $i=1,2$. Assume that there exists a constant $\tau>1$ such that $\tau^{-1}<p_i/\sqrt{n}<\tau$ for $i=1,2$ and $r$ is fixed. Then for any fixed constants $t_1,t_2,t_3\geq 2$, it holds that 
\begin{align*}
    ||(\sum_{i=1}^n Y_i\otimes Y_i)\circ \Pi_2||_2, ||(\sum_{i=1}^n Y_i^\top \otimes Y_i^\top)\circ \Pi_1||_2&\leq O(t_1\xi_{1})+\parentheses{O(t_2\xi_2)+O(t_3\xi_3)+(1-\lambda)n\sqrt{p_1p_2}}\mathbf{1}_{\lambda\neq 1}
\end{align*}
with probability at least $1-t_1^{-c_1(p_1+p_2)}-t_2^{-c_2(p_1+p_2)}-t_3^{-c_3(\sqrt{p_1}+\sqrt{p_2})}$ for some universal constants $c_1,c_2,c_3>0$, where 
\begin{align*}
    &\xi_1=\sqrt{n}(p_1+p_2), \quad \xi_2=\sqrt{(1-\lambda)}\sqrt{np_{\max}}(p_1+p_2),\quad \xi_3=(1-\lambda)\sqrt{n}p_{\max}(\sqrt{p_1}+\sqrt{p_2}). 
\end{align*}
Here $p_{\max}=\max\set{p_1,p_2}$.
\end{lemma}
\begin{proof}
We bound the $||(\sum_{i=1}^n Y_i\otimes Y_i)\circ \Pi_2||_2$ as $||(\sum_{i=1}^n Y_i^\top \otimes Y_i^\top)\circ \Pi_1||_2$ can be bounded similarly. Hence, we simply write $\Pi:=\Pi_2$. Observe that 
\begin{align*}
    \sum_{i=1}^n Y_i\otimes Y_i&=\lambda \sum_{i=1}^n E_i\otimes E_i+\sqrt{\lambda(1-\lambda)}\sum_{i=1}^n E_i\otimes U_i+\sqrt{\lambda(1-\lambda)}\sum_{i=1}^n U_i\otimes E_i\\
    &+(1-\lambda)\sum_{i=1}^n U_i\otimes U_i\\
    &=(\rom{1})+(\rom{2})+(\rom{3})+(\rom{4}).
\end{align*}
where $U_i=\sum_{j=1}^r z_{ij}A_j=A_R(z_i\otimes I_{p_2})$ for $A_R=[A_1,\ldots,A_r]$ and $z_j=[z_{1j},\ldots,z_{rj}]^\top$. Using Theorem $C.1$ of \cite{oliveira2026.supp}, we have that
\begin{align*}
    ||(\rom{1})\circ \Pi||_2\leq O(t\sqrt{n}(p_1+p_2))
\end{align*}
for every $t\geq 2$ with probability at least $1-t^{-c_1(p_1+p_2)}$ for some universal constant $c_1$ as $\lambda< 1$.\\

Now we bound $||(\rom{2})\circ \Pi||_2$ and $||(\rom{4})\circ \Pi||_2$ with high probability, where the bound on $||(\rom{3})\circ \Pi||_2$ can be obtained similarly from $||(\rom{2})\circ \Pi||_2$. Since $||\Pi||_2=1$ as $\Pi$ is an orthogonal projection operator, we initially have easy bounds; $||(A)\circ \Pi||_2\leq ||A||_2||\Pi||_2=||A||_2$ for $A=\rom{2},\rom{4}$. Also, since $\lambda< 1$, 
\begin{align*}
    ||(\rom{2})||_2&\leq \sqrt{1-\lambda}||\sum_{i=1}^n E_i\otimes A_R(z_i\otimes I_{p_2})||_2=\sqrt{1-\lambda}||(I_{p_1}\otimes A_R)\parentheses{\sum_{i=1}^n E_i\otimes z_i}(I_{p_2}\otimes I_{p_2)}||_2\\
    &\leq \sqrt{1-\lambda}||A_R||_2 ||\sum_{i=1}^n E_i\otimes z_i||_2\leq \sqrt{1-\lambda}\sqrt{p_{\max}} ||\sum_{i=1}^n E_i\otimes z_i||_2.
\end{align*}
In the last inequality, we used Lemma \ref{secA.3.3.lemma1} to obtain 
\begin{align*}
    ||A_R||_2=\sqrt{||A_RA_R^\top||_2}=\sqrt{||\sum_{i=1}^r A_iA_i^\top||_2}=\sqrt{||p_2I_{p_1}||_2}\leq \sqrt{p_{\max}}.
\end{align*}
Hence, by Lemma \ref{secA.3.3.lemma6}, for any fixed constant $t_2\geq 2$, there exists a universal constant $c_2>0$ such that with probability at least $1-t_2^{-c_2(p_1+p_2)}$,
\begin{align*}
 ||(\rom{2})\circ \Pi||_2&\leq O(t_2\sqrt{1-\lambda}\sqrt{np_{\max}}(p_1+p_2)).
\end{align*}
On the other hand, note that
\begin{align*}
    \bbE[U_1\otimes U_1]=\sum_{j=1}^r\bbE[z_{1j}^2]A_j\otimes A_j+\sum_{j\neq j'}\bbE[z_{1j}z_{1j'}]A_j\otimes A_{j'}=\sum_{j=1}^r A_j\otimes A_j
\end{align*}
and $U_i\otimes U_i=(A_R\otimes A_R)(z_{i}\otimes I_{p_2}\otimes z_{i})\otimes I_{p_2}$. Thus, 
\begin{align*}
    ||(\rom{4})||_2/(1-\lambda)&\leq ||(\rom{4})-\bbE[(\rom{4})]||_2+||\bbE[(\rom{4})]||_2\\
    &=||(A_R\otimes A_R)\parentheses{\sum_{i=1}^n(z_i\otimes I_{p_2}\otimes z_i-\bbE[z_i\otimes I_{p_2}\otimes z_i])}\otimes I_{p_2}||_2 +n||\sum_{i=1}^r A_i\otimes A_i||_2\\
    &\leq ||A_R||_2^2 ||\sum_{i=1}^n(z_i\otimes I_{p_2}\otimes z_i-\bbE[z_i\otimes I_{p_2}\otimes z_i])||_2+n\sqrt{p_1p_2}\\
    &\leq p_{\max}||\sum_{i=1}^n(z_i\otimes I_{p_2}\otimes z_i-\bbE[z_i\otimes I_{p_2}\otimes z_i])||_2 +n\sqrt{p_1p_2}.
\end{align*}
In the second inequality, we used Lemma \ref{secA.3.3.lemma1}. Take $m=p_2$ and $q\asymp \sqrt{p_1}+\sqrt{p_2}$ for sufficiently large $n$ in Lemma \ref{secA.3.3.lemma5}. Then by the assumption that $\tau^{-1}<p_1/\sqrt{n},p_2/\sqrt{n} <\tau$, $m^{\frac{1}{2q}}$ is bounded above in $n$. Thus, by Lemma, for any fixed constant $t_3\geq 2$, with probability at least $1-t_3^{-c_3(\sqrt{p_1}+\sqrt{p_2})}$, 
\begin{align*}
  ||(\rom{4})||_2/(1-\lambda)&\leq O(t_3 p_{\max}\sqrt{n}(\sqrt{p_1}+\sqrt{p_2})) +n\sqrt{p_1p_2}.
\end{align*}
for some universal constant $t_3\geq 2$.
\end{proof}

Now we prove Lemma \ref{secA.3.2.lemma2}.

\begin{proof}[\textbf{Proof of Lemma \ref{secA.3.2.lemma2}.}]
We begin the proof when $\epsilon,\eta\in(0,1)$ are constants and $\alpha=0$ so that $1-7/16\eta<\lambda$. The proof for the other cases easily follow from the below. We first prove the case when $\alpha=0$ is a constant such that $1-7/16\eta\leq \lambda$. We shall describe the events $E_1$ and $E_2$ that imply (\ref{secA.3.1.def1.eq1})--(\ref{secA.3.1.def1.eq2}) with high probability, respectively. By (2.11) of \cite{oliveira2026}, (\ref{secA.3.1.def1.eq1}) holds if $\sqrt{p_i}||\nabla_i f||_2\leq \epsilon\tr(\rho)$ and $\tr(\rho)$ is lower bounded by some constant with high probability. Take $\epsilon=1/8$ in Lemma \ref{secA.3.3.lemma2}. Then $\tr(\rho)\in(7/8,9/8)$ with high probability. On such event, substituting $\frac{7}{63}\epsilon$ into $\epsilon$ in Lemma \ref{secA.3.2.lemma1}, then we see $\sqrt{p_i}||\nabla_i f||_2\leq \epsilon\tr(\rho)$. We denote this event by $E_1$. On the other hand, from (\ref{secA.3.1.eq4}), we see that (\ref{secA.3.1.def1.eq2}) holds if $\tr(\rho)$ is lower bounded by $7/8$ and 
\begin{align*}
   \Gamma:= \max\set{\frac{1}{np}||(\sum_{i=1}^n Y_i\otimes Y_i)\circ \Pi_2||_2, \frac{1}{np}||(\sum_{i=1}^n Y_i^\top\otimes Y_i^\top)\circ \Pi_1||_2}\leq \frac{7}{8}\frac{\eta}{\sqrt{p_1p_2}}
\end{align*}
with high probability. We denote this event by $E_2$.\\

Hence, we lower bound the probability of the events $E_1$ and $E_2$. The lower bound on the probability of $E_1$ is readily obtained by substituting $1/8$ into $\frac{7}{63}\epsilon$ into $\epsilon$ of Lemma \ref{secA.3.2.lemma1} and $1/8$ into $\epsilon of$ Lemma \ref{secA.3.3.lemma2}, which gives the probability guarantee that
\begin{align*}
1-6\exp\parentheses{-c_1np_{\min}\epsilon^2}-6\exp\parentheses{-c_2\frac{n\epsilon^2}{1-\lambda}}-6\exp\parentheses{-c_3n}.
\end{align*}
for  some universal constants $c_1,c_2,c_3>0$ possibly different from those in Lemma \ref{secA.3.2.lemma1}. To lower bound the probability of the event $E_2$, take 
\begin{align*}
    t_1&=O(\frac{7\eta}{48}\frac{\sqrt{np}}{p_1+p_2}),\quad, t_2=O(\frac{7\eta}{48\sqrt{1-\lambda}}\frac{\sqrt{np_{\min}}}{p_1+p_2}),\quad  t_3=O(\frac{7\eta}{48(1-\lambda)}\frac{\sqrt{np_{\min}}}{\sqrt{p_1}+\sqrt{p_2}})
\end{align*}
in Lemma \ref{secA.3.3.lemma7}. Since $1-7/16 \eta<\lambda$ and $\Gamma$ is bounded above by $\frac{7}{8}\frac{\eta}{\sqrt{p_1p_2}}$ with probability at least
\begin{align*}
    &1-\parentheses{\frac{\sqrt{np}}{{p_1+p_2}}}^{-c_4(p_1+p_2)}-\parentheses{\frac{1}{\sqrt{1-\lambda}}\frac{\sqrt{np_{\min}}}{p_1+p_2}}^{-c_5(p_1+p_2)}-\parentheses{\frac{1}{1-\lambda}\frac{\sqrt{np_{\min}}}{\sqrt{p_1}+\sqrt{p_2}}}^{-c_6(\sqrt{p_1}+\sqrt{p_2})}
\end{align*}
for some constants $c_4,c_5,c_6>0$, we have that
\begin{align*}
  \frac{7\eta}{48}\frac{3}{\sqrt{p_1p_2}}+\frac{1-\lambda}{\sqrt{p_1p_2}}\leq \frac{7}{8}\frac{\eta}{\sqrt{p_1p_2}},
\end{align*}
proving the claim.\\

The remaining cases are $\epsilon$ is a constant with $\alpha>0$ and $\epsilon=O(n^{-\delta})$ for $\delta\in (0,\min\set{(\alpha+1/2)/2,1/2}$. In the former case, if $\alpha>0$, we have that $1-7/16\eta<\lambda$ for any constant $\eta\in(0,1)$ whenever $n$ is sufficiently large. If $\epsilon=O(n^{-\delta})$ for $\delta\in (0,\min\set{(\alpha+1/2)/2,1/2})$, $\tr(\rho)$ can be lower bounded by $1-c$ for any sufficiently small $c$ with sufficiently large $n$. Hence, we can remove the constraint that $1-7/16\eta<\lambda$.  
\end{proof}

Finally, we prove Lemma \ref{secA.3.2.lemma3}.

\begin{proof}[Proof of Lemma \ref{secA.3.2.lemma3}.]
    We prove the claim for $\Phi_Z^{(1,2)}$ as that for $\Phi_Z^{(2,1)}$ follows by symmetry. Let $R:=(\frac{1}{np_1}\sum_{j=1}Y_j^\top Y_j)^{-1/2}$. Note that with sufficiently large constant $C>0$, $R$ is non-singular with probability $1$ whenever $n\geq C\frac{p_{\max}}{p_{\min}}$. Also,
\begin{align*}
    R^{-2}-I_{p_2}=\tr(\rho)\parentheses{p_2\frac{\Phi_Z^{(2,1)}(I_{p_1})}{\tr(\Phi_Z^{(2,1)}(I_{p_1}))}-I_{p_2]}}+(\tr(\rho)-1)I_{p_2}.
\end{align*}
If $|\tr(\rho)-1|\leq O(n^{-\delta})$ and $\Phi_Y^{(2,1)}$ is a $(O(n^{-\delta}),\eta)-$quantum expander for some constants $\delta,\eta\in(0,1)$, the above implies that 
\begin{align*}
    ||R^{-2}-I_{p_2}||_2=O(n^{-\delta}) \Rightarrow ||R^2-I_{p_2}||_2&\leq ||R^{-2}-I_{p_2}||_2/||R^{-2}||_2\\
    &\leq ||R^{-2}-I_{p_2}||_2/ (1-||R^{-2}-I_{p_2}||_2)=O(n^{-\delta}).
\end{align*}
Hence, by Lemma \ref{secA.3.1.lemma2}, $\Phi_Z^{(1,2)}$ is an $(O(n^{-\delta}),\eta+O(n^{-\delta})-$quantum expander. By Lemma \ref{secA.3.2.lemma2} and Lemma \ref{secA.3.3.lemma2}, the event that $|\tr(\rho)-1|\leq O(n^{-\delta})$ and $\Phi_Y^{(2,1)}$ is a $(O(n^{-\delta}),\eta)-$quantum expander for some constants $\delta\in(0,\min\set{(\alpha+1/2)/2,1/2}),\eta\in(0,1)$ holds with the desired probability guarantee.
\end{proof}

\subsection{Proofs of the Results from Section \ref{sec4.3:asymp.null}}\label{secA.4}
We prove Lemma \ref{sec4.3.lemma1}, Corollary \ref{sec4.3.cor1}, Theorem \ref{sec4.3.thm1} and Proposition \ref{sec4.3.prop1}. Throughout this subsection, we will assume $K=I_p$ by the equivariance of the Kronecker MLE (separable component) with respect to $GL_{p_1,p_2}$. Hence, $K^{1/2}$ in \hyperlink{A3}{\textbf{(A3)}} can be set as $I_p$ so that we have $||\hat{K}-I_p||_2=O(a_n)$ with probability $1-o(1)$. Also, we will often use the following observation. For $\tilde{C}$ and $\hat{C}$ in Lemma \ref{sec4.3.lemma1}, $c(\tilde{C})$ and $\hat{C}$ share the same eigenvalues for any $K\in\calS_{p_1,p_2}^+$. Indeed, by the equivariance of the Kronecker MLE, 
\begin{align*}
   \hat{K}=k(S)=K^{1/2}k(\tilde{C})K^{1/2,\top}\equiv K^{1/2}\tilde{K}K^{1/2,\top}.
\end{align*}
Then any square root of $k(S)$ can be written as $K^{1/2}\tilde{K}^{1/2}O$ for some $O\in\calO_{p_1,p_2}$. Therefore, 
\begin{align}\label{secA.4.eq1}
\begin{split}
    c(\tilde{C})&=\tilde{K}^{-1/2}\tilde{C}\tilde{K}^{-1/2,\top},\\
c(S)&=\hat{C}=O^\top \tilde{K}^{-1/2}K^{-1/2}SK^{-1/2,\top}\tilde{K}^{-1/2,\top} O\\
&=O^\top  \tilde{K}^{-1/2}K^{-1/2}K^{1/2}\tilde{C}K^{1/2,\top}K^{-1/2,\top}\tilde{K}^{-1/2,\top} O\\
&=O^\top  \tilde{K}^{-1/2}\tilde{C}\tilde{K}^{-1/2,\top} O=O^\top c(\tilde{C})O,
\end{split}
\end{align}
proving the desired claim. Since $||\cdot||_2$ is orthogonally invariant, we also have that $||\tilde{K}-I_p||_2=O(a_n)$ with probability $1-o(1)$.\\

For the proof of Lemma \ref{sec4.3.lemma1}, we use the following ancillary result. 

\begin{lemma}\label{secA.4.lemma1}
For $A,B\in\calS_p^+$, let $A^{1/2}$ and $B^{1/2}$ denote the symmetric square roots of $A$ and $B$, respectively. Then
\begin{align*}
    ||A^{1/2}-B^{1/2}||_2\leq \frac{1}{\sqrt{\lambda_p(A)}+\sqrt{\lambda_p(B)}} ||A-B||_2. 
\end{align*}
\end{lemma}
\begin{proof}
  Suppose $x$ is a unit-norm eigenvector of $A^{1/2}-B^{1/2}$ associated with an eigenvalue $\eta$. Then it holds that 
  \begin{align*}
            x^\top(A-B)x&=x^\top (A^{1/2}-B^{1/2})A^{1/2}x+x^\top B^{1/2}(A^{1/2}-B^{1/2})x\\
      &=\tau x^\top A^{1/2}x+\tau x^\top B^{1/2}x.
  \end{align*}
  Note that $A^{1/2}-B^{1/2}$ takes either $\pm ||A^{1/2}-B^{1/2}||_2$ as an eigenvalue, and we assume $\eta= ||A^{1/2}-B^{1/2}||_2$ without loss of generality, since we can take $-x$ instead of $x$ if $\eta=-||A^{1/2}-B^{1/2}||_2$. Then by the above, 
  \begin{align*}
      x^\top(A-B)x &\geq||A^{1/2}-B^{1/2}||_2\parentheses{\lambda_p(A^{1/2})+\lambda_p(B^{1/2})}\\
      &=||A^{1/2}-B^{1/2}||_2\parentheses{\sqrt{\lambda_p(A)}+\sqrt{\lambda_p(B)}} .
\end{align*}
Thus,
\begin{align*}
    ||A^{1/2}-B^{1/2}||_2&\leq \frac{1}{\sqrt{\lambda_p(A)}+\sqrt{\lambda_p(B)}} x^\top(A-B)x\\
    &\leq \frac{1}{\sqrt{\lambda_p(A)}+\sqrt{\lambda_p(B)}}  ||A-B||_2,
\end{align*}
which concludes the proof.
\end{proof}

\begin{proof}[Proof of Lemma \ref{sec4.3.lemma1}.]
 By the equivariance of the Kronecker MLE, $K=I_p$. Also, in view of (\ref{secA.4.eq1}), we simply set $\tilde{K}^{1/2}$ as a symmetric square root of $\tilde{K}$. Under the assumption of \hyperlink{A3}{\textbf{(A3)}}, we have that $||\tilde{K}-I_p||_2=O_p(a_n)$ as discussed above. Thus,
 \begin{align*}
    ||\tilde{K}^{-1}-I_p||_2\leq ||\tilde{K}-I_p||_2/||\tilde{K}||_2\leq  ||\tilde{K}-I_p||_2/(1-||\tilde{K}-I_p||_2)=O_p(a_n).
\end{align*}
By Weyl's inequality and (\ref{secA.4.eq1}), 
\begin{align*}
    \max_{i\in [p]}|\lambda_i(\hat{C})-\lambda_i(\tilde{C})|&=\max_{i\in [p]}|\lambda_i(\tilde{K}^{-1/2}\tilde{C}\tilde{K}^{-1/2})-\lambda_i(\tilde{C})|\\
    &\leq ||\tilde{K}^{-1/2}\tilde{C}\tilde{K}^{-1/2}-\tilde{C}||_2\\
    &= ||(\tilde{K}^{-1/2}-I_p)\tilde{C}(\tilde{K}^{-1/2}-I_p)+\tilde{C}(\tilde{K}^{-1/2}-I_p)+(\tilde{K}^{-1/2}-I_p)\tilde{C}||_2\\
    &\leq ||\tilde{C}||_2||\tilde{K}^{-1/2}-I_p||_2^2+2||\tilde{C}||_2 ||\tilde{K}^{-1/2}-I_p||_2\\
&\leq  ||\tilde{C}||_2||\tilde{K}^{-1}-I_p||_2^2+2||\tilde{C}||_2 ||\tilde{K}^{-1}-I_p||_2=O_p(a_n),
\end{align*}
where the last inequality holds because $||\tilde{C}||_2=O_p(1)$ by the assumption, along with the result of Lemma \ref{secA.4.lemma1}.
\end{proof}

For the proof of Corollary \ref{sec4.3.cor1}, we use the result from \cite{gotze2011}. Recall that $F_{MP}^\gamma$ and $d_{KS}$ denote the distribution of the MP law and the Kolmogorov-Smirnov (KS) distance, respectively. Here $\gamma=\gamma_1\gamma_2$ for $\gamma_1,\gamma_2$ in \hyperlink{A1}{\textbf{(A1)}}.

\begin{lemma}\label{secA.4.lemma2}
 Assume the same as Corollary \ref{sec4.3.cor1} so that $\tilde{C}$ in Lemma \ref{sec4.3.lemma1} reduces to $1/n\sum_{i=1}^n z_iz_i^\top$ for $z_i$'s in (\ref{sec4.3.eq2}). If $F_{\tilde{C}}$ is the empirical spectral distribution of $\tilde{C}$, $d_{KS}(F_{\tilde{C}},F_{MP}^\gamma)=O_p(\log n/\sqrt{n})$.
\end{lemma}
\begin{proof}
 Since the given assumptions satisfy those of Theorem $1.1$ from \cite{gotze2011}, the result follows.
\end{proof}
Note that the rate in Lemma \ref{secA.4.lemma2} can be refined further by Theorem $1.1$ from \cite{gotze2011}. However, we take a more conservative rate for the result of Corollary \ref{sec4.3.cor1}.  
\begin{proof}[Proof of Corollary \ref{sec4.3.cor1}.]
Under the given assumptions, $\tilde{C}$ in Lemma \ref{sec4.3.lemma1} reduces to $1/n\sum_{i=1}^n z_iz_i^\top$ for $z_i$'s in (\ref{sec4.3.eq2}). Also, in view of (\ref{secA.4.eq1}), because $F_{\hat{C}}$ and $F_{\tilde{C}}$ depend on $\hat{C}$ and $\tilde{C}$ only through their eigenvalues, respectively, we shall assume $\hat{C}=\tilde{K}^{-1/2}\tilde{C}\tilde{K}^{-1/2}$ without loss of generality, where $\tilde{K}^{1/2}$ is a symmetric square root of $\tilde{K}=k(\tilde{C})$. Furthermore, we assume $\gamma:=\gamma_1\gamma_2\in(0,1]$ without loss of generality (see page $196$ 
of \cite{bai2012b} and page $2$ of \cite{gotze2011}).\\

Let $\delta_n:=||\hat{C}-\tilde{C}||_2$. We first claim that $F_{\hat{C}}(x)\leq F_{\tilde{C}}(x+\delta_n)$ for any $x\in \real$. If $x<0$, we are done. Also, if $x\geq \lambda_1(\hat{C})$, $F_{\hat{C}}(x)=1$. Also, by Weyls' inequality, 
\begin{align*}
    x+\delta_n\geq \lambda_1(\hat{C})+\delta_n\geq \lambda_1(\tilde{C}).
\end{align*}
Thus, $F_{\tilde{C}}(x)=1$. Hence, assume $x\in [0,\lambda_1(\hat{C}))$. Note that when $\gamma\leq 1$, both $\tilde{C}$ and $\hat{C}$ are strictly positive definite for sufficiently large $n$ with probability $1$. In fact, they share the same rank as $\hat{C}$ obtained by whitening $\tilde{C}$ via a non-singular matrix. Then there exists $i\in[p]$ such that $\lambda_{i+1}(\hat{C})\leq x<\lambda_{i}(\hat{C})$, where $\lambda_{p+1}(\hat{C})=0$, where $F_{\hat{C}}(x)=(p-i)/p$. Again by Weyls' inequality, 
\begin{align*}
      x+\delta_n\geq \lambda_{i+1}(\hat{C})+\delta_n\geq \lambda_{i+1}(\tilde{C}).
\end{align*}
This is true even when $i=p$, where $\lambda_{p+1}(\tilde{C})=0$. Hence, $F_{\tilde{C}}(x+\delta_n)\geq (p-i)/p=F_{\hat{C}}(x)$. \\

Next, assume $0\leq F_{\hat{C}}(x)-F_{\tilde{C}}(x)$ without loss of generality, which we will justify below. By the previous claim, 
\begin{align*}
0\leq F_{\hat{C}}(x)-F_{\tilde{C}}(x)&\leq F_{\tilde{C}}(x+\delta_n)-F_{\tilde{C}}(x)\\
&\leq   (F_{\tilde{C}}(x+\delta_n)-F_{MP}^{\gamma}(x+\delta_n))+(F_{MP}^{\gamma}(x+\delta_n)-F_{MP}^{\gamma}(x))+(F_{MP}^{\gamma}(x)-F_{\tilde{C}}(x))\\
&\leq 2\sup_{x\in\real}|F_{\tilde{C}}(x)-F_{MP}^{\gamma}(x)|+(F_{MP}^{\gamma}(x+\delta_n)-F_{MP}^{\gamma}(x)).
\end{align*}
By Lemma \ref{secA.4.lemma2}, $\sup_{x\in\real}|F_{\tilde{C}}(x)-F_{MP}^{\gamma}(x)|=O_p(\log n/\sqrt{n})$. On the other hand, if $\gamma\in (0,1]$, the density of the MP law with parameter $\gamma$, dominated by Lebesgue measure, is bounded on $\real$. Denote the upper bound by $L$. Thus, $F_{MP}^{\gamma}$ is Lipschitz continous so that 
\begin{align*}
    F_{MP}^{\gamma}(x+\delta_n)-F_{MP}^{\gamma}(x)\leq L\delta_n=O_p(\delta_n).
\end{align*}
By \hyperlink{A3}{\textbf{(A3)}} with $\lambda=1$, $O_p(\delta_n)=O_p(\log n/\sqrt{n})$ by Lemma \ref{sec4.3.lemma1}. If $ F_{\hat{C}}(x)-F_{\tilde{C}}(x)<0$, note that $F_{\hat{C}}(x)\geq F_{\tilde{C}}(x-\delta_n)$ by changing the role of $F_{\hat{C}}$ and $F_{\tilde{C}}$ above. Hence,
\begin{align*}
    0\leq F_{\tilde{C}}(x)- F_{\hat{C}}(x)&\leq F_{\tilde{C}}(x)-F_{\tilde{C}}(x-\delta_n)\\
    &\leq   (F_{\tilde{C}}(x)-F_{MP}^{\gamma}(x))+(F_{MP}^{\gamma}(x)-F_{MP}^{\gamma}(x-\delta_n))+(F_{MP}^{\gamma}(x-\delta_n)-F_{\tilde{C}}(x-\delta_n))\\
&\leq 2\sup_{x\in\real}|F_{\tilde{C}}(x)-F_{MP}^{\gamma}(x)|+(F_{MP}^{\gamma}(x)-F_{MP}^{\gamma}(x-\delta_n))
\end{align*}
so that $0\leq F_{\tilde{C}}(x)- F_{\hat{C}}(x)=O_p(\log n/\sqrt{n})$. Since the upper bound for $|F_{\tilde{C}}(x)- F_{\hat{C}}(x)|$ is uniform in $x$, $d_{KS}(\tilde{C},\hat{C})=\sup_{x\in\real }|F_{\tilde{C}}(x)- F_{\hat{C}}(x)|=O_p(\log n/\sqrt{n})$. 
\end{proof}

To proceed with the proof of Theorem \ref{sec4.3.thm1}, we introduce an additional condition for edge universality of $\tilde{C}$ in Lemma \ref{sec4.3.lemma1} (see \cite{bao2015,lee2016}).

\begin{cond}\label{secA.4.cond1}
Suppose $\Sigma\in\calS_p^+$. Let $\xi_{+,\Sigma}$ be that in Section \ref{sec4.3:asymp.null} with $\Omega=\Sigma$. Assuming  $p/n$ converges to some positive constant as $n\rightarrow\infty$, $\Sigma$ satisfies $0<\liminf_{n}\lambda_p(\Sigma)\leq \limsup_{n}\lambda_1(\Sigma)<\infty$, and $\limsup_{n}\lambda_1(\Sigma)\xi_{+,\Sigma}<1$.
\end{cond}

The following lemma is useful to prove Theorem \ref{sec4.3.thm1}.

\begin{lemma}\label{secA.4.lemma3}
Suppose $P,Q\in\calS_{p}^+$ and $Q$ satisfies Condition \ref{secA.4.cond1}. Assume $p/n$ converges to some positive constant as $n\rightarrow\infty$, $||P-Q||_2=O(a_n)$ for some positive sequence $\set{a_n}_{n=1}^\infty$ with $a_n=o(1)$, and suppose $c_1<\xi_{+,Q},\gamma_{0,Q}<c_2$ for some positive constants $c_1,c_2$. Then $P$ also satisfies Condition \ref{secA.4.cond1} and it holds that  
\begin{align}\label{secA.4.lemma3.eq1}
    \gamma_{0,P}=\gamma_{0,Q}+o(1), \quad E_{+,P}=E_{+,Q}+O(a_n)
\end{align}
\end{lemma}
\begin{proof}
To verify that $0<\liminf_{n}\lambda_p(Q)\leq \limsup_{n}\lambda_1(Q)<\infty$, observe that $\lambda_p(Q)-||P=Q||_2\leq \lambda_p(P)\leq \lambda_1(P)\leq \lambda_1(Q)+||P-Q||_2$ by Weyl's inequality. Since $||P-Q||_2=o(1)$ by the assumption, the result follows.\\

We verify the condition that $\limsup_{n}\lambda_1(P)\xi_{+,P}<1$ and (\ref{secA.4.lemma3.eq1}) simultaneously, proving the claim. We invoke the proof of Lemma $1$--$2$ of \cite{han2016}. Suppose $\alpha_i=\lambda_i(P)$ and $\beta_i=\lambda_i(Q)$ for $i\in [p]$. Note that by Weyl's inequality, $\max_{i\in[p]}|\alpha_i-\beta_i|\leq ||P-Q||_2=O(a_n)$. Let 
\begin{align*}
    f_{P}(x)\equiv\frac{1}{p}\int \parentheses{\frac{tx}{1-tx}}^2 dF_{P}(t)=\frac{1}{p}\sum_{i=1}^p\parentheses{\frac{\alpha_ix}{1-\alpha_ix}}^2.
\end{align*}
Define $f_Q(x)$ similarly. Note that $\xi_{+,Q}\in [0,\beta_1^{-1})$ and $f_Q(\xi_{+,Q})=1/\hat{\gamma}=n/p$. Then
\begin{align*}
  |f_P(\xi_{+,Q})-n/p|&=|f_P(\xi_{+,Q})-f_Q(\xi_{+,Q})|\\
  &= \frac{1}{p}\left|\sum_{i=1}^p\parentheses{\parentheses{\frac{\xi_{+,Q}}{1/\alpha_i-\xi_{+,Q}}}^2-\parentheses{\frac{\xi_{+,Q}}{1/\beta_i-\xi_{+,Q}}}^2} \right|\\
  &\leq \frac{\xi_{+,Q}^2}{p}\max_{i\in [p]}\set{\frac{|1/\alpha_i+1/\beta_i-2\xi_{+,Q}|}{(1/\alpha_i-\xi_{+,Q})^2(1/\beta_i-\xi_{+,Q})^2}} \sum_{i=1}^p |1/\alpha_i-1/\beta_i|\\
  &\leq \frac{\xi_{+,Q}^2}{p}\max_{i\in [p]}\set{\frac{|1/\alpha_i-1/\beta_i|+2|1/\beta_i-\xi_{+,Q}|}{(1/\alpha_i-\xi_{+,Q})^2(1/\beta_i-\xi_{+,Q})^2}} \sum_{i=1}^p |1/\alpha_i-1/\beta_i|.
\end{align*}
To bound further, note that
\begin{align*}
    |1/\alpha_i-1/\beta_i|=\frac{|\alpha_i-\beta_i|}{|\beta_i^2+\beta_iO(a_n)|}\leq \frac{||P-Q||_2}{\beta_p^2+\beta_pO(a_n)}=O(a_n)
\end{align*}
as $0<\liminf_n \beta_p<\infty$ by Condition \ref{secA.4.cond1}. Consequently, 
\begin{align}\label{secA.4.lemma3.eq12}
    1/p\sum_{i=1}^p |1/\alpha_i-1/\beta_i|=O(a_n).
\end{align}
Also,
\begin{align*}
\max_{i\in [p]}\set{\frac{|1/\alpha_i-1/\beta_i|+2|1/\beta_i-\xi_{+,Q}|}{(1/\alpha_i-\xi_{+,Q})^2(1/\beta_i-\xi_{+,Q})^2}}&\leq \max_{i\in [p]}\frac{|1/\alpha_i-1/\beta_i|}{(1/\alpha_i-\xi_{+,Q})^2(1/\beta_i-\xi_{+,Q})^2}\\
&+2\max_{i\in [p]}\frac{|1/\beta_i-\xi_{+,Q}|}{(1/\alpha_i-\xi_{+,Q})^2(1/\beta_i-\xi_{+,Q})^2}\\
&=:(\rom{1})+(\rom{2}).
\end{align*}
Observe that 
\begin{align*}
    (\rom{1})&\leq \max_{i\in [p]}|1/\alpha_i-1/\beta_i| \cdot\max_{i\in [p]}\frac{1}{(1/\alpha_i-\xi_{+,Q})^2(1/\beta_i-\xi_{+,Q})^2}\\
    &=O(a_n)\max_{i\in [p]}\frac{1}{(1/\beta_i-\xi_{+,Q}+O(a_n))^2(1/\beta_i-\xi_{+,Q})^2}\\
    &=\beta_1^2 O(a_n)\max_{i\in [p]}\frac{1}{(\beta_1/\beta_i-\beta_1\xi_{+,Q}+O(a_n))^2(\beta_1/\beta_i-\beta_1\xi_{+,Q})^2}.
\end{align*}
Since $\beta_1/\beta_i\geq 1$ and $\limsup_{n}\beta_1\xi_{+,Q}<1$, we have $\max_{i\in [p]}1/(\beta_1/\beta_i-\beta_1\xi_{+,Q})^2=O(1)$. Moreover, $\beta_1=O(1)$ by Condition \ref{secA.4.cond1}. Therefore, $(\rom{1})=O(a_n)$. Similarly, one can observe that $(\rom{2})=O(1)$. Consequently, 
\begin{align*}
    \max_{i\in [p]}\set{\frac{|1/\alpha_i-1/\beta_i|+2|1/\beta_i-\xi_{+,Q}|}{(1/\alpha_i-\xi_{+,Q})^2(1/\beta_i-\xi_{+,Q})^2}}=O(1).
\end{align*}
Since there exist positive constants $c_1,c_2$ for which $c_1<\xi_{+,Q}<c_2$, we consequently have that
\begin{align*}
    |f_{P}(\xi_{+,Q})-n/p|=O(a_n).
\end{align*}
For sufficiently large $n$, suppose $0<\xi_{+,Q}-||P-Q||_2<x<\xi_{+,Q}+||P-Q||_2$. If $n$ is large enough, then 
\begin{align*}
  \xi_{+,Q}+||P-Q||_2=\xi_{+,Q}\beta_1(1/\beta_1-1/\alpha_1)+\xi_{+,Q}\beta_1/\alpha_1+O(a_n)<1/\alpha_1
\end{align*}
as $\limsup_{n}\xi_{+,Q}\beta_1<1$ and $1/\beta_1-1/\alpha_1=O(a_n)$. Also, for such $x$,
\begin{align*}
    |f_P(\xi_{+,Q})-f_P(x)|&=\frac{1}{p}\left|\sum_{i=1}^p \frac{(\xi_{+,Q}^2-x^2)/\alpha_i^2-2\xi_{+,Q}x/\alpha_i(\xi_{+,Q}-x)}{(1/\alpha_i-\xi_{+,Q})^2(1/\alpha_i-x)^2}\right|\\
    &\leq|\xi_{+,Q}-x|\frac{1}{p}\left|\sum_{i=1}^p \frac{(\xi_{+,Q}+x)/\alpha_i^2}{(1/\alpha_i-\xi_{+,Q})^2(1/\alpha_i-x)^2}\right|\\
    &+|\xi_{+,Q}-x|\frac{1}{p}\left|\sum_{i=1}^p \frac{\xi_{+,Q}x/\alpha_i}{(1/\alpha_i-\xi_{+,Q})^2(1/\alpha_i-x)^2}\right|.
\end{align*}
Noting that $\lim_n \beta_i/\alpha_i=1$ for each $i$, $\limsup_{n}\beta_1\xi_{+,Q}<1$, and $\beta_1=O(1)$, $\liminf_n\beta_p>0$, we see that 
\begin{align*}
    \xi_{+,Q}+x,\xi_{+,Q}x,\max_{i\in [p]}\frac{1}{\alpha_i},\max_{i\in [p]}\frac{1}{(1/\alpha_i-\xi_{+,Q})^2},\max_{i\in [p]}\frac{1}{(1/\alpha_i-x)^2}=O(1).
\end{align*}
Consequently, $|f_P(\xi_{+,Q})-f_P(x)|=O(a_n)$. For any sufficiently large $n$, since
\begin{align*}
    f_{P}'(x)>\frac{2}{x}f_{P}(x)>\frac{2}{\xi_{+,Q}+O(a_n)}f_P(x),
\end{align*}
$|f_P(\xi_{+,Q})-f_P(x)|,|f_P(\xi_{+,Q})-n/p|=O(a_n)$ implies that 
$f_P'(x)>M$ for some constant $M$ with $x\in(0,\alpha_1^{-1})$. Since $f_P(x)$ is continuous on $(0,\alpha_1^{-1})$, the mean value theorem leads to 
\begin{align*}
    f_P(\xi_{+,Q}-||P-Q||_2)< n/p <  f_P(\xi_{+,Q}+||P-Q||_2).
\end{align*}
for any sufficiently large $n$. Thus, the unique solution $\xi_{+,P}\in [0,\alpha_1^{-1})$ of $f_P(x)=n/p$ in $x$ satisfies $|\xi_{+,Q}-\xi_{+,P}|=O(a_n)$. Furthermore,
\begin{align*}
    \limsup_n \xi_{+,P}\alpha_1\leq \limsup_n(\beta_1+||P-Q||_2)(\xi_{+,Q}+O(a_n))=\limsup_n \beta_1\xi_{+,Q}<1.
\end{align*}
Hence, for large enough $n$, the assumption on $\xi_{+,Q}$ implies that $\xi_{+,P}$ is bounded away from $0$ and above. Therefore,
\begin{align*}
    |E_{+,P}-E_{+,Q}&|\leq \left|\frac{1}{\xi_{+,P}}-\frac{1}{\xi_{+,Q}}\right|+\frac{1}{n}\left|\sum_{i=1}^p \frac{1}{1/\alpha_i-\xi_{+,P}}-\frac{1}{1/\beta_i-\xi_{+,Q}}\right|\\
    &\leq \frac{|\xi_{+,P}-\xi_{+,Q}|}{\xi_{+,P}\xi_{+,Q}}+\frac{O(a_n)}{n}\sum_{i=1}^p\left|\frac{1}{(1/\beta_i-\xi_{+,Q})(1/\beta_i-\xi_{+,Q}+O(a_n))}\right|\\
    &=O(a_n).
\end{align*}
Similarly, 
\begin{align*}
    \left|\frac{1}{\gamma_{0,P}^3}-\frac{1}{\gamma_{0,Q}^3}\right|=O(a_n),
\end{align*}
concluding the proof.
\end{proof}

Note that if $Q=I_p$ or $1+\sqrt{\gamma}>\lambda_1(Q)\geq \cdots \geq \lambda_r(Q)>\lambda_{r+1}(Q)=\cdots=\lambda_p(Q)=1$ for $r\in \bbN$ fixed in $(p_1,p_2,n)$ and the limit $\gamma$ of $p/n$ in $n$, the condition that $c_1<\gamma_{0,Q},\xi_{+,Q}<c_2$ is readily met. For the latter case, see \cite{bao2015,feral2009}.

\begin{proof}[Proof of Theorem \ref{sec4.3.thm1}.]
For simplicity, let $\tilde{T}_1(Y):=\gamma_0n^{2/3}(T_1(Y)-\hat{E}_+)$. Under the given assumptions, it suffices to prove that 
\begin{align*}
    \lim_{n\rightarrow}\bbP(\tilde{T}_1(Y)\leq x) =F_1(x)
\end{align*}
for any $x\in\real$ and the distribution function $F_1$ of $TW_1$. Again, by the equivariance of the Kronecker MLE (separable component), we assume $K=I_p$ in \hyperlink{A2}{\textbf{(A2)}}. Also, as in the proofs of Lemma \ref{sec4.3.lemma1} and Corollary \ref{sec4.3.cor1}, and (\ref{secA.4.eq1}), we assume $\hat{K}=\tilde{K}$ and $\hat{C}=\tilde{K}^{-1/2}\tilde{C}\tilde{K}^{-1/2}$ for symmetric square root $\tilde{K}^{1/2}$ of $\tilde{K}=k(\tilde{C})$ without loss of generality. Denote the event that $||\tilde{K}-I_p||_2=O(a_n)$ by $E_n$. Under \hyperlink{A3}{\textbf{(A3)}}, $\bbP(E_n)=1-o(1)$. For given $x\in\real$, by the law of iterated expectation,
\begin{align*}
  \bbP(\tilde{T}_1(Y)\leq x)&= \bbP(\tilde{T}_1(Y)\leq x,E_n)+ \bbP(\tilde{T}_1(Y)\leq x,E_n^c)\\
  &=\bbE[\bbP(\tilde{T}_1(Y)\leq x,E_n|\tilde{K})]+ \bbP(\tilde{T}_1(Y)\leq x,E_n^c)
\end{align*}
Observe that $0\leq \bbP(\tilde{T}_1(Y)\leq x,E_n^c)\leq \bbP(E_n^c)=o(1)$. Since the probability is always in $[0,1]$, by dominated convergence theorem, it suffices to show that 
\begin{align}\label{secA.4.eq2}
 \lim_{n\rightarrow}\bbP(\tilde{T}_1(Y)\leq x,E_n|\tilde{K})=F_1(x).
\end{align}
Let $Z=[z_1,\ldots,z_n] \in\real^{p\times n}$ for $z_i$'s in (\ref{sec4.3.eq2}) and \hyperlink{A2}{\textbf{(A2)}}. Then with $C^{1/2}=UDV^\top$, 
\begin{align*}
    T_1(Y)=\lambda_1(\hat{C})=\lambda_1(\tilde{K}^{-1/2}C^{1/2}ZZ^\top C^{1/2,\top}\tilde{K}^{-1/2}n)=\lambda_1(Z^\top C^{1/2,\top} \tilde{K}^{-1}C^{1/2}Z/n).
\end{align*}
On the event $E_n$, we have that $||\tilde{K}^{-1}-I_p||_2=O(a_n)$.  Hence, we prove (\ref{secA.4.eq2}) by arguing that if $||\tilde{K}^{-1}-I_p||_2=O(a_n)$ and $C$ satisfies one of (\rom{1})--(\rom{3}) in Theorem \ref{sec4.3.thm1}, 
\begin{align}\label{secA.4.eq3}
    \gamma_0n^{2/3}(\lambda_1(Z^\top C^{1/2,\top} \tilde{K}^{-1}C^{1/2}Z/n)-\hat{E}_+)|\tilde{K}\Rightarrow TW_1.
\end{align}
We prove the above claim for each of conditions (\rom{1})--(\rom{3}). Note that if $\lambda$ and $\alpha$ satisfy either (\rom{2}) or (\rom{3}), we have $\lambda\rightarrow 1$ as $n\rightarrow\infty$, whereas $\lambda=1$ in (\rom{1}). For (\rom{3}), \hyperlink{A4}{\textbf{(A4)}} ensures this. Throughout the proof below, we will write $\hat{\gamma}_0:=\gamma_{0,C^{1/2,\top} \tilde{K}^{-1}C^{1/2}}$ and $\hat{E}_+:=E_{+,C^{1/2,\top} \tilde{K}^{-1}C^{1/2}}$ for $\gamma_{0,\Omega}$ and $E_{+,\Omega}$ defined in Section \ref{sec4.3:asymp.null} with $\Omega=C^{1/2,\top} \tilde{K}^{-1}C^{1/2}$.  \\

\noindent(\rom{1}): In this case, $C^{1/2}=UV^\top$ so that $C^{1/2}$ is orthogonal. Hence, given $\tilde{K}$, $||C^{1/2,\top} \tilde{K}^{-1}C^{1/2}-I_p||_2=||\tilde{K}^{-1}-I_p||_2=O(a_n)$. By Lemma \ref{secA.4.lemma3} with $Q=I_p$, $C^{1/2,\top}\tilde{K}^{-1}C^{1/2}$ satisfies Condition \ref{secA.4.cond1} and given $\tilde{K}$,
\begin{align*}
    \hat{\gamma}_0=\gamma_0+o(1),\quad \hat{E}_+=E_+ +O(a_n).
\end{align*}
Note that $\hat{E}_+$ converges to $E_+$ in probability unconditionally on $\tilde{K}$. Since $Z$ matches $N(0,1)$ up to order $4$ and \hyperlink{A1}{\textbf{(A1)}}--\hyperlink{A3}{\textbf{(A3)}} hold, combining Theorem $1.3$ of \cite{bao2015} and Corollary $2.7$ of \cite{lee2016} leads to 
\begin{align*}
    \hat{\gamma}_0n^{2/3}(\lambda_1(Z^\top C^{1/2,\top} \tilde{K}^{-1}C^{1/2}Z/n)-\hat{E}_+)|\tilde{K}\Rightarrow TW_1.
\end{align*}
Then (\ref{secA.4.eq3}) can be concluded using Slutsky's theorem as $\gamma_0/\hat{\gamma}_0$ converges to $1$ in probability.\\

\noindent(\rom{2}): Recall that $\lambda\rightarrow 1$ as $n\rightarrow\infty$ in this case. By \hyperlink{A4}{\textbf{(A4)}}, for each $i=1,\ldots,r$, 
\begin{align*}
    \frac{\lambda_i(C)}{\lambda}=\frac{1-\lambda}{\lambda}\sigma_i(A)^2+1=O(n^{-\alpha+1})+1 \Rightarrow \frac{\lambda_i(C)-\lambda}{\lambda}=O(n^{-\alpha+1}).
\end{align*}
where $O(n^{-\alpha+1})=o(1)$ if $\alpha\in(1,\infty)$. Hence, if (\rom{2}) holds, 
\begin{align*}
    ||C/\lambda-I_p||_2=||C^{1/2,\top}C^{1/2}-\lambda I_p||_2/\lambda=\max_{i\in[r]}| \lambda_i(C)-\lambda|/\lambda=O(n^{-\alpha+1}).
\end{align*}
Consequently, given $\tilde{K}$,
\begin{align*}
    ||C^{1/2,\top}\tilde{K}^{-1}C^{1/2}/\lambda-I_p||_2&\leq ||C^{1/2,\top}(\tilde{K}^{-1}-I_p)C^{1/2}||_2/\lambda+||C^{1/2,\top}C^{1/2}-\lambda I_p||_2/\lambda\\
    &\leq ||\tilde{K}^{-1}-I_p||_2||C||_2/\lambda +||C^{1/2,\top}C^{1/2}-\lambda I_p||_2/\lambda\\
    &=O(\max\set{a_n,n^{-\alpha+1}}).
\end{align*}
Thus,
\begin{align*}
      ||C^{1/2,\top}\tilde{K}^{-1}C^{1/2}-I_p||_2&=\lambda    ||C^{1/2,\top}\tilde{K}^{-1}C^{1/2}/\lambda-1/\lambda I_p||_2\\
      &\leq\lambda    ||C^{1/2,\top}\tilde{K}^{-1}C^{1/2}/\lambda-I_p||_2+(1-\lambda)=O(\max\set{a_n,n^{-\alpha+1}}).
\end{align*}
as $1-\lambda\asymp n^{-\alpha}$. Therefore, by Lemma \ref{secA.4.lemma3} with $Q=I_p$, $C^{1/2,\top}\tilde{K}^{-1}C^{1/2}$ satisfies Condition \ref{secA.4.cond1} and given $\tilde{K}$,
\begin{align*}
    \hat{\gamma}_0=\gamma_0+o(1),\quad \hat{E}_+=E_+ +O(\max\set{a_n,n^{-\alpha+1}}).
\end{align*}
Again, $\hat{E}_+$ converges to $E_+$ in probability unconditionally on $\tilde{K}$. Also, by the moment-matching condition on the entries of $z_i$'s and \hyperlink{A1}{\textbf{(A1)}}--\hyperlink{A4}{\textbf{(A4)}}, Theorem $1.3$ of \cite{bao2015} and Corollary $2.7$ of \cite{lee2016} imply (\ref{secA.4.eq3}).\\

\noindent(\rom{3}): Again, note that $\lambda\rightarrow 1$ as $n\rightarrow\infty$. By the choice of $\lambda$, $C/\lambda$ and $C^{1/2,\top}C^{1/2}/\lambda$ share the same eigenvalues with $r$ spiked eigenvalues and the non-spiked eigenvalues all equal to $1$. Moreover, $\lambda_1(C)/\lambda=\lambda_1(C^{1/2,\top}C^{1/2})/\lambda=1+c$ with $c<\sqrt{\gamma_1\gamma_2}$. By Lemma $4.2$ of \cite{bao2015}, $C^{1/2,\top}C^{1/2}/\lambda$ satisfies Condition \ref{secA.4.cond1} and we have that  
\begin{align*}
    \gamma_{0,C^{1/2,\top}C^{1/2}/\lambda}=\gamma_0+o(1),\quad E_{0,C^{1/2,\top}C^{1/2}/\lambda}=E_++O(n^{-1}).
\end{align*}
Moreover,
\begin{align*}
    ||C^{1/2,\top}\tilde{K}^{-1}C^{1/2}-C^{1/2,\top}C^{1/2}/\lambda||_2&\leq  ||C^{1/2,\top}(\tilde{K}^{-1}-I_p)C^{1/2}||_2+(1-\lambda)/\lambda ||C||_2\\
    &\leq ||C||_2 ||\tilde{K}^{-1}-I_p||_2+(1-\lambda)/\lambda ||C||_2=O(a_n),
\end{align*}
where the last equality holds because $||C||_2=1+c$ and $1-\lambda\asymp n^{-1}$, which is absorbed into $O(a_n)$. By Lemma \ref{secA.4.lemma3} with $Q=C^{1/2,\top}C^{1/2}/\lambda$, $C^{1/2,\top}\tilde{K}^{-1}C^{1/2}$ satisfies Condition \ref{secA.4.cond1} and given $\tilde{K}$,
\begin{align*}
    \hat{\gamma}_0&=\gamma_{0,C^{1/2,\top}C^{1/2}/\lambda}+o(1)=\gamma_0+o(1)\\
    \hat{E}_+&=E_{0,C^{1/2,\top}C^{1/2}/\lambda}+O(a_n)=E_{+}+O(a_n),
\end{align*}
as $O(n^{-1})$ is absorbed into $O(a_n)$. Again, $\hat{E}_+$ converges to $E_+$ in probability unconditionally on $\tilde{K}$. Therefore, applying the moment-matching condition on $Z$ and \hyperlink{A1}{\textbf{(A1)}}--\hyperlink{A4}{\textbf{(A4)}}, (\ref{secA.4.eq3}) is implied by Theorem $1.3$, Corollary $1.7$, and Lemma $4.2$ of \cite{bao2015} and Corollary $2.7$ of \cite{lee2016}.
\end{proof}

Finally, we prove Proposition \ref{sec4.3.prop1}. To introduce the additional notation, for a sequence of random variables $\set{X_n}_{n=1}^\infty$ and $\set{Y_n}_{n=1}^\infty$, if both $X_n/Y_n=O_p(1)$ and $Y_n/X_n=O_p(1)$ hold, write $X_n\asymp_p Y_n$.

\begin{proof}[Proof of Proposition \ref{sec4.3.prop1}.]
From the proof of Theorem \ref{sec4.3.thm1}, recall that for $Z=[z_1,\ldots,z_n]$ with $z_i$'s in (\ref{sec4.3.eq2}) and $\tilde{C}=ZZ^\top/n$, $\tilde{K}=k(\tilde{C})$. Throughout the proof, we will write $\beta_n(\varphi):=\bbP_{I,C}(\varphi(Y)=1)$ instead of $\bbP(\varphi(Y)=1|I,C)$, not to be confused with the conditioning notation below. Here $\varphi$ is either $\phi_1^1$ or $\phi_1^2$. Also, recall that $\hat{E}_+:=E_{+,\tilde{K}^{-1}}$ and $E_+=E_{+,I_p}$. Since $||\tilde{K}^{-1}-I_p||_2=O_p(a_n)$ by \hyperlink{A4}{\textbf{(A4)}}, Lemma \ref{secA.4.lemma3} with $Q=I_p$ implies that $\hat{E}_+=E_++O(a_n)$ given $\tilde{K}$. Lastly, let 
\begin{align}\label{sec4.3.prop1.eq1}
    \psi=\lambda_1(C)\parentheses{1+\gamma\frac{\lambda}{\lambda_1(C)-\lambda} }
\end{align}
for $\gamma=\gamma_1\gamma_2$. The proof is based on the result of \cite{jiang2021,jiang2021b,cai2020}, particularly their Theorem $3.1$--$3.2$ and Corollary $3.1$. Note that our assumptions satisfy Assumption (a)--(c) of \cite{jiang2021b}. In particular, their assumption (b) is satisfied by the sub-exponential tail assumption in \hyperlink{A2}{\textbf{(A2)}}. We also need the result of \cite{cai2020} when $\lambda_1(C)$ diverges, where their Assumption $1$--$2$ are met by our assumptions.\\ 

\noindent(\rom{1}): Note that 
\begin{align}\label{sec4.3.prop1.eq2}
    \lambda_1(C)=(1-\lambda)\sigma_1^2(A)+\lambda=p(1-\lambda)\sigma_1^2(A)/p+\lambda\asymp n^{-\alpha+1}.
\end{align}
Thus, $\lambda_1(C)$ diverges as $n$ grows. Also, $\psi\asymp n^{-\alpha+1}$. Observe that
\begin{align*}
  T_1^1(Y)=\gamma_0n^{2/3}(\lambda_1(\hat{C})-E_+)&=\lambda_1(\tilde{C})\gamma_0n^{2/3}\parentheses{\frac{\lambda_1(\hat{C})-\lambda_1(\tilde{C})}{\lambda_1(\tilde{C})}}+\lambda_1(\tilde{C})\gamma_0n^{2/3}\parentheses{1-\frac{E_+}{\lambda_1(\tilde{C})}}\\
  &=:(1)+(2).
\end{align*}
By Theorem $2.1$ of \cite{cai2020} and (\ref{sec4.3.prop1.eq2}), $\lambda_1(\tilde{C})/\lambda_1(C)-1=O_p(n^{-\alpha+1})$. Thus, $E_+/\lambda_1(\tilde{C})=o_p(1)$ and $(2)\asymp_p n^{5/3-\alpha}$. For (1), from the proof of Lemma \ref{sec4.3.lemma1}, one can deduce that $|\lambda_1(\hat{C})-\lambda_1(\tilde{C})|/\lambda_1(\tilde{C})=O_p(a_n)$. Hence, $(1)=O_p(n^{5/3-\alpha-\delta}\log n)=o_p(n^{5/3-\alpha})$ for a fixed constant $\delta$ as in \hyperlink{A3}{\textbf{(A3)}}. Therefore, if $C_{1,n}=o(n^{}5/3-\alpha)$, 
\begin{align*}
     \beta_n(\phi_1^1)=\bbP_{I,C}(\phi_1^1(Y)=1)=\bbP_{I,C}(\gamma_0n^{2/3}(\lambda_1(\hat{C})-E_+)>C_{1,n})\rightarrow 1.
\end{align*}
For the consistency of $\phi_1^2$, denote the event that $||\tilde{K}-I_p||_2=O(a_n)$ by $E_n$.Then by the law of iterated expectation,
\begin{align}\label{sec4.3.prop1.eq3}
    \beta_n(\phi_1^2)=\bbP_{I,C}(\phi_1^2(Y)=1)=\bbE[\bbP_{I,C}(\phi_1^2(Y)=1,E_n|\tilde{K})]+\bbP_{I,C}(\phi_1^2(Y)=1,E_n^c).
\end{align}
The latter term is $o(1)$ by \hyperlink{A3}{\textbf{(A3)}}. Hence, the claim is proved if $\bbP_{I,C}(\phi_1^1(Y)=1,E_n|\tilde{K})$ converges to $1$ by the dominated convergence theorem. Given $\tilde{K}$, $\hat{E}_+=E_++O(a_n)$, a similar analysis shows that $\bbP_{I,C}(\phi_1^1(Y)=1,E_n|\tilde{K})\rightarrow 1$ as $n$ grows.\\  

\noindent(\rom{2}): The proof is similar to the case of (\rom{1}). In view of (\ref{sec4.3.prop1}), it suffices to prove the consistency for the test $\phi_1^1$. Observe that $\psi$ in (\ref{sec4.3.prop1.eq1}) becomes $(1+c+\gamma(1+c)/c)\lambda>E_+$ for sufficiently large $n$, which converges to $1+c+\gamma(1+c)/c$ as $1-\lambda\asymp n^{-1}$. Also, note that $\lambda_1(C)/\lambda=1+c$. Hence, by \cite{baik2006,baik2005,jiang2021,bai2012}, we have that 
\begin{align*}
    ||\tilde{C}||_2/\lambda \overset{a.s.}{\rightarrow}\begin{cases}
        (1+\sqrt{\gamma_1\gamma_2})^2,&\quad c\leq \sqrt{\gamma_1\gamma_2},\\
        1+c+\gamma_1\gamma_2(1+c)/c,&\quad c>\sqrt{\gamma_1\gamma_2}
    \end{cases}.
\end{align*}
Because $\lambda$ converges to $1$ under the given assumption, $||\tilde{C}||_2=O_p(1)$. By Theorem $2$ of \cite{jiang2021b}, we have that $n^{1/2}(\lambda_1(\tilde{C})-\psi)=O_p(1)$. Therefore,
\begin{align*}
    T_1^1(Y)&=\gamma_0n^{2/3}(\lambda_1(\hat{C})-\lambda_1(\tilde{C})+\gamma_0n^{1/6}\cdot n^{1/2}(\lambda_1(\tilde{C})-\psi)+\gamma_0n^{2/3}(\psi-E_+)\\
    &=:(4)+(5)+(6).
\end{align*}
Since $||\tilde{C}||_2=O_p(1)$, Lemma \ref{sec4.3.lemma1} implies that $(4)=O_p(n^{2/3-\delta}\log n)$. Also, Theorem $2$ of \cite{jiang2021b} implies $(5)=O_p(n^{1/6})$. Since $(6)\asymp n^{2/3}=n^{5/3-\alpha}$ as $\alpha=1$, we conclude that $\beta_n(\phi_1^1)\rightarrow1$ as $n\rightarrow\infty$ as in the proof of (\rom{1}).

\end{proof}

\subsection{Proofs of the Results from Section \ref{sec4.4:consist}}\label{secA.5}

We prove Proposition \ref{sec4.4.prop1}.

\begin{proof}[Proof of Proposition \ref{sec4.4.prop1}.]
  From the proofs of the results in Section \ref{sec4.3:asymp.null}, deduce that 
  \begin{align*}
      \hat{C}=\tilde{K}^{-1/2}\tilde{C}\tilde{K}^{-1/2,\top}
  \end{align*}
for $\tilde{C}=1/n\sum_{i=1}^n z_iz_i^\top$ and $\tilde{K}=k(\tilde{C})$. Since $||\cdot||_F$ is orthogonally invariant, we assume $\tilde{K}^{1/2}$ is symmetric without loss of generality so that $\hat{C}=\tilde{K}^{-1/2}\tilde{C}\tilde{K}^{-1/2}.$ Again, under the assumption of \hyperlink{A3}{\textbf{(A3)}}, $||\tilde{K}^{-1}-I_p||_2,||\tilde{K}^{-1/2}-I_p||_2=O_p(a_n)$ from the proof of Lemma \ref{sec4.3.lemma1}. Observe that
\begin{align*}
    \frac{||\hat{C}||_F^2}{p}=\frac{||\hat{C}-\tilde{C}||_F^2+2\tr\parentheses{\tilde{C}(\hat{C}-\tilde{C})}+||\tilde{C}||_F^2}{p}.
\end{align*}
We verify that 
\begin{align*}
    ||\hat{C}-\tilde{C}||_F^2/p=o_p(1),\quad \tr\parentheses{\tilde{C}(\hat{C}-\tilde{C})}/p=o_p(1). 
\end{align*}
Under the given assumption on $z_i$'s and \hyperlink{A1}{\textbf{(A1)}}, Theorem $2.2$ of \cite{heiny2023} implies that 
\begin{align*}
   ||\tilde{C}||_F^2/p\overset{p}{\rightarrow}\gamma_1\gamma_2.
\end{align*}
Hence, the desired claim is proved. Note that
\begin{align*}
    ||\hat{C}-\tilde{C}||_F^2/p&=(\tr\parentheses{\tilde{C}\tilde{K}^{-1}\tilde{C}\tilde{K}^{-1}}-2\tr\parentheses{\tilde{C}\tilde{K}^{-1/2}\tilde{C}\tilde{K}^{-1/2}}+\tr\parentheses{\tilde{C}^2})/p\\
    &=(\tr\parentheses{\tilde{C}(\tilde{K}^{-1}-I_p)\tilde{C}(\tilde{K}^{-1}-I_p)}+2\tr\parentheses{\tilde{C}^2(\tilde{K}^{-1}-I_p)})/p\\
    &-\parentheses{2\tr\parentheses{\tilde{C}(\tilde{K}^{-1/2}-I_p)\tilde{C}(\tilde{K}^{-1/2}-I_p)}+4\tr\parentheses{\tilde{C}^2(\tilde{K}^{-1/2}-I_p)}}/p.
\end{align*}
To verify that each term in the last equality is $o_p(1)$, note that under the given assumption on $z_i$'s and \hyperlink{A1}{\textbf{(A1)}}, $||\tilde{C}||_2=O_p(1)$ \cite{baik2006,baik2005,jiang2021,bai2012}. Also, by \hyperlink{A3}{\textbf{(A3)}}, 
\begin{align*}
    |\tr\parentheses{\tilde{C}(\tilde{K}^{-1}-I_p)\tilde{C}(\tilde{K}^{-1}-I_p)}/p|&\leq ||\tilde{C}(\tilde{K}^{-1}-I_p)\tilde{C}(\tilde{K}^{-1}-I_p) ||_2\leq ||\tilde{K}^{-1}-I_p||_2^2||\tilde{C}||_2^2=O_p(a_n^2),\\
      |\tr\parentheses{\tilde{C}^2(\tilde{K}^{-1}-I_p)}/p|&\leq ||\tilde{C}^2(\tilde{K}^{-1}-I_p)||_2\leq ||\tilde{K}^{-1}-I_p||_2||\tilde{C}||_2^2=O_p(a_n).
\end{align*}
The similar results hold when $\tilde{K}^{-1}$ is replaced with $\tilde{K}^{-1/2}$ above. Lastly, 
\begin{align*}
    |\tr\parentheses{\tilde{C}(\hat{C}-\tilde{C})}|/p\leq ||\tilde{C}||_2||\hat{C}-\tilde{C}||_2=O_p(a_n)=o_p(1),
\end{align*}
noting that $||\tilde{C}||_2=O_p(1)$ under the given assumptions and $||\hat{C}-C||_2=O_p(a_n)$ by the proof of Lemma \ref{sec4.3.lemma1}.

\end{proof}

\subsection{Separability Rank of the Core Covariance Matrix}\label{secA.6}
We shall examine the separability rank of partial-isotropy rank$-r$ core covariance matrices for $(p_1,p_2,r)$ such that $\max\set{p_1/p_2,p_2/p_1}=r$. Recall that $C\in\calC_{p_1,p_2}^+$ has a separability rank $1$ if and only if  $C$ is separable, i.e., $C=I_p$. Hence, it is natural to ask the separability rank of non-separable $C$. Of course, by Proposition $2.2$ of \cite{puchkin2024}, one can observe that the separability rank of $\Sigma\in\calS_{p}^+$ is the same as that of $c(\Sigma)$. Nevertheless, how the singular values of $\calR(C)$ look, which imply the separability rank, might also be of interest. Also, while we know $\sigma_1(\calR(C))$ as $\sqrt{p_1p_2}$ by Proposition \ref{sec3.3.prop1}, we do not know the rest of the singular values.\\

Hence, we shall explore these singular values for a special case of partial-isotropy rank$-r$ core with $\max\set{p_1/p_2,p_2/p_1}=r$ admitting their closed form. By Lemma \ref{sec4.1.lemma1} and Theorem \ref{sec4.1.thm1}, we have that such a core is written as $(1-\lambda)AA^\top+\lambda I_p$ for $A\in\real^{p\times r}$ as in Theorem \ref{sec4.1.thm1}. Note that the formulas of $A$ when $p_1<p_2$ and $p_2=p_1r$ can be written similarly by symmetry. 

\begin{prop}\label{secA.6.prop1}
    Suppose $\max\set{p_1/p_2,p_2/p_1}=r$ for generic $r\in\bbN$. Let $C=(1-\lambda)AA^\top+\lambda I_p\in\calC_{p_1,p_2}^+$ for $A$ as in Theorem \ref{sec4.1.thm1} and $\lambda\in(0,1)$, and $E=\calR(I_{p_1})\calR(I_{p_1})^\top/p_1$. Then
    \begin{align*}
        \calR(C)\calR(C)^\top=(1-\lambda)^2pI_{p_2^2}+(p_1\lambda^2+2\lambda(1-\lambda)p_1)E.
    \end{align*}
\end{prop}
\begin{proof}
By the linearity of $\calR$, one can see that 
\begin{align*}
    \calR(C)\calR(C)^\top=(1-\lambda)^2 \calR(AA^\top)\calR(AA^\top)^\top+(p_1\lambda^2+2\lambda(1-\lambda)p_1)E.
\end{align*}
Thus, it suffices to verify the formula of $ \calR(AA^\top)\calR(AA^\top)^\top$. Without loss of generality, assume $p_1\geq p_2$ so that $p_1=p_2r$. Then by Theorem \ref{sec4.1.thm1}, we know that $[A_1,\ldots,A_r]=\sqrt{p_2}[O_1,\ldots,O_r]=:\sqrt{p_2}O$ for some $O\in\calO_{p_1}$, with each $A_i$ and $O_i$ having $p_2$ columns. By Lemma \ref{secA.1.lemma1} and the linearity of $\calR$, we have that 
\begin{align*}
    \calR(AA^\top)=p_2\sum_{i=1}^r O_i^\top\otimes O_i^\top.
\end{align*}
Consequently,
\begin{align*}
     \calR(AA^\top)\calR(AA^\top)^\top=p_2^2 \sum_{i=1}^r I_{p_2}\otimes I_{p_2}=rp_2^2I_{p_2^2}=p_1p_2I_{p_2^2}=pI_{p_2^2}.
\end{align*}
\end{proof}

\begin{cor}\label{secA.6.cor1}
Assume the same for $(p_1,p_2,r)$ as in Proposition \ref{secA.6.prop1}, and let $C\in\calC_{p_1,p_2}^+$ be the matrix defined therein. Then
\begin{align*}
    \sigma_1(\calR(C))=\sqrt{p}, \sigma_2(\calR(C))=\cdots=\sigma_{p_1^2\wedge p_2^2}(\calR(C))=(1-\lambda)\sqrt{p}.
\end{align*}
\end{cor}
\begin{proof}
      We identify the singular values through the square roots of the non-zero eigenvalues of $\calR(C)\calR(C)^\top$. Note that $E$ has rank$-1$ because $\calR(I_p)$ has rank$-1$, as $I_p$ is separable. Also, $||\calR(I_p)||_F^2=||I_p||_F^2=p$ as $\calR$ is an isometry under $||\cdot||_F$. Consequently, the only nonzero eigenvalue of $E$ is $p/p_1=p_2$. In view of $\calR(C)\calR(C)^\top$ in Proposition \ref{secA.6.prop1}, we have that 
    \begin{align*}
        \lambda_1(\calR(C)\calR(C)^\top)=p,\quad  \lambda_2(\calR(C)\calR(C)^\top)=\cdots=\lambda_{p_1^2\wedge p_2^2}(\calR(C)\calR(C)^\top)=(1-\lambda)p.
    \end{align*}
    Taking the square root of each eigenvalue above, we conclude the proof. 
\end{proof}

Consequently, a partial-isotropy rank$-r$ $C$ has a separability rank $p_1^2\wedge p_2^2$ if $\max\set{p_1/p_2,p_2/p_1}=r$ with two distinct singular values. 

\subsection{Distributional Invariance and Maximal Invariance of the Sample Core Covariance Matrix}\label{secA.7}

In this subsection, extending the result of Proposition $3$ of \cite{hoff2023a}, we prove that the distribution of the sample core covariance matrix $\hat{C}$, is invariant with respect to $K=k(\Sigma)$ if $\calH=\calL_{p_1,p_2}^+$ and $UDV^\top$ in the model (\ref{sec3.eq1}) of the main text is taken to be the Cholesky factor of $C=c(\Sigma)$. Let $K_L$ and $C_L$ be the Cholesky factors of $K$ and $C$, respectively. Then $K_LC_L\in\calL_{p}^+$ and the uniqueness of the Cholesky factor implies $K_LC_L$ is a unique Cholesky factor of $\Sigma$. We also introduce a counterexample demonstrating that $\hat{C}$ may fail to be maximal invariant under the action of $\calL_{p_1}^+\times \calL_{p_2}^+$ on the data when $n<p$. We first show the distributional invariance of $\hat{C}$ with respect to $K=k(\Sigma)$.
\begin{prop}\label{secA.7.prop1}
   Let $(K,C)\in \calS_{p_1,p_2}^+\times \calC_{p_1,p_2}^+$, and denote the Cholesky factors of $K$ and $C$ by $K_L$ and $C_L$, respectively. Suppose $Z_1,\ldots, Z_n\overset{i.i.d.}{\sim}\calP_0$, where $\bbE[Z_1]=0$ and $V[Z_1]=I_p$. Assume that $Y_1,\ldots,Y_n$ are random samples such that $y_i=K_LC_L z_i$ for $y_i=\text{vec}(Y_i)$ and $z_i=\text{vec}(Z_i)$. Let $S=1/n\sum_{i=1}^ny_iy_i^\top$, and put $\hat{C}=c(S)$. Then parameterizing the distribution of $\hat{C}$ through $(K,C)$, we have that for any $K'\in\calS_{p_1,p_2}^+$,
    \begin{align*}
        \hat{C}|K,C\overset{d}{\equiv}\hat{C}|K',C.
    \end{align*}
\end{prop}

\begin{proof}
Observe that 
 \begin{align*}
     S=K_LC_L \underbrace{1/n\sum_{i=1}^n z_iz_i^\top}_{\tilde{S}} (K_LC_L)^\top. 
 \end{align*}
 Then, by the equivariance of the Kronecker map $k$ under the action of $GL_{p_1,p_2}$,
 \begin{align*}
     k(S)=K_L k\parentheses{C_L\tilde{S}C_L^\top} K_L^\top.
 \end{align*}
 If $\tilde{K}_L$ is a Cholesky factor of $k(C_L\tilde{S}C_L^\top)$, the Cholesky factor of $k(S)$ is $K_L\tilde{K}_L$ as the set $\calL_{p_1,p_2}^+$ is a group under matrix multiplication. Therefore, 
 \begin{align*}
     \hat{C}=c(S)=(K_L\tilde{K}_L)^{-1} S (K_L\tilde{K}_L)^{-\top}=\tilde{K}_{L}^{-1}C_L\tilde{S}C_L^\top \tilde{K}_{L}^{-\top}=c(C_L\tilde{S}C_L^\top).
 \end{align*}
 Because 
 \begin{align*}
     C_L\tilde{S}C_L^\top|K,C\overset{d}{\equiv}  C_L\tilde{S}C_L^\top|K',C
 \end{align*}
 for any $K,K'\in \calS_{p_1,p_2}^+$, the distribution of $c(C_L\tilde{S}C_L^\top)$ is also independent of $K$, proving the claim.
\end{proof}

Now we discuss the maximal invariance of $\hat{C}$. Note that a statistic $U(X)$ is maximal invariant under the action $g$ of a group $G$ if $U(X)=U(Y)$ implies that $X=gY$ for some $g\in G$. Define the action of $g=(L_1,L_2)\in \calL_{p_1}^+\times \calL_{p_2}^+$ on $Y=(Y_1,\ldots,Y_n)\in(\real^{p_1\times p_2})^n$ by 
\begin{align*}
    gY=(L_1Y_1L_2^\top,\ldots,L_1Y_nL_2^\top).
\end{align*}
Note that under usual matrix multiplication, $\calL_{q}^+$ is a group for any $q$ and so $\calL_{p_1}^+\times \calL_{p_2}^+$ is indeed a group under the natural action $(L_1,L_2)\cdot (U_1,U_2)\rightarrow (L_1U_1,L_2U_2)$. If $S(Y):=1/n\sum_{i=1}^n y_iy_i^\top$ for $y_i=\text{vec}(Y_i)$, then the action of $g$ induces the action of $g'=L_2\otimes L_1\in\calL_{p_1,p_2}^+$ on the space of $S$ as  
\begin{align*}
    S(gY)=g'S(Y)=(L_2\otimes L_1)S(Y)(L_2\otimes L_1)^\top.
\end{align*}
Let $C(Y)\equiv \hat{C}$, the core component of $S(Y)$. Since the test statistics proposed in the article are based on $C(Y)$, if $\calH=\calL_{p_1,p_2}^+$ in Definition \ref{sec2.2.def1} and $UDV^\top=C_L$ in the model (\ref{sec3.eq1}), the associated tests become invariant tests \cite{lehmann2022}. According to Proposition $3$ of \cite{hoff2023a}, with the action of $\calL_{p_1,p_2}^+$ on the space of rank $n\wedge p$ positive semidefinite matrices as above, $C(Y)$ is maximal invariant when $n\geq p$ under the action of $\calL_{p_1}^+\times \calL_{p_2}^+$ on $(\real^{p_1\times p_2})^n$. Thus, the tests $\phi_1,\phi_2,$ and $\phi_3$ are invariant in this case. However, when $n<p$, $C(Y)$ may not be strictly positive semi-definite, and thus not maximal invariant.        

To see this, let $X,Y\in (\real^{p_1\times p_2})^n$. We introduce a counterexample where there may not exist an action $g$ of $\calL_{p_1}^+\times \calL_{p_2}^+$ such that $X=gY$ for which $C(X)=C(Y)$, when $n<p$. Hence, in high-dimensional settings, $C(Y)$ may not be maximal invariant under the action of $\calL_{p_1}^+\times \calL_{p_2}^+$. The concrete counterexample below illustrates the failure of the maximal invariance of $C(Y)$, when $n<p$. 

\begin{cex}\label{secA.7.cex1}
    Suppose $Y=(Y_1,Y_2)\in(\real^{2\times 2})^2$, where 
    \begin{align*}
        Y_1=\left[\begin{array}{cc}
        1 & 0 \\
        0 & 1
        \end{array}\right], \quad         Y_2=\left[\begin{array}{cc}
        0 & 1 \\
        -1 & 0
        \end{array}\right]
    \end{align*}
    Take $\calH=\calL_{p_1,p_2}^+$ in Definition \ref{sec2.2.def1}. For some generic $X=(X_1,X_2)\in(\real^{2\times 2})^2$, assume $C(X)=C(Y)$. However, there does not exist an action $g$ of $\calL_{p_1}^+\times\calL_{p_2}^+$ such that $X=gY$.
\end{cex}
To prove Counterexample \ref{secA.7.cex1}, we need the following ancillary result.
 \begin{lemma}\label{secA.7.lemma1}
        Suppose that $n> p_1/p_2+p_2/p_1$. Define the action of $\calL_{p_1}^+\times\calL_{p_2}^+$ on $(\real^{p_1\times p_2})^n$ by $gY=\parentheses{L_1Y_1L_2^\top,\ldots,L_1Y_nL_2^\top}$ for $Y=(Y_1,\ldots,Y_n)\in (\real^{p_1\times p_2})^n$ and $g=(L_1,L_2)\in\calL_{p_1}^+\times\calL_{p_2}^+$. Then the maximal invariance of $C(Y)$ under the action of $\calL_{p_1}^+\times\calL_{p_2}^+$ on $(\real^{p_1\times p_2})^n$ holds if and only if the action of $\calL_{p_1}^+\times\calL_{p_2}^+$ is transitive over the set $\mathcal{F}=\set{X\in(\real^{p_1\times p_2})^n:S(X)=C}$ for any core $C$, possibly positive semidefinite.
    \end{lemma}
    \begin{proof}[\textbf{\upshape Proof.}]
        We begin with the proof of if part. Suppose $X,Y\in(\real^{p_1\times p_2})^n$ satisfy $C(X)=C(Y)=C$ for some  core $C$, possibly positive semidefinite. For $Z\in (\real^{p_1\times p_2})^n$, let $K(Z)\equiv k(S(Z))$. Denote the Cholesky factors of $K(X)$ and $K(Y)$ by $L_{2,X}\otimes L_{1,X}$ and $L_{2,Y}\otimes L_{1,Y}$, respectively. If $g_X=(L_{1,X},L_{2,X})$ and $g_Y=(L_{1,Y},L_{2,Y})$, $S(g_X^{-1}X)=S(g_Y^{-1}Y)=C$. By the assumption of if part, there exists $g\in\calL_{p_1}^+\times\calL_{p_2}^+$ such that $g_X^{-1}X=gg_Y^{-1}Y$. Thus, $X=g_Xgg_Y^{-1}Y$. Since $\calL_{p_1}^+\times \calL_{p_2}^+$ is a group, $g_Xgg_Y^{-1}\in\calL_{p_1}^+\times \calL_{p_2}^+$, which concludes the if part.\\

        To prove the only if part, suppose $S(X)=S(Y)=C$. By the assumption of only if part, there exists $g\in\calL_{p_1}^+\times \calL_{p_2}^+$ such that $X=gY$, which implies that the action of $\calL_{p_1}^+\times \calL_{p_2}^+$ is transitive over $\mathcal{F}$.  
    \end{proof}

Note that if $n\geq p$, Proposition $3$ of \cite{hoff2023a} implies that the statistic $\hat{C}$ is maximal invariant. Therefore, by Lemma \ref{secA.7.lemma1}, the action of $\calL_{p_1,p_2}^+$ is transitive over $\mathcal{F}$ for a given $C\in\calC_{p_1,p_2}^+$. However, as shown in the proof of Counterexample \ref{secA.7.cex1}, the transitivity of the action of $(\real^{p_1\times p_2})^n$ may not hold if $n<p$.   

    \begin{proof}[\textbf{\upshape Proof of Counterexample \ref{secA.7.cex1}.}]
We first compute the sample core of $Y=(Y_1,Y_2)$. Then
    \begin{align*}
        S(Y)=\left[\begin{array}{cccc}
        1/2 & 0 & 0 & 1/2\\
        0 & 1/2 & -1/2 & 0\\
        0 & -1/2 & 1/2 & 0 \\
        1/2 & 0 & 0 & 1/2
        \end{array}\right].
    \end{align*}
For $Z=(Z_1,\ldots,Z_n)\in(\real^{p_1\times p_2})^n$, denote the separable component of $S(Z)$ by $K(Z)$, namely the Kronecker MLE based on $Z$. To find $K(Y)$, note that $\hat{\Sigma}_2\otimes\hat{\Sigma}_1$ is the Kronecker MLE based on $(Z_1,\ldots,Z_n)\in(\real^{p_1\times p_2})^n$ if and only if 
    \begin{align}\label{secA.7.eq1}
\begin{split}
    \hat{\Sigma}_1&=\sum_{i=1}^n Z_i\hat{\Sigma}_2^{-1}Z_i^\top/np_2,\\
\hat{\Sigma}_2&=\sum_{i=1}^n Z_i^\top \hat{\Sigma}_1^{-1}Z_i/np_1.
\end{split}
\end{align}
(see (3) and Definition $1$ of \cite{hoff2023a}). Suppose $K(Y)=\Sigma_2\otimes \Sigma_1$. By (\ref{secA.7.eq1}), 
    \begin{align*}
        \Sigma_1&=\parentheses{Y_1\Sigma_2^{-1}Y_1^\top+Y_2\Sigma_2^{-1}Y_2^\top}/4,\\
        \Sigma_2&=\parentheses{Y_1^\top \Sigma_1^{-1}Y_1+Y_2^\top \Sigma_1^{-1}Y_2}/4.
    \end{align*}
    Writing $\Sigma_2^{-1}=(\omega_{ij}^2)$, the above implies that 
    \begin{align*}
        \Sigma_1&=(\omega_{11}^2+\omega_{22}^2)/4I_2,\\
        \Sigma_2&=2/(\omega_{11}^2+\omega_{22}^2)I_2.
    \end{align*}
    Hence, $K(Y)=1/2 I_4$. Because the Cholesky factor of $K(Y)$ is clearly $1/\sqrt{2}I_4$, 
    \begin{align*}
        C(Y)=(1/\sqrt{2}I_4)^{-1}S_Y(1/\sqrt{2}I_4)^{-\top}=\left[\begin{array}{cccc}
        1 & 0 & 0 & 1\\
        0 & 1 & -1 & 0\\
        0 & -1 & 1 & 0 \\
        1 & 0 & 0 & 1
        \end{array}\right]=:C.
    \end{align*}
    In view of Lemma \ref{secA.7.lemma1}, we characterize the set of $X=(X_1,X_2)\in (\real^{p_1\times p_2})^2$ such that $S(X)=C$. Write $X_1=(x_{1,ij})$ and $X_2=(x_{2,ij})$, and suppose that $S(X)=C$. Then 
    \begin{align*}
        2S_X&=\left[\begin{array}{cccc}
        x_{1,11}^2+x_{2,11}^2 & x_{1,11}x_{1,21}+x_{2,11}x_{2,21} & x_{1,11}x_{1,12}+x_{2,11}x_{2,12} &x_{1,11}x_{1,22}+x_{2,11}x_{2,22} \\
        * &      x_{1,21}^2+x_{2,21}^2& x_{1,21}x_{1,12}+x_{2,21}x_{2,12}& x_{1,21}x_{1,22}+x_{2,21}x_{2,22} \\
        * & * & x_{1,12}^2+x_{2,12}^2 & x_{1,12}x_{1,22}+x_{2,12}x_{2,22} \\
        * & * &* &\ x_{1,22}^2+x_{2,22}^2\
        \end{array}\right]\\
        &=2\left[\begin{array}{cccc}
        1 & 0 & 0 & 1\\
        0 & 1 & -1 & 0\\
        0 & -1 & 1 & 0 \\
        1 & 0 & 0 & 1
        \end{array}\right]=2C.
    \end{align*}
    Comparing the diagonal entries, $X_1$ and $X_2$ can be initially parameterized by $(\theta,\phi,\delta,\tau)$ as 
    \begin{align*}
        X_1=\sqrt{2}\left[\begin{array}{cc}
     \cos\theta    & \cos\delta \\
       \cos\phi  & \cos\tau
        \end{array}\right], \quad X_2=\sqrt{2}\left[\begin{array}{cc}
       \sin\theta  & \sin\delta \\
        \sin\phi &\sin\tau
        \end{array}\right].
    \end{align*}
    In fact, the parameterization above may be redundant. For example, if one compares $(1,2)-$entries of $2S(X)$ and $2C$ above, it holds that $\cos\theta\cos\phi+\sin\theta\sin\phi=\cos\parentheses{\theta-\phi}=0$. Thus, $\theta-\phi=(2k-1)\pi/2$ for some $k\in\mathbb{Z}$. Comparing other off-diagonal entries similarly, one can deduce that 
    \begin{align*}
        X_1=\sqrt{2}\left[\begin{array}{cc}
      \cos\theta  & \pm \sin\theta\\
      \mp \sin\theta & \cos\theta
        \end{array}\right], \quad X_2=\sqrt{2}\left[\begin{array}{cc}
     \sin\theta  & \mp \cos\theta \\
       \pm \cos\theta &\sin\theta
        \end{array}\right].
    \end{align*}
    Thus, the set $\mathcal{F}$ in Lemma \ref{secA.7.lemma1} with $C$ above, is exactly 
    \begin{align*}
        \mathcal{F}=\set{\parentheses{\sqrt{2}\left[\begin{array}{cc}
      \cos\theta  & \pm \sin\theta\\
      \mp \sin\theta & \cos\theta
        \end{array}\right],\sqrt{2}\left[\begin{array}{cc}
     \sin\theta  & \mp \cos\theta \\
       \pm \cos\theta &\sin\theta
        \end{array}\right]}:\theta\in\real}.
    \end{align*}
    Now we claim that the action of $\calL_{p_1}^+\times \calL_{p_2}^+$ is not transitive over $\mathcal{F}$, and thus the maximal invariance of the sample covariance matrix does not hold by Lemma \ref{secA.7.lemma1}. Suppose otherwise. Then there exists $g=(L_1,L_2)\in\calL_{p_1}^+\times \calL_{p_2}^+$ such that 
    \begin{align*}
        gY=\parentheses{L_1Y_1L_2^\top,L_1Y_2L_2^\top}=\parentheses{\sqrt{2}\left[\begin{array}{cc}
      \cos\theta  & \pm \sin\theta\\
      \mp \sin\theta & \cos\theta
        \end{array}\right],\sqrt{2}\left[\begin{array}{cc}
     \sin\theta  & \mp \cos\theta \\
       \pm \cos\theta &\sin\theta
        \end{array}\right]}
    \end{align*}
    for any $\theta$. Writing $L_1=(\ell_{1,ij})$ and $L_2=(\ell_{2,ij})$, observe that 
    \begin{align*}
        L_1Y_2L_2^\top=\left[\begin{array}{cc}
           0  & -\ell_{1,11}\ell_{2,22}  \\
         \ell_{1,22}\ell_{2,11}    & \ell_{2,21}\ell_{1,22}-\ell_{2,22}\ell_{1,21}
        \end{array}\right].
    \end{align*}
    Since $\sin\theta$ is not necessarily $0$, this proves the desired claim.
\end{proof}

\section{Miscellaneous Figures and Tables}\label{sec.B:misc}

We provide figures and tables that support the results and discussions in the main text. Regarding Lemma \ref{sec4.3.lemma1} and Corollary \ref{sec4.3.cor1}, we numerically examine two things. First, we examine whether the Marchenko-Pastur (MP) law \cite{marchenko1967} holds for the empirical spectral distribution of a sample core covariance matrix under the null of separability, i.e., $C=I_p$. \cite{marchenko1967} showed that if $p/n\rightarrow\gamma\in(0,\infty)$ as $n\rightarrow\infty$, and $X=(x_{ij})\in \real^{n\times p}$, where $x_{ij}\overset{i.i.d.}{\sim}N(0,\sigma^2)$, then the empirical distribution of the eigenvalues of the sample covariance matrix $S=X^\top X/n$ converges weakly to the distribution $H$, whose density, dominated by Lebesgue measure, is given by: 
\begin{align*}
f_{\text{MP}}(x)=\frac{1}{2\pi\sigma^2}\frac{\sqrt{(\lambda_+-x)(x-\lambda_-)}}{\gamma x}I(x\in[\lambda_-,\lambda_+])+(1-1/\gamma)\delta_0(x)I(\gamma>1).
\end{align*}
Here $\lambda_{\pm}=\sigma^2(1\pm \sqrt{\gamma})^2$. This result was generalized to the case where $x_{ij}$'s are i.i.d. with zero mean and finite second moment \cite{silverstein1995}. With $C=I_p$, we verify whether MP law holds for the empirical spectral distribution of the sample core covariance matrix $\hat{C}=c(S)$. We consider $(p_1,p_2,n)\in\set{(40,40,6400),(40,80,6400),(40,80,1600),(80,80,1600)}$, representing the cases when $p/n\in \{1/4, 1/2,2,4\}$, respectively. Figure \ref{secB.figure1} shows the empirical spectral distribution of $\hat{C}$, and the blue line denotes the density of MP law for selected values of $(p_1,p_2,n)$. It suggests that the empirical spectral distribution may converge to the MP law. Note that for $p/n>1$, there are $p-n$ zero sample core eigenvalues, which are omitted in Figure \ref{secB.figure1}.\\ 

\begin{figure}[!ht]
    \centering
    \includegraphics{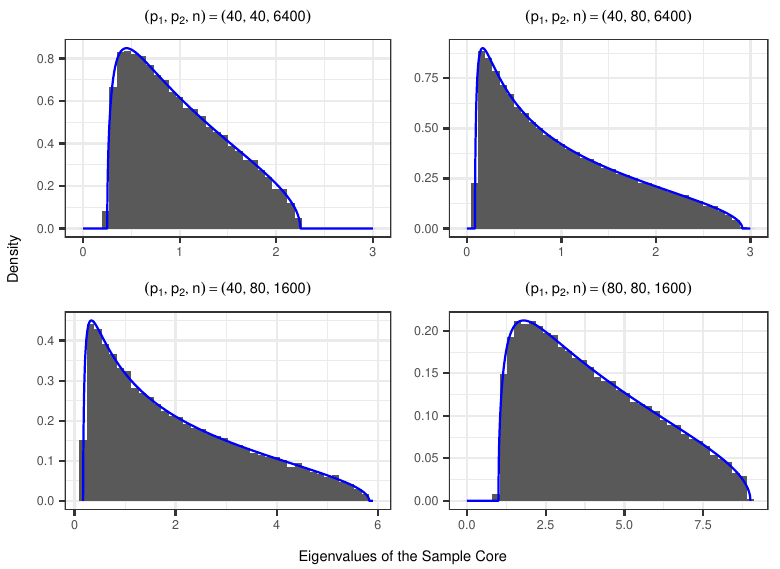}
    \caption{The empirical spectral distribution of the sample core covariance matrix $\hat{C}$ and the density of MP law (blue line) when $C=I_p$ with $(p_1,p_2,n)\in\set{(40,40,6400),(40,80,6400),(40,80,1600),(80,80,1600)}$.}
    \label{secB.figure1}
\end{figure}

Next, we verify the sample eigenvalue bias derived by \cite{baik2006} when $C$ exhibits a rank$-r$ partial-isotropy structure for $r=1,2$, i.e., $C=(1-\lambda)AA^\top+\lambda I_p$ for $\lambda\in(0,1)$ and $A\in\real^{p\times r}$ as in Theorem \ref{sec4.1.thm1}. Recall from Section \ref{sec4.1} that assuming $p_1\geq p_2$, $r=1$ is possible only when $p_1=p_2$ and $r=2$ only when $p_1=p_2$ or $(p_1,p_2)=((k+1)m,km)$ for some $k,m\in\bbN$. We present a version of the results of \cite{baik2006}. Suppose the covariance matrix $\Sigma$ has a spectrum as 
\begin{align*}
    \text{spec}(\Sigma)=\sigma^2(\underbrace{a_1,\ldots,a_1}_{n_1},\ldots,\underbrace{a_k,\ldots,a_k}_{n_k},\underbrace{1,\ldots,1}_{p-M}),
\end{align*}
where $n_1,\ldots,n_k\geq 1$ are fixed numbers in $n$, $M=\sum_{i=1}^k n_i$, $a_1\geq \cdots \geq a_k>1$ are fixed spike values, with $\sigma^2$ known. Let $\tilde{S}=\Sigma^{1/2}S\Sigma^{1/2}$ for the symmetric square root $\Sigma^{1/2}$ of $\Sigma$. Denote the eigenvalues of $\tilde{S}$ by $\ell_1\geq \cdots \geq \ell_M$ corresponding to the spike eigenvalues of $\Sigma$, and let $\mathcal{J}_i=\set{\sum_{j=1}^{i-1}n_i+1,\ldots,\sum_{j=1}^in_j}$ for $j=1,\ldots,k$. Under the regime where $p/n\rightarrow\gamma\in(0,\infty)$ as $n\rightarrow\infty$, \cite{baik2006} derived the sample eigenvalue bias, as follows (see Theorem 1.1--1.2 of \cite{baik2006}): for $j\in \mathcal{J}_i$,  
\begin{align*}
    \ell_j\overset{a.s.}{\rightarrow}\begin{cases}
     \sigma^2\parentheses{a_i+\frac{\gamma a_i}{a_i-\sigma^2}},\quad & \text{if } a_i>1+\sqrt{\gamma},\\
     \sigma^2(1+\sqrt{\gamma})^2,\quad & \text{o.t.}
    \end{cases}.
\end{align*}

We examine whether the above result holds when $\Sigma$ is a core covariance matrix with a rank$-r$ partial isotropic structure, focusing on $r=1,2$. To align with the assumption of $\Sigma$, we set $\lambda=1/(1+rc/p_1p_2)$ for a constant $c>0$. By Corollary \ref{sec4.1.cor1}, the ratio of each spike eigenvalue to the non-spike eigenvalue is $1+c$ for the chosen $\lambda$. Although the ratio is not constant in $p$ with the same choice of $\lambda$ if $p_1=p_2$ and $r=2$, we still consider such a $\lambda$ as the ratio is bounded in $p$.  We take $c\in\set{0.2,0.8,1.2,1.8,2.4}$ and $(p_1,p_2,n)$ such that $p/n=1/4$ . Given $c$ and $(p_1,p_2,n)$, we randomly generate a core covariance matrix $A$ according to Theorem \ref{sec4.1.thm1}, and generate a sample covariance matrix $\check{S}$ of the data generated from $N_{p_1\times p_2}(0,C)$. We then compute the sample core eigenvalues corresponding to the spike eigenvalues. Repeating this procedure $1000$ times, we compute the mean of the sample core eigenvalues across $1000$ iterations and compare it with the almost sure limits of the sample eigenvalues by \cite{baik2006}. Tables \ref{secB.table1}--\ref{secB.table3} report the mean of sample core eigenvalues across $1000$ iterations and the almost sure limits of the sample eigenvalue by \cite{baik2006}. The results suggest that the sample eigenvalue bias by \cite{baik2006} may also hold for the sample core eigenvalue.

\begin{table}[!ht]
    \centering
    \begin{tabular}{cccccc}
     \hline    & \multicolumn{5}{c}{$c$} \\ \cmidrule{2-6}  
         & $0.2$ & $0.8$ & $1.2$ & $1.8$ & $2.4$ \\ \hline
         $\tau_1$ & 2.249 & 2.358& 2.650 &3.175 & 3.732 \\
     $\hat{\tau}_1$ & 2.225 & 2.352 & 2.644 & 3.172 & 3.730  \\ \hline 
    \end{tabular}
    \caption{The limiting value of the largest sample eigenvalue by \cite{baik2006} ($\tau_1$), and the mean of the largest sample eigenvalue across $1000$ iterations ($\hat{\tau}_1$), where $C$ is randomly generated as a rank$-1$ partial isotropic structure with $(p_1,p_2,r)=(p_1,p_1,1)$ and $\lambda=1/(1+c/p_1^2)$ in the second item of Corollary \ref{sec4.1.cor1}. Here $(p_1,p_2,n)=(20,20,1600)$.}
    \label{secB.table1}
\end{table}

\begin{table}[!ht]
    \centering
    \begin{tabular}{cccccc}
    \hline & \multicolumn{5}{c}{$c$} \\ \cmidrule{2-6}
           & $0.2$ & $0.8$ & $1.2$ & $1.8$ & $2.4$ \\ \hline
      $\tau_1$     & 2.248 & 2.441 & 2.954 & 3.392 &  4.104 \\ 
      $\hat{\tau}_1$ & 2.225 & 2.442 & 2.962 & 3.392 & 4.097\\
      $\tau_2$ & 2.248 & 2.283 & 2.372 & 2.934 & 3.322 \\
      $\hat{\tau}_2$ & 2.193 &2.269  &2.364 & 2.915 & 3.304  \\ \hline 
    \end{tabular}
    \caption{The limiting values of the two largest sample eigenvalues by \cite{baik2006} ($\tau_1\geq \tau_2$), and the means of the two largest sample eigenvalue across $1000$ iterations ($\hat{\tau}_1\geq \hat{\tau}_2$), where $C$ is randomly generated as a rank$-2$ partial isotropic structure with $(p_1,p_2,r)=(p_1,p_1,2)$ and $\lambda=1/(1+2c/p_1^2)$ in the first item of Corollary \ref{sec4.1.cor1}. Here $(p_1,p_2,n)=(20,20,1600)$.}
    \label{secB.table2}
\end{table}

\begin{table}[!ht]
    \centering
    \begin{tabular}{cccccc}
    \hline & \multicolumn{5}{c}{$c$} \\ \cmidrule{2-6}
           & $0.2$ & $0.8$ & $1.2$ & $1.8$ & $2.4$ \\ \hline
      $\tau$     &  2.249 & 2.356 &2.648 & 3.170 & 3.724 \\ 
      $\hat{\tau}_1$ & 2.232 & 2.388  & 2.696  & 3.235 & 3.809 \\
      $\hat{\tau}_2$ & 2.207 &2.317  &2.594 & 3.099 & 3.636 \\ \hline 
    \end{tabular}
    \caption{The limiting value of the two largest sample eigenvalues by \cite{baik2006} ($\tau$), and the means of the two largest sample eigenvalue across $1000$ iterations ($\hat{\tau}_1\geq \hat{\tau}_2$), where $C$ is randomly generated as a rank$-2$ partial isotropic structure with $(p_1,p_2,r)=((k+1)m,km,2)$ for some $k,m\in\bbN$ and $\lambda=1/(1+2c/p_1p_2)$ in the second item of Corollary \ref{sec4.1.cor1}. Note that the limiting values for the two largest sample eigenvalues are the same in this case. Here $(p_1,p_2,n)=(30,20,2400)$.}
    \label{secB.table3}
\end{table}

We conclude this section with a discussion on the result of Proposition \ref{sec4.4.prop1}. By Proposition \ref{sec4.4.prop1}, $T_3(Y)$ converges to $\gamma_1\gamma_2$ in probability under the null hypothesis of separability, where $\gamma_i$ is a limit of $p_i/\sqrt{n}$ in $n$ as in \hyperlink{A1}{\textbf{(A1)}}. In Proposition \ref{sec4.4.prop1}, we assumed that $k(\Sigma)=I_p$ for the population covariance matrix $\Sigma$ by Kronecker-invariance. Through Monte Carlo simulations, we verify the results of Proposition \ref{sec4.4.prop1}. In addition to this, we verify that the values of $\bar{T}_3(Y):=||S||_F^2/\sigma_1(\calR(S))^2-1$ will depend on the value of $\Sigma=\Sigma_2\otimes \Sigma_1\in\calS_{p_1,p_2}^+$ in general, even under the null hypothesis of separability. Consequently, this provides another motivation for using $T_3(Y)$ rather than $\bar{T}_3(Y)$ to test the separability.\\ 

To this end, we consider two different values of $\Sigma_2\otimes \Sigma_1$  as follows. Let $\Gamma_i\Lambda_i\Gamma_i^\top$ be the eigendecomposition of $\Sigma_i$ for $\Gamma_i\in\calO_p$ and a diagonal $\Lambda_i$ with positive diagonal entries. We randomly generate $\Gamma_i$. The diagonal entries of $\Lambda_i$ are generated as follows.
\begin{itemize}
    \item[]\textbf{(B1)} $\text{diag}(\Lambda_1)=(\underbrace{4,\ldots,4}_{p_1/2},\underbrace{1,\ldots,1}_{p_1/2})$, $\text{diag}(\Lambda_2)=(\underbrace{3,\ldots,3}_{p_2/2},\underbrace{2,\ldots,2}_{p_2/2})$.
    \item[]\textbf{(B2)} $\text{diag}(\Lambda_1)=(\underbrace{8,\ldots,8}_{p_1/2},\underbrace{3,\ldots,3}_{p_1/2})$, $\text{diag}(\Lambda_2)=(\underbrace{5,\ldots,5}_{p_2/2},\underbrace{1,\ldots,1}_{p_2/2})$.
\end{itemize}

Take $(n,p_1,p_2)=(1600,40,20)$. For each model \textbf{(B1)} and \textbf{(B2)}, we simulate $T_3(Y)$ and $\bar{T}_3(Y)$ based on $Y_1,\ldots,Y_n\overset{i.i.d.}{\sim}N_{p_1\times p_2}(0,\Sigma_2\otimes\Sigma_1)$, and repeat this procedure for $1000$ times. Table \ref{secB.table4} provides the mean of simulated $T_3(Y)$ and $\bar{T}_3(Y)$ across $1000$ simulations for each model \textbf{(B1)} and \textbf{(B2)}. We can clearly observe that the means of simulated $T_3(Y)$ do not vary by the models, whereas $\bar{T}_3(Y)$ do. 

\begin{table}[!ht]
    \centering
    \begin{tabular}{ccc}
    \hline 
      Model   & $T_3(Y)$ & $\bar{T}_3(Y)$ \\ \hline 
    \textbf{B1} & 0.499 & 0.352\\ 
      \textbf{B2} & 0.499 & 0.285 \\ \hline
    \end{tabular}
    \caption{Mean of $T_3(Y)$ and $\bar{T}_3(Y)$ across $1000$ Monte Carlo simulations for $Y_1,\ldots,Y_n\overset{i.i.d.}{\sim}N_{p_1\times p_2}(0,\Sigma_2\otimes\Sigma_1)$ with $(n,p_1,p_2)=(1600,40,20)$, where $\Sigma_2\otimes\Sigma_1$ is generated according to the model either \textbf{(B1)} or \textbf{(B2)}.}
    \label{secB.table4}
\end{table}

\section{Additional Simulation Studies}\label{sec.C:add}

\subsection{Simulated Null Distributions for Non-Gaussian Populations}\label{secC.1}

We provide the Monte Carlo-approximated (simulated) null distributions of $T_1(Y),T_2(Y),$ and $T_3(Y)$ for non-Gaussian populations, particularly those satisfying the sub-exponential tail assumption. For $T_1(Y)$, we consider its two versions, $T_1^1(Y)$ and $T_1^2(Y)$, as defined in Section \ref{sec4.3:asymp.null}. Note that by Kronecker-invariance, we assume that the population covariance matrix $\Sigma$ is $I_p$. We generate random matrices $Y_1,\ldots,Y_n\in\real^{p_1\times p_2}$ whose entries are i.i.d. with a distribution $\calP_0$ whose mean and variance are $0$ and $1$, respectively. Suppose $Z\sim \calP_0$. Then $Z\overset{d}{\equiv}(W-\bbE[W])/\sqrt{V[W]}$ for $W\sim \text{Gamma}(\alpha,\beta)$, where $\bbE[W]=\alpha/\beta$ and $V[W]=\alpha/\beta^2$. We consider $(\alpha,\beta)\in\set{(1,1),(2,1),(1,2)}$.\\

Figure \ref{secC.1.figure1} provides the simulated null distributions of $T_1^1(Y)$ and $T_1^2(Y)$ based on $1000$ Monte Carlo simulations, and the asymptotic null distribution, $TW_1$, with $(n,p_1,p_2)=(1600,80,80)$. From Figure \ref{secC.1.figure1}, we see that the center of the simulated null distribution based on $T_1^2(Y)$ is slightly shifted to the left compared to that of $TW_1$, as observed in Figures \ref{sec5.1.figure1}--\ref{sec5.1.figure2}. Hence, the asymptotic test based on $T_1^2(Y)$ may be conservative. Again, this may be due to the estimation error of $\hat{K}$. On the other hand, whether that based on $T_1^1(Y)$ is conservative may depend on the heavy-tailedness of the data-generating distribution $\calP_0$. Note that the kurtosis of $\text{Gamma}(\alpha,\beta)$ is $6/\alpha$. Hence, the larger $\alpha$ is, the more the center of the simulated null distribution is shifted toward to the left. Indeed, from the first and third rows of Figure \ref{secC.1.figure1}, one can expect that the asymptotic test based on $T_1^1(Y)$ is unlikely to be conservative, whereas it is conservative based on the second row of the figure.\\

\begin{figure}[!ht]
    \centering
    \includegraphics{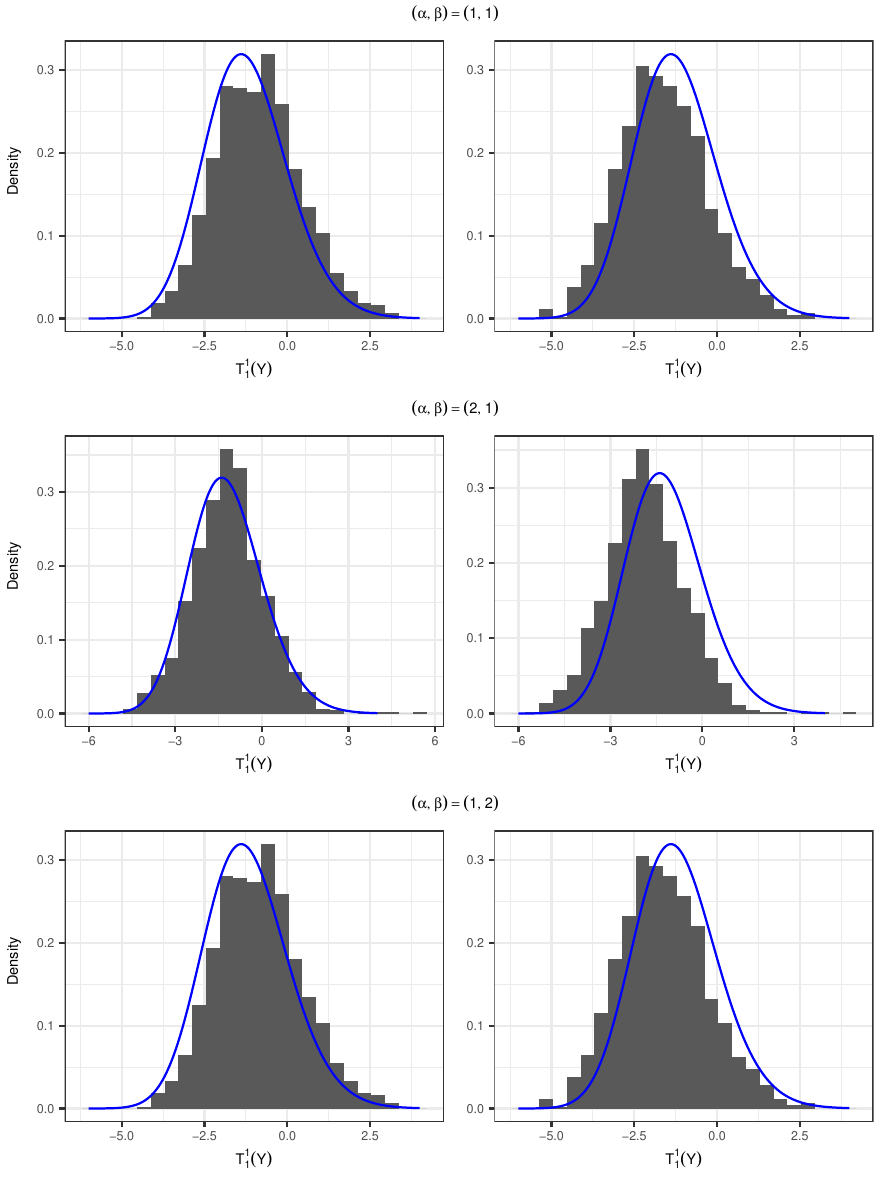}
    \caption{Monte Carlo-approximated distributions of $T_1^1(Y)$ (left panel) and $T_1^2(Y)$ (right panel) based on $1000$ simulations, and the asymptotic null distribution $TW_1$ (blue line) with $(n,p_1,p_2)=(1600,80,80)$. Here the entries of the random matrices $Y_1,\ldots,Y_n$ are i.i.d. with $(Z-\bbE[Z])/\sqrt{V[Z]}$, where $Z\sim \text{Gamma}(\alpha,\beta)$. The first, second, and third rows denote the cases of $(\alpha,\beta)=(1,1),(2,1),(1,2)$, respectively.}
    \label{secC.1.figure1}
\end{figure}

Figure \ref{secC.1.figure2} provides the simulated null distributions of suitably transformed $T_2(Y)$ and $T_3(Y)$ based on $1000$ Monte Carlo simulations, and the asymptotic null distributions by \cite{wang2021} and \cite{li2016}, respectively. We refer the transformations of $T_2(Y)$ to Theorem $1$--$2$ of \cite{wang2021}. The transformation of $T_3(Y)$ is given as $nT_3(Y)-p-1$. Note that compared to the results by \cite{ledoit2002}, \cite{li2016} studied the asymptotic null distribution of the test statistic of John's sphericity test in more general regime, where the fourth moment is finite. From Figure \ref{secC.1.figure2}, we observe that the centers of the simulated null distributions of these statistics are shifted to the left compared to the reference asymptotic null distributions, as observed in Figure \ref{sec5.1.figure3}.

\begin{figure}[!ht]
    \centering
    \includegraphics{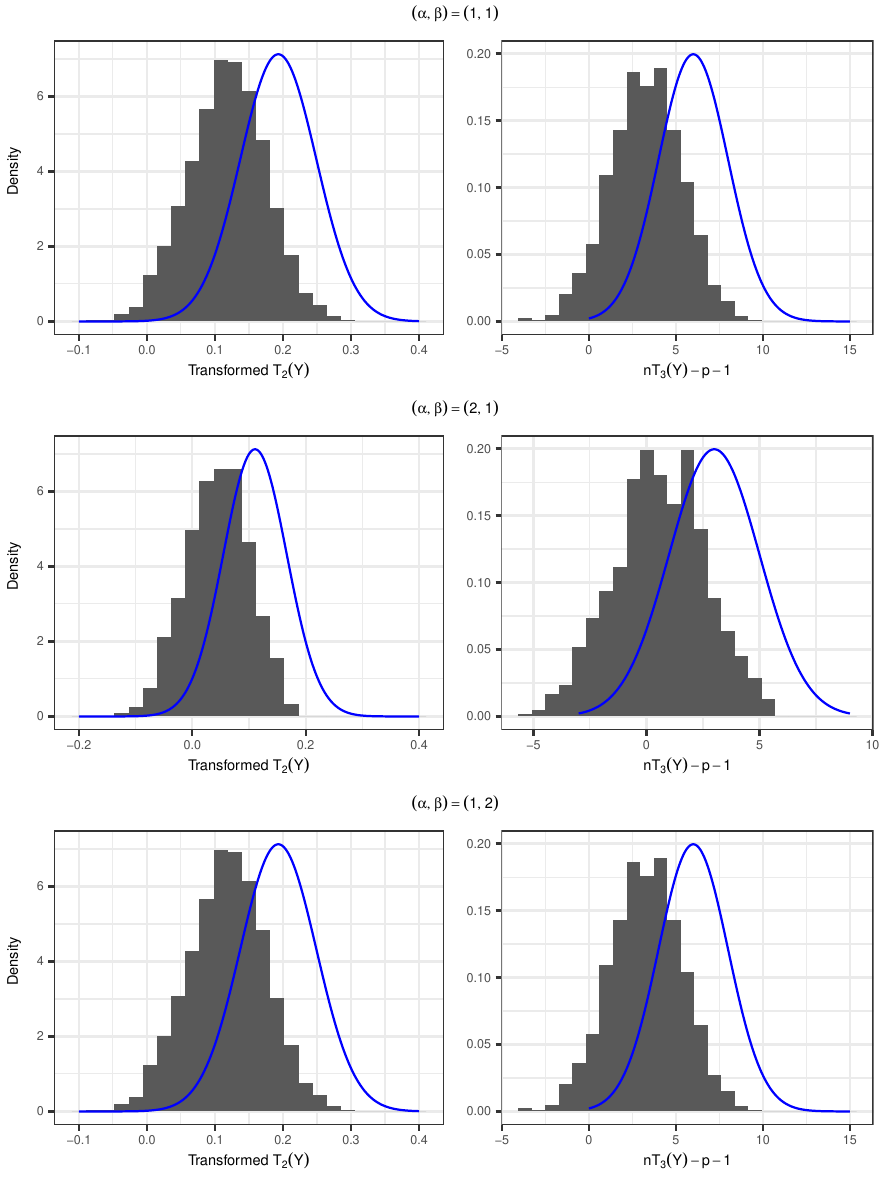}
    \caption{Monte Carlo-approximated distributions of $T_2(Y)$ (left panel) and $T_3(Y)$ (right panel) based on $1000$ simulations and the asymptotic null distributions in blue line by \cite{wang2021} (left panel) and \cite{li2016} (right panel) with $(n,p_1,p_2)=(1600,20,20)$. Here the entries of the random matrices $Y_1,\ldots,Y_n$ are i.i.d. with $(Z-\bbE[Z])/\sqrt{V[Z]}$, where $Z\sim \text{Gamma}(\alpha,\beta)$. The first, second, and third rows denote the cases of $(\alpha,\beta)=(1,1),(2,1),(1,2)$, respectively.}
    \label{secC.1.figure2}
\end{figure}

\subsection{Simulated Distributions of $T_1(Y)$ under Alternative Regimes}\label{secC.2}

We verify the second claim of Theorem \ref{sec4.1.thm1} on the asymptotic distribution of transformed $T_1(Y)$ under the local alternative regimes. We provide the empirical distributions of the versions of $T_1(Y)$, $T_1^1(Y)$ and $T_1^2(Y)$ defined in Section \ref{sec4.3:asymp.null}, under alternative regimes based on $1000$ Monte Carlo simulations. For each simulation, we generate the data $Y_1,\ldots,Y_n\overset{i.i.d.}{\sim}N_{p_1\times p_2}(0,(1-\lambda)AA^\top+\lambda I_p)$ with $(n,p_1,p_2)=(10000,60,50)$. Here $\lambda\in(0,1)$ and $A\in\real^{p\times 2}$ generated according to the second item of Theorem \ref{sec4.1.thm1}. \\

We consider two values of $\lambda$ as $1/(1+2\cdot 0.2/p)$ and $1/(1+2\cdot 4/p)$. By Corollary \ref{sec4.1.cor1}, the former value represents the local alternative regimes as $\lambda_1(C)/\lambda=1+0.2<1+\sqrt{0.3}$ for $C=(1-\lambda)AA^\top+\lambda$ and $\lambda$ is close enough to $1$, whereas $\lambda_1(C)/\lambda=1+4>1+\sqrt{0.3}$ for the latter value of $\lambda$. For the former value, the empirical distributions of $T_1^1(Y)$ and $T_1^2(Y)$ are expected to approximate $TW_1$ well. Recall that since $\lambda_1(C)/\lambda=(1-\lambda)\sigma_1^2(A)/\lambda+1$, the result of Theorem \ref{sec4.3.thm1} implies that the asymptotic distribution of transformed $T_1(Y)$ under local alternative is the same as the asymptotic null distribution if the signal-to-noise ratio (SNR) in the core is low enough. Hence, to illustrate the contrast, we also provide the empirical distribution under the alternative regime with a large SNR in the population core. From Figures \ref{secC.2.figure1}--\ref{secC.2.figure2}, we observe that both empirical distributions of $T_1^1(Y)$ and $T_1^2(Y)$ approximate $TW_1$ with a low SNR in the population core, whereas these distributions significantly deviate from $TW_1$ with a large SNR.  

\begin{figure}[!ht]
    \centering
    \includegraphics{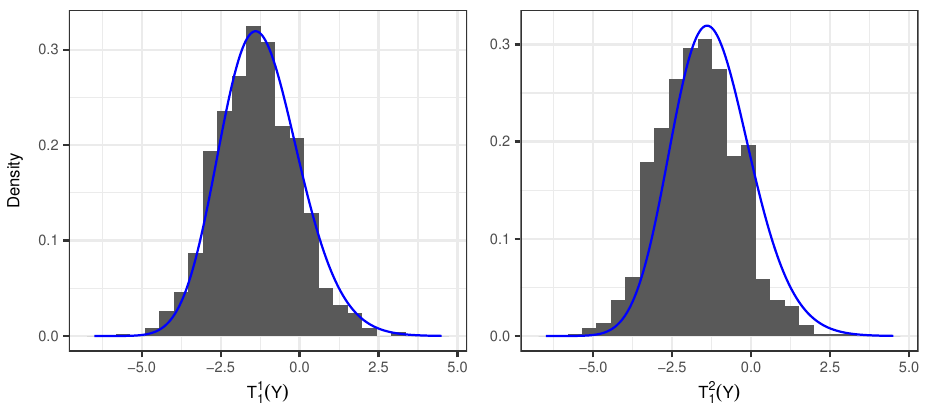}
    \caption{Monte Carlo-approximated distributions of $T_1^1(Y)$ (left panel) and $T_1^2(Y)$ (right panel) based on $1000$ simulations under $N_{p_1\times p_2}(0,(1-\lambda)AA^\top+\lambda I_p)$ and the asymptotic null distribution $TW_1$ (blue line) with $(n,p_1,p_2)=(10000,60,50)$. Here $A\in\real^{p\times 2}$ generated according to the second item of Theorem \ref{sec4.1.thm1} and $\lambda=1/(1+2\cdot 0.2/p)$.}
    \label{secC.2.figure1}
\end{figure}

\begin{figure}[!ht]
    \centering
    \includegraphics{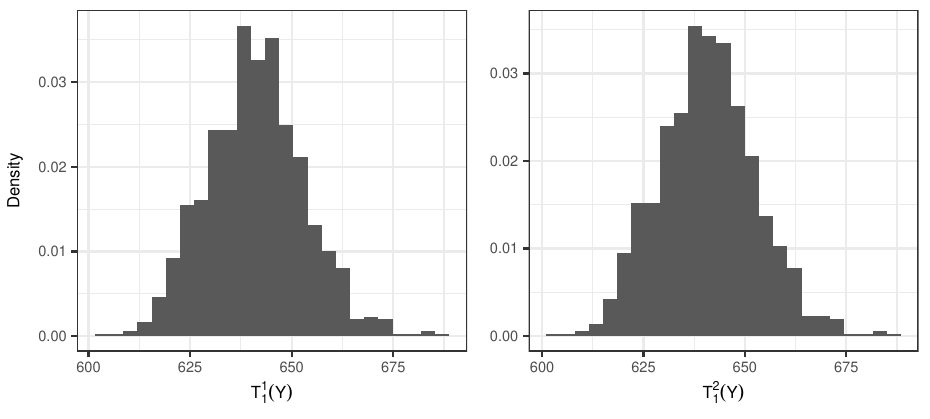}
    \caption{Monte Carlo-approximated distributions of $T_1^1(Y)$ (left panel) and $T_1^2(Y)$ (right panel) based on $1000$ simulations under $N_{p_1\times p_2}(0,(1-\lambda)AA^\top+\lambda I_p)$ with $(n,p_1,p_2)=(10000,60,50)$. Here $A\in\real^{p\times 2}$ generated according to the second item of Theorem \ref{sec4.1.thm1} and $\lambda=1/(1+2\cdot 4/p)$.}
    \label{secC.2.figure2}
\end{figure}

\subsection{Empirical Power of Non-Gaussian Populations}\label{secC.3}
We evaluate the empirical power of $\phi_1,\phi_2,\phi_3$, \texttt{LRT}, and \texttt{PTCLT} under non-Gaussian populations. Namely, let $\calP_0$ be a distribution with zero mean and a finite variance. Suppose $Z_1,\ldots,Z_n\in\real^{p_1 \times p_2}$ are i.i.d. random matrices, where the entries of $Z_i$ are i.i.d. with $\calP_0$. Given a core covariance matrix $C$, we apply the aforementioned separability tests to the data $Y_1,\ldots,Y_n\in \real^{p_1\times p_2}$, where $y_i:=\text{vec}(Y_i)=C^{1/2}z_i$ for $z_i=\text{vec}(Z_i)$. Recall that, due to the Kronecker-invariance, it suffices to consider the case where the population covariance matrix itself is a core. In this section, we take $C=C_{2,w}$ for $C_{2,w}$ defined in Section \ref{sec5.2:emp.power}. Note that the empirical power of each test is computed using the procedure described in Section \ref{sec5.2:emp.power}.\\

 For the distribution $\calP_0$, we consider $\text{Gamma}(4,2)-2$, $t_{6}$ and $t_{10}$, where $t_\nu$ denotes the $t-$distribution with degree of freedom $\nu$. Note that $\text{Gamma}(4,2)-2$ is sub-exponential and thus has moments of all orders, which is not the case for the $t-$ distributions. The $t-$distributions with different degrees of freedom are considered to examine whether the trend in power observed from Figures \ref{sec5.2.figure1}--\ref{sec5.2.figure4} still holds when the underlying distribution is heavy-tailed.\\

Tables \ref{secC.3.table1}--\ref{secC.3.table4} report the empirical powers of the separability tests for different values of $w$ and choices of $\calP_0$. For each choice of $w$ and $\calP_0$, the largest value of the power is bold-faced. Each table corresponds to a different value of $(p_1,p_2,n)\in\{(8,8,256),(8,16,256), (16,32,256),\allowbreak (32,32,256)\}$, representing the cases where $p/n\in\set{1/4,1/2,2,4}$, respectively. \texttt{LRT} is omitted when $p/n>1$ as its test statistic is not well-defined in this case. While the trends from Figures \ref{sec5.2.figure1}--\ref{sec5.2.figure4} are observed similarly, note that the power of each test is similar for $\text{Gamma}(4,2)-2$ and $t_{10}$, given a value of $w$ and $(p_1,p_2,n)$. On the other hand, the powers under $t_6$ are generally higher than those under the other two distributions for each test. Furthermore, under the distribution $t_6$, the difference in power between $\phi_2$ and $\phi_3$ becomes more evident, particularly for small values of $w$ as $p/n$ increases.       

 \begin{table}[!ht]
    \centering \scriptsize 
    \begin{tabular}{cccccccccccc} \hline 
       \multirowcell{2}{$\calP_0$}  &  &\multicolumn{10}{c}{$w$} \\ \cmidrule{3-12} 
         & & $0.1$ &  $0.2$ & $0.3$ & $0.4$ & $0.5$ & $0.6$ & $0.7$ & $0.8$ & $0.9$ & $1.0$ \\ \hline 
               \multirowcell{5}{$\text{Gamma}(4,2)-2$}    & $\phi_1$ & $0.046$& $0.051$ & $0.068$ & $0.078$  & $0.140$ & $0.197$ & $0.300$& $0.391$  & $0.528$ & $0.664$\\
                & $\phi_2$ & $0.057$ &$\bf 0.087$ & $0.124$ & $0.216$& $0.385$ & $\bf 0.587$ & $0.810$ & $0.948$  & $\bf 0.990$ & $\bf 1.000$\\
      & $\phi_3$ &$0.057$ & $0.086$ & $\bf 0.126$& $\bf 0.228$ & $\bf 0.395$ & $0.586$ & $\bf 0.821$ & $\bf 0.952$  &$\bf 0.990$ & $\bf 1.000$\\
       & \texttt{LRT} & $0.055$& $0.078$ & $0.118$ & $0.192$ & $0.351$ & $0.547$ & $0.758$ & $0.921$  & $0.981$& $0.999$\\
         & \texttt{PTCLT} & $\bf 0.062$ &$0.058$ & $0.070$& $0.054$& $0.095$ & $0.072$ &$0.079$ & $0.088$  &$0.096$ & $0.097$\\[-1em] \\
          \multirowcell{5}{$t_6$}    & $\phi_1$ & $0.058$ & $0.074$ & $0.091$ & $0.104$  & $0.138$ & $0.209$ & $0.298$ & $0.391$  & $0.507$ & $0.656$\\
            & $\phi_2$ & $0.041$ & $0.074$ & $0.114$ &  $\bf 0.192$ & $\bf 0.334$  & $\bf 0.529$ & $\bf 0.746$ & $\bf 0.920$ &  $\bf 0.982$ & $\bf 0.999$\\
      & $\phi_3$ & $0.051$ &$0.072$ & $0.114$ & $0.190$ & $0.330$ & $0.517$ & $0.736$ & $\bf 0.920$  & $0.980$& $\bf 0.999$\\
       & \texttt{LRT}  & $0.042$ & $0.064$ & $0.112$ &  $0.182$ &  $0.315$ & $0.494$ &  $0.695$ &  $0.897$ &  $0.971$ &  $\bf 0.999$\\
         & \texttt{PTCLT} & $\bf 0.128$ &  $\bf 0.132$  & $\bf 0.130$ &  $0.130$ & $0.141$ &  $0.149$ & $0.123$ &  $0.137$ & $0.139$ &  $0.147$\\[-1em] \\
          \multirowcell{5}{$t_{10}$}    & $\phi_1$ & $0.049$ & $0.053$ & $0.073$ & $0.129$ & $0.153$ & $0.209$ & $0.309$ & $0.434$ & $0.572$ & $0.702$\\
                  & $\phi_2$ & 0.044 & 0.070 & 0.122 &\bf  0.207 &\bf  0.382 & \bf 0.550 & \bf 0.783 & \bf 0.939 & \bf 0.986 & \bf 0.998\\
      & $\phi_3$ & $0.038$ & $0.059$ & $0.118$ & $\bf 0.207$ & $0.370$ & $0.549$ & $0.781$ & $\bf 0.939$ & $0.985$ & $0.997$\\
       & \texttt{LRT} & $0.054$& $\bf 0.083$ & $\bf 0.125$ &  $0.184$ & $0.353$ & $0.534$ & $0.742$ & $0.927$ & $0.976$ & $\bf 0.998$\\
         & \texttt{PTCLT} & \bf 0.056 & 0.061 & 0.064 & 0.064 & 0.078 & 0.076 & 0.088 & 0.087 & 0.099 & 0.099\\ \hline 
    \end{tabular}
    \caption{The empirical power of $\phi_1$, $\phi_2$, $\phi_3$, \texttt{LRT}, and \texttt{PTCLT} when the population core covariance matrix is $C_{2,w}$ and $\calP_0$ is one of $\text{Gamma}(4,2)-2$, $t_6$, and $t_{10}$, across $w\in\set{0.1,\ldots,1}$ with $(p_1.p_2,n)=(8,8,256)$.}
    \label{secC.3.table1}
\end{table}

 \begin{table}[!ht]
    \centering \scriptsize 
    \begin{tabular}{cccccccccccc} \hline 
      \multirowcell{2}{$\calP_0$}   &  &\multicolumn{10}{c}{$w$} \\ \cmidrule{3-12} 
         & & $0.1$ &  $0.2$ & $0.3$ & $0.4$ & $0.5$ & $0.6$ & $0.7$ & $0.8$ & $0.9$ & $1.0$ \\ \hline 
          \multirowcell{5}{$\text{Gamma}(4,2)-2$}    & $\phi_1$  & $0.049$ &  $0.051$ & $0.063$ & $0.078$ & $0.133$ & $0.157$ & $0.225$ & $0.318$ & $0.454$ & $0.562$\\
                  & $\phi_2$ & $0.068$ & $0.066$ & $0.109$ & $0.186$ & $0.344$ &  $0.544$ &  $0.770$ & $0.929$  & $0.989$ & $\bf 0.999$\\
      & $\phi_3$ & $\bf 0.076$ & $\bf 0.074$ & $\bf 0.115$ & $\bf 0.199$ & $\bf 0.374$ & $\bf 0.573$ & $\bf 0.792$ & $\bf 0.942$ & $\bf 0.992$ & $0.998$ \\
       & \texttt{LRT} & $0.051$ & $0.065$ & $0.092$ & $0.159$ & $0.274$ &  $0.414$ & $0.629$ & $0.838$ & $0.954$ & $0.995$\\
         & \texttt{PTCLT} & $0.055$ &  $0.057$  & $0.065$ &  $0.057$ &  $0.060$ & $0.073$ & $0.066$ & $0.063$ & $0.071$  & $0.084$ \\[-1em] \\
\multirowcell{5}{$t_6$}    & $\phi_1$ & $0.103$ &  $0.141$ & $0.150$ & $0.192$ & $0.240$ & $0.281$ & $0.353$ & $0.456$ & $0.610$ & $0.714$\\
        & $\phi_2$ &  $0.204$ & $0.247$ & $0.291$ & $0.450$ & $0.630$ & $0.783$ & $0.932$ & $0.989$ & $\bf 0.997$ & $\bf 1.000$\\
      & $\phi_3$ & $\bf 0.214$ &  $\bf 0.268$ & $\bf 0.305$ & $\bf 0.460$ & $\bf 0.642$  & $\bf 0.790$ &  $\bf 0.940$ & $\bf 0.991$ & $0.995$ & $\bf 1.000$\\
       & \texttt{LRT} & $0.155$ &  $0.159$ & $0.211$ & $0.324$ & $0.480$ & $0.657$ & $0.823$ & $0.934$ & $0.982$ &  $0.999$\\
         & \texttt{PTCLT} & $0.100$ & $0.099$ & $0.118$ & $0.117$ & $0.121$ & $0.120$ & $0.132$ & $0.118$ &  $0.126$&  $0.152$\\[-1em] \\
          \multirowcell{5}{$t_{10}$}    & $\phi_1$ & 0.049 & 0.057 & 0.073 & 0.093 & 0.122 & 0.189 & 0.265 & 0.365 & 0.453 & 0.615\\
                  & $\phi_2$ & 0.054 & 0.076  & 0.115 & 0.180 & 0.336 & 0.600 & 0.798 & 0.942 & \bf 0.991 & \bf 0.998\\
      & $\phi_3$ & 0.057 &\bf  0.081 & \bf 0.121 &\bf  0.199 & \bf 0.364 &\bf  0.623 &\bf  0.816 &\bf  0.946 & \bf 0.991 &\bf  0.998\\
       & \texttt{LRT} & 0.045 & 0.078 & 0.107 & 0.159 & 0.278 & 0.508 & 0.690 & 0.881 & 0.960 & 0.998\\
         & \texttt{PTCLT} &
        \bf  0.060 &  0.061 & 0.065 & 0.059 & 0.074 & 0.070 & 0.078 & 0.074&  0.088 & 0.082\\ \hline 
    \end{tabular}
    \caption{The empirical power of $\phi_1$, $\phi_2$, $\phi_3$, \texttt{LRT}, and \texttt{PTCLT} when the population core covariance matrix is $C_{2,w}$ and $\calP_0$ is one of $\text{Gamma}(4,2)-2$, $t_6$, and $t_{10}$, across $w\in\set{0.1,\ldots,1}$ with $(p_1.p_2,n)=(8,16,256)$.}
    \label{secC.3.table2}
\end{table}

 \begin{table}[!ht]
    \centering \scriptsize 
    \begin{tabular}{cccccccccccc} \hline 
      \multirowcell{2}{$\calP_0$}   &  &\multicolumn{10}{c}{$w$} \\ \cmidrule{3-12} 
         & & $0.1$ &  $0.2$ & $0.3$ & $0.4$ & $0.5$ & $0.6$ & $0.7$ & $0.8$ & $0.9$ & $1.0$ \\ \hline 
               \multirowcell{4}{$\text{Gamma}(4,2)-2$}    & $\phi_1$ & $0.037$ & $0.049$ & $0.062$ & $0.078$ & $0.092$ & $0.132$ & $0.173$ & $0.263$ &  $0.317$ & $0.429$\\
                & $\phi_2$ & $\bf 0.051$ & $0.066$ & $0.103$ & $0.183$ & $0.296$ & $0.496$ & $0.733$ & $0.921$ & $0.985$ & $\bf 1.000$\\
      & $\phi_3$ & $0.049$ & $\bf 0.072$ & $\bf 0.125$ & $\bf 0.227$ & $\bf 0.335$ & $\bf 0.572$ & $\bf 0.801$ & $\bf 0.944$ & $\bf 0.992$ & $\bf 1.000$\\
         & \texttt{PTCLT} & $0.044$ & $0.063$ & $0.063$ & $0.059$ & $0.055$ & $0.046$ & $0.063$ & $0.053$ & $0.072$ & $0.058$\\[-1em] \\
                   \multirowcell{4}{$t_6$}    & $\phi_1$ & $0.102$ & $0.103$ & $0.146$ & $0.162$ & $0.180$ & $0.249$ & $0.299$ & $0.389$ & $0.505$ & $0.622$\\
                          & $\phi_2$ & $0.206$ & $0.239$ & $0.330$ & $0.492$ & $0.636$ & $0.824$ & $0.923$ & $0.989$ & $0.995$ & $\bf 1.000$\\
      & $\phi_3$  & $\bf 0.245$ & $\bf 0.300$ & $\bf 0.387$ & $\bf 0.559$ & $\bf 0.671$ & $\bf 0.865$ & $\bf 0.955$ & $\bf 0.996$ & $\bf 0.999$ & $\bf 1.000$\\
         & \texttt{PTCLT} & $0.104$ & $0.095$ & $0.102$ & $0.101$ & $0.089$ & $0.102$ & $0.117$ & $0.114$ & $0.101$ & $0.100$\\[-1em] \\
          \multirowcell{4}{$t_{10}$}    & $\phi_1$ & 0.023 & 0.019 & 0.033 & 0.041 & 0.043 & 0.072 & 0.117 & 0.153 & 0.210 & 0.314\\
                & $\phi_2$ & 0.011 & 0.014 & 0.030 & 0.050 &\bf  0.145 & 0.237 & 0.494 & 0.768 & 0.936 & \bf 0.994\\
      & $\phi_3$ & 0.011 & 0.017 & 0.025 & 0.047 & 0.141 & \bf 0.243 &\bf  0.515 & \bf 0.788 &\bf  0.951 & \bf 0.994\\
         & \texttt{PTCLT} &\bf  0.063 & \bf 0.050 & \bf 0.055 &  \bf 0.059 & 0.063 & 0.063 & 0.059 & 0.059 & 0.055 & 0.059\\ \hline 
    \end{tabular}
    \caption{The empirical power of $\phi_1$, $\phi_2$, $\phi_3$, \texttt{LRT}, and \texttt{PTCLT} when the population core covariance matrix is $C_{2,w}$ and $\calP_0$ is one of $\text{Gamma}(4,2)-2$, $t_6$, and $t_{10}$, across $w\in\set{0.1,\ldots,1}$ with $(p_1.p_2,n)=(16,32,256)$.}
    \label{secC.3.table3}
\end{table}

 \begin{table}[!ht]
    \centering \scriptsize 
    \begin{tabular}{cccccccccccc} \hline 
       \multirowcell{2}{$\calP_0$}  &  &\multicolumn{10}{c}{$w$} \\ \cmidrule{3-12} 
         & & $0.1$ &  $0.2$ & $0.3$ & $0.4$ & $0.5$ & $0.6$ & $0.7$ & $0.8$ & $0.9$ & $1.0$ \\ \hline 
               \multirowcell{4}{$\text{Gamma}(4,2)-2$}    & $\phi_1$ & $0.030$ & $0.033$ & $0.049$ &  $0.052$ & $0.057$ & $0.108$ & $0.140$ & $0.212$ & $0.282$ & $0.366$\\
                      & $\phi_2$ &  $0.057$ & $0.074$ & $0.139$ & $0.173$ & $0.354$ & $0.560$ & $0.803$ & $0.932$ & $0.993$ & $\bf 1.000$\\
      & $\phi_3$ & $0.050$ &$\bf 0.078$ & $\bf 0.149$ & $\bf 0.186$ & $\bf 0.376$ & $\bf 0.580$ & $\bf 0.839$ & $\bf 0.952$ & $\bf 0.996$ & $\bf 1.000$\\
         & \texttt{PTCLT} & $\bf 0.070$ & $0.040$ & $0.071$ & $0.053$ & $0.056$ & $0.054$ & $0.059$ & $0.065$ & $0.059$ & $0.056$\\[-1em] \\
          \multirowcell{4}{$t_6$}    & $\phi_1$ & $0.100$ & $0.089$ & $0.109$ & $0.120$ & $0.134$ & $0.189$ & $0.229$ & $0.301$ & $0.402$ & $0.479$\\
                  & $\phi_2$ & $0.244$ & $0.282$ & $0.376$ & $0.502$ & $0.680$ & $0.835$ & $0.932$ & $0.991$ & $0.997$ & $\bf 1.000$\\
      & $\phi_3$ & $\bf 0.287$ & $\bf 0.320$ & $\bf 0.412$ & $\bf 0.559$ & $\bf 0.729$ & $\bf 0.872$ & $\bf 0.944$ & $\bf 0.996$ & $\bf 0.999$ & $\bf 1.000$\\
         & \texttt{PTCLT} & $0.093$ & $0.080$ & $0.085$ & $0.086$ & $0.089$ & $0.099$ & $0.076$ & $0.110$ & $0.104$ & $0.095$\\[-1em] \\
          \multirowcell{4}{$t_{10}$}    & $\phi_1$ & 0.038 & 0.048 & 0.071 & 0.074 & 0.073 & 0.130 & 0.179 & 0.222 & 0.313 & 0.399\\
                  & $\phi_2$ & 0.048 & \bf 0.079 & \bf 0.124 &\bf  0.181 & 0.353 & 0.515 & 0.782 & 0.913 & 0.984 & \bf 1.000\\
      & $\phi_3$ & 0.045 & 0.075 & 0.106 & 0.178 &\bf  0.366 & \bf 0.525 & \bf 0.801 & \bf 0.930 &\bf 0.991 & 0.999\\
         & \texttt{PTCLT} & \bf 0.067 & 0.071 & 0.072 & 0.064 & 0.058 & 0.063 &  0.060 & 0.062 & 0.063 & 0.046\\ \hline 
    \end{tabular}
    \caption{The empirical power of $\phi_1$, $\phi_2$, $\phi_3$, \texttt{LRT}, and \texttt{PTCLT} when the population core covariance matrix is $C_{2,w}$ and $\calP_0$ is one of $\text{Gamma}(4,2)-2$, $t_6$, and $t_{10}$, across $w\in\set{0.1,\ldots,1}$ with $(p_1.p_2,n)=(32,32,256)$.}
    \label{secC.3.table4}
\end{table}

\subsection{Empirical Power of Non-Zero Mean Gaussian Populations}\label{secC.4}

We evaluate the empirical power of $\phi_1,\phi_2,\phi_3$, \texttt{LRT}, and \texttt{PTCLT} under non-zero mean Gaussian populations. We replicate the simulation studies for Figures \ref{sec5.2.figure3}--\ref{sec5.2.figure4} but with centering. Recall that the test statistic of each test is based on the sample core $\hat{C}=c(S)$ for the sample covariance matrix $S=1/n\sum_{i=1}^n y_iy_i^\top$, except for \texttt{PTCLT}. Here $y_i=\text{vec}(Y_i)$ for random matrices $Y_i\in \real^{p_1\times p_2}$. The sample covariance matrix was computed as above because we assumed $\bbE[Y_1]=0$. Now we assume the unknown mean and thus replace $S$ with its centered version as $S=1/n\sum_{i=1}^n (y_i-\bar{y})(y_i-\bar{y})^\top$ for $\bar{y}=1/n\sum_{i=1}^n y_i$. Then we do the same simulation studies for Figures \ref{sec5.2.figure3}--\ref{sec5.2.figure4}. We refer the reader to Section \ref{sec5.2:emp.power} for details. From Figures \ref{secC.4.figure1}--\ref{secC.4.figure2}, one can see that the tendency observed from Figures \ref{sec5.2.figure3}--\ref{sec5.2.figure4} does not alter after centering. 

\begin{figure}[!ht]
    \centering
    \includegraphics{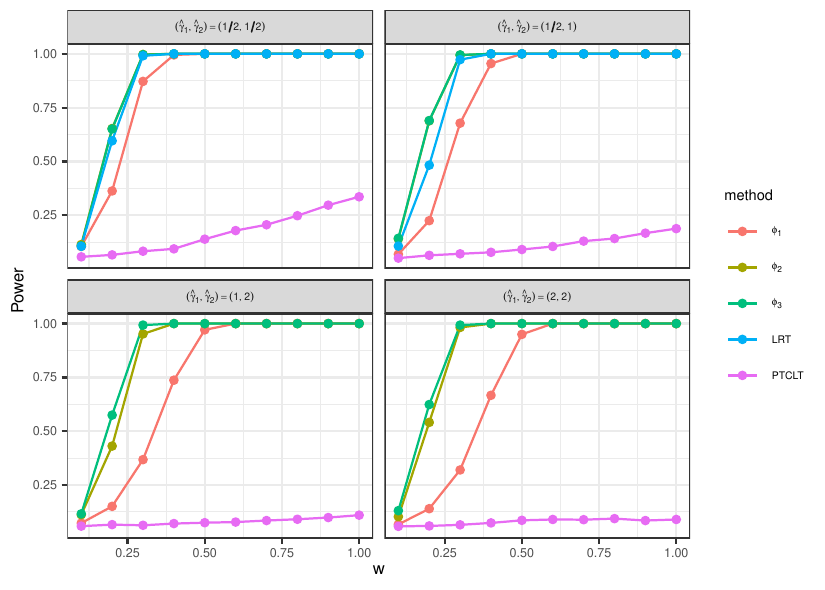}
    \caption{Empirical power of the separability tests under $N_{p_1\times p_2}(0,C_{1,w})$ as functions of $w$ across $(\hat{\gamma}_1,\hat{\gamma}_2)$. Here $n=256$ and $(p_1,p_2)$ is determined according to the value of $(\hat{\gamma}_1,\hat{\gamma}_2)$ and $C_{1,w}$ is defined in Section \ref{sec5.2:emp.power}. The test statistics of  $\phi_1,\phi_2,\phi_3$ and \texttt{LRT} are computed based on the core of the centered sample covariance matrix.}
    \label{secC.4.figure1}
\end{figure}

\begin{figure}[!ht]
    \centering
    \includegraphics{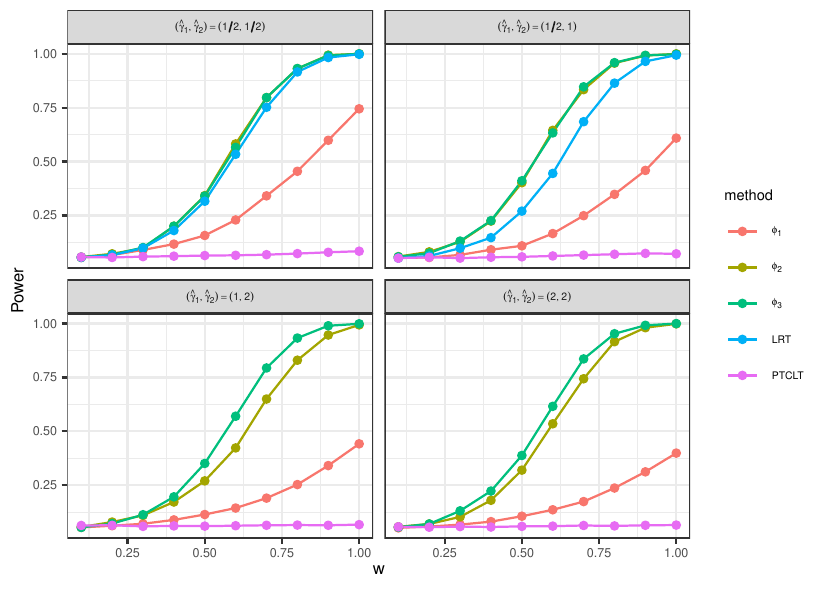}
    \caption{Empirical power of the separability tests under $N_{p_1\times p_2}(0,C_{2,w})$ as functions of $w$ across $(\hat{\gamma}_1,\hat{\gamma}_2)$. Here $n=256$ and $(p_1,p_2)$ is determined according to the value of $(\hat{\gamma}_1,\hat{\gamma}_2)$ and $C_{2,w}$ is defined in Section \ref{sec5.2:emp.power}. The test statistics of  $\phi_1,\phi_2,\phi_3$ and \texttt{LRT} are computed based on the core of the centered sample covariance matrix.}
    \label{secC.4.figure2}
\end{figure}

\section{Critical Values for $\phi_1,\phi_2,$ and $\phi_3$}\label{sec.D:cutoff}

We provide the approximated $q_{1-\alpha}$ values, denoted by $\hat{q}_{1-\alpha}$, for $\phi_1,\phi_2$, and $\phi_3$ at level $\alpha=0.05$ under Gaussian populations. These values are based on $1000$ Monte Carlo simulations for various values of $(p_1,p_2,n)$. To account for the uncertainty in the sampling distribution, we apply transformations to the test statistics $T_1$, $T_2$, and $T_3$ as follows: 
\begin{align*}
    \tilde{T}_1(Y)&=\gamma_0 n^{2/3}\parentheses{T_1(Y)-E_+},\\ 
    \tilde{T}_2(Y)&=T_2(Y)+pF^{\gamma_n}\parentheses{\log(x+1)} -p\log p,\\
    \tilde{T}_3(Y)&=nT_3(Y)-p-1.
\end{align*}
Here $\gamma_0$ and $E_+$ are those in Section \ref{sec4.3:asymp.null}. Also, $\gamma_n=p/n$ and for $a_1(y)=(y+2-\sqrt{y^2+4})/(2\sqrt{y})$, 
\begin{align*}
    F^{y}\parentheses{\log(x+1)}&=-\frac{1}{2}\left[\frac{2a_1(y)}{\sqrt{y}}-\frac{y+1}{y}\log\parentheses{\frac{\sqrt{y}}{a_1(y)}}\right]\\
    &-\frac{1}{2}\left[\frac{1-y}{y}\log(1-a_1(y)\sqrt{y})+\frac{1-y}{y}\log\parentheses{\frac{\sqrt{y}}{a_1(y)}-y}\right]\mathbf{1}_{y\in(0,1)}\\
    &-\frac{1}{2}\left[\frac{y-1}{y}\log\parentheses{1-\frac{a_1(y)}{\sqrt{y}}}+\frac{y-1}{y}\log\parentheses{\frac{\sqrt{y}}{a_1(y)}-1}\right]\mathbf{1}_{y\in[1,\infty)}.
\end{align*}
Recall that the asymptotic null distribution of a version of $\tilde{T}_1(Y)$ was studied in Section  \ref{sec4.3:asymp.null}. The transformation of $T_2$ as in $\tilde{T}_2(Y)$ follows from Theorem $1$--$2$ of \cite{wang2021}. Also, that of $T_3$ as in $\tilde{T}_3(Y)$ follows the transformation of the test statistic of John's sphericity test, admitting the closed form of the asymptotic distribution. The values of $\hat{q}_{1-\alpha}$ for $\tilde{T}_1$, $\tilde{T}_2$, and $\tilde{T}_3$, instead of $T_1,T_2$, and $T_3$, respectively, are reported in Table \ref{secE.table1}.  
\begin{table}[!ht]
    \centering
    \begin{tabular}{cccc}
    \hline 
    $(p_1,p_2,n)$ & $\tilde{T}_1$ & $\tilde{T}_2$ & $\tilde{T}_3$ \\ \hline 
     $(6,6,144)$    & $-0.130$ & $0.060$ & $0.726$\\ 
      $(8,8,256)$ & $-0.075$ & $0.064$ & $1.017$\\
        $(10,10,400)$ & $0.007$ & $0.060$ & $0.810$ \\
        $(12,12,576)$ & $0.093$ & $0.056$ & $1.000$ \\
         $(14,14,784)$ & $0.338$ & $0.062$ & $1.187$ \\
          $(16,16,1024)$ & $0.307$ & $0.067$ & $1.246$\\
           $(6,12,144)$    & $-0.262$ & $0.083$  & $0.603$ \\ 
      $(8,16,256)$ & $0.171$ & $0.089$ & $0.731$\\
        $(10,20,400)$ & $0.154$ & $0.082$ & $0.623$ \\
        $(12,24,576)$ & $0.279$ & $0.091$ & $0.595$ \\
         $(14,28,784)$ & $0.215$ & $0.094$ & $0.624$ \\
         $(16,32,1024)$ & $0.332$ & $0.093$ & $0.728$ \\
           $(12,24,144)$    & $-0.02$ & $0.175$ & $0.416$ \\ 
      $(16,32,256)$ & $0.221$ & $0.167$ & $0.594$ \\
        $(20,40,400)$ & $0.382$ & $0.173$ & $0.780$ \\
        $(24,48,576)$ & $0.349$ & $0.163$ & $0.697$ \\
         $(28,56,784)$ & $0.347$ & $0.183$ & $0.730$ \\
          $(32,64,1024)$ & $0.598$ & $0.185$ & $0.780$ \\
           $(24,24,144)$    & $0.028$ & $0.180$ & $1.033$ \\ 
      $(32,32,256)$ & $0.414$ & $0.182$ & $1.006$ \\
        $(40,40,400)$ & $0.339$ & $0.183$ & $1.256$ \\
        $(48,48,576)$ & $0.690$ & $0.174$ & $1.151$ \\
         $(56,56,784)$ & $0.671$ & $0.185$ & $1.169$ \\
          $(64,64,1024)$ & $0.631$ & $0.172$ & $1.341$ \\ \hline 
    \end{tabular}
    \caption{ The values of $\hat{q}_{1-\alpha}$ for $\tilde{T}_1$, $\tilde{T}_2$, and $\tilde{T}_3$ based on $1000$ Monte Carlo simulations at level $\alpha=0.05$.}
    \label{secE.table1}
\end{table}

\end{document}